\documentclass[12pt]{article}
\pdfoutput=1

\usepackage[utf8]{inputenc}

\usepackage{geometry,enumerate,comment,color}
\usepackage{amsmath,amsthm,amssymb,mathrsfs,bbm,nicefrac,stmaryrd}

\newcommand{\N}{\ensuremath{\mathbb N}}
\newcommand{\Z}{\ensuremath{\mathbb Z}}

\newcommand{\R}{\ensuremath{\mathbb R}}

\renewcommand{\P}{\ensuremath{\mathbb P}}
\newcommand{\E}{\ensuremath{\mathbb E}}
\newcommand{\ud}{\,\mathrm{d}}

\newcommand{\Cost}{\mathrm{Cost}}
\newcommand{\ltriplevert}{\lvert\kern-0.25ex\lvert\kern-0.25ex\lvert}
\newcommand{\rtriplevert}{\rvert\kern-0.25ex\rvert\kern-0.25ex\rvert}

\DeclareMathOperator{\smallsumtext}{\textstyle\sum}
\DeclareMathOperator*{\smallsum}{\textstyle\sum}
\DeclareMathOperator*{\smallprod}{\textstyle\prod}

\newtheorem{theorem}{Theorem}[section]

\newtheorem{proposition}[theorem]{Proposition}
\newtheorem{corollary}[theorem]{Corollary}
\newtheorem{lemma}[theorem]{Lemma}

\newtheorem{setting}{Setting}[section]

\hypersetup{colorlinks=true}
\newcommand\yesnumber{\refstepcounter{equation}\tag{\theequation}}

\title{Generalised multilevel Picard approximations}

\author{Michael B.\ Giles$^1$, Arnulf Jentzen$^{2,3}$, and Timo Welti$^{4}$\bigskip\\
\small{$^1$ Mathematical Institute, University of Oxford, Oxford,}\\
\small{United Kingdom, e-mail: {\tt Mike.Giles@maths.ox.ac.uk}}\\
\small{$^2$ SAM, Department of Mathematics, ETH Z\"urich, Z\"urich,}\\
\small{Switzerland, e-mail: {\tt arnulf.jentzen@sam.math.ethz.ch}}\\
\small{$^3$ Faculty of Mathematics and Computer Science, University of M\"unster,}\\
\small{M\"unster, Germany, e-mail: {\tt ajentzen@uni-muenster.de}}\\
\small{$^4$ SAM, Department of Mathematics, ETH Z\"urich, Z\"urich,}\\
\small{Switzerland, e-mail: {\tt timo.welti@sam.math.ethz.ch}}
}

\usepackage[backend=biber,style=numeric-comp,giveninits=true,sorting=nyt,maxbibnames=99,useprefix=true]{biblatex}
\DeclareNameAlias{default}{last-first}
\renewbibmacro{in:}{}
\AtBeginBibliography{%

}
\DefineBibliographyStrings{english}{%
	andothers = {\addcomma\addspace\textsc{et\addabbrvspace al}\adddot},
	and = {\textsc{and}}
}

\DeclareFieldFormat
[article,inbook,incollection,inproceedings,patent,thesis,unpublished]
{title}{#1}
\renewbibmacro{in:}{%
	\ifentrytype{article}{%
	}{%
	\printtext{\bibstring{in}\intitlepunct}%
}%
}
\renewbibmacro*{volume+number+eid}{%
	\printfield{volume}%
	\setunit*{\addcomma\space}%
	\printfield{number}%
	\setunit{\addcomma\space}%
	\printfield{eid}}
\DeclareFieldFormat{pages}{#1}
\renewbibmacro*{publisher+location+date}{%
	\printlist{publisher}%
	\setunit*{\addcomma\space}%
	\printlist{location}%
	\setunit*{\addcomma\space}%
	\usebibmacro{date}%
	\newunit}

\begin{document}

\maketitle

\begin{abstract}
It is one of the most challenging problems in applied mathematics
to approximatively solve high-dimensional partial differential equations (PDEs).
In particular,
most of the numerical approximation schemes studied in the scientific literature
suffer under the curse of dimensionality
in the sense that
the number of computational operations needed to compute
an approximation with an error of size at most $ \varepsilon > 0 $
grows at least exponentially in the PDE dimension $ d \in \N $ or in the reciprocal of $ \varepsilon $.
%Therefore,
%a fundamental goal of current research activities
%is to propose and analyse numerical methods with the power
%to beat the curse of dimensionality
%in such way that
%the number of computational operations needed to compute
%an approximation with an error of size at most $ \varepsilon > 0 $
%grows at most polynomially in both the PDE dimension $ d \in \N $ and the reciprocal of $ \varepsilon $.
Recently,
so-called
full-history recursive multilevel Picard (MLP) approximation methods
have been introduced
to tackle the problem of approximately solving high-dimensional PDEs.
MLP approximation methods currently are, to the best of our knowledge,
the only methods
for parabolic semi-linear PDEs with general time horizons and general initial conditions
for which there is a rigorous proof
that they are indeed able to beat the curse of dimensionality.
The main purpose of this work
is to investigate MLP approximation methods in more depth,
to reveal more clearly how these methods can overcome the curse of dimensionality,
and
to propose a generalised class of MLP approximation schemes,
which covers previously analysed MLP approximation schemes as special cases.
In particular,
we develop an abstract framework in which
this class of generalised MLP approximations can be formulated and analysed
and,
thereafter, apply this abstract framework to derive
a computational complexity result
for suitable MLP approximations for semi-linear heat equations.
These resulting MLP approximations for semi-linear heat equations
essentially are generalisations of
previously introduced MLP approximations for semi-linear heat equations.
\end{abstract}

\begin{center}
MLP,
curse of dimensionality,
semi-linear partial differential\\ equations,
PDEs,
semi-linear heat equations
\end{center}

\tableofcontents

\section{Introduction}

It is one of the most challenging problems in applied mathematics
to approximatively solve high-dimensional partial differential equations (PDEs).
In particular,
most of the numerical approximation schemes studied in the scientific literature,
such as finite differences, finite elements, and sparse grids,
suffer under the curse of dimensionality
(cf.\ Bellman~\cite{Bellman1957})
in the sense that
the number of computational operations needed to compute
an approximation with an error of size at most $ \varepsilon > 0 $
grows at least exponentially in the PDE dimension $ d \in \N = \{ 1, 2, 3, \ldots \} $ or in the reciprocal of $ \varepsilon $.
Computing such an approximation with reasonably small error thus becomes unfeasible in dimension greater than, say, ten.
Therefore,
a fundamental goal of current research activities
is to propose and analyse numerical methods with the power
to beat the curse of dimensionality
in such way that
the number of computational operations needed to compute
an approximation with an error of size at most $ \varepsilon > 0 $
grows at most polynomially in both the PDE dimension $ d \in \N $ and the reciprocal of $ \varepsilon $
(cf., e.g., Novak \& Wo\'{z}niakowski~\cite[Chapter~1]{NovakWozniakowski2008} and \cite[Chapter~9]{NovakWozniakowski2010}).
In the recent years
a number of numerical schemes have been proposed
to tackle the problem of approximately solving high-dimensional PDEs,
which include
deep learning based approximation methods
(cf., e.g.,~\cite{HanJentzenE2018,
EHanJentzen2017,
BeckBeckerCheriditoJentzenNeufeld2019arXiv,
BeckBeckerGrohsJaafariJentzen2018arXiv,
BeckEJentzen2019,
BeckerCheriditoJentzen2019,
BeckerCheriditoJentzenWelti2019arXiv,
BergNystroem2018,
Chan-Wai-NamMikaelWarin2019,
Dockhorn2019arXiv,
EYu2018,
FarahmandNabiNikovski2017,
FujiiTakahashiTakahashi2019,
GoudenegeMolentZanette2019arXiv,
HanLong2018arXiv,
Henry-Labordere2017SSRN,
HurePhamWarin2019arXiv,
JacquierOumgari2019arXiv,
LongLuMaDong2018,
LyeMishraRay2019arXiv,
MagillQureshiDeHaan2018,
PhamWarin2019arXiv,
Raissi2018,
SirignanoSpiliopoulos2018,
ChenWan2019arXiv}
and the references mentioned therein),
branching diffusion approximation methods
(cf.~\cite{AgarwalClaisse2017arXiv,
BelakHoffmannSeifried2019SSRN,
BernalDosReisSmith2017,
BouchardTanWarin2019,
BouchardTanWarinZou2017,
ChangLiuXiong2016,
Henry-Labordere2012,
Henry-LabordereOudjaneTanTouziWarin2019,
Henry-LabordereTanTouzi2014,
Henry-LabordereTouzi2018arXiv,
McKean1975,
RasulovRaimovaMascagni2010,
Skorohod1964,
Warin2017arXiv,
Watanabe1965}),
approximation methods based on discretising a corresponding backward stochastic differential equation
using iterative regression on function Hamel bases
(cf., e.g.,~\cite{BallyPages2003,
BenderDenk2007,
BenderSchweizerZhuo2017,
BouchardTouzi2004,
Chassagneux2014,
ChassagneuxCrisan2014,
ChassagneuxRichou2015,
ChassagneuxRichou2016,
CrisanManolarakis2010,
CrisanManolarakis2012,
CrisanManolarakis2014,
CrisanManolarakisTouzi2010,
DelarueMenozzi2006,
DouglasMaProtter1996,
GobetLabart2010,
GobetLemor2008arXiv,
GobetLemorWarin2005,
GobetLopez-SalasTurkedjiev2016,
GobetTurkedjiev2016a,
GobetTurkedjiev2016b,
HuijskensRuijterOosterlee2016,
LabartLelong2013,
LemorGobetWarin2006,
LionnetDosReisSzpruch2015,
MaProtterSanMartinTorres2002,
MaProtterYong1994,
MaYong1999,
MilsteinTretyakov2006,
MilsteinTretyakov2007,
Pham2015,
RuijterOosterlee2015,
RuijterOosterlee2016,
RuszczynskiYao2017arXiv,
Turkedjiev2015,
Zhang2004}
and the references mentioned therein)
or using
Wiener chaos expansions
(cf.\ Briand \& Labart~\cite{BriandLabart2014}
and
Geiss \& Labart~\cite{GeissLabart2016}),
and full-history recursive multilevel Picard (MLP) approximation methods
(cf.~\cite{EHutzenthalerJentzenKruse2016arXiv,EHutzenthalerJentzenKruse2019,HutzenthalerJentzenKruseNguyenVonWurstemberger2018arXivv2,HutzenthalerKruse2017arXiv,HutzenthalerJentzenVonWurstemberger2019arXiv,BeckHornungHutzenthalerJentzenKruse2019arXiv}).
So far,
deep learning based approximation methods for PDEs
seem to work very well
in the case of high-dimensional PDEs
judging from a large number of numerical experiments.
However,
there exist only partial results
(cf.~\cite{BernerGrohsJentzen2018arXiv,
ElbraechterGrohsJentzenSchwab2018arXiv,
GrohsHornungJentzenVonWurstemberger2018arXiv,
GrohsHornungJentzenZimmermann2019arXiv,
GrohsJentzenSalimova2019arXiv,
HutzenthalerJentzenKruseNguyen2019arXiv,
JentzenSalimovaWelti2018arXiv,
KutyniokPetersenRaslanSchneider2019arXiv,
ReisingerZhang2019arXiv})
and no full error analysis in the scientific literature
rigorously justifying their effectiveness in the numerical approximation of high-dimensional PDEs.
Moreover,
while for branching diffusion methods
there is a full error analysis available
proving that the curse of dimensionality can be beaten
for instances of PDEs
with small time horizon and small initial condition, respectively,
numerical simulations suggest
that such methods fail to work
if the time horizon or the initial condition, respectively,
are not small anymore.
To sum it up,
MLP approximation methods currently are, to the best of our knowledge,
the only methods
for parabolic semi-linear PDEs with general time horizons and general initial conditions
for which there is a rigorous proof
that they are indeed able to beat the curse of dimensionality.

The main purpose of this work
is to investigate MLP methods in more depth,
to reveal more clearly how these methods can overcome the curse of dimensionality,
and to generalise the MLP scheme
proposed in Hutzenthaler et al.~\cite{HutzenthalerJentzenKruseNguyenVonWurstemberger2018arXivv2}.
In particular,
in the main result of this article,
Theorem~\ref{thm:complexity_analysis} in Subsection~\ref{sec:complexity_analysis} below,
we develop an abstract framework in which
suitably generalised MLP approximations can be formulated
(cf.~\eqref{eq:scheme_intro} in Theorem~\ref{thm:complexity_analysis_intro} below)
and analysed
(cf.~(\ref{item:thm:complexity_analysis_intro1})--(\ref{item:thm:complexity_analysis_intro3}) in Theorem~\ref{thm:complexity_analysis_intro} below)
and,
thereafter, apply this abstract framework to derive
a computational complexity result
for suitable MLP approximations for semi-linear heat equations
(cf.\ Corollary~\ref{cor:MLP_heat_intro} below).
These resulting MLP approximations for semi-linear heat equations
essentially are generalisations of
the MLP approximations introduced in~\cite{HutzenthalerJentzenKruseNguyenVonWurstemberger2018arXivv2}.
To make the reader more familiar with the contributions of this article,
we now illustrate in Theorem~\ref{thm:complexity_analysis_intro} below
the findings of the main result of this article,
Theorem~\ref{thm:complexity_analysis} in Subsection~\ref{sec:complexity_analysis} below,
in a simplified situation.

\begin{theorem}
\label{thm:complexity_analysis_intro}
Let
$ ( \Omega, \mathscr{F}, \P ) $
be a probability space,
let $ ( \mathcal{Y}, \lVert \cdot \rVert_{ \mathcal{Y} } ) $
be a separable $ \R $-Banach space,
let
$ \mathfrak{z}, \mathfrak{B}, \kappa, C, c \in [ 1, \infty ) $,
$ \Theta = \cup_{ n = 1 }^\infty \Z^n $,
$ ( M_j )_{ j \in \N } \subseteq \N $,
$ y \in \mathcal{Y} $
satisfy
$ \liminf_{ j \to \infty } M_j = \infty $,
$ \sup_{ j \in \N }
\nicefrac{ M_{ j + 1 } }{ M_j } \leq \mathfrak{B} $,
and
$ \sup_{ j \in \N }
\nicefrac{ M_j }{ j } \leq \kappa $,
let
$ ( \mathcal{Z}, \mathscr{Z} ) $ be a measurable space,
let
$ Z^{ \theta } \colon \Omega \to \mathcal{Z} $,
$ \theta \in \Theta $,
be i.i.d.\
$ \mathscr{F} $/$ \mathscr{Z} $-measurable functions,
let $ ( \mathcal{H}, \langle \cdot, \cdot \rangle_{ \mathcal{H} }, \lVert \cdot \rVert_{ \mathcal{H} } ) $
be a separable $ \R $-Hilbert space,
let
$ \mathscr{S} $
be the strong $ \sigma $-algebra on $ L( \mathcal{Y}, \mathcal{H} ) $,
let
$ \psi_k \colon \Omega \to L( \mathcal{Y}, \mathcal{H} ) $, $ k \in \N_0 = \{ 0, 1, 2, \ldots \} $,
be
$ \mathscr{F} $/$ \mathscr{S} $-measurable functions,
assume that
$ ( Z^{ \theta } )_{ \theta \in \Theta } $
and
$ ( \psi_k )_{ k \in \N_0 } $
are independent,
let
$ \Phi_l \colon \mathcal{Y} \times \mathcal{Y} \times \mathcal{Z} \to \mathcal{Y} $, $ l \in \N_0 $,
be
$ ( \mathscr{B}( \mathcal{Y} ) \otimes \mathscr{B}( \mathcal{Y} ) \otimes \mathscr{Z} ) $/$ \mathscr{B}( \mathcal{Y} ) $-measurable
functions,
let
$ Y_{ n, j }^\theta \colon \Omega \to \mathcal{Y} $,
$ \theta \in \Theta $,
$ j \in \N $,
$ n \in ( \N_0 \cup \{ -1 \} ) $,
satisfy
for all
$ n, j \in \N $,
$ \theta \in \Theta $
that
$ Y_{ -1, j }^\theta
= Y_{ 0, j }^\theta = 0 $
and
\begin{equation}
\label{eq:scheme_intro}
Y_{ n, j }^\theta
=
\smallsum_{l=0}^{n-1}
\tfrac{1}{ ( M_j )^{ n - l } }
\biggl[
    \smallsum_{i=1}^{ ( M_j )^{ n - l } }
        \Phi_{ l } \bigl(
        Y^{ ( \theta, l, i ) }_{ l, j },
        Y^{ ( \theta, -l, i ) }_{ l - 1, j },
        Z^{ ( \theta, l, i ) }
        \bigr)
\biggr]
,
\end{equation}
let
$ ( \Cost_{ n, j } )_{ ( n, j ) \in ( \N_0 \cup \{ -1 \} ) \times \N }
\subseteq [ 0, \infty ) $
satisfy
for all $ n, j \in \N $
that
$ \Cost_{ -1, j } = \Cost_{ 0, j } = 0 $
and
\begin{equation}
\label{eq:intro_cost}
\Cost_{ n, j }
\leq
( M_j )^{ n }
\mathfrak{z}
+
\smallsum\nolimits_{ l = 0 }^{ n - 1 }
    \bigl[
    ( M_j )^{ n - l }
    ( \Cost_{ l, j } + \Cost_{ l - 1, j } + 2 \mathfrak{z} )
    \bigr]
,
\end{equation}
and
assume
for all
$ k \in \N_0 $,
$ n, j \in \N $,
$ u, v \in \mathcal{Y} $
that
$ \E\bigl[
    \lVert 
        \Phi_{ k }(
        Y^{ 0 }_{ k, j },
        Y^{ 1 }_{ k - 1, j },
        Z^{ 0 }
        )
    \rVert_{ \mathcal{Y} }
\bigr]
< \infty $
and
\begin{gather}
\label{eq:intro_hypothesis_I}
\max\bigl\{
\E\bigl[
\lVert \psi_k(
    \Phi_{ 0 }(
    0,
    0,
    Z^{ 0 }
    )
) \rVert_{ \mathcal{H} }^2
\bigr]
,
\E\bigl[
\lVert \psi_k(
    y
) \rVert_{ \mathcal{H} }^2
\bigr]
\bigr\}
\leq
\tfrac{ C^2 }{ k! }
,
\\[0.5\baselineskip]
\label{eq:intro_hypothesis_II}
\E\bigl[
\lVert \psi_k(
    \Phi_{ n }( u, v, Z^{ 0 } )
) \rVert_{ \mathcal{H} }^2
\bigr]
\leq
c
\, \E\bigl[
\lVert \psi_{ k + 1 }(
    u - v
) \rVert_{ \mathcal{H} }^2
\bigr]
,
\\[0.5\baselineskip]
\label{eq:intro_hypothesis_III}
\E\Bigl[
\bigl\lVert \psi_{ k }
\bigl(
    y -
    \smallsum\nolimits_{l=0}^{n-1}
        \E\bigl[
            \Phi_{ l }\bigl(
            Y^0_{ l, j },
            Y^1_{ l - 1, j },
            Z^0
            \bigr)
        \bigr]
\bigr)
\bigr\rVert_{ \mathcal{H} }^2
\Bigr]
\leq
c
\, \E\Bigl[
\bigl\lVert \psi_{ k + 1 }\bigl(
    Y_{ n - 1, j }^0 - y
\bigr) \bigr\rVert_{ \mathcal{H} }^2
\Bigr]
.
\end{gather}
Then
\begin{enumerate}[(i)]
\item
\label{item:thm:complexity_analysis_intro1}
it holds
for all
$ n \in \N $
that
$ \bigl(
\E\bigl[
    \lVert \psi_0(
        Y_{ n, n }^0 - y
    ) \rVert_{ \mathcal{H} }^2
\bigr]
\bigr)^{ \nicefrac{1}{2} }
\leq
C
\bigl[
    \frac{ 8 c e^\kappa }{ M_n }
\bigr]^{ \nicefrac{n}{2} }
< \infty $,
\item
\label{item:thm:complexity_analysis_intro2}
it holds
for all
$ n \in \N $
that
$ \Cost_{ n, n }
\leq
( 5 M_n )^{ n }
\mathfrak{z}
$,
and
\item
\label{item:thm:complexity_analysis_intro3}
there exists
$ ( N_\varepsilon )_{ \varepsilon \in ( 0, 1 ] } \subseteq \N $
such that
it holds
for all
$ \varepsilon \in ( 0, 1 ] $,
$ \delta \in ( 0, \infty ) $
that
$ \sup_{ n \in \{ N_\varepsilon, N_\varepsilon + 1, \ldots \} }
\bigl(
\E\bigl[
    \lVert \psi_0(
        Y_{ n, n }^0 - y
    ) \rVert_{ \mathcal{H} }^2
\bigr]
\bigr)^{ \nicefrac{1}{2} }
\leq
\varepsilon $
and
\begin{equation}
\Cost_{ N_\varepsilon, N_\varepsilon }
\leq
5
\mathfrak{z}
e^\kappa
C^{ 2 ( 1 + \delta ) }
\bigl(
    1
    +
    \sup\nolimits_{ n \in \N }
    \bigl[
    ( M_{ n } )^{ -\delta } ( 40 c e^{ 2 \kappa } \mathfrak{B} )^{ ( 1 + \delta ) }
    \bigr]^n
\bigr)
\varepsilon^{ -2 ( 1 + \delta ) }
<
\infty
.
\end{equation}
\end{enumerate}
\end{theorem}

Theorem~\ref{thm:complexity_analysis_intro}
follows directly from the more general result in
Corollary~\ref{cor:complexity_analysis} in Subsection~\ref{sec:complexity_analysis} below,
which, in turn, is a consequence of the main result of this work,
Theorem~\ref{thm:complexity_analysis}
in Subsection~\ref{sec:complexity_analysis} below.
In the following we provide some intuitions and further explanations
for Theorem~\ref{thm:complexity_analysis_intro}
and illustrate how it is applied
%to show that MLP approximations are overcoming the curse of dimensionality
in the context of numerically approximating semi-linear heat equations
(cf.\ Corollary~\ref{cor:MLP_heat_intro} below
and Setting~\ref{setting:MLP_heat} in Section~\ref{sec:MLP_heat} below).
Theorem~\ref{thm:complexity_analysis_intro}
establishes
an upper error bound
(cf.~(\ref{item:thm:complexity_analysis_intro1}) in Theorem~\ref{thm:complexity_analysis_intro})
and an upper cost bound
(cf.~(\ref{item:thm:complexity_analysis_intro2}) in Theorem~\ref{thm:complexity_analysis_intro})
for the generalised MLP approximations in~\eqref{eq:scheme_intro}
as well as
an abstract complexity result
(cf.~(\ref{item:thm:complexity_analysis_intro3}) in Theorem~\ref{thm:complexity_analysis_intro}),
which states that
for an approximation accuracy $ \varepsilon $
in a suitable root mean square sense
the computational cost is essentially of order $ \varepsilon^{ -2 } $.
The separable $ \R $-Banach space
$ ( \mathcal{Y}, \lVert \cdot \rVert_{ \mathcal{Y} } ) $
is a set
which the exact solution $ y $ is an element of
and where
the generalised MLP approximations
$ Y_{ n, j }^\theta \colon \Omega \to \mathcal{Y} $,
$ \theta \in \Theta $,
$ j \in \N $,
$ n \in ( \N_0 \cup \{ -1 \} ) $,
which are random variables approximating $ y \in \mathcal{Y} $ in an appropriate sense,
take values in.
When $ y \in \mathcal{Y} $ is the solution of a suitable semi-linear heat equation
(cf.~\eqref{eq:intro_heat_equation} below),
elements of $ \mathcal{Y} $
are at most polynomially growing functions
in $ C( [ 0, T ] \times \R^d, \R ) $,
where $ T \in ( 0, \infty ) $, $ d \in \N $
(cf.~\eqref{eq:def_Y} in Subsection~\ref{sec:setting_MLP_heat} below).
The randomness of
the generalised MLP approximations
$ Y_{ n, j }^\theta $, $ \theta \in \Theta $, $ j \in \N $, $ n \in ( \N_0 \cup \{ -1 \} ) $,
stems from the i.i.d.\ random variables
$ Z^{ \theta } \colon \Omega \to \mathcal{Z} $,
$ \theta \in \Theta $,
taking values in a measurable space
$ ( \mathcal{Z}, \mathscr{Z} ) $,
which
in our example about semi-linear heat equations
correspond to
standard Brownian motions
and
on $ [ 0, 1 ] $ uniformly distributed random variables
(cf.~\eqref{eq:def_Y_heat_intro} below).
Observe that the generalised MLP approximations in~\eqref{eq:scheme_intro}
are full-history recursive
since each iterate depends on all previous iterates.
Together with the random variables
$ Z^{ \theta } $, $ \theta \in \Theta $,
the previous iterates enter
through the functions
$ \Phi_l \colon \mathcal{Y} \times \mathcal{Y} \times \mathcal{Z} \to \mathcal{Y} $, $ l \in \N_0 $,
which thus govern the dynamics of the generalised MLP approximations.
%Moreover,
This recursive dependence,
the consequential nesting of the generalised MLP approximations,
and the Monte Carlo sums in~\eqref{eq:scheme_intro}
necessitate 
a large number of i.i.d.\ samples
indexed by
$ \theta \in \Theta = \cup_{ n = 1 }^\infty \Z^n $
in order to formulate the generalised MLP approximations.
In connection with this note that it holds
for every
$ n \in ( \N_0 \cup \{ - 1 \} ) $,
$ j \in \N $
that
$ Y_{ n, j }^\theta $, $ \theta \in \Theta $,
are identically distributed
(cf.~(\ref{item:independence5}) in Proposition~\ref{prop:independence} below).
On the other hand,
the parameter
$ j \in \N $
of the generalised MLP approximations
$ Y_{ n, j }^\theta $, $ \theta \in \Theta $, $ j \in \N $, $ n \in ( \N_0 \cup \{ -1 \} ) $,
specifies the respective element of
the sequence of Monte Carlo numbers
$ ( M_j )_{ j \in \N } \subseteq \N $
(which are assumed to grow to infinity
%$ \liminf_{ j \to \infty } M_j = \infty $
not faster than linearly)
and thereby determines the numbers of Monte Carlo samples to be used
in~\eqref{eq:scheme_intro}.
Thus for every
$ j \in \N $
%$ \theta \in \Theta $
we can consider the family
$ ( Y_{ n, j }^0 )_{ n \in ( \N_0 \cup \{ - 1 \} ) } $
of generalised MLP approximations
with Monte Carlo sample numbers based on $ M_j $,
of which we pick the $ j $-th element
$ Y_{ j, j }^0 $
to approximate $ y \in \mathcal{Y} $
(cf.~(\ref{item:thm:complexity_analysis_intro3}) in Theorem~\ref{thm:complexity_analysis_intro}).
More precisely,
for every $ n \in \N $
the approximation error for $ Y_{ n, n }^0 $
is measured in the root mean square sense
in a separable $ \R $-Hilbert space
$ ( \mathcal{H}, \langle \cdot, \cdot \rangle_{ \mathcal{H} }, \lVert \cdot \rVert_{ \mathcal{H} } ) $,
after linearly mapping it
from $ \mathcal{Y} $ to $ \mathcal{H} $
using the possibly random function
$ \psi_0 \colon \Omega \to L( \mathcal{Y}, \mathcal{H} ) $
(cf.~(\ref{item:thm:complexity_analysis_intro1})
and~(\ref{item:thm:complexity_analysis_intro3}) in Theorem~\ref{thm:complexity_analysis_intro} above).
In our example about semi-linear heat equations,
$ \mathcal{H} $
is nothing but the set of real numbers $ \R $
and $ \psi_0 $
is the deterministic evaluation of a function in
$ \mathcal{Y} \subseteq C( [ 0, T ] \times \R^d, \R ) $
at a deterministic approximation point in $ [ 0, T ] \times \R^d $
(cf.~\eqref{eq:definition_psi_k} in Subsection~\ref{sec:setting_MLP_heat} below).
Conversely,
the functions
$ \psi_k \colon \Omega \to L( \mathcal{Y}, \mathcal{H} ) $, $ k \in \N $,
correspond in our example to
evaluations at suitable random points in $ [ 0, T ] \times \R^d $
multiplied with random factors that diminish quickly as $ k \in \N $ increases
(cf.~\eqref{eq:definition_psi_k} in Subsection~\ref{sec:setting_MLP_heat} below).
Indeed
assumption~\eqref{eq:intro_hypothesis_I}
essentially demands that
mean square norms of point evaluations of the functions
$ \psi_k $, $ k \in \N_0 $,
diminish at least as fast as
the reciprocal of the factorial of their index.
%$ \nicefrac{1}{k!} $, $ k \in \N_0 $.
Due to this, the functions
$ \psi_k $, $ k \in \N_0 $,
can be thought of encoding magnitude in an appropriate randomised sense.
Assumption~\eqref{eq:intro_hypothesis_II}
hence essentially requires
for every
$ k \in \N_0 $,
$ n \in \N $
that
the $ k $-magnitude of the dynamics function $ \Phi_n $
can be bounded
(up to a constant)
by the $ ( k + 1 ) $-magnitude of the difference of its first two arguments,
while
assumption~\eqref{eq:intro_hypothesis_III},
roughly speaking,
calls for suitable telescopic cancellations
(cf.~\eqref{eq:sum_Phi} in Subsection~\ref{sec:estimates} below)
such that
for every
$ k \in \N_0 $
the $ k $-magnitude of
the probabilistically weak approximation error of a given MLP iterate
(cf.~\eqref{eq:expectation_Y_n}--\eqref{eq:hypothesis_III_consequence} in Subsection~\ref{sec:error_analysis} below)
can be bounded
(up to a constant)
by the $ ( k + 1 ) $-magnitude
of the approximation error of the previous MLP iterate.
Furthermore,
we think of the real number
$ \mathfrak{z} \in [ 1, \infty ) $
as a parameter
associated to the computational cost of one realisation of $ Z^0 $
and
for every
$ n \in ( \N_0 \cup \{ - 1 \} ) $,
$ j \in \N $
we think of the real number
$ \Cost_{ n, j } \in [ 0, \infty ) $
as an upper bound for
the computational cost associated to one realisation of $ \psi_0( Y_{ n, j }^0 ) $
(cf.~\eqref{eq:intro_cost} above).
In our application of the abstract framework outlined above,
we have that
$ \mathfrak{z} $
corresponds to the spacial dimension of the considered semi-linear heat equation
and
we have
for every
$ n \in ( \N_0 \cup \{ - 1 \} ) $,
$ j \in \N $
that the number
$ \Cost_{ n, j } $
corresponds to an upper bound for
the sum of the number of realisations of standard normal random variables
and
the number of realisations of on $ [ 0, 1 ] $ uniformly distributed random variables
used to compute one realisation of
$ \psi_0( Y_{ n, j }^0 ) $
(cf.~\eqref{eq:cost_MLP_heat_fixed} in Subsection~\ref{sec:MLP_heat_fixed} below).

The abstract framework in Theorem~\ref{thm:complexity_analysis_intro}
can be applied to prove
convergence and computational complexity results for
MLP approximations in more specific settings.
We demonstrate this for the example of MLP approximations for semi-linear heat equations.
In particular,
Corollary~\ref{cor:MLP_heat_intro} below
establishes that
the MLP approximations in~\eqref{eq:def_Y_heat_intro},
which essentially are generalisations of the MLP approximations introduced in~\cite{HutzenthalerJentzenKruseNguyenVonWurstemberger2018arXivv2},
approximate solutions of semi-linear heat equations~\eqref{eq:intro_heat_equation}
at the origin without the curse of dimensionality
(cf.\ \cite[Theorem~1.1]{HutzenthalerJentzenKruseNguyenVonWurstemberger2018arXivv2}
and
\cite[Theorem~1.1 and Theorem~4.1]{HutzenthalerJentzenVonWurstemberger2019arXiv}).

\begin{corollary}
\label{cor:MLP_heat_intro}
Let
$ T \in ( 0, \infty ) $,
$ p \in [ 0, \infty ) $,
$ \Theta = \cup_{ n = 1 }^\infty \Z^n $,
$ ( M_j )_{ j \in \N } \subseteq \N $
satisfy
$ \sup_{ j \in \N }
\allowbreak
( \nicefrac{ M_{ j + 1 } }{ M_j }
+
\nicefrac{ M_j }{ j } )
< \infty
=
\liminf_{ j \to \infty } M_j $,
let
$ f \colon \R \to \R $
be a Lipschitz continuous function,
let
$ g_d \in C( \R^d, \R ) $, $ d \in \N $,
satisfy
$ \sup_{ d \in \N }
\sup_{ x \in \R^d }
\nicefrac{ \lvert g_d( x ) \rvert }{ \max\{ 1, \lVert x \rVert^p_{ \R^d } \} }
< \infty $,
for every $ d \in \N $
let
$ y_d \in C( [ 0, T ] \times \R^d, \R ) $
be an at most polynomially growing
viscosity solution of
\begin{equation}
\label{eq:intro_heat_equation}
\bigl( \tfrac{ \partial y_d }{ \partial t } \bigr)( t, x )
+
\tfrac{1}{2}
( \Delta_x y_d )( t, x )
+
f( y_d( t, x ) )
= 0
\end{equation}
with
$ y_d( T, x ) = g_d( x ) $
for $ ( t, x ) \in ( 0, T ) \times \R^d $,
let
$ ( \Omega, \mathscr{F}, \P ) $
be a probability space,
let
$ U^\theta \colon \Omega \to [ 0, 1 ] $, $ \theta \in \Theta $,
be independent on $ [ 0, 1 ] $ uniformly distributed random variables,
let
$ W^{ d, \theta } \colon [ 0, T ] \times \Omega \to \R^d $, $ \theta \in \Theta $, $ d \in \N $,
be independent standard Brownian motions,
assume
that
$ ( U^\theta )_ { \theta \in \Theta } $
and
$ ( W^{ d, \theta } )_ { ( d, \theta ) \in \N \times \Theta } $
are independent,
let
$ Y_{ n, j }^{ d, \theta } \colon [ 0, T ] \times \R^d \times \Omega \to \R $,
$ \theta \in \Theta $,
$ d, j \in \N $,
$ n \in ( \N_0 \cup \{ -1 \} ) $,
satisfy
for all
$ n, j, d \in \N $,
$ \theta \in \Theta $,
$ t \in [ 0, T ] $,
$ x \in \R^d $
that
$ Y_{ -1, j }^{ d, \theta }( t, x ) = Y_{ 0, j }^{ d, \theta }( t, x ) = 0 $
and
\begin{align}
\label{eq:def_Y_heat_intro}
& Y_{ n, j }^{ d, \theta }( T - t, x )
=
\sum_{l=0}^{n-1}
\tfrac{t}{ ( M_j )^{ n - l } }
\Biggl[
    \sum_{i=1}^{ ( M_j )^{ n - l } }
    \Bigl[
        f\Bigl(
            Y^{ d, ( \theta, l, i ) }_{ l, j }\bigl(
                T - t + U^{ ( \theta, l, i ) } t,
                x + W^{ d, ( \theta, l, i ) }_{ U^{ ( \theta, l, i ) } t }
            \bigr)
        \Bigr)
\\ & \quad \nonumber
        -
        \mathbbm{1}_{ \N }( l )
        f\Bigl(
            Y^{ d, ( \theta, -l, i ) }_{ l - 1, j }\bigl(
                T - t + U^{ ( \theta, l, i ) } t,
                x + W^{ d, ( \theta, l, i ) }_{ U^{ ( \theta, l, i ) } t }
            \bigr)
        \Bigr)
    \Bigr]
\Biggr]
+
\tfrac{1}{ ( M_j )^{ n } }
\Biggl[
\sum_{ i = 1 }^{ (M_j )^n }
    g_d\bigl( x + W^{ d, ( \theta, 0, i ) }_{ t } \bigr)
\Biggr]
,
\end{align}
and
for every
$ d, n \in \N $
let
$ \Cost_{ d, n } \in \N_0 $
be the number of realisations of standard normal random variables
used to compute one realisation of
$ Y_{ n, n }^{ d, 0 }( 0, 0 ) $
(cf.~\eqref{eq:cost_MLP_heat} below for a precise definition).
Then
there exist
$ ( N_{ d, \varepsilon } )_{ ( d, \varepsilon ) \in \N \times ( 0, 1 ] } \subseteq \N $
and
$ ( C_\delta )_{ \delta \in ( 0, \infty ) } \subseteq ( 0, \infty ) $
such that
it holds
for all
$ d \in \N $,
$ \varepsilon \in ( 0, 1 ] $,
$ \delta \in ( 0, \infty ) $
that
$ \Cost_{ d, N_{ d, \varepsilon } }
\leq
C_\delta
\:\! d^{ 1 + p( 1 + \delta ) }
\varepsilon^{ -2 ( 1 + \delta ) } $
and
\begin{equation}
\sup\nolimits_{ n \in \{ N_{ d, \varepsilon }, N_{ d, \varepsilon } + 1, \ldots \} }
\bigl(
\E\bigl[
    \lvert Y_{ n, n }^{ d, 0 }( 0, 0 ) - y_d( 0, 0 ) \rvert^2
\bigr]
\bigr)^{ \nicefrac{1}{2} }
\leq
\varepsilon
.
\end{equation}
\end{corollary}

Corollary~\ref{cor:MLP_heat_intro}
is a direct consequence of
Corollary~\ref{cor:MLP_heat} in Subsection~\ref{sec:MLP_heat_variable} below,
which itself is a direct consequence of
Theorem~\ref{thm:MLP_heat} in Subsection~\ref{sec:MLP_heat_variable} below.
Theorem~\ref{thm:MLP_heat}, in turn, follows from
Corollary~\ref{cor:complexity_analysis} in Subsection~\ref{sec:complexity_analysis} below
and Theorem~\ref{thm:complexity_analysis_intro} above, respectively,
and essentially is a slight generalisation of
\cite[Theorem~1.1]{HutzenthalerJentzenKruseNguyenVonWurstemberger2018arXivv2}.
More specifically,
the MLP approximations in~\eqref{eq:def_Y_heat} in Theorem~\ref{thm:MLP_heat}
and in~\eqref{eq:def_Y_heat_intro} above, respectively,
allow for general sequences
of Monte Carlo numbers
$ ( M_j )_{ j \in \N } \subseteq \N $
satisfying
$ \sup_{ j \in \N }
( \nicefrac{ M_{ j + 1 } }{ M_j }
+
\nicefrac{ M_j }{ j } )
< \infty
=
\liminf_{ j \to \infty } M_j $.
This includes, in particular, the special case
where
$ \forall \, j \in \N \colon M_j = j $,
which essentially corresponds to
the MLP approximations in~\cite{HutzenthalerJentzenKruseNguyenVonWurstemberger2018arXivv2}
(cf.\ \cite[(1) in Theorem~1.1]{HutzenthalerJentzenKruseNguyenVonWurstemberger2018arXivv2}).
%The only other difference between
%the MLP approximations in~\eqref{eq:def_Y_heat} in Theorem~\ref{thm:MLP_heat}
%and
%the MLP approximations in~\cite[(1) in Theorem~1.1]{HutzenthalerJentzenKruseNguyenVonWurstemberger2018arXivv2}
%is that we use
%the Brownian motions appearing in the first summation
%on the right hand side of~\eqref{eq:def_Y_heat}
%again in the second summation
%on the right hand side of~\eqref{eq:def_Y_heat},
%instead of using independent Brownian motions for all summations.
%In practice
%both versions of MLP approximations are equally straightforward to implement
%and in this article we only chose this option
%because it fits more naturally into the framework
%of the generalised MLP approximations in~\eqref{eq:scheme_intro}.

This work is structured as follows.
Section~\ref{sec:MLP_generalised}
is devoted to the abstract framework
of generalised MLP approximations.
In particular,
we study
several elementary but crucial properties of
such approximations
in Proposition~\ref{prop:independence}
in Subsection~\ref{sec:properties_MLP}.
Moreover,
we derive an error analysis
for generalised MLP approximations
(see Subsection~\ref{sec:error_analysis}),
which relies on suitably generalised versions of
well-known identities involving bias and variance in Hilbert spaces
(see Corollary~\ref{cor:decomposition}
and Corollary~\ref{cor:variances} in Subsection~\ref{sec:identities}).
The error analysis in Subsection~\ref{sec:error_analysis}
is subsequently combined with
the cost analysis in Subsection~\ref{sec:cost_analysis}
to establish
in Subsection~\ref{sec:complexity_analysis}
a complexity analysis
for generalised MLP approximations
(see Theorem~\ref{thm:complexity_analysis} and Corollary~\ref{cor:complexity_analysis}
in Subsection~\ref{sec:complexity_analysis}).
Throughout Section~\ref{sec:MLP_generalised}
the measurability results
in Subsection~\ref{sec:measurability}
are used.
In Section~\ref{sec:MLP_heat}
we employ
the abstract framework for generalised MLP approximations
from Section~\ref{sec:MLP_generalised}
to analyse numerical approximations for semi-linear heat equations.
Subsection~\ref{sec:polynomial_spaces}
collects several elementary and well-known auxiliary results,
which are used in Subsection~\ref{sec:verification}
to verify that the main assumptions of
the abstract complexity result in Corollary~\ref{cor:complexity_analysis} are fulfilled
in the case of the example setting for numerical approximations for semi-linear heat equations.
Finally,
in Subsection~\ref{sec:complexity_analysis_heat}
we combine the results from
Subsection~\ref{sec:verification}
with Corollary~\ref{cor:complexity_analysis}
to obtain a complexity analysis for MLP approximations for semi-linear heat equations
(see Proposition~\ref{prop:MLP_heat_fixed}, Theorem~\ref{thm:MLP_heat}, and Corollary~\ref{cor:MLP_heat}
in Subsection~\ref{sec:complexity_analysis_heat}).

\section{Generalised full-history recursive multilevel Picard (MLP)}
\label{sec:MLP_generalised}

In this section we introduce
generalised full-history recursive multilevel Picard (MLP) approximations
and provide an error analysis
(see Subsection~\ref{sec:error_analysis}),
cost analysis
(see Subsection~\ref{sec:cost_analysis}),
and complexity analysis
(see Subsection~\ref{sec:complexity_analysis})
for such approximations.

For the formulation of the error analysis
for generalised MLP approximations
we require random variables which take values
in the Banach space
$ L( \mathcal{Y}, \mathcal{H} ) $
of continuous linear functions between
a separable Banach space
$ ( \mathcal{Y}, \lVert \cdot \rVert_{ \mathcal{Y} } ) $
and a separable Hilbert space
$ ( \mathcal{H}, \langle \cdot, \cdot \rangle_{ \mathcal{H} }, \lVert \cdot \rVert_{ \mathcal{H} } ) $
equipped with the strong $ \sigma $-algebra.
Let us recall that the strong $ \sigma $-algebra
on
$ L( \mathcal{Y}, \mathcal{H} ) $
%the Banach space of continuous linear functions
is nothing but the trace $ \sigma $-algebra
of the product $ \sigma $-algebra
on the set
$ \mathcal{H}^\mathcal{Y} $
of all functions
from $ \mathcal{Y} $ to $ \mathcal{H} $.
Having this in mind,
Subsection~\ref{sec:measurability}
collects three elementary measurability results
(see Lemmas~\ref{lem:strong_sigma_algebra}--\ref{lem:prod_measurability}
and
Corollary~\ref{cor:prod_measurability})
about
functions whose domains or codomains
involve a set of functions equipped with the trace $ \sigma $-algebra
of the product $ \sigma $-algebra.

In Subsection~\ref{sec:identities}
we first recall
the elementary and well-known
bias--variance decomposition of the mean square error
for random variables that take values in a separable Hilbert space
(see Lemma~\ref{lem:decomposition}).
Thereafter,
we present in Corollary~\ref{cor:decomposition}
a generalised bias--variance decomposition,
where the mean square error, the bias, and the variance
are measured in a certain randomised sense.
Analogously,
we also recall the elementary and well-known result that
the mean square norm
of the sum of independent zero mean random variables in a separable Hilbert space
is equal
to the sum of the individual mean square norms
(see Lemma~\ref{lem:variances})
and prove a randomised generalisation thereof
(see Corollary~\ref{cor:variances}).
This generalisation
as well as the generalised bias--variance decomposition
in Corollary~\ref{cor:decomposition}
are used in our error analysis
for generalised MLP approximations
(see~\eqref{eq:decomposition}, \eqref{eq:variance_Y1}, and \eqref{eq:decomposition2}
in the proof of Proposition~\ref{prop:error_analysis} below).

Subsequently,
Proposition~\ref{prop:independence}
in Subsection~\ref{sec:properties_MLP}
establishes
several elementary but crucial properties of
generalised MLP approximations,
which are a consequence of their definition.

Subsection~\ref{sec:error_analysis}
is devoted to the
error analysis
for generalised MLP approximations.
Proposition~\ref{prop:error_analysis}
specifies the most general hypotheses in this article
(see~\eqref{eq:EA_definition_Y}--\eqref{eq:EA_hypothesis_III} below)
under which we prove an error estimate
for generalised MLP approximations.
%(see~\eqref{eq:EA_error_estimate} below).
The upper bound for the error in Proposition~\ref{prop:error_analysis}
(see~\eqref{eq:EA_error_estimate} below)
can be much shortened by choosing
a natural number $ \mathfrak{M} \in \N $
such that for every
$ n \in \N $,
$ l \in \{ 0, 1, \ldots, n - 1 \} $
the Monte Carlo sample number
$ M_{ n, l } \in \N $
in the generalised MLP approximations
(see~\eqref{eq:EA_definition_Y} in Proposition~\ref{prop:error_analysis})
is equal to
$ \mathfrak{M}^{ n - l } $,
which is the assertion of Corollary~\ref{cor:error_analysis}
(see~\eqref{eq:EAcor_definition_Y} below).

The subject of Subsection~\ref{sec:cost_analysis}
is the cost analysis for generalised MLP approximations.
The cost estimate in Proposition~\ref{prop:cost_analysis}
follows from an application of the discrete Gronwall-type inequality
in Agarwal~\cite[Theorem~4.1.1]{Agarwal2000}.
The second cost estimate in Subsection~\ref{sec:cost_analysis}
(see Corollary~\ref{cor:cost_analysis}),
in turn,
is a consequence of Proposition~\ref{prop:cost_analysis}
and the elementary and well-known estimate in Lemma~\ref{lem:a_b}.

In Subsection~\ref{sec:complexity_analysis}
the error analysis from Subsection~\ref{sec:error_analysis}
and the cost analysis from Subsection~\ref{sec:cost_analysis}
are combined to derive a complexity analysis
for generalised MLP approximations.
More precisely,
the main result of this article,
Theorem~\ref{thm:complexity_analysis},
relates the error estimate in
Corollary~\ref{cor:error_analysis}
to the cost estimate in
Corollary~\ref{cor:cost_analysis}
in order to arrive at a complexity estimate
(see~\eqref{eq:complexity_estimate} below).
The subsequent result,
Corollary~\ref{cor:complexity_analysis},
is obtained by
replacing assumption~\eqref{eq:complexity_hypothesis_II}
in Theorem~\ref{thm:complexity_analysis}
by the simpler assumption~\eqref{eq:CA_hypothesis_II}
and choosing
for every $ k \in \N_0 $
the coefficient
$ \mathfrak{c}_k \in ( 0, \infty ) $
in Theorem~\ref{thm:complexity_analysis}
to be equal to $ k! $.
Finally,
the elementary result in
Lemma~\ref{lem:M_j_conditions}
shows that a strictly increasing and at most linearly growing
sequence of natural numbers automatically fulfils the
hypotheses on the sequence
$ ( M_j )_{ j \in \N } \subseteq \N $
in Corollary~\ref{cor:complexity_analysis}.

\subsection{Measurability involving the strong \texorpdfstring{$\sigma $}{sigma}-algebra}
\label{sec:measurability}

\begin{lemma}
\label{lem:strong_sigma_algebra}
Let $ \mathcal{E} $ be a set,
let
$ ( \mathcal{F}, \mathscr{F} ) $
and
$ ( \mathcal{G}, \mathscr{G} ) $
be measurable spaces,
let $ \mathcal{S} \subseteq \mathcal{F}^{ \mathcal{E} } $,
let
%$ \mathscr{S} \subseteq \{ \mathcal{A} \colon \mathcal{A} \subseteq \mathcal{S} \} $
%be the set given by
$ \mathscr{S}
=
\sigma_{ \mathcal{S} } \bigl( \bigl\{
    \{ \varphi \in \mathcal{S} \colon \varphi( e ) \in \mathcal{A} \} \subseteq \mathcal{S} \colon
    e \in \mathcal{E},\, \mathcal{A} \in \mathscr{F}
\bigr\} \bigr) $,
and
let
$ \psi \colon \mathcal{G} \to \mathcal{S} $
be a function.
Then it holds that
$ \psi $ %\colon \mathcal{G} \to \mathcal{S} $
is
$ \mathscr{G} $/$ \mathscr{S} $-measurable
if and only if
it holds for all $ e \in \mathcal{E} $ that
$ \mathcal{G} \ni \omega \mapsto [ \psi( \omega ) ]( e ) \in \mathcal{F} $
is $ \mathscr{G} $/$ \mathscr{F} $-measurable.
\end{lemma}
\begin{proof}[Proof of Lemma~\ref{lem:strong_sigma_algebra}]
Throughout this proof
let $ P_e \colon \mathcal{S} \to \mathcal{F} $, $ e \in \mathcal{E} $,
satisfy
for all
$ e \in \mathcal{E} $, $ \varphi \in \mathcal{S} $
that
$ P_e( \varphi ) = \varphi( e ) $.
Observe that it holds that
\begin{equation}
\label{eq:characterisation_strong_sigma_algebra}
\mathscr{S}
=
\{
    ( \mathcal{A} \cap \mathcal{S} ) \subseteq \mathcal{S} \colon
    \mathcal{A} \in ( \otimes_{ e \in \mathcal{E} } \mathscr{F} )
\}
=
\sigma_{ \mathcal{S} } \bigl( ( P_e )_{ e \in \mathcal{E} } \bigr)
.
\end{equation}
This in ensures for all
$ e \in \mathcal{E} $
that
$ P_e \colon \mathcal{S} \to \mathcal{F} $
is an $ \mathscr{S} $/$ \mathscr{F} $-measurable function.
Equation~\eqref{eq:characterisation_strong_sigma_algebra}
hence shows that
$ \psi \colon \mathcal{G} \to \mathcal{S} $ is
$ \mathscr{G} $/$ \mathscr{S} $-measurable
if and only if
it holds for all $ e \in \mathcal{E} $ that
$ P_e \circ \psi \colon \mathcal{G} \to \mathcal{F} $
is $ \mathscr{G} $/$ \mathscr{F} $-measurable.
The proof of Lemma~\ref{lem:strong_sigma_algebra} is thus complete.
\end{proof}

\begin{lemma}
\label{lem:prod_measurability}
Let $ ( \mathcal{ E }, d_{ \mathcal{ E } } ) $
be a separable metric space,
let $ ( \mathcal{ F }, d_{ \mathcal{ F } } ) $
be a metric space,
let $ \mathcal{ S } \subseteq C( \mathcal{ E }, \mathcal{ F } ) $,
and
let
%$ \mathscr{ S } \subseteq \{ \mathcal{ A } \colon \mathcal{ A } \subseteq \mathcal{ S } \} $
%be the set given by
$ \mathscr{S}
=
\sigma_{ \mathcal{S} } \bigl( \bigl\{
    \{ \varphi \in \mathcal{S} \colon \varphi( e ) \in \mathcal{ A } \} \subseteq \mathcal{S} \colon
    e \in \mathcal{ E },\, \mathcal{ A } \in \mathscr{B}( \mathcal{ F } )
\bigr\} \bigr)
%=
%\sigma_{ \mathcal{S} }((
%    \mathcal{S} \ni \varphi \mapsto \varphi( e ) \in \mathcal{ F }
%)_{ e \in \mathcal{ E } })
$.
Then it holds that
$ \mathcal{ S } \times \mathcal{ E } \ni ( \varphi, e ) \mapsto \varphi( e ) \in \mathcal{ F } $
is an
$ ( \mathscr{S} \otimes \mathscr{ B }( \mathcal{ E } ) ) $/$ \mathscr{ B }( \mathcal{ F } ) $-measurable function.
\end{lemma}
\begin{proof}[Proof of Lemma~\ref{lem:prod_measurability}]
Throughout this proof let
$ f \colon \mathcal{ S } \times \mathcal{ E } \to \mathcal{ F } $
satisfy
for all
$ \varphi \in \mathcal{ S } $,
$ e \in \mathcal{ E } $
that
$
f( \varphi, e )
=
\varphi( e )
$.
Note that it holds for all
$ \varphi \in \mathcal{ S } $
that
\begin{equation}
\label{eq:cara_cont}
( \mathcal{ E } \ni e \mapsto f( \varphi, e ) \in \mathcal{ F } )
= \varphi
\in \mathcal{ S } \subseteq C( \mathcal{ E }, \mathcal{ F } )
.
\end{equation}
In addition,
observe that it holds for all
$ e \in \mathcal{ E } $,
$ \mathcal{ A } \in \mathscr{B}( \mathcal{ F } ) $
that
\begin{equation}
\{ \varphi \in \mathcal{S} \colon f( \varphi, e ) \in \mathcal{ A } \}
=
\{ \varphi \in \mathcal{S} \colon \varphi( e ) \in \mathcal{ A } \}
\in \mathscr{S}
.
\end{equation}
This proves
for all
$ e \in \mathcal{ E } $
that
$ \mathcal{ S } \ni \varphi \mapsto f( \varphi, e ) \in \mathcal{ F } $
is an
$ \mathscr{S} $/$ \mathscr{ B }( \mathcal{ F } ) $-measurable function.
Combining this and~\eqref{eq:cara_cont}
with
Aliprantis \& Border~\cite[Lemma~4.51]{AliprantisBorder2006}
(also see, e.g.,
\cite[Lemma~2.4]{BeckBeckerGrohsJaafariJentzen2018arXiv})
establishes that
$ f \colon \mathcal{ S } \times \mathcal{ E } \to \mathcal{ F } $
is an
$ ( \mathscr{S} \otimes \mathscr{ B }( \mathcal{ E } ) ) $/$ \mathscr{ B }( \mathcal{ F } ) $-measurable function.
The proof of Lemma~\ref{lem:prod_measurability} is thus complete.
\end{proof}

\begin{corollary}
\label{cor:prod_measurability}
Let
$ ( \Omega, \mathscr{F} ) $
be a measurable space,
let
$ ( \mathcal{V}, \lVert \cdot \rVert_{ \mathcal{V} } ) $
be a separable normed $ \R $-vector space,
let
$ Y \colon \Omega \to \mathcal{V} $
be an
$ \mathscr{F} $/$ \mathscr{B}( \mathcal{V} ) $-measurable function,
let
$ ( \mathcal{W}, \lVert \cdot \rVert_{ \mathcal{W} } ) $
be a normed $ \R $-vector space,
let
%$ \mathscr{S} \subseteq \{ \mathcal{S} \colon \mathcal{S} \subseteq L( \mathcal{V}, \mathcal{W} ) \} $
%be the set given by
$ \mathscr{S}
=
\sigma_{ L( \mathcal{V}, \mathcal{W} ) } \bigl( \bigl\{
    \{ \varphi \in L( \mathcal{V}, \mathcal{W} ) \colon \varphi( v ) \in \mathcal{B} \} \subseteq L( \mathcal{V}, \mathcal{W} ) \colon
    v \in \mathcal{V},\, \mathcal{ B } \in \mathscr{B}( \mathcal{W} )
\bigr\} \bigr) $,
and
let
$ \psi \colon \Omega \to L( \mathcal{V}, \mathcal{W} ) $
be an
$ \mathscr{F} $/$ \mathscr{S} $-measurable function.
Then
\begin{enumerate}[(i)]
\item
\label{item:prod_measurability1}
it holds that
$ L( \mathcal{V}, \mathcal{W} ) \times \mathcal{V} \ni
( \varphi, v ) \mapsto \varphi( v ) \in \mathcal{W} $
is an
$ ( \mathscr{S} \otimes \mathscr{ B }( \mathcal{ V } ) ) $/$ \mathscr{ B }( \mathcal{ W } ) $-measurable function
and
\item
\label{item:prod_measurability2}
it holds that
$ \psi( Y ) =
( \Omega \ni
\omega \mapsto [ \psi( \omega ) ]( Y( \omega ) ) \in \mathcal{W} ) $
is an
$ \mathscr{F} $/$ \mathscr{ B }( \mathcal{ W } ) $-measurable function.
\end{enumerate}
\end{corollary}
\begin{proof}[Proof of Corollary~\ref{cor:prod_measurability}]
Observe that
Lemma~\ref{lem:prod_measurability}
(with
$ \mathcal{ E } = \mathcal{ V } $,
$ \mathcal{ F } = \mathcal{ W } $,
$ \mathcal{ S } = L( \mathcal{V}, \mathcal{W} ) $
in the notation of Lemma~\ref{lem:prod_measurability})
implies~(\ref{item:prod_measurability1}).
In addition,
the fact that
$ \Omega \ni
\omega \mapsto ( \psi( \omega ), Y( \omega ) )
\in L( \mathcal{V}, \mathcal{W} ) \times \mathcal{V} $
is an
$ \mathscr{F} $/$ ( \mathscr{S} \otimes \mathscr{ B }( \mathcal{ V } ) ) $-measurable function
and~(\ref{item:prod_measurability1})
show~(\ref{item:prod_measurability2}).
The proof of Corollary~\ref{cor:prod_measurability} is thus complete.
\end{proof}

\subsection{Identities involving bias and variance in Hilbert spaces}
\label{sec:identities}

\subsubsection{Bias--variance decomposition}

\begin{lemma}
\label{lem:decomposition}
Let
$ ( \Omega, \mathscr{F}, \P ) $
be a probability space,
let
$ ( \mathcal{H}, \langle \cdot, \cdot \rangle_{ \mathcal{H} }, \lVert \cdot \rVert_{ \mathcal{H} } ) $
be a separable $ \R $-Hilbert space,
let $ h \in \mathcal{H} $,
and
let
$ X \colon \Omega \to \mathcal{H} $
be an
$ \mathscr{F} $/$ \mathscr{B}( \mathcal{H} ) $-measurable function
which satisfies
$ \E[
\lVert X \rVert_{ \mathcal{H} }
] < \infty $.
Then
\begin{equation}
\E\bigl[
    \lVert X - h \rVert_{ \mathcal{H} }^2
\bigr]
=
\E\bigl[
    \lVert X - \E[ X ] \rVert_{ \mathcal{H} }^2
\bigr]
+
\lVert \E[ X ] - h \rVert_{ \mathcal{H} }^2
.
\end{equation}
\end{lemma}
\begin{proof}[Proof of Lemma~\ref{lem:decomposition}]
Note that
the Cauchy--Schwarz inequality
implies that
\begin{equation}
\E\bigl[
    \lvert
    \langle X - \E[ X ],  \E[ X ] - h \rangle_{ \mathcal{H} }
    \rvert
\bigr]
\leq
\E\bigl[
    \lVert X - \E[ X ] \rVert_{ \mathcal{H} }
\bigr]
\lVert \E[ X ] - h \rVert_{ \mathcal{H} }
< \infty
.
\end{equation}
This ensures that
\begin{equation}
\begin{split}
\E\bigl[
    \lVert X - h \rVert_{ \mathcal{H} }^2
\bigr]
& =
\E\bigl[
    \lVert X - \E[ X ] + \E[ X ] - h \rVert_{ \mathcal{H} }^2
\bigr]
\\ &
=
\E\bigl[
    \lVert X - \E[ X ] \rVert_{ \mathcal{H} }^2
+
\lVert \E[ X ] - h \rVert_{ \mathcal{H} }^2
+
2 \langle X - \E[ X ],  \E[ X ] - h \rangle_{ \mathcal{H} }
\bigr]
\\ &
=
\E\bigl[
    \lVert X - \E[ X ] \rVert_{ \mathcal{H} }^2
\bigr]
+
\lVert \E[ X ] - h \rVert_{ \mathcal{H} }^2
+
2
\langle \E[ X ] - \E[ X ],  \E[ X ] - h \rangle_{ \mathcal{H} }
\\ &
=
\E\bigl[
    \lVert X - \E[ X ] \rVert_{ \mathcal{H} }^2
\bigr]
+
\lVert \E[ X ] - h \rVert_{ \mathcal{H} }^2
.
\end{split}
\end{equation}
The proof of Lemma~\ref{lem:decomposition} is thus complete.
\end{proof}

\subsubsection{Generalised bias--variance decomposition}

\begin{corollary}
\label{cor:decomposition}
Let
$ ( \Omega, \mathscr{F}, \P ) $
be a probability space,
let
$ ( \mathcal{Y}, \lVert \cdot \rVert_{ \mathcal{Y} } ) $
be a separable $ \R $-Banach space,
let $ y \in \mathcal{Y} $,
let
$ Y \colon \Omega \to \mathcal{Y} $
be an
$ \mathscr{F} $/$ \mathscr{B}( \mathcal{Y} ) $-measurable function
which satisfies
$ \E[
\lVert Y \rVert_{ \mathcal{Y} }
] < \infty $,
let
$ ( \mathcal{H}, \langle \cdot, \cdot \rangle_{ \mathcal{H} }, \lVert \cdot \rVert_{ \mathcal{H} } ) $
be a separable $ \R $-Hilbert space,
let
%$ \mathscr{S} \subseteq \{ \mathcal{S} \colon \mathcal{S} \subseteq L( \mathcal{Y}, \mathcal{H} ) \} $
%be the set given by
$ \mathscr{S}
=
\sigma_{ L( \mathcal{Y}, \mathcal{H} ) } \bigl( \bigl\{
    \{ \varphi \in L( \mathcal{Y}, \mathcal{H} ) \colon \varphi( x ) \in \mathcal{B} \} \subseteq L( \mathcal{Y}, \mathcal{H} ) \colon
    x \in \mathcal{Y},\, \mathcal{ B } \in \mathscr{B}( \mathcal{H} )
\bigr\} \bigr) $,
let
$ \psi \colon \Omega \to L( \mathcal{Y}, \mathcal{H} ) $
be an
$ \mathscr{F} $/$ \mathscr{S} $-measurable function,
and assume that
$ Y $ and $ \psi $
are independent.
Then
\begin{equation}
\E\bigl[
    \lVert \psi( Y - y ) \rVert_{ \mathcal{H} }^2
\bigr]
=
\E\bigl[
    \lVert \psi( Y - \E[ Y ] ) \rVert_{ \mathcal{H} }^2
\bigr]
+
\E\bigl[
    \lVert \psi( \E[ Y ] - y ) \rVert_{ \mathcal{H} }^2
\bigr]
\end{equation}
(cf.~(\ref{item:prod_measurability2}) in Corollary~\ref{cor:prod_measurability}).
\end{corollary}
\begin{proof}[Proof of Corollary~\ref{cor:decomposition}]
The fact that it holds
for all
$ x \in \mathcal{Y} $
that
$ L( \mathcal{Y}, \mathcal{H} ) \times \Omega \ni
( \varphi, \omega ) \mapsto ( \varphi, Y( \omega ) - x )
\in L( \mathcal{Y}, \mathcal{H} ) \times \mathcal{Y} $
is an
$ ( \mathscr{S} \otimes \sigma_\Omega( Y ) ) $/$ ( \mathscr{S} \otimes \mathscr{ B }( \mathcal{ Y } ) ) $-measurable function
and
the fact that
$ L( \mathcal{Y}, \mathcal{H} ) \times \mathcal{Y} \ni
( \varphi, x ) \mapsto \varphi( x ) \in \mathcal{H} $
is an
$ ( \mathscr{S} \otimes \mathscr{ B }( \mathcal{ Y } ) ) $/$ \mathscr{ B }( \mathcal{ H } ) $-measurable function
(cf.~(\ref{item:prod_measurability1}) in Corollary~\ref{cor:prod_measurability})
imply
for all
$ x \in \mathcal{Y} $
that
\begin{equation}
\label{eq:decomposition_measurability}
L( \mathcal{Y}, \mathcal{H} ) \times \Omega
\ni ( \varphi, \omega ) \mapsto
\varphi( Y( \omega ) - x )
\in \mathcal{ H }
\end{equation}
is an
$ ( \mathscr{S} \otimes \sigma_\Omega( Y ) ) $/$ \mathscr{ B }( \mathcal{ H } ) $-measurable function.
Lemma~2.2 in Hutzenthaler et al.~\cite{HutzenthalerJentzenKruseNguyenVonWurstemberger2018arXivv2}
(with
%$ ( \Omega, \mathcal{F}, \P ) = ( \Omega, \mathscr{F}, \P ) $,
$ \mathcal{G} = \sigma_\Omega( Y ) $,
$ ( S, \mathcal{S} ) = ( L( \mathcal{Y}, \mathcal{H} ), \mathscr{S} ) $,
$ U =
( L( \mathcal{Y}, \mathcal{H} ) \times \Omega
\ni ( \varphi, \omega ) \mapsto
\lVert \varphi( Y( \omega ) - y ) \rVert_{ \mathcal{H} }^2
\in [ 0, \infty ) ) $,
$ Y = \psi $
in the notation of~\cite[Lemma~2.2]{HutzenthalerJentzenKruseNguyenVonWurstemberger2018arXivv2})
and
Lemma~\ref{lem:decomposition}
hence
yield that
\begin{equation}
\begin{split}
& \E\bigl[
    \lVert \psi( Y - y ) \rVert_{ \mathcal{H} }^2
\bigr]
=
\int_{ L( \mathcal{Y}, \mathcal{H} ) }
\E\bigl[
    \lVert \varphi( Y - y ) \rVert_{ \mathcal{H} }^2
\bigr]
\, ( \psi( \P )_{ \mathscr{S} } )( \mathrm{d} \varphi )
\\ &
=
\int_{ L( \mathcal{Y}, \mathcal{H} ) }
\E\bigl[
    \lVert \varphi( Y ) - \varphi( y ) \rVert_{ \mathcal{H} }^2
\bigr]
\, ( \psi( \P )_{ \mathscr{S} } )( \mathrm{d} \varphi )
\\ &
=
\int_{ L( \mathcal{Y}, \mathcal{H} ) }
\E\bigl[
    \lVert \varphi( Y ) - \E[ \varphi( Y ) ] \rVert_{ \mathcal{H} }^2
\bigr]
+
\lVert \E[ \varphi( Y ) ] - \varphi( y ) \rVert_{ \mathcal{H} }^2
\, ( \psi( \P )_{ \mathscr{S} } )( \mathrm{d} \varphi )
.
\end{split}
\end{equation}
This,
the fact that
$ L( \mathcal{Y}, \mathcal{H} ) \ni
\varphi \mapsto \varphi( \E[ Y ] - y ) \in \mathcal{H} $
is an
$ \mathscr{S} $/$ \mathscr{ B }( \mathcal{ H } ) $-measurable function
(cf.~Lemma~\ref{lem:strong_sigma_algebra}),
\eqref{eq:decomposition_measurability},
and
again
\cite[Lemma~2.2]{HutzenthalerJentzenKruseNguyenVonWurstemberger2018arXivv2}
(with
%$ ( \Omega, \mathcal{F}, \P ) = ( \Omega, \mathscr{F}, \P ) $,
$ \mathcal{G} = \sigma_\Omega( Y ) $,
$ ( S, \mathcal{S} ) = ( L( \mathcal{Y}, \mathcal{H} ), \mathscr{S} ) $,
$ U =
( L( \mathcal{Y}, \mathcal{H} ) \times \Omega
\ni ( \varphi, \omega ) \mapsto
\lVert \varphi( Y( \omega ) - \E[ Y ] ) \rVert_{ \mathcal{H} }^2
\in [ 0, \infty ) ) $,
$ Y = \psi $
in the notation of~\cite[Lemma~2.2]{HutzenthalerJentzenKruseNguyenVonWurstemberger2018arXivv2})
establish that
\begin{equation}
\begin{split}
\E\bigl[
    \lVert \psi( Y - y ) \rVert_{ \mathcal{H} }^2
\bigr]
& =
\int_{ L( \mathcal{Y}, \mathcal{H} ) }
\E\bigl[
    \lVert \varphi( Y - \E[ Y ] ) \rVert_{ \mathcal{H} }^2
\bigr]
\, ( \psi( \P )_{ \mathscr{S} } )( \mathrm{d} \varphi )
\\ & \quad
+
\int_{ L( \mathcal{Y}, \mathcal{H} ) }
\lVert \varphi( \E[ Y ] - y ) \rVert_{ \mathcal{H} }^2
\, ( \psi( \P )_{ \mathscr{S} } )( \mathrm{d} \varphi )
\\ &
=
\E\bigl[
    \lVert \psi( Y - \E[ Y ] ) \rVert_{ \mathcal{H} }^2
\bigr]
+
\E\bigl[
    \lVert \psi( \E[ Y ] - y ) \rVert_{ \mathcal{H} }^2
\bigr]
.
\end{split}
\end{equation}
The proof of Corollary~\ref{cor:decomposition} is thus complete.
\end{proof}

\subsubsection{Variance identity}

\begin{lemma}
\label{lem:variances}
Let $ n \in \N $,
let
$ ( \Omega, \mathscr{F}, \P ) $
be a probability space,
let
$ ( \mathcal{H}, \langle \cdot, \cdot \rangle_{ \mathcal{H} }, \lVert \cdot \rVert_{ \mathcal{H} } ) $
be a separable $ \R $-Hilbert space,
and
let
$ X_1, X_2, \ldots, X_n \colon \Omega \to \mathcal{H} $
be independent
$ \mathscr{F} $/$ \mathscr{B}( \mathcal{H} ) $-measurable functions
which satisfy
for all
$ i \in \{ 1, 2, \ldots, n \} $
that
$ \E[
\lVert X_i \rVert_{ \mathcal{H} }
] < \infty $.
Then
\begin{equation}
\E \Biggl[
\biggl\lVert
    \smallsum_{ i = 1 }^{n}
        ( X_i - \E[ X_i ] )
\biggr\rVert_{ \mathcal{H} }^2
\Biggr]
=
\smallsum_{ i = 1 }^{n}
\E\bigl[
\lVert
    X_i - \E[ X_i ]
\rVert_{ \mathcal{H} }^2
\bigr]
.
\end{equation}
\end{lemma}
\begin{proof}[Proof of Lemma~\ref{lem:variances}]
Observe that
the Cauchy--Schwarz inequality,
the fact that
$ X_1, X_2,
\allowbreak
\ldots, X_n $ are independent,
and
the fact that it holds
for all independent random variables $ Y, Z \colon \Omega \to \R $
with $ \E[ \lvert Y \rvert + \lvert Z \rvert ] < \infty $
that
$ \E[ \lvert Y Z \rvert ] < \infty $
and
$ \E[ Y Z ] = \E[ Y ] \E[ Z ] $
(cf., e.g., Klenke~\cite[Theorem~5.4]{Klenke2014})
demonstrate
for all
$ i, j \in \{ 1, 2, \ldots, n \} $
with $ i \neq j $
that
\begin{equation}
\label{eq:correlation_finite}
\begin{split}
\E \bigl[
\lvert
\langle
    X_i - \E[ X_i ],
    X_j - \E[ X_j ]
\rangle_{ \mathcal{H} }
\rvert
\bigr]
& \leq
\E \bigl[
\lVert
    X_i - \E[ X_i ]
\rVert_{ \mathcal{H} }
\lVert
    X_j - \E[ X_j ]
\rVert_{ \mathcal{H} }
\bigr]
\\ &
=
\E \bigl[
\lVert
    X_i - \E[ X_i ]
\rVert_{ \mathcal{H} }
\bigr]
\E \bigl[
\lVert
    X_j - \E[ X_j ]
\rVert_{ \mathcal{H} }
\bigr]
< \infty
.
\end{split}
\end{equation}
Moreover,
the fact that
$ X_1, X_2, \ldots, X_n $ are independent
ensures
for all
$ i, j \in \{ 1, 2, \ldots, n \} $
with $ i \neq j $
that
\begin{equation}
( X_i, X_j )( \P )_{ \mathscr{B}( \mathcal{H} ) \otimes \mathscr{B}( \mathcal{H} ) }
=
[ ( X_i )( \P )_{ \mathscr{B}( \mathcal{H} ) } ] \otimes [ ( X_j )( \P )_{ \mathscr{B}( \mathcal{H} ) } ]
.
\end{equation}
Fubini's theorem
and
\eqref{eq:correlation_finite}
hence
show
for all
$ i, j \in \{ 1, 2, \ldots, n \} $
with $ i \neq j $
that
\begin{equation}
\begin{split}
& \E \bigl[
\langle
    X_i - \E[ X_i ],
    X_j - \E[ X_j ]
\rangle_{ \mathcal{H} }
\bigr]
\\ &
=
\int_{ \mathcal{H} \times \mathcal{H} }
    \langle
        x - \E[ X_i ],
        y - \E[ X_j ]
    \rangle_{ \mathcal{H} }
\, \bigl( ( X_i, X_j )( \P )_{ \mathscr{B}( \mathcal{H} ) \otimes \mathscr{B}( \mathcal{H} ) } \bigr)( \mathrm{d} x, \mathrm{d} y )
\\ &
=
\int_{ \mathcal{H} \times \mathcal{H} }
    \langle
        x - \E[ X_i ],
        y - \E[ X_j ]
    \rangle_{ \mathcal{H} }
\, \bigl( [ ( X_i )( \P )_{ \mathscr{B}( \mathcal{H} ) } ] \otimes [ ( X_j )( \P )_{ \mathscr{B}( \mathcal{H} ) } ] \bigr)( \mathrm{d} x, \mathrm{d} y )
\\ &
=
\int_{ \mathcal{H} }
\int_{ \mathcal{H} }
    \langle
        x - \E[ X_i ],
        y - \E[ X_j ]
    \rangle_{ \mathcal{H} }
\, \bigl( ( X_i )( \P )_{ \mathscr{B}( \mathcal{H} ) } \bigr)( \mathrm{d} x )
\bigl( ( X_j )( \P )_{ \mathscr{B}( \mathcal{H} ) } \bigr)( \mathrm{d} y )
\\ &
=
\int_{ \mathcal{H} }
    \E \bigl[
    \langle
        X_i - \E[ X_i ],
        y - \E[ X_j ]
    \rangle_{ \mathcal{H} }
    \bigr]
\, \bigl( ( X_j )( \P )_{ \mathscr{B}( \mathcal{H} ) } \bigr)( \mathrm{d} y )
\\ &
=
\int_{ \mathcal{H} }
    \langle
        \E[ X_i ] - \E[ X_i ],
        y - \E[ X_j ]
    \rangle_{ \mathcal{H} }
\, \bigl( ( X_j )( \P )_{ \mathscr{B}( \mathcal{H} ) } \bigr)( \mathrm{d} y )
\\ &
=
\langle
    \E[ X_i ] - \E[ X_i ],
    \E[ X_j ] - \E[ X_j ]
\rangle_{ \mathcal{H} }
=
0
.
\end{split}
\end{equation}
This and
again
\eqref{eq:correlation_finite}
prove that
\begin{equation}
\begin{split}
& \E \Biggl[
\biggl\lVert
    \smallsum_{ i = 1 }^{n}
        ( X_i - \E[ X_i ] )
\biggr\rVert_{ \mathcal{H} }^2
\Biggr]
=
\E \Biggl[
\biggl\langle
    \smallsum_{ i = 1 }^{n} ( X_i - \E[ X_i ] ),
    \smallsum_{ j = 1 }^{n} ( X_j - \E[ X_j ] )
\biggr\rangle_{ \mathcal{H} }
\Biggr]
\\ &
=
\E \biggl[
\smallsum_{ i, j = 1 }^{n}
\langle
    X_i - \E[ X_i ],
    X_j - \E[ X_j ]
\rangle_{ \mathcal{H} }
\biggr]
\\ &
=
\biggl[
\smallsum_{ i = 1 }^{n}
\E \bigl[
\langle
    X_i - \E[ X_i ],
    X_i - \E[ X_i ]
\rangle_{ \mathcal{H} }
\bigr]
\biggr]
+
\smallsum_{ i, j = 1, \, i \neq j }^{n}
\E \bigl[
\langle
    X_i - \E[ X_i ],
    X_j - \E[ X_j ]
\rangle_{ \mathcal{H} }
\bigr]
\\ &
=
\smallsum_{ i = 1 }^{n}
\E\bigl[
\lVert
    X_i - \E[ X_i ]
\rVert_{ \mathcal{H} }^2
\bigr]
.
\end{split}
\end{equation}
The proof of Lemma~\ref{lem:variances} is thus complete.
\end{proof}

\subsubsection{Generalised variance identity}
%{Generalised variance identity for the sum of independent random variables}

\begin{corollary}
\label{cor:variances}
Let $ n \in \N $,
let
$ ( \Omega, \mathscr{F}, \P ) $
be a probability space,
let
$ ( \mathcal{Y}, \lVert \cdot \rVert_{ \mathcal{Y} } ) $
be a separable $ \R $-Banach space,
let
$ Y_1, Y_2, \ldots, Y_n \colon \Omega \to \mathcal{Y} $
be independent
$ \mathscr{F} $/$ \mathscr{B}( \mathcal{Y} ) $-measurable functions
which satisfy
for all
$ i \in \{ 1, 2, \ldots, n \} $
that
$ \E[
\lVert Y_i \rVert_{ \mathcal{Y} }
] < \infty $,
let
$ ( \mathcal{H}, \langle \cdot, \cdot \rangle_{ \mathcal{H} }, \lVert \cdot \rVert_{ \mathcal{H} } ) $
be a separable $ \R $-Hilbert space,
let
%$ \mathscr{S} \subseteq \{ \mathcal{S} \colon \mathcal{S} \subseteq L( \mathcal{Y}, \mathcal{H} ) \} $
%be the set given by
$ \mathscr{S}
=
\sigma_{ L( \mathcal{Y}, \mathcal{H} ) } \bigl( \bigl\{
    \{ \varphi \in L( \mathcal{Y}, \mathcal{H} ) \colon \varphi( y ) \in \mathcal{B} \} \subseteq L( \mathcal{Y}, \mathcal{H} ) \colon
    y \in \mathcal{Y},\, \mathcal{ B } \in \mathscr{B}( \mathcal{H} )
\bigr\} \bigr) $,
let
$ \psi \colon \Omega \to L( \mathcal{Y}, \mathcal{H} ) $
be an
$ \mathscr{F} $/$ \mathscr{S} $-measurable function,
and assume that
$ ( Y_i )_{ i \in \{ 1, 2, \ldots, n \} } $
and
$ \psi $
are independent.
Then
\begin{equation}
\E \Biggl[
\biggl\lVert
    \smallsum_{ i = 1 }^{n}
        \psi( Y_i - \E[ Y_i ] )
\biggr\rVert_{ \mathcal{H} }^2
\Biggr]
=
\smallsum_{ i = 1 }^{n}
\E\bigl[
\lVert
\psi( Y_i - \E[ Y_i ] )
\rVert_{ \mathcal{H} }^2
\bigr]
\end{equation}
(cf.~(\ref{item:prod_measurability2}) in Corollary~\ref{cor:prod_measurability}).
\end{corollary}
\begin{proof}[Proof of Corollary~\ref{cor:variances}]
The fact that it holds
for all
$ i \in \{ 1, 2, \ldots, n \} $
that
$ L( \mathcal{Y}, \mathcal{H} ) \times \Omega \ni
( \varphi, \omega ) \mapsto ( \varphi, Y_i( \omega ) - \E[ Y_i ] )
\in L( \mathcal{Y}, \mathcal{H} ) \times \mathcal{Y} $
is an
$ ( \mathscr{S} \otimes \sigma_\Omega( ( Y_j )_{ j \in \{ 1, 2, \ldots, n \} } ) ) $/$ ( \mathscr{S} \otimes \mathscr{ B }( \mathcal{ Y } ) ) $-measurable function
and
the fact that
$ L( \mathcal{Y}, \mathcal{H} ) \times \mathcal{Y} \ni
( \varphi, y ) \mapsto \varphi( y ) \in \mathcal{H} $
is an
$ ( \mathscr{S} \otimes \mathscr{ B }( \mathcal{ Y } ) ) $/$ \mathscr{ B }( \mathcal{ H } ) $-measurable function
(cf.~(\ref{item:prod_measurability1}) in Corollary~\ref{cor:prod_measurability})
ensure
for all
$ i \in \{ 1, 2, \ldots, n \} $
that
\begin{equation}
\label{eq:product_measurability}
L( \mathcal{Y}, \mathcal{H} ) \times \Omega
\ni ( \varphi, \omega ) \mapsto
\varphi( Y_i( \omega ) - \E[ Y_i ] )
\in \mathcal{ H }
\end{equation}
is an
$ ( \mathscr{S} \otimes \sigma_\Omega( ( Y_j )_{ j \in \{ 1, 2, \ldots, n \} } ) ) $/$ \mathscr{ B }( \mathcal{ H } ) $-measurable function.
This and
\cite[Lemma~2.2]{HutzenthalerJentzenKruseNguyenVonWurstemberger2018arXivv2}
(with
%$ ( \Omega, \mathcal{F}, \P ) = ( \Omega, \mathscr{F}, \P ) $,
$ \mathcal{G} = \sigma_\Omega( ( Y_i )_{ i \in \{ 1, 2, \ldots, n \} } ) $,
$ ( S, \mathcal{S} ) = ( L( \mathcal{Y}, \mathcal{H} ), \mathscr{S} ) $,
$ U =
( L( \mathcal{Y}, \mathcal{H} ) \times \Omega
\ni ( \varphi, \omega ) \mapsto
\lVert \sum_{ i = 1 }^{n} \varphi( Y_i( \omega ) - \E[ Y_i ] ) \rVert_{ \mathcal{H} }^2
\in [ 0, \infty ) ) $,
$ Y = \psi $
in the notation of~\cite[Lemma~2.2]{HutzenthalerJentzenKruseNguyenVonWurstemberger2018arXivv2})
establish that
\begin{equation}
\begin{split}
\E \Biggl[
\biggl\lVert
    \smallsum_{ i = 1 }^{n}
        \psi( Y_i - \E[ Y_i ] )
\biggr\rVert_{ \mathcal{H} }^2
\Biggr]
& =
\int_{ L( \mathcal{Y}, \mathcal{H} ) }
\E \Bigl[
\bigl\lVert
    \smallsumtext_{ i = 1 }^{n}
        \varphi( Y_i - \E[ Y_i ] )
\bigr\rVert_{ \mathcal{H} }^2
\Bigr]
\, ( \psi( \P )_{ \mathscr{S} } )( \mathrm{d} \varphi )
\\ &
=
\int_{ L( \mathcal{Y}, \mathcal{H} ) }
\E \Bigl[
\bigl\lVert
    \smallsumtext_{ i = 1 }^{n}
        ( \varphi( Y_i ) - \E[ \varphi( Y_i ) ] )
\bigr\rVert_{ \mathcal{H} }^2
\Bigr]
\, ( \psi( \P )_{ \mathscr{S} } )( \mathrm{d} \varphi )
.
\end{split}
\end{equation}
Lemma~\ref{lem:variances},
\eqref{eq:product_measurability},
and
\cite[Lemma~2.2]{HutzenthalerJentzenKruseNguyenVonWurstemberger2018arXivv2}
(with
%$ ( \Omega, \mathcal{F}, \P ) = ( \Omega, \mathscr{F}, \P ) $,
$ \mathcal{G} = \sigma_\Omega( ( Y_j )_{ j \in \{ 1, 2, \ldots, n \} } ) $,
$ ( S, \mathcal{S} ) = ( L( \mathcal{Y}, \mathcal{H} ),
\allowbreak
\mathscr{S} ) $,
$ U =
( L( \mathcal{Y}, \mathcal{H} ) \times \Omega
\ni ( \varphi, \omega ) \mapsto
\lVert \varphi( Y_i( \omega ) - \E[ Y_i ] ) \rVert_{ \mathcal{H} }^2
\in [ 0, \infty ) ) $,
$ Y = \psi $
for
$ i \in \{ 1, 2, \ldots, n \} $
in the notation of~\cite[Lemma~2.2]{HutzenthalerJentzenKruseNguyenVonWurstemberger2018arXivv2})
hence
show that
\begin{equation}
\begin{split}
\E \Biggl[
\biggl\lVert
    \smallsum_{ i = 1 }^{n}
        \psi( Y_i - \E[ Y_i ] )
\biggr\rVert_{ \mathcal{H} }^2
\Biggr]
&
=
\int_{ L( \mathcal{Y}, \mathcal{H} ) }
\smallsum_{ i = 1 }^{n}
\E \bigl[
\lVert \varphi( Y_i ) - \E[ \varphi( Y_i ) ] \rVert_{ \mathcal{H} }^2
\bigr]
\, ( \psi( \P )_{ \mathscr{S} } )( \mathrm{d} \varphi )
\\ &
=
\smallsum_{ i = 1 }^{n}
\int_{ L( \mathcal{Y}, \mathcal{H} ) }
\E \bigl[
\lVert \varphi( Y_i - \E[ Y_i ] ) \rVert_{ \mathcal{H} }^2
\bigr]
\, ( \psi( \P )_{ \mathscr{S} } )( \mathrm{d} \varphi )
\\ &
=
\smallsum_{ i = 1 }^{n}
\E \bigl[
\lVert \psi( Y_i - \E[ Y_i ] ) \rVert_{ \mathcal{H} }^2
\bigr]
.
\end{split}
\end{equation}
The proof of Corollary~\ref{cor:variances}
is thus complete.
\end{proof}

\subsection{Properties of generalised MLP approximations}
\label{sec:properties_MLP}

\begin{proposition}
\label{prop:independence}
Let
$ ( \Omega, \mathscr{F}, \P ) $
be a probability space,
let $ ( \mathcal{Y}, \lVert \cdot \rVert_{ \mathcal{Y} } ) $
be a separable $ \R $-Banach space,
let
$ \Theta = \cup_{ n = 1 }^\infty \Z^n $,
$ ( M_{ n, l } )_{ ( n, l ) \in \N \times \N_0 } \subseteq \N $,
let
$ ( \mathcal{Z}, \mathscr{Z} ) $ be a measurable space,
let
$ Z^{ \theta } \colon \Omega \to \mathcal{Z} $,
$ \theta \in \Theta $,
be i.i.d.\
$ \mathscr{F} $/$ \mathscr{Z} $-measurable functions,
let
$ \Phi_l \colon \mathcal{Y} \times \mathcal{Y} \times \mathcal{Z} \to \mathcal{Y} $, $ l \in \N_0 $,
be
$ ( \mathscr{B}( \mathcal{Y} ) \otimes \mathscr{B}( \mathcal{Y} ) \otimes \mathscr{Z} ) $/$ \mathscr{B}( \mathcal{Y} ) $-measurable
functions,
let
$ Y_{ -1 }^\theta \colon \Omega \to \mathcal{Y} $, $ \theta \in \Theta $,
be i.i.d.\
$ \mathscr{F} $/$ \mathscr{B}( \mathcal{Y} ) $-measurable functions,
let
$ Y_{ 0 }^\theta \colon \Omega \to \mathcal{Y} $, $ \theta \in \Theta $,
be i.i.d.\
$ \mathscr{F} $/$ \mathscr{B}( \mathcal{Y} ) $-measurable functions,
assume that
$ ( Y_{ -1 }^{ \theta } )_{ \theta \in \Theta } $,
$ ( Y_{ 0 }^{ \theta } )_{ \theta \in \Theta } $,
and
$ ( Z^{ \theta } )_{ \theta \in \Theta } $
are independent,
and let
$ Y_{ n }^\theta \colon \Omega \to \mathcal{Y} $,
$ \theta \in \Theta $,
$ n \in \N $,
satisfy
for all
$ n \in \N $,
$ \theta \in \Theta $
that
\begin{equation}
\label{eq:independence_scheme}
Y_{ n }^\theta
=
\sum_{l=0}^{n-1}
\tfrac{1}{ M_{ n, l } }
\Biggl[
    \sum_{i=1}^{ M_{ n, l } }
        \Phi_{ l } \bigl(
        Y^{ ( \theta, l, i ) }_{ l },
        Y^{ ( \theta, -l, i ) }_{ l - 1 },
        Z^{ ( \theta, l, i ) }
        \bigr)
\Biggr]
.
\end{equation}
Then
\begin{enumerate}[(i)]
\item
\label{item:independence1}
it holds for all
$ n \in ( \N_0 \cup \{ - 1 \} ) $,
$ \theta \in \Theta $
that
$ Y_{ n }^\theta \colon \Omega \to \mathcal{Y} $
is an
$ \mathscr{F} $/$ \mathscr{B}( \mathcal{Y} ) $-measurable function,
\item
\label{item:independence2}
it holds for all
$ n \in \N $,
$ \theta \in \Theta $
that
$ \sigma_{ \Omega }( Y_{ n }^\theta ) \subseteq
\sigma_{ \Omega }\bigl(
( Y_{ -1 }^{ ( \theta, \vartheta ) } )_{ \vartheta \in \Theta },
( Y_{ 0 }^{ ( \theta, \vartheta ) } )_{ \vartheta \in \Theta },
( Z^{ ( \theta, \vartheta ) } )_{ \vartheta \in \Theta }
\bigr) $,
\item
\label{item:independence3}
it holds for every
$ n, m \in ( \N_0 \cup \{ - 1 \} ) $,
$ k \in \N $,
$ \theta_1, \theta_2, \vartheta \in \Z^k $
with
$  \theta_1 \neq  \theta_2 $
that
$ Y_{ n }^{ \theta_1 } $,
$ Y_{ m }^{ \theta_2 } $,
and $ Z^\vartheta $
are independent,
\item
\label{item:independence4}
it holds for every
$ \theta \in \Theta $
that
$ \bigl(
    Y^{ ( \theta, l, i ) }_{ l },
    Y^{ ( \theta, -l, i ) }_{ l - 1 },
    Z^{ ( \theta, l, i ) }
\bigr) $,
$ i \in \N $,
$ l \in \N_0 $,
are independent,
\item
\label{item:independence5}
it holds for every
$ n \in ( \N_0 \cup \{ - 1 \} ) $
that
$ Y_{ n }^\theta $, $ \theta \in \Theta $,
are identically distributed,
and
\item
\label{item:independence6}
it holds for every
$ \theta \in \Theta $,
$ l \in \N_0 $,
$ i \in \N $
that
$ \Omega \ni \omega \mapsto
\Phi_{ l } \bigl(
    Y^{ ( \theta, l, i ) }_{ l }( \omega ),
    Y^{ ( \theta, -l, i ) }_{ l - 1 }( \omega ),
\allowbreak
    Z^{ ( \theta, l, i ) }( \omega )
\bigr)
\in \mathcal{Y} $
and
$ \Omega \ni \omega \mapsto
\Phi_{ l } \bigl(
    Y^{ 0 }_{ l }( \omega ),
    Y^{ 1 }_{ l - 1 }( \omega ),
    Z^{ 0 }( \omega )
\bigr)
\in \mathcal{Y} $
are identically distributed.
\end{enumerate}
\end{proposition}
\begin{proof}[Proof of Proposition~\ref{prop:independence}]
Throughout this proof
let
$ R^{ \theta, l, i } \colon \Omega \to \mathcal{Y} \times \mathcal{Y} \times \mathcal{Z} $,
$ i \in \N $,
$ l \in \N_0 $,
$ \theta \in \Theta $,
satisfy
for all
$ \theta \in \Theta $,
$ l \in \N_0 $,
$ i \in \N $
that
\begin{equation}
\label{eq:R_theta_l_i_definition}
R^{ \theta, l, i }( \omega )
=
\bigl(
    Y^{ ( \theta, l, i ) }_{ l }( \omega ),
    Y^{ ( \theta, -l, i ) }_{ l - 1 }( \omega ),
    Z^{ ( \theta, l, i ) }( \omega )
\bigr)
\end{equation}
and
let
$ \Psi_n \colon ( \mathcal{Y} \times \mathcal{Y} \times \mathcal{Z} )^{ \{ 0, 1, \ldots, n - 1 \} \times \N } \to \mathcal{Y} $,
$ n \in \N $,
satisfy
for all
$ n \in \N $,
$ r =
$\linebreak$
( r^{ l, i } )_{ ( l, i ) \in \{ 0, 1, \ldots, n - 1 \} \times \N } \in ( \mathcal{Y} \times \mathcal{Y} \times \mathcal{Z} )^{ \{ 0, 1, \ldots, n - 1 \} \times \N } $
that
$ \Psi_n( r )
=
\sum_{l=0}^{n-1}
\tfrac{1}{ M_{ n, l } }
\bigl[
    \sum_{i=1}^{ M_{ n, l } }
        \Phi_{ l }( r^{ l, i } )
\bigr]
$.

First,
note that
the assumption that it holds for all
$ \theta \in \Theta $
that
$ Y_{ -1 }^\theta \colon \Omega \to \mathcal{Y} $
and
$ Y_{ 0 }^\theta \colon \Omega \to \mathcal{Y} $
are
$ \mathscr{F} $/$ \mathscr{B}( \mathcal{Y} ) $-measurable functions,
the assumption that it holds for all
$ \theta \in \Theta $
that
$ Z^{ \theta } \colon \Omega \to \mathcal{Z} $
is an
$ \mathscr{F} $/$ \mathscr{Z} $-measurable function,
the assumption that it holds for all
$ l \in \N_0 $
that
$ \Phi_l \colon \mathcal{Y} \times \mathcal{Y} \times \mathcal{Z} \to \mathcal{Y} $
is an
$ ( \mathscr{B}( \mathcal{Y} ) \otimes \mathscr{B}( \mathcal{Y} ) \otimes \mathscr{Z} ) $/$ \mathscr{B}( \mathcal{Y} ) $-measurable function,
the assumption that
$ ( \mathcal{Y}, \lVert \cdot \rVert_{ \mathcal{Y} } ) $
is a separable $ \R $-Banach space,
and
induction on $ \N_0 $
prove~(\ref{item:independence1}).

Second,
we show~(\ref{item:independence2}) by induction on $ n \in \N $.
For the base case $ n = 1 $ observe that it holds for all
$ \theta \in \Theta $
that
\begin{equation}
Y_{ 1 }^\theta
=
\tfrac{1}{ M_{ 1, 0 } }
    \sum_{i=1}^{ M_{ 1, 0 } }
        \Phi_{ 0 } \bigl(
        Y^{ ( \theta, 0, i ) }_{ 0 },
        Y^{ ( \theta, 0, i ) }_{ -1 },
        Z^{ ( \theta, 0, i ) }
        \bigr)
.
\end{equation}
This demonstrates for all
$ \theta \in \Theta $
that
\begin{equation}
\begin{split}
\sigma_{ \Omega }( Y_{ 1 }^\theta )
& \subseteq
\sigma_{ \Omega }\bigl( ( Y_{ -1 }^{ ( \theta, 0, i ) } )_{ i \in \N }, ( Y_{ 0 }^{ ( \theta, 0, i ) } )_{ i \in \N }, ( Z^{ ( \theta, 0, i ) } )_{ i \in \N }
\bigr)
\\ &
\subseteq
\sigma_{ \Omega }\bigl( ( Y_{ -1 }^{ ( \theta, \vartheta ) } )_{ \vartheta \in \Theta }, ( Y_{ 0 }^{ ( \theta, \vartheta ) } )_{ \vartheta \in \Theta }, ( Z^{ ( \theta, \vartheta ) } )_{ \vartheta \in \Theta }
\bigr)
.
\end{split}
\end{equation}
This establishes~(\ref{item:independence2}) in the base case $ n = 1 $.
For the induction step $ \N \ni n - 1 \to n \in \{ 2, 3, \ldots \} $
let $ n \in \{ 2, 3, \ldots \} $
and assume for all
$ l \in \{ 1, \ldots, n - 1 \} $,
$ \theta \in \Theta $
that
\begin{equation}
\label{eq:independence2_induction_step}
\sigma_{ \Omega }( Y_{ l }^\theta )
\subseteq
\sigma_{ \Omega }\bigl( ( Y_{ -1 }^{ ( \theta, \vartheta ) } )_{ \vartheta \in \Theta }, ( Y_{ 0 }^{ ( \theta, \vartheta ) } )_{ \vartheta \in \Theta }, ( Z^{ ( \theta, \vartheta ) } )_{ \vartheta \in \Theta }
\bigr)
.
\end{equation}
This and~\eqref{eq:independence_scheme}
imply for all
$ \theta \in \Theta $
that
\begin{equation}
\begin{split}
\sigma_{ \Omega }( Y_{ n }^\theta )
& \subseteq
\sigma_{ \Omega }\bigl(
    ( Y^{ ( \theta, -l, i ) }_{ l - 1 } )_{ ( l, i ) \in \{ 0, 1, \ldots, n - 1 \} \times \N },
    ( Y^{ ( \theta, l, i ) }_{ l } )_{ ( l, i ) \in \{ 0, 1, \ldots, n - 1 \} \times \N },
    ( Z^{ ( \theta, l, i ) } )_{ ( l, i ) \in \N_0 \times \N }
\bigr)
\\ &
\subseteq
\sigma_{ \Omega }\bigl(
    ( Y^{ ( \theta, \mathfrak{z} ) }_{ l } )_{ ( l, \mathfrak{z} ) \in \{ 1, \ldots, n - 1 \} \times \Z^2 },
    ( Y_{ -1 }^{ ( \theta, \mathfrak{z} ) } )_{ \mathfrak{z} \in \Z^2 },
    ( Y_{ 0 }^{ ( \theta, \mathfrak{z} ) } )_{ \mathfrak{z} \in \Z^2 },
    ( Z^{ ( \theta, \mathfrak{z} ) } )_{ \mathfrak{z} \in \Z^2 }
\bigr)
\\ &
\subseteq
\sigma_{ \Omega }\bigl(
    ( Y_{ -1 }^{ ( \theta, \mathfrak{z}, \vartheta ) } )_{ ( \mathfrak{z}, \vartheta ) \in \Z^2 \times \Theta },
    ( Y_{ 0 }^{ ( \theta, \mathfrak{z}, \vartheta ) } )_{ ( \mathfrak{z}, \vartheta ) \in \Z^2 \times \Theta },
    ( Z^{ ( \theta, \mathfrak{z}, \vartheta ) } )_{ ( \mathfrak{z}, \vartheta ) \in \Z^2 \times \Theta },
\\ &
\phantom{
\subseteq
\sigma_{ \Omega }\bigl( \
}
    ( Y_{ -1 }^{ ( \theta, \vartheta ) } )_{ \vartheta \in \Theta },
    ( Y_{ 0 }^{ ( \theta, \vartheta ) } )_{ \vartheta \in \Theta },
    ( Z^{ ( \theta, \vartheta ) } )_{ \vartheta \in \Theta }
\bigr)
\\ &
=
\sigma_{ \Omega }\bigl(
    ( Y_{ -1 }^{ ( \theta, \vartheta ) } )_{ \vartheta \in \Theta },
    ( Y_{ 0 }^{ ( \theta, \vartheta ) } )_{ \vartheta \in \Theta },
    ( Z^{ ( \theta, \vartheta ) } )_{ \vartheta \in \Theta }
\bigr)
.
\end{split}
\end{equation}
%This and the induction hypothesis~\eqref{eq:independence2_induction_step}
%ensure
%for all
%$ l \in \{ 1, \ldots, n \} $,
%$ \theta \in \Theta $
%that
%$
%\sigma_{ \Omega }( Y_{ l }^\theta )
%\subseteq
%\sigma_{ \Omega }\bigl( ( Y_{ -1 }^{ ( \theta, \vartheta ) } )_{ \vartheta \in \Theta }, ( Y_{ 0 }^{ ( \theta, \vartheta ) } )_{ \vartheta \in \Theta }, ( Z^{ ( \theta, \vartheta ) } )_{ \vartheta \in \Theta }
%\bigr)
%$.
Induction hence establishes~(\ref{item:independence2}).

Third,
observe that
the assumption that
$ ( Y_{ -1 }^{ \theta } )_{ \theta \in \Theta } $,
$ ( Y_{ 0 }^{ \theta } )_{ \theta \in \Theta } $,
and
$ ( Z^{ \theta } )_{ \theta \in \Theta } $
are independent
ensures that it holds for every
$ k \in \N $,
$ \theta_1, \theta_2, \vartheta \in \Z^k $
with
$  \theta_1 \neq  \theta_2 $
that
$ \sigma_{ \Omega }\bigl(
( Y_{ -1 }^{ ( \theta_1, \mathfrak{z} ) } )_{ \mathfrak{z} \in \Theta },
\allowbreak
( Y_{ 0 }^{ ( \theta_1, \mathfrak{z} ) } )_{ \mathfrak{z} \in \Theta },
( Z^{ ( \theta_1, \mathfrak{z} ) } )_{ \mathfrak{z} \in \Theta }
\bigr) $,
$ \sigma_{ \Omega }\bigl(
( Y_{ -1 }^{ ( \theta_2, \mathfrak{z} ) } )_{ \mathfrak{z} \in \Theta },
( Y_{ 0 }^{ ( \theta_2, \mathfrak{z} ) } )_{ \mathfrak{z} \in \Theta },
( Z^{ ( \theta_2, \mathfrak{z} ) } )_{ \mathfrak{z} \in \Theta }
\bigr) $,
and
$ \sigma_{ \Omega }( Z^\vartheta ) $
are independent.
Combining this with~(\ref{item:independence2})
proves~(\ref{item:independence3}).

Fourth,
note that
the assumption that
the family
$ ( Y_{ -1 }^{ \theta } )_{ \theta \in \Theta } $
is independent,
the assumption that
the family
$ ( Y_{ 0 }^{ \theta } )_{ \theta \in \Theta } $
is independent,
the assumption that
the family
$ ( Z^{ \theta } )_{ \theta \in \Theta } $
is independent,
and
the assumption that
$ ( Y_{ -1 }^{ \theta } )_{ \theta \in \Theta } $,
$ ( Y_{ 0 }^{ \theta } )_{ \theta \in \Theta } $,
and
$ ( Z^{ \theta } )_{ \theta \in \Theta } $
are independent
imply for every
$ \theta \in \Theta $
that
the family
\begin{equation*}
\N_0 \times \N
\ni
( l, i )
\mapsto
\begin{cases}
\sigma_{ \Omega }\bigl(
Y_{ -1 }^{ ( \theta, 0, i ) },
Y_{ 0 }^{ ( \theta, 0, i ) },
Z^{ ( \theta, 0, i ) }
\bigr)
& \colon l = 0
\\
\!\!\!\!\:
\begin{array}{l}
\sigma_{ \Omega }\bigl(
( Y_{ -1 }^{ ( \theta, l, i, \vartheta ) } )_{ \vartheta \in \Theta },
( Y_{ 0 }^{ ( \theta, l, i, \vartheta ) } )_{ \vartheta \in \Theta },
( Z^{ ( \theta, l, i, \vartheta ) } )_{ \vartheta \in \Theta },
( Y_{ -1 }^{ ( \theta, -l, i, \vartheta ) } )_{ \vartheta \in \Theta },
\\
\phantom{\sigma_{ \Omega }\bigl(}
( Y_{ 0 }^{ ( \theta, -l, i, \vartheta ) } )_{ \vartheta \in \Theta },
( Z^{ ( \theta, -l, i, \vartheta ) } )_{ \vartheta \in \Theta },
Z^{ ( \theta, l, i ) }
\bigr)
\end{array}
& \colon l \neq 0
\end{cases}
\end{equation*}
is independent.
This,
\eqref{eq:R_theta_l_i_definition},
and~(\ref{item:independence2})
ensure for every
$ \theta \in \Theta $
that
the family
\begin{equation}
\label{eq:R_theta_l_i_indep}
\bigl[
\N_0 \times \N
\ni
( l, i )
\mapsto
\bigl(
    Y^{ ( \theta, l, i ) }_{ l },
    Y^{ ( \theta, -l, i ) }_{ l - 1 },
    Z^{ ( \theta, l, i ) }
\bigr)
\bigr]
=
( R^{ \theta, l, i } )_{ ( l, i ) \in \N_0 \times \N }
\end{equation}
is independent.
This finishes the proof of~(\ref{item:independence4}).

Fifth,
we establish~(\ref{item:independence5}) by induction on $ n \in \N $.
For the base case $ n = 1 $ note that
the assumption that
$ Y_{ -1 }^{ \theta } $, $ \theta \in \Theta $,
are identically distributed,
the assumption that
$ Y_{ 0 }^{ \theta } $, $ \theta \in \Theta $,
are identically distributed,
the assumption that
$ Z^{ \theta } $, $ \theta \in \Theta $,
are identically distributed,
the assumption that
$ ( Y_{ -1 }^{ \theta } )_{ \theta \in \Theta } $,
$ ( Y_{ 0 }^{ \theta } )_{ \theta \in \Theta } $,
and
$ ( Z^{ \theta } )_{ \theta \in \Theta } $
are independent,
and~\eqref{eq:R_theta_l_i_definition}
establish for every
$ \theta \in \Theta $,
$ i \in \N $
that
\begin{equation}
\label{eq:R_theta_0_i_id}
R^{ \theta, 0, i }
=
\bigl(
    Y^{ ( \theta, 0, i ) }_{ 0 },
    Y^{ ( \theta, 0, i ) }_{ -1 },
    Z^{ ( \theta, 0, i ) }
\bigr)
\quad\text{and}\quad
(
Y^{ 0 }_{ 0 },
Y^{ 1 }_{ -1 },
Z^{ 0 }
)
\end{equation}
are identically distributed.
In particular, this shows for every
$ i \in \N $
that
$ R^{ \theta, 0, i } $, $ \theta \in \Theta $,
are identically distributed.
Combining this with~\eqref{eq:R_theta_l_i_indep}
proves that
$ \bigl( \Omega \ni \omega \mapsto ( R^{ \theta, 0, i }( \omega ) )_{ i \in \N }
\in ( \mathcal{Y} \times \mathcal{Y} \times \mathcal{Z} )^{ \N } \bigr) $,
$ \theta \in \Theta $,
are identically distributed.
The fact that
$ \forall \, \theta \in \Theta, \, \omega \in \Omega \colon
Y_{ 1 }^\theta( \omega )
=
\Psi_1\bigl( ( R^{ \theta, 0, i }( \omega ) )_{ i \in \N } \bigr) $
hence implies that
$ Y_{ 1 }^\theta $, $ \theta \in \Theta $,
are identically distributed.
This,
the assumption that
$ Y_{ -1 }^{ \theta } $, $ \theta \in \Theta $,
are identically distributed,
and
the assumption that
$ Y_{ 0 }^{ \theta } $, $ \theta \in \Theta $,
are identically distributed
show~(\ref{item:independence5}) in the base case $ n = 1 $.
For the induction step $ \N \ni n - 1 \to n \in \{ 2, 3, \ldots \} $
let $ n \in \{ 2, 3, \ldots \} $
and assume for every
$ l \in \{ -1, 0, 1, \ldots, n - 1 \} $
that
$ Y_{ l }^\theta $, $ \theta \in \Theta $,
are identically distributed.
This,
the assumption that
$ Z^{ \theta } $, $ \theta \in \Theta $,
are identically distributed,
(\ref{item:independence3}),
and~\eqref{eq:R_theta_l_i_definition}
ensure for every
$ \theta \in \Theta $,
$ l \in \{ 1, \ldots, n - 1 \} $,
$ i \in \N $
that
\begin{equation}
\label{eq:R_theta_l_i_id}
R^{ \theta, l, i }
=
\bigl(
    Y^{ ( \theta, l, i ) }_{ l },
    Y^{ ( \theta, -l, i ) }_{ l - 1 },
    Z^{ ( \theta, l, i ) }
\bigr)
\quad\text{and}\quad
(
Y^{ 0 }_{ l },
Y^{ 1 }_{ l - 1 },
Z^{ 0 }
) 
\end{equation}
are identically distributed.
Combining this and \eqref{eq:R_theta_0_i_id}
establishes for every
$ l \in \{ 0, 1, \ldots, n - 1 \} $,
$ i \in \N $
that
$ R^{ \theta, l, i } $, $ \theta \in \Theta $,
are identically distributed.
This and~\eqref{eq:R_theta_l_i_indep}
demonstrate that
$ \bigl( \Omega \ni \omega \mapsto ( R^{ \theta, l, i }( \omega ) )_{ ( l, i ) \in \{ 0, 1, \ldots, n - 1 \} \times \N }
\in ( \mathcal{Y} \times \mathcal{Y} \times \mathcal{Z} )^{ \{ 0, 1, \ldots, n - 1 \} \times \N } \bigr) $,
$ \theta \in \Theta $,
are identically distributed.
Therefore,
the fact that
$ \forall \, \theta \in \Theta, \, \omega \in \Omega \colon
Y_{ n }^\theta( \omega )
=
\Psi_n\bigl( ( R^{ \theta, l, i }( \omega ) )_{ ( l, i ) \in \{ 0, 1, \ldots, n - 1 \} \times \N } \bigr) $
shows that
$ Y_{ n }^\theta $, $ \theta \in \Theta $,
are identically distributed.
Induction hence proves~(\ref{item:independence5}).

Sixth,
observe that~\eqref{eq:R_theta_0_i_id} and~\eqref{eq:R_theta_l_i_id}
establish~(\ref{item:independence6}).
The proof of Proposition~\ref{prop:independence} is thus complete.
\end{proof}

\subsection{Error analysis}
\label{sec:error_analysis}

\begin{proposition}
\label{prop:error_analysis}
Let
$ ( \Omega, \mathscr{F}, \P ) $
be a probability space,
let $ ( \mathcal{Y}, \lVert \cdot \rVert_{ \mathcal{Y} } ) $
be a separable $ \R $-Banach space,
let
$ C, c \in ( 0, \infty ) $,
$ ( \mathfrak{c}_k )_{ k \in \N_0 } \subseteq ( 0, \infty ) $,
$ \Theta = \cup_{ n = 1 }^\infty \Z^n $,
$ y \in \mathcal{Y} $,
for every $ n \in \N $
let $ ( M_{ n, l } )_{ l \in \{ 0, 1, \ldots, n \} } \subseteq \N $
satisfy
$ M_{ n, 1 } \geq M_{ n, 2 } \geq \ldots \geq M_{ n, n } $,
let
$ ( \mathcal{Z}, \mathscr{Z} ) $ be a measurable space,
let
$ Z^{ \theta } \colon \Omega \to \mathcal{Z} $,
$ \theta \in \Theta $,
be i.i.d.\
$ \mathscr{F} $/$ \mathscr{Z} $-measurable functions,
let $ ( \mathcal{H}, \langle \cdot, \cdot \rangle_{ \mathcal{H} }, \lVert \cdot \rVert_{ \mathcal{H} } ) $
be a separable $ \R $-Hilbert space,
let
%$ \mathscr{S} \subseteq \{ \mathcal{S} \colon \mathcal{S} \subseteq L( \mathcal{Y}, \mathcal{H} ) \} $
%be the set given by
$ \mathscr{S}
=
\sigma_{ L( \mathcal{Y}, \mathcal{H} ) } \bigl( \bigl\{
    \{ \varphi \in L( \mathcal{Y}, \mathcal{H} ) \colon \varphi( x ) \in \mathcal{B} \} \subseteq L( \mathcal{Y}, \mathcal{H} ) \colon
    x \in \mathcal{Y},\, \mathcal{ B } \in \mathscr{B}( \mathcal{H} )
\bigr\} \bigr) $,
let
$ \psi_k \colon \Omega \to L( \mathcal{Y}, \mathcal{H} ) $, $ k \in \N_0 $,
be
$ \mathscr{F} $/$ \mathscr{S} $-measurable functions,
let
$ \Phi_l \colon \mathcal{Y} \times \mathcal{Y} \times \mathcal{Z} \to \mathcal{Y} $, $ l \in \N_0 $,
be
$ ( \mathscr{B}( \mathcal{Y} ) \otimes \mathscr{B}( \mathcal{Y} ) \otimes \mathscr{Z} ) $/$ \mathscr{B}( \mathcal{Y} ) $-measurable
functions,
let
$ Y_{ -1 }^\theta \colon \Omega \to \mathcal{Y} $, $ \theta \in \Theta $,
be i.i.d.\
$ \mathscr{F} $/$ \mathscr{B}( \mathcal{Y} ) $-measurable functions,
let
$ Y_{ 0 }^\theta \colon \Omega \to \mathcal{Y} $, $ \theta \in \Theta $,
be i.i.d.\
$ \mathscr{F} $/$ \mathscr{B}( \mathcal{Y} ) $-measurable functions,
assume that
$ ( Y_{ -1 }^{ \theta } )_{ \theta \in \Theta } $,
$ ( Y_{ 0 }^{ \theta } )_{ \theta \in \Theta } $,
$ ( Z^{ \theta } )_{ \theta \in \Theta } $,
and
$ ( \psi_k )_{ k \in \N_0 } $
are independent,
let
$ Y_{ n }^\theta \colon \Omega \to \mathcal{Y} $,
$ \theta \in \Theta $,
$ n \in \N $,
satisfy
for all
$ n \in \N $,
$ \theta \in \Theta $
that
\begin{equation}
\label{eq:EA_definition_Y}
Y_{ n }^\theta
=
\sum_{l=0}^{n-1}
\tfrac{1}{ M_{ n, l } }
\Biggl[
    \sum_{i=1}^{ M_{ n, l } }
        \Phi_{ l } \bigl(
        Y^{ ( \theta, l, i ) }_{ l },
        Y^{ ( \theta, -l, i ) }_{ l - 1 },
        Z^{ ( \theta, l, i ) }
        \bigr)
\Biggr]
,
\end{equation}
and
assume
for all
$ k \in \N_0 $,
$ n \in \N $
that
$ \E\bigl[
    \lVert 
        \Phi_{ k }(
        Y^{ 0 }_{ k },
        Y^{ 1 }_{ k - 1 },
        Z^{ 0 }
        )
    \rVert_{ \mathcal{Y} }
\bigr]
< \infty $
and
\begin{gather}
\label{eq:EA_hypothesis_I}
\max\bigl\{
\E\bigl[
\lVert \psi_k(
    \Phi_{ 0 }(
    Y_0^0,
    Y_{ -1 }^1,
    Z^{ 0 }
    )
) \rVert_{ \mathcal{H} }^2
\bigr]
,
\mathbbm{1}_{ \N }( k )
\, \E\bigl[
\lVert \psi_k(
    Y_0^0 - y
) \rVert_{ \mathcal{H} }^2
\bigr]
\bigr\}
\leq
\tfrac{ C^2 }{ \mathfrak{c}_k }
,
\\[0.5\baselineskip]
\label{eq:EA_hypothesis_II}
\E\bigl[
\lVert \psi_k(
    \Phi_{ n } (
    Y^{ 0 }_{ n },
    Y^{ 1 }_{ n - 1 },
    Z^{ 0 }
    )
) \rVert_{ \mathcal{H} }^2
\bigr]
\leq
c
\, \E\bigl[
\lVert \psi_{ k + 1 }(
    Y^{ 0 }_{ n }
    -
    Y^{ 1 }_{ n - 1 }
) \rVert_{ \mathcal{H} }^2
\bigr]
,
\\[0.5\baselineskip]
\label{eq:EA_hypothesis_III}
\E\Biggl[
\biggl\lVert \psi_{ k }
\biggl(
    y -
    \smallsum_{l=0}^{n-1}
        \E\bigl[
            \Phi_{ l }(
            Y^0_{ l },
            Y^1_{ l - 1 },
            Z^0
            )
        \bigr]
\biggr)
\biggr\rVert_{ \mathcal{H} }^2
\Biggr]
\leq
\tfrac{ 2 c }{ M_{ n, n } }
\, \E\bigl[
\lVert \psi_{ k + 1 }(
    Y_{ n - 1 }^0 - y
) \rVert_{ \mathcal{H} }^2
\bigr]
.
\end{gather}
Then it holds
for all
$ N \in \N $
that
\begin{equation}
\label{eq:EA_error_estimate}
\begin{split}
&
\bigl(
\E\bigl[
    \lVert \psi_0(
        Y_{ N }^0 - y
    ) \rVert_{ \mathcal{H} }^2
\bigr]
\bigr)^{ \nicefrac{1}{2} }
\\ &
\leq
C ( 1 + 4 c )^{ \nicefrac{N}{2} }
\biggl[
\min
\biggl(
\biggl\{
        \min\{ M_{ l_k, 0 } \mathfrak{c}_k, M_{ l_k, 1 } \mathfrak{c}_{ k + 1 } \}
        \smallprod_{ j = 1}^{ k }
        M_{ l_{ j - 1 }, l_j + 1 }
    \colon
\\ & \quad
    \begin{array}{c}
        k \in \N \cap [ 0, N - 1 ],
        \, ( l_i )_{ i \in \{ 0, 1, \ldots, k \} } \subseteq \\
        \{ 1, 2, \ldots, N \},
        \, N = l_0 > l_{ 1 } > \ldots > l_{ k }
    \!\! \end{array}
\biggr\}
\cup
\bigl\{
    \min\{ M_{ N, 0 } \mathfrak{c}_0, M_{ N, 1 } \mathfrak{c}_{ 1 } \}
\bigr\}
\biggr)
\biggr]^{ -\nicefrac{1}{2} }
< \infty
.
\end{split}
\end{equation}
\end{proposition}
\begin{proof}[Proof of Proposition~\ref{prop:error_analysis}]
First of all,
note that
the assumption that
$ \forall \, l \in \N_0 \colon
\E\bigl[
    \lVert 
        \Phi_{ l }(
        Y^{ 0 }_{ l },
\allowbreak
        Y^{ 1 }_{ l - 1 },
        Z^{ 0 }
        )
    \rVert_{ \mathcal{Y} }
\bigr]
< \infty $
and
(\ref{item:independence6})
in Proposition~\ref{prop:independence}
establish
for all
$ l \in \N_0 $,
$ i \in \N $
that
\begin{equation}
\label{eq:finite_moment_Phi}
\E\Bigl[
    \bigl\lVert
    \Phi_{ l } \bigl(
    Y^{ ( 0, l, i ) }_{ l },
    Y^{ ( 0, -l, i ) }_{ l - 1 },
    Z^{ ( 0, l, i ) }
    \bigr)
    \bigr\rVert_{ \mathcal{Y} }
\Bigr]
=
\E\bigl[
    \lVert 
        \Phi_{ l }(
        Y^{ 0 }_{ l },
        Y^{ 1 }_{ l - 1 },
        Z^{ 0 }
        )
    \rVert_{ \mathcal{Y} }
\bigr]
< \infty
.
\end{equation}
This,
(\ref{item:independence1})
in Proposition~\ref{prop:independence},
and
\eqref{eq:EA_definition_Y}
ensure
for all
$ n \in \N $
that
$ Y_{ n }^0 \colon \Omega \to \mathcal{Y} $
is an
$ \mathscr{F} $/$ \mathscr{B}( \mathcal{Y} ) $-measurable function
and
\begin{equation}
\label{eq:finite_moment_Y}
\begin{split}
\E\bigl[
    \lVert
        Y_{ n }^0
    \rVert_{ \mathcal{Y} }
\bigr]
& =
\E\Biggl[
    \biggl\lVert
    \smallsum_{l=0}^{n-1}
    \tfrac{1}{ M_{ n, l } }
    \biggl[
        \smallsum_{i=1}^{ M_{ n, l } }
            \Phi_{ l } \bigl(
            Y^{ ( 0, l, i ) }_{ l },
            Y^{ ( 0, -l, i ) }_{ l - 1 },
            Z^{ ( 0, l, i ) }
            \bigr)
    \biggr]
    \biggr\rVert_{ \mathcal{Y} }
\Biggr]
\\ &
\leq
\smallsum_{l=0}^{n-1}
\tfrac{1}{ M_{ n, l } }
\biggl[
    \smallsum_{i=1}^{ M_{ n, l } }
    \E\Bigl[
        \bigl\lVert
        \Phi_{ l } \bigl(
        Y^{ ( 0, l, i ) }_{ l },
        Y^{ ( 0, -l, i ) }_{ l - 1 },
        Z^{ ( 0, l, i ) }
        \bigr)
        \bigr\rVert_{ \mathcal{Y} }
    \Bigr]
\biggr]
\\ &
=
\smallsum_{l=0}^{n-1}
\tfrac{1}{ M_{ n, l } }
\biggl[
    \smallsum_{i=1}^{ M_{ n, l } }
    \E\bigl[
        \lVert
        \Phi_{ l }(
        Y^{ 0 }_{ l },
        Y^{ 1 }_{ l - 1 },
        Z^{ 0 }
        )
        \rVert_{ \mathcal{Y} }
    \bigr]
\biggr]
\\ &
=
\smallsum_{l=0}^{n-1}
    \E\bigl[
        \lVert
        \Phi_{ l }(
        Y^{ 0 }_{ l },
        Y^{ 1 }_{ l - 1 },
        Z^{ 0 }
        )
        \rVert_{ \mathcal{Y} }
    \bigr]
< \infty
.
\end{split}
\end{equation}
In addition,
(\ref{item:independence2})
in Proposition~\ref{prop:independence}
yields
for all
$ n \in \N $,
$ \theta \in \Theta $
that
\begin{equation}
\begin{split}
\sigma_{ \Omega }( Y_{ n }^\theta )
& \subseteq
\sigma_{ \Omega }\bigl(
( Y_{ -1 }^{ ( \theta, \vartheta ) } )_{ \vartheta \in \Theta },
( Y_{ 0 }^{ ( \theta, \vartheta ) } )_{ \vartheta \in \Theta },
( Z^{ ( \theta, \vartheta ) } )_{ \vartheta \in \Theta }
\bigr)
\\ &
\subseteq
\sigma_{ \Omega }\bigl(
( Y_{ -1 }^{ \vartheta } )_{ \vartheta \in \Theta },
( Y_{ 0 }^{ \vartheta } )_{ \vartheta \in \Theta },
( Z^{ \vartheta } )_{ \vartheta \in \Theta }
\bigr)
.
\end{split}
\end{equation}
Note that
this implies that
$
\sigma_{ \Omega }\bigl(
    ( Y_{ n }^\theta )_{ ( n, \theta ) \in ( \N_0 \cup \{ - 1 \} ) \times \Theta },
    ( Z^{ \theta } )_{ \theta \in \Theta }
\bigr)
\subseteq
\sigma_{ \Omega }\bigl(
( Y_{ -1 }^{ \theta } )_{ \theta \in \Theta },
( Y_{ 0 }^{ \theta } )_{ \theta \in \Theta },
\allowbreak
( Z^{ \theta } )_{ \theta \in \Theta }
\bigr)
$.
This and
the assumption that
$ ( Y_{ -1 }^{ \theta } )_{ \theta \in \Theta } $,
$ ( Y_{ 0 }^{ \theta } )_{ \theta \in \Theta } $,
$ ( Z^{ \theta } )_{ \theta \in \Theta } $,
and
$ ( \psi_k )_{ k \in \N_0 } $
are independent
demonstrate
for every
$ k \in \N_0 $
that
\begin{equation}
\label{eq:Y_psi_independent}
\sigma_{ \Omega }\bigl(
    ( Y_{ n }^\theta )_{ ( n, \theta ) \in ( \N_0 \cup \{ - 1 \} ) \times \Theta },
    ( Z^{ \theta } )_{ \theta \in \Theta }
\bigr)
\quad
\text{and}
\quad
\psi_k
\end{equation}
are independent.
Corollary~\ref{cor:decomposition}
and
\eqref{eq:finite_moment_Y}
hence
show
for all
$ k \in \N_0 $,
$ n \in \N $
that
\begin{equation}
\label{eq:decomposition}
\E\bigl[
    \lVert \psi_k(
        Y_{ n }^0 - y
    ) \rVert_{ \mathcal{H} }^2
\bigr]
=
\E\bigl[
    \lVert \psi_k(
        Y_{ n }^0 - \E[ Y_{ n }^0 ]
    ) \rVert_{ \mathcal{H} }^2
\bigr]
+
\E\bigl[
    \lVert \psi_k(
        \E[ Y_{ n }^0 ] - y
    ) \rVert_{ \mathcal{H} }^2
\bigr]
.
\end{equation}
Next observe that
\eqref{eq:EA_definition_Y},
\eqref{eq:finite_moment_Phi},
(\ref{item:independence4}) in Proposition~\ref{prop:independence},
\eqref{eq:Y_psi_independent},
and
Corollary~\ref{cor:variances}
(with
$ n = \sum_{l=0}^{\mathfrak{n}-1} M_{ \mathfrak{n}, l } $,
%$ Y_{ i + \sum_{\ell=0}^{l-1} M_{ \mathfrak{n}, \ell } }
%=
%\tfrac{1}{ M_{ \mathfrak{n}, l } }
%\Phi_{ l }( Y^{ ( 0, l, i ) }_{ l }, Y^{ ( 0, -l, i ) }_{ l - 1 }, Z^{ ( 0, l, i ) } ) $,
$ \psi = \psi_k $
for
%$ i \in \{ 1, 2, \ldots, M_{ \mathfrak{n}, l } \} $,
%$ l \in \{ 0, 1, \ldots, \mathfrak{n} - 1 \} $,
$ \mathfrak{n} \in \N $,
$ k \in \N_0 $
in the notation of Corollary~\ref{cor:variances})
prove
for all
$ k \in \N_0 $,
$ n \in \N $
that
\begin{align*}
\label{eq:variance_Y1}
&
\E\bigl[
    \lVert \psi_k(
        Y_{ n }^0 - \E[ Y_{ n }^0 ]
    ) \rVert_{ \mathcal{H} }^2
\bigr]
\\ &
=
\E\Biggl[
\biggl\lVert \psi_k\biggl(
\smallsum_{l=0}^{n-1}
    \tfrac{1}{ M_{ n, l } }
    \smallsum_{i=1}^{ M_{ n, l } }
    \bigl(
        \Phi_{ l } \bigl(
        Y^{ ( 0, l, i ) }_{ l },
        Y^{ ( 0, -l, i ) }_{ l - 1 },
        Z^{ ( 0, l, i ) }
        \bigr)
        -
        \E\bigl[
            \Phi_{ l } \bigl(
            Y^{ ( 0, l, i ) }_{ l },
            Y^{ ( 0, -l, i ) }_{ l - 1 },
            Z^{ ( 0, l, i ) }
            \bigr)
        \bigr]
    \bigr)
\biggr) \biggr\rVert_{ \mathcal{H} }^2
\Biggr]
\\ &
=
\E\Biggl[
\biggl\lVert
\smallsum_{l=0}^{n-1}
    \smallsum_{i=1}^{ M_{ n, l } }
    \psi_k\Bigl(
        \tfrac{1}{ M_{ n, l } }\bigl(
        \Phi_{ l } \bigl(
        Y^{ ( 0, l, i ) }_{ l },
        Y^{ ( 0, -l, i ) }_{ l - 1 },
        Z^{ ( 0, l, i ) }
        \bigr)
        -
        \E\bigl[
            \Phi_{ l } \bigl(
            Y^{ ( 0, l, i ) }_{ l },
            Y^{ ( 0, -l, i ) }_{ l - 1 },
            Z^{ ( 0, l, i ) }
            \bigr)
        \bigr]
    \bigr) \Bigr)
\biggr\rVert_{ \mathcal{H} }^2
\Biggr]
\\ &
=
\smallsum_{l=0}^{n-1}
\smallsum_{i=1}^{ M_{ n, l } }
\E\Bigl[
\bigl\lVert
\tfrac{1}{ M_{ n, l } }
    \psi_k\bigl(
        \Phi_{ l } \bigl(
        Y^{ ( 0, l, i ) }_{ l },
        Y^{ ( 0, -l, i ) }_{ l - 1 },
        Z^{ ( 0, l, i ) }
        \bigr)
        -
        \E\bigl[
            \Phi_{ l } \bigl(
            Y^{ ( 0, l, i ) }_{ l },
            Y^{ ( 0, -l, i ) }_{ l - 1 },
            Z^{ ( 0, l, i ) }
            \bigr)
        \bigr]
    \bigr)
\bigr\rVert_{ \mathcal{H} }^2
\Bigr]
\\ & \yesnumber
=
\smallsum_{l=0}^{n-1}
\tfrac{1}{ ( M_{ n, l } )^2 }
\biggl[
\smallsum_{i=1}^{ M_{ n, l } }
\E\Bigl[
\bigl\lVert
    \psi_k\bigl(
        \Phi_{ l } \bigl(
        Y^{ ( 0, l, i ) }_{ l },
        Y^{ ( 0, -l, i ) }_{ l - 1 },
        Z^{ ( 0, l, i ) }
        \bigr)
\\ & \quad
\hphantom{
\smallsum_{l=0}^{n-1}
\tfrac{1}{ ( M_{ n, l } )^2 }
\biggl[
\smallsum_{i=1}^{ M_{ n, l } }
\E\Bigl[
\bigl\lVert
    \psi_k\bigl(
}
        \mathop{-}
        \E\bigl[
            \Phi_{ l } \bigl(
            Y^{ ( 0, l, i ) }_{ l },
            Y^{ ( 0, -l, i ) }_{ l - 1 },
            Z^{ ( 0, l, i ) }
            \bigr)
        \bigr]
    \bigr)
\bigr\rVert_{ \mathcal{H} }^2
\Bigr]
\biggr]
.
\end{align*}
Moreover,
the fact that it holds
for all
$ k \in \N_0 $,
$ x \in \mathcal{Y} $
that
$ \Omega \ni \omega
\mapsto [ \psi_k( \omega ) ]( x ) \in \mathcal{H} $
is an
$ \mathscr{F} $/$ \mathscr{B}( \mathcal{H} ) $-measurable function
(cf.~Lemma~\ref{lem:strong_sigma_algebra})
and
the fact that it holds
for all
$ k \in \N_0 $,
$ \omega \in \Omega $
that
$ \mathcal{Y} \ni x
\mapsto [ \psi_k( \omega ) ]( x ) \in \mathcal{H} $
is a continuous function
demonstrate that
\begin{equation}
\label{eq:random_field}
\mathcal{Y} \times \Omega \ni ( x, \omega )
\mapsto [ \psi_k( \omega ) ]( x ) \in \mathcal{H} 
\end{equation}
is a continuous random field.
This,
\eqref{eq:Y_psi_independent},
(\ref{item:independence6}) in Proposition~\ref{prop:independence},
and
Hutzenthaler, Jentzen, \& von Wurstemberger~\cite[Lemma~3.5]{HutzenthalerJentzenVonWurstemberger2019arXiv}
(with
$ S = \mathcal{Y} $,
$ E = \mathcal{H} $,
$ U = V
= ( \mathcal{Y} \times \Omega \ni ( x, \omega )
\mapsto [ \psi_k( \omega ) ]( x ) \in \mathcal{H} ) $,
$ X = \Phi_{ l }( Y^{ ( 0, l, i ) }_{ l }, Y^{ ( 0, -l, i ) }_{ l - 1 }, Z^{ ( 0, l, i ) } ) $,
$ Y = \Phi_{ l }( Y^{ 0 }_{ l }, Y^{ 1 }_{ l - 1 }, Z^{ 0 } ) $
for
$ i \in \N $,
$ l, k \in \N_0 $
in the notation of \cite[Lemma~3.5]{HutzenthalerJentzenVonWurstemberger2019arXiv})
ensure
for all
$ k, l \in \N_0 $,
$ i \in \N $
that
\begin{equation}
\begin{split}
\Omega \ni \omega & \mapsto
[ \psi_k( \omega ) ]\bigl(
    \Phi_{ l }\bigl(
    Y^{ ( 0, l, i ) }_{ l }( \omega ),
    Y^{ ( 0, -l, i ) }_{ l - 1 }( \omega ),
    Z^{ ( 0, l, i ) }( \omega )
    \bigr)
\bigr) \in \mathcal{H}
\quad
\text{and}
\\
 \Omega \ni \omega & \mapsto
[ \psi_k( \omega ) ]\bigl(
    \Phi_{ l }(
    Y^{ 0 }_{ l }( \omega ),
    Y^{ 1 }_{ l - 1 }( \omega ),
    Z^{ 0 }( \omega )
    )
\bigr) \in \mathcal{H}
\end{split}
\end{equation}
are
identically distributed.
This,
\eqref{eq:variance_Y1},
\eqref{eq:finite_moment_Phi},
\eqref{eq:Y_psi_independent},
and
Corollary~\ref{cor:decomposition}
(with
$ y = 0 $,
$ Y = \Phi_{ l }( Y^{ ( 0, l, i ) }_{ l }, Y^{ ( 0, -l, i ) }_{ l - 1 }, Z^{ ( 0, l, i ) } ) $,
$ \psi = \psi_k $
for
$ i \in \{ 1, 2, \ldots, M_{ n, l } \} $,
$ l \in \{ 0, 1, \ldots, n - 1 \} $,
$ n \in \N $,
$ k \in \N_0 $
in the notation of Corollary~\ref{cor:decomposition})
imply
for all
$ k \in \N_0 $,
$ n \in \N $
that
\begin{equation}
\label{eq:decomposition2}
\begin{split}
&
\E\bigl[
    \lVert \psi_k(
        Y_{ n }^0 - \E[ Y_{ n }^0 ]
    ) \rVert_{ \mathcal{H} }^2
\bigr]
\leq
\smallsum_{l=0}^{n-1}
\tfrac{1}{ ( M_{ n, l } )^2 }
\biggl[
\smallsum_{i=1}^{ M_{ n, l } }
\E\Bigl[
\bigl\lVert
    \psi_k\bigl(
        \Phi_{ l } \bigl(
        Y^{ ( 0, l, i ) }_{ l },
        Y^{ ( 0, -l, i ) }_{ l - 1 },
        Z^{ ( 0, l, i ) }
        \bigr)
    \bigr)
\bigr\rVert_{ \mathcal{H} }^2
\Bigr]
\biggr]
\\ &
=
\smallsum_{l=0}^{n-1}
\tfrac{1}{ ( M_{ n, l } )^2 }
\biggl[
\smallsum_{i=1}^{ M_{ n, l } }
\E\bigl[
\lVert
    \psi_k(
        \Phi_{ l }(
        Y^{ 0 }_{ l },
        Y^{ 1 }_{ l - 1 },
        Z^{ 0 }
        ))
\rVert_{ \mathcal{H} }^2
\bigr]
\biggr]
\\ &
=
\smallsum_{l=0}^{n-1}
\bigl(
\tfrac{1}{ M_{ n, l } }
\, \E\bigl[
\lVert
    \psi_k(
        \Phi_{ l }(
        Y^{ 0 }_{ l },
        Y^{ 1 }_{ l - 1 },
        Z^{ 0 }
        ))
\rVert_{ \mathcal{H} }^2
\bigr]
\bigr)
\\ &
=
\tfrac{1}{ M_{ n, 0 } }
\, \E\bigl[
\lVert
    \psi_k(
        \Phi_{ 0 }(
        Y^{ 0 }_{ 0 },
        Y^{ 1 }_{ -1 },
        Z^{ 0 }
        ))
\rVert_{ \mathcal{H} }^2
\bigr]
+
\smallsum_{l=1}^{n-1}
\bigl(
\tfrac{1}{ M_{ n, l } }
\, \E\bigl[
\lVert
    \psi_k(
        \Phi_{ l }(
        Y^{ 0 }_{ l },
        Y^{ 1 }_{ l - 1 },
        Z^{ 0 }
        ))
\rVert_{ \mathcal{H} }^2
\bigr]
\bigr)
.
\end{split}
\end{equation}
Assumptions~\eqref{eq:EA_hypothesis_I}--\eqref{eq:EA_hypothesis_II},
the fact that
$ \forall \, a, b \in \R \colon
( a + b )^2 \leq 2 ( a^2 + b^2 ) $,
\eqref{eq:random_field},
\eqref{eq:Y_psi_independent},
(\ref{item:independence5}) in Proposition~\ref{prop:independence},
\cite[Lemma~3.5]{HutzenthalerJentzenVonWurstemberger2019arXiv}
(with
$ S = \mathcal{Y} $,
$ E = \mathcal{H} $,
$ U = V
= ( \mathcal{Y} \times \Omega \ni ( x, \omega )
\mapsto [ \psi_k( \omega ) ]( x ) \in \mathcal{H} ) $,
$ X = Y^{ 1 }_{ l } - y $,
$ Y = Y^{ 0 }_{ l } - y $
for
$ l, k \in \N_0 $
in the notation of \cite[Lemma~3.5]{HutzenthalerJentzenVonWurstemberger2019arXiv}),
and
the assumption that
$ \forall \, n \in \N \colon
M_{ n, 1 } \geq M_{ n, 2 } \geq \ldots \geq M_{ n, n } $
hence
prove
for all
$ k \in \N_0 $,
$ n \in \N $
that
\begin{equation}
\label{eq:variance_Y}
\begin{split}
& \E\bigl[
    \lVert \psi_k(
        Y_{ n }^0 - \E[ Y_{ n }^0 ]
    ) \rVert_{ \mathcal{H} }^2
\bigr]
\leq
\tfrac{ C^2 }{ M_{ n, 0 } \mathfrak{c}_k }
+
\biggl[
\smallsum_{l=1}^{n-1}
\tfrac{ c }{ M_{ n, l } }
\, \E\bigl[
    \lVert \psi_{ k + 1 }(
        Y^{ 0 }_{ l } - Y^{ 1 }_{ l - 1 }
    ) \rVert_{ \mathcal{H} }^2
\bigr]
\biggr]
\\ &
=
\tfrac{ C^2 }{ M_{ n, 0 } \mathfrak{c}_k }
+
\biggl[
\smallsum_{l=1}^{n-1}
\tfrac{ c }{ M_{ n, l } }
\, \E\bigl[
    \lVert \psi_{ k + 1 }(
        Y^{ 0 }_{ l } - y
    )
    +
    \psi_{ k + 1 }(
            y - Y^{ 1 }_{ l - 1 }
    ) \rVert_{ \mathcal{H} }^2
\bigr]
\biggr]
\\ &
\leq
\tfrac{ C^2 }{ M_{ n, 0 } \mathfrak{c}_k }
+
\biggl[
\smallsum_{l=1}^{n-1}
\tfrac{ c }{ M_{ n, l } }
\, \E\bigl[
    ( \lVert \psi_{ k + 1 }(
        Y^{ 0 }_{ l } - y
    ) \rVert_{ \mathcal{H} }
    +
    \lVert \psi_{ k + 1 }(
        Y^{ 1 }_{ l - 1 } - y
    ) \rVert_{ \mathcal{H} } )^2
\bigr]
\biggr]
\\ &
\leq
\tfrac{ C^2 }{ M_{ n, 0 } \mathfrak{c}_k }
+
\biggl[
\smallsum_{l=1}^{n-1}
\tfrac{ 2 c }{ M_{ n, l } }
\, \E\bigl[
    \lVert \psi_{ k + 1 }(
        Y^{ 0 }_{ l } - y
    ) \rVert_{ \mathcal{H} }^2
\bigr]
\biggr]
+
\biggl[
\smallsum_{l=1}^{n-1}
\tfrac{ 2 c }{ M_{ n, l } }
\, \E\bigl[
    \lVert \psi_{ k + 1 }(
        Y^{ 1 }_{ l - 1 } - y
    ) \rVert_{ \mathcal{H} }^2
\bigr]
\biggr]
\\ &
=
\tfrac{ C^2 }{ M_{ n, 0 } \mathfrak{c}_k }
+
\biggl[
\smallsum_{l=1}^{n-1}
\tfrac{ 2 c }{ M_{ n, l } }
\, \E\bigl[
    \lVert \psi_{ k + 1 }(
        Y^{ 0 }_{ l } - y
    ) \rVert_{ \mathcal{H} }^2
\bigr]
\biggr]
+
\biggl[
\smallsum_{l=0}^{n-2}
\tfrac{ 2 c }{ M_{ n, l + 1 } }
\, \E\bigl[
    \lVert \psi_{ k + 1 }(
        Y^{ 0 }_{ l } - y
    ) \rVert_{ \mathcal{H} }^2
\bigr]
\biggr]
\\ &
\leq
\tfrac{ C^2 }{ M_{ n, 0 } \mathfrak{c}_k }
+
\biggl[
\smallsum_{l=1}^{n-1}
\tfrac{ 2 c }{ M_{ n, l + 1 } }
\, \E\bigl[
    \lVert \psi_{ k + 1 }(
        Y^{ 0 }_{ l } - y
    ) \rVert_{ \mathcal{H} }^2
\bigr]
\biggr]
+
\biggl[
\smallsum_{l=0}^{n-2}
\tfrac{ 2 c }{ M_{ n, l + 1 } }
\, \E\bigl[
    \lVert \psi_{ k + 1 }(
        Y^{ 0 }_{ l } - y
    ) \rVert_{ \mathcal{H} }^2
\bigr]
\biggr]
\\ &
=
\tfrac{ C^2 }{ M_{ n, 0 } \mathfrak{c}_k }
+
\biggl[
\smallsum_{l=0}^{n-1}
\tfrac{ 2 ( 2 - \mathbbm{1}_{ \{ 0 \} }( l ) - \mathbbm{1}_{ \{ n - 1 \} }( l ) ) c }{ M_{ n, l + 1 } }
\, \E\bigl[
    \lVert \psi_{ k + 1 }(
        Y^{ 0 }_{ l } - y
    ) \rVert_{ \mathcal{H} }^2
\bigr]
\biggr]
.
\end{split}
\end{equation}
Furthermore,
note that
\eqref{eq:finite_moment_Y},
\eqref{eq:EA_definition_Y},
\eqref{eq:finite_moment_Phi},
and
(\ref{item:independence6}) in Proposition~\ref{prop:independence}
show
for all
$ n \in \N $
that
\begin{equation}
\label{eq:expectation_Y_n}
\begin{split}
\E[ Y_{ n }^0 ]
& =
\E\biggl[
\smallsum_{l=0}^{n-1}
    \tfrac{1}{ M_{ n, l } }
    \biggl[
    \smallsum_{i=1}^{ M_{ n, l } }
        \Phi_{ l } \bigl(
        Y^{ ( 0, l, i ) }_{ l },
        Y^{ ( 0, -l, i ) }_{ l - 1 },
        Z^{ ( 0, l, i ) }
        \bigr)
    \biggr]
\biggr]
\\ &
=
\smallsum_{l=0}^{n-1}
    \tfrac{1}{ M_{ n, l } }
    \biggl[
    \smallsum_{i=1}^{ M_{ n, l } }
    \E\bigl[
        \Phi_{ l } \bigl(
        Y^{ ( 0, l, i ) }_{ l },
        Y^{ ( 0, -l, i ) }_{ l - 1 },
        Z^{ ( 0, l, i ) }
        \bigr)
    \bigr]
    \biggr]
\\ &
=
\smallsum_{l=0}^{n-1}
    \tfrac{1}{ M_{ n, l } }
    \biggl[
    \smallsum_{i=1}^{ M_{ n, l } }
    \E\bigl[
        \Phi_{ l }(
        Y^0_{ l },
        Y^1_{ l - 1 },
        Z^0 )
    \bigr]
    \biggr]
=
\smallsum_{l=0}^{n-1}
    \E\bigl[
        \Phi_{ l }(
        Y^0_{ l },
        Y^1_{ l - 1 },
        Z^0 )
    \bigr]
.
\end{split}
\end{equation}
This and assumption~\eqref{eq:EA_hypothesis_III}
establish
for all
$ k \in \N_0 $,
$ n \in \N $
that
\begin{equation}
\label{eq:hypothesis_III_consequence}
\begin{split}
\E\bigl[
    \lVert \psi_k(
        \E[ Y_{ n }^0 ] - y
    ) \rVert_{ \mathcal{H} }^2
\bigr]
& =
\E\Biggl[ \biggl\lVert \psi_{ k }
\biggl(
    \biggl[
    \smallsum_{l=0}^{n-1}
        \E\bigl[
            \Phi_{ l } (
            Y^0_{ l },
            Y^1_{ l - 1 },
            Z^0
            )
        \bigr]
    \biggr]
    -
    y
\biggr)
\biggr\rVert_{ \mathcal{H} }^2 \Biggr]
\\ &
\leq
\tfrac{ 2 c }{ M_{ n, n } }
\, \E\bigl[ \lVert \psi_{ k + 1 } (
    Y_{ n - 1 }^0 - y
) \rVert_{ \mathcal{H} }^2 \bigr]
.
\end{split}
\end{equation}
Combining
\eqref{eq:decomposition}
with
\eqref{eq:variance_Y}
and
assumption~\eqref{eq:EA_hypothesis_I}
hence
proves
for all
$ k \in \N_0 $,
$ n \in \N $
that
\begin{equation}
\label{eq:recursion}
\begin{split}
& \E\bigl[
    \lVert \psi_k(
        Y_{ n }^0 - y
    ) \rVert_{ \mathcal{H} }^2
\bigr]
\leq
\tfrac{ C^2 }{ M_{ n, 0 } \mathfrak{c}_k }
+
\biggl[
\smallsum_{l=0}^{n-1}
\tfrac{ 2 ( 2 - \mathbbm{1}_{ \{ 0 \} }( l ) ) c }{ M_{ n, l + 1 } }
\, \E\bigl[
    \lVert \psi_{ k + 1 }(
        Y^{ 0 }_{ l } - y
    ) \rVert_{ \mathcal{H} }^2
\bigr]
\biggr]
\\ &
=
\tfrac{ C^2 }{ M_{ n, 0 } \mathfrak{c}_k }
+
\tfrac{ 2 c }{ M_{ n, 1 } }
\, \E\bigl[
    \lVert \psi_{ k + 1 }(
        Y^{ 0 }_{ 0 } - y
    ) \rVert_{ \mathcal{H} }^2
\bigr]
+
\biggl[
\smallsum_{l=1}^{n-1}
\tfrac{ 4 c }{ M_{ n, l + 1 } }
\, \E\bigl[
    \lVert \psi_{ k + 1 }(
        Y^{ 0 }_{ l } - y
    ) \rVert_{ \mathcal{H} }^2
\bigr]
\biggr]
\\ &
\leq
\tfrac{ C^2 }{ M_{ n, 0 } \mathfrak{c}_k }
+
\tfrac{ 2 C^2 c }{ M_{ n, 1 } \mathfrak{c}_{ k + 1 } }
+
\biggl[
\smallsum_{l=1}^{n-1}
\tfrac{ 4 c }{ M_{ n, l + 1 } }
\, \E\bigl[
    \lVert \psi_{ k + 1 }(
        Y^{ 0 }_{ l } - y
    ) \rVert_{ \mathcal{H} }^2
\bigr]
\biggr]
.
\end{split}
\end{equation}
Next we introduce some additional notation.
For the remainder of this proof
let
$ N \in \N $,
let
$ \varepsilon_n \in [ 0, \infty ] $,
$ n \in \{ 1, 2, \ldots, N \} $,
satisfy
for all
$ n \in \{ 1, 2, \ldots, N \} $
that
\begin{equation}
\label{eq:definition_varepsilon}
\begin{split}
&
\varepsilon_n
=
\max
\Biggl(
\Biggl\{
    \biggl(
        M_{ l_{ k - 1 }, n + 1 }
        \smallprod_{ j = 1 }^{ k - 1 }
        M_{ l_{ j - 1 }, l_j + 1 }
    \biggr)^{ -1 }
    \E\bigl[
        \lVert \psi_k(
            Y_{ n }^0 - y
        ) \rVert_{ \mathcal{H} }^2
    \bigr]
    \colon
\\ &
    \begin{array}{c}
        k \in \N \cap [ 0, N - n ],
        \, ( l_i )_{ i \in \{ 0, 1, \ldots, k - 1 \} } \subseteq \{ n + 1, \!\! \\
        n + 2, \ldots, N \},
        \, N = l_0 > l_{ 1 } > \ldots > l_{ k - 1 }
    \end{array}
\Biggr\}
\cup
\bigl\{
    \mathbbm{1}_{ \{ N \} }( n )
    \, \E\bigl[
        \lVert \psi_0(
            Y_{ N }^0 - y
        ) \rVert_{ \mathcal{H} }^2
    \bigr]
\bigr\}
\Biggr)
,
\end{split}
\end{equation}
and
let
$ a \in [ 0, \infty ) $
be
given by
\begin{equation}
\label{eq:definition_a}
\begin{split}
a
& =
C^2 ( 1 + 2 c )
\biggl[
\min
\biggl(
\biggl\{
        \min\{ M_{ l_k, 0 } \mathfrak{c}_k, M_{ l_k, 1 } \mathfrak{c}_{ k + 1 } \}
        \smallprod_{ j = 1}^{ k }
        M_{ l_{ j - 1 }, l_j + 1 }
    \colon
\\ & \quad
    \begin{array}{c}
        k \in \N \cap [ 0, N - 1 ],
        \, ( l_i )_{ i \in \{ 0, 1, \ldots, k \} } \subseteq \\
        \{ 1, 2, \ldots, N \},
        \, N = l_0 > l_{ 1 } > \ldots > l_{ k }
    \!\! \end{array}
\biggr\}
\cup
\bigl\{
    \min\{ M_{ N, 0 } \mathfrak{c}_0, M_{ N, 1 } \mathfrak{c}_{ 1 } \}
\bigr\}
\biggr)
\biggr]^{ -1 }
.
\end{split}
\end{equation}
Observe that
\eqref{eq:recursion}%and~\eqref{eq:definition_varepsilon}
--\eqref{eq:definition_a}
establish
for all
$ n \in \N \cap [ 0, N - 1 ] $,
$ k \in \{ 1, 2, \ldots, N - n \} $,
$ ( l_i )_{ i \in \{ 0, 1, \ldots, k - 1 \} } \subseteq \{ n + 1, n + 2, \ldots, N \} $
with
$ l_0 = N $
and
$ \forall \, i \in \N \cap [ 0, k - 1 ] \colon
l_{ i - 1 } > l_{ i } $
that
\begin{align*}
\label{eq:varepsilon_estimate}
& \biggl(
    M_{ l_{ k - 1 }, n + 1 }
    \smallprod_{ j = 1}^{ k - 1 }
    M_{ l_{ j - 1 }, l_j + 1 }
\biggr)^{ -1 }
\E\bigl[
    \lVert \psi_k(
        Y_{ n }^0 - y
    ) \rVert_{ \mathcal{H} }^2
\bigr]
\\ &
\leq
\Bigl(
    \tfrac{ C^2 }{ M_{ n, 0 } \mathfrak{c}_k }
    +
    \tfrac{ 2 C^2 c }{ M_{ n, 1 } \mathfrak{c}_{ k + 1 } }
\Bigr)
\biggl(
    M_{ l_{ k - 1 }, n + 1 }
    \smallprod_{ j = 1}^{ k - 1 }
    M_{ l_{ j - 1 }, l_j + 1 }
\biggr)^{ -1 }
\\ & \quad
+
4 c
\smallsum_{\ell=1}^{n-1}
\Biggl[
\biggl(
    M_{ n, \ell + 1 }
    M_{ l_{ k - 1 }, n + 1 }
    \smallprod_{ j = 1}^{ k - 1 }
    M_{ l_{ j - 1 }, l_j + 1 }
\biggr)^{ -1 }
\E\bigl[
    \lVert \psi_{ k + 1 }(
        Y^{ 0 }_{ \ell } - y
    ) \rVert_{ \mathcal{H} }^2
\bigr]
\Biggr]
\\ &
\leq
C^2 ( 1 + 2 c )
\Biggl[
\max_{ l_k \in \{ 1, 2, \ldots, l_{ k - 1 } - 1 \} }
\biggl(
    \min\{ M_{ l_k, 0 } \mathfrak{c}_k, M_{ l_k, 1 } \mathfrak{c}_{ k + 1 } \}
    M_{ l_{ k - 1 }, l_k + 1 }
    \smallprod_{ j = 1}^{ k - 1 }
    M_{ l_{ j - 1 }, l_j + 1 }
\biggr)^{ -1 }
\Biggr]
\\ & \quad
+
4 c
\smallsum_{\ell=1}^{n-1}
\max_{
l_k \in \{ \ell + 1, \ell + 2, \ldots, l_{ k - 1 } - 1 \}
}
\Biggl[
\biggl(
    M_{ l_k, \ell + 1 }
    M_{ l_{ k - 1 }, l_k + 1 }
    \smallprod_{ j = 1}^{ k - 1 }
    M_{ l_{ j - 1 }, l_j + 1 }
\biggr)^{ -1 }
\E\bigl[
    \lVert \psi_{ k + 1 }(
        Y^{ 0 }_{ \ell } - y
    ) \rVert_{ \mathcal{H} }^2
\bigr]
\Biggr]
\\ & \yesnumber
=
C^2 ( 1 + 2 c )
\biggl[
\min_{ l_k \in \{ 1, 2, \ldots, l_{ k - 1 } - 1 \} }
\biggl(
    \min\{ M_{ l_k, 0 } \mathfrak{c}_k, M_{ l_k, 1 } \mathfrak{c}_{ k + 1 } \}
    \smallprod_{ j = 1}^{ k }
    M_{ l_{ j - 1 }, l_j + 1 }
\biggr)
\biggr]^{ -1 }
\\ & \quad
+
4 c
\smallsum_{\ell=1}^{n-1}
\max_{
l_k \in \{ \ell + 1, \ell + 2, \ldots, l_{ k - 1 } - 1 \}
}
\Biggl[
\biggl(
    M_{ l_k, \ell + 1 }
    \smallprod_{ j = 1}^{ k }
    M_{ l_{ j - 1 }, l_j + 1 }
\biggr)^{ -1 }
\E\bigl[
    \lVert \psi_{ k + 1 }(
        Y^{ 0 }_{ \ell } - y
    ) \rVert_{ \mathcal{H} }^2
\bigr]
\Biggr]
\\ & \leq
a
+
4 c
\smallsum_{\ell=1}^{n-1}
\varepsilon_\ell
.
\end{align*}
In addition,
\eqref{eq:recursion}%and~\eqref{eq:definition_varepsilon}
--\eqref{eq:definition_a}
ensure that
\begin{equation}
\begin{split}
&
\varepsilon_N
=
\E\bigl[
    \lVert \psi_0(
        Y_{ N }^0 - y
    ) \rVert_{ \mathcal{H} }^2
\bigr]
\leq
\tfrac{ C^2 }{ M_{ N, 0 } \mathfrak{c}_0 }
+
\tfrac{ 2 C^2 c }{ M_{ N, 1 } \mathfrak{c}_{ 1 } }
+
\biggl[
\smallsum_{\ell=1}^{N-1}
\tfrac{ 4 c }{ M_{ N, \ell + 1 } }
\, \E\bigl[
    \lVert \psi_{ 1 }(
        Y^{ 0 }_{ \ell } - y
    ) \rVert_{ \mathcal{H} }^2
\bigr]
\biggr]
\\ &
\leq
C^2 ( 1 + 2 c )
( \min\{ M_{ N, 0 } \mathfrak{c}_0, M_{ N, 1 } \mathfrak{c}_{ 1 } \} )^{ -1 }
+
4 c
\biggl[
\smallsum_{\ell=1}^{N-1}
( M_{ N, \ell + 1 } )^{ -1 }
\, \E\bigl[
    \lVert \psi_{ 1 }(
        Y^{ 0 }_{ \ell } - y
    ) \rVert_{ \mathcal{H} }^2
\bigr]
\biggr]
\\ &
\leq
a
+
4 c
\smallsum_{\ell=1}^{N-1}
\varepsilon_\ell
.
\end{split}
\end{equation}
This,
\eqref{eq:definition_varepsilon},
and
\eqref{eq:varepsilon_estimate}
show
for all
$ n \in \{ 1, 2, \ldots, N \} $
that
\begin{equation}
\varepsilon_n
\leq
a
+
4 c
\smallsum_{\ell=1}^{n-1}
\varepsilon_\ell
.
\end{equation}
The fact that
$ a + c < \infty $
and
the discrete Gronwall-type inequality in
Agarwal~\cite[Corollary~4.1.2]{Agarwal2000}
hence establish
for all
$ n \in \{ 1, 2, \ldots, N \} $
that
\begin{equation}
\varepsilon_n
\leq
a ( 1 + 4 c )^{ n - 1 }
< \infty
.
\end{equation}
This
and \eqref{eq:definition_varepsilon}--\eqref{eq:definition_a}
imply that
\begin{equation}
\begin{split}
&
\E\bigl[
    \lVert \psi_0(
        Y_{ N }^0 - y
    ) \rVert_{ \mathcal{H} }^2
\bigr]
=
\varepsilon_N
\leq
a ( 1 + 4 c )^{ N - 1 }
\\ &
\leq
C^2 ( 1 + 4 c )^N
\biggl[
\min
\biggl(
\biggl\{
        \min\{ M_{ l_k, 0 } \mathfrak{c}_k, M_{ l_k, 1 } \mathfrak{c}_{ k + 1 } \}
        \smallprod_{ j = 1}^{ k }
        M_{ l_{ j - 1 }, l_j + 1 }
    \colon
\\ & \quad
    \begin{array}{c}
        k \in \N \cap [ 0, N - 1 ],
        \, ( l_i )_{ i \in \{ 0, 1, \ldots, k \} } \subseteq \\
        \{ 1, 2, \ldots, N \},
        \, N = l_0 > l_{ 1 } > \ldots > l_{ k }
    \!\! \end{array}
\biggr\}
\cup
\bigl\{
    \min\{ M_{ N, 0 } \mathfrak{c}_0, M_{ N, 1 } \mathfrak{c}_{ 1 } \}
\bigr\}
\biggr)
\biggr]^{ -1 }
< \infty
.
\end{split}
\end{equation}
The proof of Proposition~\ref{prop:error_analysis}
is thus complete.
\end{proof}

\begin{corollary}
\label{cor:error_analysis}
Let
$ ( \Omega, \mathscr{F}, \P ) $
be a probability space,
let $ ( \mathcal{Y}, \lVert \cdot \rVert_{ \mathcal{Y} } ) $
be a separable $ \R $-Banach space,
let
$ C, c \in ( 0, \infty ) $,
$ ( \mathfrak{c}_k )_{ k \in \N_0 } \subseteq ( 0, \infty ) $,
$ \Theta = \cup_{ n = 1 }^\infty \Z^n $,
$ M \in \N $,
$ y \in \mathcal{Y} $,
let
$ ( \mathcal{Z}, \mathscr{Z} ) $ be a measurable space,
let
$ Z^{ \theta } \colon \Omega \to \mathcal{Z} $,
$ \theta \in \Theta $,
be i.i.d.\
$ \mathscr{F} $/$ \mathscr{Z} $-measurable functions,
let $ ( \mathcal{H}, \langle \cdot, \cdot \rangle_{ \mathcal{H} }, \lVert \cdot \rVert_{ \mathcal{H} } ) $
be a separable $ \R $-Hilbert space,
let
%$ \mathscr{S} \subseteq \{ \mathcal{S} \colon \mathcal{S} \subseteq L( \mathcal{Y}, \mathcal{H} ) \} $
%be the set given by
$ \mathscr{S}
=
\sigma_{ L( \mathcal{Y}, \mathcal{H} ) } \bigl( \bigl\{
    \{ \varphi \in L( \mathcal{Y}, \mathcal{H} ) \colon \varphi( x ) \in \mathcal{B} \} \subseteq L( \mathcal{Y}, \mathcal{H} ) \colon
    x \in \mathcal{Y},\, \mathcal{ B } \in \mathscr{B}( \mathcal{H} )
\bigr\} \bigr) $,
let
$ \psi_k \colon \Omega \to L( \mathcal{Y}, \mathcal{H} ) $, $ k \in \N_0 $,
be
$ \mathscr{F} $/$ \mathscr{S} $-measurable functions,
let
$ \Phi_l \colon \mathcal{Y} \times \mathcal{Y} \times \mathcal{Z} \to \mathcal{Y} $, $ l \in \N_0 $,
be
$ ( \mathscr{B}( \mathcal{Y} ) \otimes \mathscr{B}( \mathcal{Y} ) \otimes \mathscr{Z} ) $/$ \mathscr{B}( \mathcal{Y} ) $-measurable
functions,
let
$ Y_{ -1 }^\theta \colon \Omega \to \mathcal{Y} $, $ \theta \in \Theta $,
be i.i.d.\
$ \mathscr{F} $/$ \mathscr{B}( \mathcal{Y} ) $-measurable functions,
let
$ Y_{ 0 }^\theta \colon \Omega \to \mathcal{Y} $, $ \theta \in \Theta $,
be i.i.d.\
$ \mathscr{F} $/$ \mathscr{B}( \mathcal{Y} ) $-measurable functions,
assume that
$ ( Y_{ -1 }^{ \theta } )_{ \theta \in \Theta } $,
$ ( Y_{ 0 }^{ \theta } )_{ \theta \in \Theta } $,
$ ( Z^{ \theta } )_{ \theta \in \Theta } $,
and
$ ( \psi_k )_{ k \in \N_0 } $
are independent,
let
$ Y_{ n }^\theta \colon \Omega \to \mathcal{Y} $,
$ \theta \in \Theta $,
$ n \in \N $,
satisfy
for all
$ n \in \N $,
$ \theta \in \Theta $
that
\begin{equation}
\label{eq:EAcor_definition_Y}
Y_{ n }^\theta
=
\sum_{l=0}^{n-1}
\tfrac{1}{ M^{ n - l } }
\Biggl[
    \sum_{i=1}^{ M^{ n - l } }
        \Phi_{ l } \bigl(
        Y^{ ( \theta, l, i ) }_{ l },
        Y^{ ( \theta, -l, i ) }_{ l - 1 },
        Z^{ ( \theta, l, i ) }
        \bigr)
\Biggr]
,
\end{equation}
and
assume
for all
$ k \in \N_0 $,
$ n \in \N $
that
$ \E\bigl[
    \lVert 
        \Phi_{ k }(
        Y^{ 0 }_{ k },
        Y^{ 1 }_{ k - 1 },
        Z^{ 0 }
        )
    \rVert_{ \mathcal{Y} }
\bigr]
< \infty $
and
\begin{gather}
\max\bigl\{
\E\bigl[
\lVert \psi_k(
    \Phi_{ 0 }(
    Y_0^0,
    Y_{ -1 }^1,
    Z^{ 0 }
    )
) \rVert_{ \mathcal{H} }^2
\bigr]
,
\mathbbm{1}_{ \N }( k )
\, \E\bigl[
\lVert \psi_k(
    Y_0^0 - y
) \rVert_{ \mathcal{H} }^2
\bigr]
\bigr\}
\leq
\tfrac{ C^2 }{ \mathfrak{c}_k }
,
\\[0.5\baselineskip]
\E\bigl[
\lVert \psi_k(
    \Phi_{ n } (
    Y^{ 0 }_{ n },
    Y^{ 1 }_{ n - 1 },
    Z^{ 0 }
    )
) \rVert_{ \mathcal{H} }^2
\bigr]
\leq
c
\, \E\bigl[
\lVert \psi_{ k + 1 }(
    Y^{ 0 }_{ n }
    -
    Y^{ 1 }_{ n - 1 }
) \rVert_{ \mathcal{H} }^2
\bigr]
,
\\[0.5\baselineskip]
\E\Biggl[
\biggl\lVert \psi_{ k }
\biggl(
    y -
    \smallsum_{l=0}^{n-1}
        \E\bigl[
            \Phi_{ l }(
            Y^0_{ l },
            Y^1_{ l - 1 },
            Z^0
            )
        \bigr]
\biggr)
\biggr\rVert_{ \mathcal{H} }^2
\Biggr]
\leq
2 c
\, \E\bigl[
\lVert \psi_{ k + 1 }(
    Y_{ n - 1 }^0 - y
) \rVert_{ \mathcal{H} }^2
\bigr]
.
\end{gather}
Then it holds
for all
$ N \in \N $
that
\begin{equation}
\bigl(
\E\bigl[
    \lVert \psi_0(
        Y_{ N }^0 - y
    ) \rVert_{ \mathcal{H} }^2
\bigr]
\bigr)^{ \nicefrac{1}{2} }
\leq
C
\bigl[
    \tfrac{ 1 + 4 c }{ M }
\bigr]^{ \nicefrac{N}{2} }
\max_{ k \in \{ 0, 1, \ldots, N \} }
\sqrt{
    \tfrac{ M^k }{ \mathfrak{c}_k }
}
< \infty
.
\end{equation}
\end{corollary}
\begin{proof}[Proof of Corollary~\ref{cor:error_analysis}]
Throughout this proof
let $ \mathfrak{M}_{ n, l } \in \N $,
$ l \in \{ 0, 1, \ldots, n \} $,
$ n \in \N $,
be the natural numbers which satisfy
for all
$ n \in \N $,
$ l \in \{ 0, 1, \ldots, n \} $
that
$ \mathfrak{M}_{ n, l } = M^{ n - l } $.
Note that it holds
for all
$ n \in \N $
that
$ \mathfrak{M}_{ n, 1 } \geq \mathfrak{M}_{ n, 2 } \geq \ldots \geq \mathfrak{M}_{ n, n } $.
The fact that
$ \forall \, n \in \N \colon
\mathfrak{M}_{ n, n } = 1 $
and
Proposition~\ref{prop:error_analysis}
hence
ensure
for all $ N \in \N $
that
\begin{equation}
\label{eq:error_estimate}
\begin{split}
&
\bigl(
\E\bigl[
    \lVert \psi_0(
        Y_{ N }^0 - y
    ) \rVert_{ \mathcal{H} }^2
\bigr]
\bigr)^{ \nicefrac{1}{2} }
\\ &
\leq
C ( 1 + 4 c )^{ \nicefrac{N}{2} }
\biggl[
\min
\biggl(
\biggl\{
        \min\{ \mathfrak{M}_{ l_k, 0 } \mathfrak{c}_k, \mathfrak{M}_{ l_k, 1 } \mathfrak{c}_{ k + 1 } \}
        \smallprod_{ j = 1}^{ k }
        \mathfrak{M}_{ l_{ j - 1 }, l_j + 1 }
    \colon
\\ & \quad
    \begin{array}{c}
        k \in \N \cap [ 0, N - 1 ],
        \, ( l_i )_{ i \in \{ 0, 1, \ldots, k \} } \subseteq \\
        \{ 1, 2, \ldots, N \},
        \, N = l_0 > l_{ 1 } > \ldots > l_{ k }
    \!\! \end{array}
\biggr\}
\cup
\bigl\{
    \min\{ \mathfrak{M}_{ N, 0 } \mathfrak{c}_0, \mathfrak{M}_{ N, 1 } \mathfrak{c}_{ 1 } \}
\bigr\}
\biggr)
\biggr]^{ -\nicefrac{1}{2} }
.
\end{split}
\end{equation}
Next observe that it holds
for all
$ N \in \N $,
$ k \in \N \cap [ 0, N - 1 ] $,
$ ( l_i )_{ i \in \{ 0, 1, \ldots, k \} } \subseteq \{ 1, 2, \ldots, N \} $
with
$ l_0 = N $
and
$ \forall \, i \in \{ 1, 2, \ldots, k \} \colon
l_{ i - 1 } > l_{ i } $
that
\begin{equation}
\begin{split}
\min\{
    \mathfrak{M}_{ l_k, 0 } \mathfrak{c}_k, \mathfrak{M}_{ l_k, 1 } \mathfrak{c}_{ k + 1 }
\}
\smallprod_{ j = 1}^{ k }
\mathfrak{M}_{ l_{ j - 1 }, l_j + 1 }
& =
\min\{
    M^{ l_k } \mathfrak{c}_k, M^{ l_k - 1 } \mathfrak{c}_{ k + 1 }
\}
\smallprod_{ j = 1}^{ k }
M^{ l_{ j - 1 } - l_j - 1 }
\\ =
\min\{
    M^{ l_k } \mathfrak{c}_k, M^{ l_k - 1 } \mathfrak{c}_{ k + 1 }
\}
M^{ l_{ 0 } - l_k - k }
& =
\min\{
    M^{ N - k } \mathfrak{c}_k, M^{ N - ( k + 1 ) } \mathfrak{c}_{ k + 1 }
\}
.
\end{split}
\end{equation}
This and \eqref{eq:error_estimate}
establish
for all
$ N \in \N $
that
\begin{equation}
\begin{split}
&
\bigl(
\E\bigl[
    \lVert \psi_0(
        Y_{ N }^0 - y
    ) \rVert_{ \mathcal{H} }^2
\bigr]
\bigr)^{ \nicefrac{1}{2} }
\\ &
\leq
C ( 1 + 4 c )^{ \nicefrac{N}{2} }
\Bigl[
\min
\Bigl(
\bigl\{
    \min\{
        M^{ N - k } \mathfrak{c}_k, M^{ N - ( k + 1 ) } \mathfrak{c}_{ k + 1 }
    \}
    \colon
    k \in \N \cap [ 0, N - 1 ]
\bigr\}
\\ & \quad\quad
\cup
\bigl\{
    \min\{ M^N \mathfrak{c}_0, M^{ N - 1 } \mathfrak{c}_{ 1 } \}
\bigr\}
\Bigr)
\Bigr]^{ -\nicefrac{1}{2} }
\\ &
=
C ( 1 + 4 c )^{ \nicefrac{N}{2} }
\biggl[
    \min_{ k \in \{ 0, 1, \ldots, N - 1 \} }
    \min\{
        M^{ N - k } \mathfrak{c}_k, M^{ N - ( k + 1 ) } \mathfrak{c}_{ k + 1 }
    \}
\biggr]^{ -\nicefrac{1}{2} }
\\ &
=
C ( 1 + 4 c )^{ \nicefrac{N}{2} }
\biggl[
    \min_{ k \in \{ 0, 1, \ldots, N \} }
        ( M^{ N - k } \mathfrak{c}_k )
\biggr]^{ -\nicefrac{1}{2} }
\\ &
=
C
\bigl[
    \tfrac{ 1 + 4 c }{ M }
\bigr]^{ \nicefrac{N}{2} }
\biggl[
    \min_{ k \in \{ 0, 1, \ldots, N \} }
        \tfrac{ \mathfrak{c}_k }{ M^k }
\biggr]^{ -\nicefrac{1}{2} }
\\ &
=
C
\bigl[
    \tfrac{ 1 + 4 c }{ M }
\bigr]^{ \nicefrac{N}{2} }
\max_{ k \in \{ 0, 1, \ldots, N \} }
\sqrt{
    \tfrac{ M^k }{ \mathfrak{c}_k }
}
< \infty
.
\end{split}
\end{equation}
The proof of Corollary~\ref{cor:error_analysis} is thus complete.
\end{proof}

\subsection{Cost analysis}
\label{sec:cost_analysis}

\begin{proposition}%[Generalised cost analysis]
\label{prop:cost_analysis}
Let
$ M \in ( 0, \infty ) $,
$ ( \alpha_l )_{ l \in \N_0 },
( \beta_l )_{ l \in \N_0 },
( \gamma_l )_{ l \in \N_0 },
( \Cost_n )_{ n \in \N_0 \cup \{ -1 \} }
\subseteq [ 0, \infty ) $
satisfy
for all $ n \in \N $
that
\begin{equation}
\Cost_{ n }
\leq
\sum_{ l = 0 }^{ n - 1 }
    \bigl[
    M^{ n - l }
    ( \alpha_l \Cost_{ l } + \beta_l \Cost_{ l - 1 } + \gamma_l )
    \bigr]
.
\end{equation}
Then
it holds
for all
$ n \in \N $
that
\begin{equation}
\Cost_n
\leq
M^{ n }
\Biggl(
\beta_0 \Cost_{ -1 }
+
\bigl( \alpha_0 + \tfrac{ \beta_{ 1 } }{M} \bigr) \Cost_{ 0 }
+
\sum_{ l = 0 }^{ n - 1 }
    [    
    M^{ - l }
    \gamma_l
    ]
\Biggr)
\prod_{ l = 1 }^{ n - 1 }
    \bigl(
    1 + \alpha_l + \tfrac{ \beta_{ l + 1 } }{M}
    \bigr)
.
\end{equation}
\end{proposition}
\begin{proof}[Proof of Proposition~\ref{prop:cost_analysis}]
Observe that it holds
for all
$ n \in \N $
that
\begin{align*}
&
\Cost_{ n }
\leq
M^{ n }
\sum_{ l = 0 }^{ n - 1 }
    \bigl[
    M^{ - l }
    ( \alpha_l \Cost_{ l } + \beta_l \Cost_{ l - 1 } + \gamma_l )
    \bigr]
\\ & =
M^{ n }
\Biggl(
\beta_0 \Cost_{ -1 }
+
\Biggl[
\sum_{ l = 0 }^{ n - 1 }
    M^{ - l }
    ( \alpha_l \Cost_{ l } + \gamma_l )
\Biggr]
+
\Biggl[
\sum_{ l = 1 }^{ n - 1 }
    M^{ - l }
    \beta_l \Cost_{ l - 1 }
\Biggr]
\Biggr)
\\ & =
M^{ n }
\Biggl(
\beta_0 \Cost_{ -1 }
+
\Biggl[
\sum_{ l = 0 }^{ n - 1 }
    M^{ - l }
    ( \alpha_l \Cost_{ l } + \gamma_l )
\Biggr]
+
\tfrac{1}{M}
\Biggl[
\sum_{ l = 0 }^{ n - 2 }
    M^{ - l }
    \beta_{ l + 1 } \Cost_{ l }
\Biggr]
\Biggr)
\\ & \leq \yesnumber
M^{ n }
\Biggl(
\beta_0 \Cost_{ -1 }
+
\Biggl[
\sum_{ l = 0 }^{ n - 1 }
    M^{ - l }
    \gamma_l
\Biggr]
+
\Biggl[
\sum_{ l = 0 }^{ n - 1 }
    M^{ - l }
    \alpha_l \Cost_{ l }
\Biggr]
+
\tfrac{1}{M}
\Biggl[
\sum_{ l = 0 }^{ n - 1 }
    M^{ - l }
    \beta_{ l + 1 } \Cost_{ l }
\Biggr]
\Biggr)
\\ & =
M^{ n }
\Biggl(
\beta_0 \Cost_{ -1 }
+
\Biggl[
\sum_{ l = 0 }^{ n - 1 }
    M^{ - l }
    \gamma_l
\Biggr]
+
\Biggl[
\sum_{ l = 0 }^{ n - 1 }
    M^{ - l }
    \bigl( \alpha_l + \tfrac{ \beta_{ l + 1 } }{M} \bigr) \Cost_{ l }
\Biggr]
\Biggr)
\\ & =
M^{ n }
\Biggl(
\beta_0 \Cost_{ -1 }
+
\bigl( \alpha_0 + \tfrac{ \beta_{ 1 } }{M} \bigr) \Cost_{ 0 }
+
\sum_{ l = 0 }^{ n - 1 }
    [
    M^{ - l }
    \gamma_l
    ]
\Biggr)
+
M^{ n }
\Biggl[
\sum_{ l = 1 }^{ n - 1 }
    M^{ - l }
    \bigl( \alpha_l + \tfrac{ \beta_{ l + 1 } }{M} \bigr) \Cost_{ l }
\Biggr]
.
\end{align*}
Theorem~4.1.1
in
Agarwal~\cite{Agarwal2000}
(with
$ a = 1 $,
$ u( k ) = \Cost_k $,
$ p( k ) =
M^{ k }
(
\beta_0 \Cost_{ -1 }
+
( \alpha_0 + \tfrac{ \beta_{ 1 } }{M} ) \Cost_{ 0 }
+
\sum_{ l = 0 }^{ k - 1 }
    [    
    M^{ - l }
    \gamma_l
    ]
) $,
$ q( k ) = M^k $,
$ f( k ) = M^{ - k } ( \alpha_k + \tfrac{ \beta_{ k + 1 } }{M} ) $
for $ k \in \N $
in the notation of \cite[Theorem~4.1.1]{Agarwal2000})
hence establishes
for all
$ n \in \N $
that
\begin{equation}
\begin{split}
&
\Cost_n
\leq
M^{ n }
\Biggl(
\beta_0 \Cost_{ -1 }
+
\bigl( \alpha_0 + \tfrac{ \beta_{ 1 } }{M} \bigr) \Cost_{ 0 }
+
\sum_{ l = 0 }^{ n - 1 }
    [    
    M^{ - l }
    \gamma_l
    ]
\Biggr)
\\ & \quad
+
M^{ n }
\sum_{ l = 1 }^{ n - 1 }
\Biggl[
    M^{ l }
    \Biggl(
    \beta_0 \Cost_{ -1 }
    +
    \bigl( \alpha_0 + \tfrac{ \beta_{ 1 } }{M} \bigr) \Cost_{ 0 }
    +
    \sum_{ i = 0 }^{ l - 1 }
        [    
        M^{ - i }
        \gamma_i
        ]
    \Biggr)
    M^{ - l }
    \bigl( \alpha_l + \tfrac{ \beta_{ l + 1 } }{M} \bigr)
\\ & \quad
\hphantom{
\mathop{+}
M^{ n }
\sum_{ l = 1 }^{ n - 1 }
\Biggl[
}
    \cdot
    \prod_{ i = l + 1 }^{ n - 1 }
        \bigl[
        1 + M^{ i } M^{ - i } \bigl( \alpha_i + \tfrac{ \beta_{ i + 1 } }{M} \bigr)
        \bigr]
\Biggr]
\\ & \leq
M^{ n }
\Biggl(
\beta_0 \Cost_{ -1 }
+
\bigl( \alpha_0 + \tfrac{ \beta_{ 1 } }{M} \bigr) \Cost_{ 0 }
+
\sum_{ l = 0 }^{ n - 1 }
    [    
    M^{ - l }
    \gamma_l
    ]
\Biggr)
\\ & \quad
\cdot
\Biggl(
1
+
\sum_{ l = 1 }^{ n - 1 }
\Biggl[
    \bigl( \alpha_l + \tfrac{ \beta_{ l + 1 } }{M} \bigr)
    \prod_{ i = l + 1 }^{ n - 1 }
        \bigl(
        1 + \alpha_i + \tfrac{ \beta_{ i + 1 } }{M}
        \bigr)
\Biggr]
\Biggr)
.
\end{split}
\end{equation}
This and
\cite[Problem~1.9.10]{Agarwal2000}
show
for all
$ n \in \N $
that
\begin{equation}
\Cost_n
\leq
M^{ n }
\Biggl(
\beta_0 \Cost_{ -1 }
+
\bigl( \alpha_0 + \tfrac{ \beta_{ 1 } }{M} \bigr) \Cost_{ 0 }
+
\sum_{ l = 0 }^{ n - 1 }
    [    
    M^{ - l }
    \gamma_l
    ]
\Biggr)
\prod_{ l = 1 }^{ n - 1 }
    \bigl(
    1 + \alpha_l + \tfrac{ \beta_{ l + 1 } }{M}
    \bigr)
.
\end{equation}
The proof of Proposition~\ref{prop:cost_analysis} is thus complete.
\end{proof}

\begin{lemma}
\label{lem:a_b}
Let $ a, b \in [ 0, \infty ) $.
Then it holds
for all
$ n \in \N $
that
\begin{equation}
\label{eq:a_b_induction}
( a n + b )
b^{ n - 1 }
\leq
( a + b )^{ n }
.
\end{equation}
\end{lemma}
\begin{proof}[Proof of Lemma~\ref{lem:a_b}]
We prove~\eqref{eq:a_b_induction}
by induction on $ n \in \N $.
Note that the base case $ n = 1 $ is clear.
For the induction step $ \N \ni n - 1 \to n \in \{ 2, 3, \ldots \} $
let $ n \in \{ 2, 3, \ldots \} $
and assume that
$ ( a ( n - 1 ) + b )
b^{ n - 2 }
\leq
( a + b )^{ n - 1 } $.
This ensures that
\begin{equation}
\begin{split}
( a n + b )
b^{ n - 1 }
& =
a b^{ n - 1 }
+
( a ( n - 1 ) + b )
b^{ n - 1 }
\leq
a b^{ n - 1 }
+
b
( a + b )^{ n - 1 }
\\ &
\leq
a
( a + b )^{ n - 1 }
+
b
( a + b )^{ n - 1 }
=
( a + b )^{ n }
.
\end{split}
\end{equation}
Induction hence completes the proof of Lemma~\ref{lem:a_b}.
\end{proof}

\begin{corollary}%[Generalised cost analysis]
\label{cor:cost_analysis}
Let
$ M \in [ 1, \infty ) $,
$ \mathfrak{z}, \alpha, \beta, \gamma \in [ 0, \infty ) $,
$ ( \Cost_n )_{ n \in \N_0 \cup \{ -1 \} }
\subseteq [ 0, \infty ) $
satisfy
for all
$ n \in \N $
that
$ \Cost_{ -1 } = \Cost_{ 0 } = 0 $
and
\begin{equation}
\Cost_{ n }
\leq
M^{ n }
\mathfrak{z}
+
\sum_{ l = 0 }^{ n - 1 }
    \bigl[
    M^{ n - l }
    ( \alpha \Cost_{ l } + \beta \Cost_{ l - 1 } + \gamma \mathfrak{z} )
    \bigr]
.
\end{equation}
Then
it holds
for all
$ n \in \N $
that
\begin{equation}
\Cost_n
\leq
( 1 + \alpha + \beta + \gamma )^{ n }
M^{ n }
\mathfrak{z}
.
\end{equation}
\end{corollary}
\begin{proof}[Proof of Corollary~\ref{cor:cost_analysis}]
Note that
Proposition~\ref{prop:cost_analysis}
demonstrates
for all
$ n \in \N $
that
\begin{equation}
\label{eq:cost_estimate_raw}
\begin{split}
\Cost_n
& \leq
M^{ n }
\Biggl(
\mathfrak{z}
+
\gamma \mathfrak{z}
\sum_{ l = 0 }^{ n - 1 }
    M^{ - l }
\Biggr)
\prod_{ l = 1 }^{ n - 1 }
    \bigl(
    1 + \alpha + \tfrac{ \beta }{M}
    \bigr)
\\ &
\leq
\Biggl(
1
+
\gamma
\sum_{ l = 0 }^{ n - 1 }
    M^{ - l }
\Biggr)
( 1 + \alpha + \beta )^{ n - 1 }
M^{ n }
\mathfrak{z}
.
\end{split}
\end{equation}
In addition,
observe that it holds
for all
$ n \in \N $
that
\begin{equation}
\label{eq:sum_M}
\sum_{ l = 0 }^{ n - 1 }
    M^{ - l }
\leq
\sum_{ l = 0 }^{ n - 1 }
    1
=
n
.
\end{equation}
Furthermore,
Lemma~\ref{lem:a_b}
implies
for all
$ n \in \N $
that
\begin{equation}
( 1 + \gamma n )
( 1 + \alpha + \beta )^{ n - 1 }
\leq
( \gamma n + 1 + \alpha + \beta )
( 1 + \alpha + \beta )^{ n - 1 }
\leq
( 1 + \alpha + \beta + \gamma )^{ n }
.
\end{equation}
This,
\eqref{eq:cost_estimate_raw},
and~\eqref{eq:sum_M}
prove
for all
$ n \in \N $
that
\begin{equation}
\Cost_n
\leq
( 1 + \gamma n )
( 1 + \alpha + \beta )^{ n - 1 }
M^{ n }
\mathfrak{z}
\leq
( 1 + \alpha + \beta + \gamma )^{ n }
M^{ n }
\mathfrak{z}
.
\end{equation}
The proof of Corollary~\ref{cor:cost_analysis} is thus complete.
\end{proof}

\subsection{Complexity analysis}
\label{sec:complexity_analysis}

\begin{theorem}
\label{thm:complexity_analysis}
Let
$ ( \Omega, \mathscr{F}, \P ) $
be a probability space,
let $ ( \mathcal{Y}, \lVert \cdot \rVert_{ \mathcal{Y} } ) $
be a separable $ \R $-Banach space,
let
$ \mathfrak{z}, \gamma \in [ 0, \infty ) $,
$ \mathfrak{B}, \mathfrak{b}, C \in [ 1, \infty ) $,
$ c \in ( 0, \infty ) $,
$ ( \mathfrak{c}_k )_{ k \in \N_0 } \subseteq ( 0, \infty ) $,
$ \Theta = \cup_{ n = 1 }^\infty \Z^n $,
$ ( M_j )_{ j \in \N } \subseteq \N $,
$ y, \mathfrak{y}_{ -1 }, \mathfrak{y}_{ 0 } \in \mathcal{Y} $
satisfy
$ \liminf_{ j \to \infty } M_j = \infty $,
$ \sup_{ j \in \N }
\nicefrac{ M_{ j + 1 } }{ M_j } \leq \mathfrak{B} $,
and
$ \forall \, n \in \N \colon
\max_{ k \in \{ 0, 1, \ldots, n \} }
    \nicefrac{ ( M_n )^{ k } }{ \mathfrak{ c }_k }
\leq
\mathfrak{b}^{ n }
$,
let
$ ( \mathcal{Z}, \mathscr{Z} ) $ be a measurable space,
let
$ Z^{ \theta } \colon \Omega \to \mathcal{Z} $,
$ \theta \in \Theta $,
be i.i.d.\
$ \mathscr{F} $/$ \mathscr{Z} $-measurable functions,
let $ ( \mathcal{H}, \langle \cdot, \cdot \rangle_{ \mathcal{H} }, \lVert \cdot \rVert_{ \mathcal{H} } ) $
be a separable $ \R $-Hilbert space,
let
%$ \mathscr{S} \subseteq \{ \mathcal{S} \colon \mathcal{S} \subseteq L( \mathcal{Y}, \mathcal{H} ) \} $
%be the set given by
$ \mathscr{S}
=
\sigma_{ L( \mathcal{Y}, \mathcal{H} ) } \bigl( \bigl\{
    \{ \varphi \in L( \mathcal{Y}, \mathcal{H} ) \colon \varphi( x ) \in \mathcal{B} \} \subseteq L( \mathcal{Y}, \mathcal{H} ) \colon
    x \in \mathcal{Y},\, \mathcal{ B } \in \mathscr{B}( \mathcal{H} )
\bigr\} \bigr) $,
let
$ \psi_k \colon \Omega \to L( \mathcal{Y}, \mathcal{H} ) $, $ k \in \N_0 $,
be
$ \mathscr{F} $/$ \mathscr{S} $-measurable functions,
assume that
$ ( Z^{ \theta } )_{ \theta \in \Theta } $
and
$ ( \psi_k )_{ k \in \N_0 } $
are independent,
let
$ \Phi_l \colon \mathcal{Y} \times \mathcal{Y} \times \mathcal{Z} \to \mathcal{Y} $, $ l \in \N_0 $,
be
$ ( \mathscr{B}( \mathcal{Y} ) \otimes \mathscr{B}( \mathcal{Y} ) \otimes \mathscr{Z} ) $/$ \mathscr{B}( \mathcal{Y} ) $-measurable
functions,
let
$ Y_{ n, j }^\theta \colon \Omega \to \mathcal{Y} $,
$ \theta \in \Theta $,
$ j \in \N $,
$ n \in ( \N_0 \cup \{ -1 \} ) $,
satisfy
for all
$ n, j \in \N $,
$ \theta \in \Theta $
that
$ Y_{ -1, j }^\theta = \mathfrak{y}_{ -1 } $,
$ Y_{ 0, j }^\theta = \mathfrak{y}_{ 0 } $,
and
\begin{equation}
Y_{ n, j }^\theta
=
\sum_{l=0}^{n-1}
\tfrac{1}{ ( M_j )^{ n - l } }
\Biggl[
    \sum_{i=1}^{ ( M_j )^{ n - l } }
        \Phi_{ l } \bigl(
        Y^{ ( \theta, l, i ) }_{ l, j },
        Y^{ ( \theta, -l, i ) }_{ l - 1, j },
        Z^{ ( \theta, l, i ) }
        \bigr)
\Biggr]
,
\end{equation}
let
$ ( \Cost_{ n, j } )_{ ( n, j ) \in ( \N_0 \cup \{ -1 \} ) \times \N }
\subseteq [ 0, \infty ) $
satisfy
for all $ n, j \in \N $
that
$ \Cost_{ -1, j } = \Cost_{ 0, j } = 0 $
and
\begin{equation}
\label{eq:cost_complexity}
\Cost_{ n, j }
\leq
( M_j )^{ n }
\mathfrak{z}
+
\sum_{ l = 0 }^{ n - 1 }
    \bigl[
    ( M_j )^{ n - l }
    ( \Cost_{ l, j } + \Cost_{ l - 1, j } + \gamma \mathfrak{z} )
    \bigr]
,
\end{equation}
and
assume
for all
$ k \in \N_0 $,
$ n, j \in \N $
that
$ \E\bigl[
    \lVert 
        \Phi_{ k }(
        Y^{ 0 }_{ k, j },
        Y^{ 1 }_{ k - 1, j },
        Z^{ 0 }
        )
    \rVert_{ \mathcal{Y} }
\bigr]
< \infty $
and
\begin{gather}
\max\bigl\{
\E\bigl[
\lVert \psi_k(
    \Phi_{ 0 }(
    \mathfrak{y}_{ 0 },
    \mathfrak{y}_{ -1 },
    Z^{ 0 }
    )
) \rVert_{ \mathcal{H} }^2
\bigr]
,
\mathbbm{1}_{ \N }( k )
\, \E\bigl[
\lVert \psi_k(
    \mathfrak{y}_{ 0 } - y
) \rVert_{ \mathcal{H} }^2
\bigr]
\bigr\}
\leq
\tfrac{ C^2 }{ \mathfrak{c}_k }
,
\\[0.5\baselineskip]
\label{eq:complexity_hypothesis_II}
\E\Bigl[
\bigl\lVert \psi_k\bigl(
    \Phi_{ n } \bigl(
    Y^{ 0 }_{ n, j },
    Y^{ 1 }_{ n - 1, j },
    Z^{ 0 }
    \bigr)
\bigr) \bigr\rVert_{ \mathcal{H} }^2
\Bigr]
\leq
c
\, \E\Bigl[
\bigl\lVert \psi_{ k + 1 }\bigl(
    Y^{ 0 }_{ n, j }
    -
    Y^{ 1 }_{ n - 1, j }
\bigr) \bigr\rVert_{ \mathcal{H} }^2
\Bigr]
,
\\[0.5\baselineskip]
\E\Biggl[
\biggl\lVert \psi_{ k }
\biggl(
    y -
    \smallsum_{l=0}^{n-1}
        \E\bigl[
            \Phi_{ l }\bigl(
            Y^0_{ l, j },
            Y^1_{ l - 1, j },
            Z^0
            \bigr)
        \bigr]
\biggr)
\biggr\rVert_{ \mathcal{H} }^2
\Biggr]
\leq
2 c
\, \E\Bigl[
\bigl\lVert \psi_{ k + 1 }\bigl(
    Y_{ n - 1, j }^0 - y
\bigr) \bigr\rVert_{ \mathcal{H} }^2
\Bigr]
.
\end{gather}
Then
\begin{enumerate}[(i)]
\item
\label{item:complexity_analysis1}
it holds
for all
$ n \in \N $
that
\begin{equation}
\bigl(
\E\bigl[
    \lVert \psi_0(
        Y_{ n, n }^0 - y
    ) \rVert_{ \mathcal{H} }^2
\bigr]
\bigr)^{ \nicefrac{1}{2} }
\leq
C
\biggl[
    \frac{ \mathfrak{ b }( 1 + 4 c ) }{ M_n }
\biggr]^{ \nicefrac{n}{2} }
< \infty
,
\end{equation}
\item
\label{item:complexity_analysis2}
it holds
for all
$ n \in \N $
that
$ \Cost_{ n, n }
\leq
( 3 + \gamma )^n
( M_n )^{ n }
\mathfrak{z}
$,
%\item
%it holds that
%$
%\limsup_{ n \to \infty }
%\bigl( \E\bigl[ \lVert \psi_{ 0 } (
%    Y_{ n, n }^0 - y
%) \rVert_{ \mathcal{H} }^2 \bigr] \bigr)^{ \nicefrac{1}{2} }
%=
%0
%$,
and
\item
\label{item:complexity_analysis3}
there exists
$ ( N_\varepsilon )_{ \varepsilon \in ( 0, 1 ] } \subseteq \N $
such that
it holds
for all
$ \varepsilon \in ( 0, 1 ] $,
$ \delta \in ( 0, \infty ) $
that
$ \sup_{ n \in \{ N_\varepsilon, N_\varepsilon + 1, \ldots \} }
\bigl(
\E\bigl[
    \lVert \psi_0(
        Y_{ n, n }^0 - y
    ) \rVert_{ \mathcal{H} }^2
\bigr]
\bigr)^{ \nicefrac{1}{2} }
\leq
\varepsilon $
and
\begin{equation}
\label{eq:complexity_estimate}
\!%\linebreak
\Cost_{ N_\varepsilon, N_\varepsilon }
\leq
\mathfrak{z}
\mathfrak{b}
\mathfrak{c}_1
( 3 + \gamma )
C^{ 2 ( 1 + \delta ) }
\biggl(
    1
    +
    \sup_{ n \in \N }
    \biggl[
        \frac{
        [ \mathfrak{B} \mathfrak{ b }^2 ( 3 + \gamma )( 1 + 4 c ) ]^{ ( 1 + \delta ) }
        }{
        ( M_{ n } )^{ \delta }
        }
    \biggr]^n
\biggr)
\varepsilon^{ -2 ( 1 + \delta ) }
<
\infty
.
\end{equation}
\end{enumerate}
\end{theorem}
\begin{proof}[Proof of Theorem~\ref{thm:complexity_analysis}]
Throughout this proof
let
$ ( N_\varepsilon )_{ \varepsilon \in ( 0, 1 ] } \subseteq \N $
be the family of natural numbers which satisfies
for all
$ \varepsilon \in ( 0, 1 ] $
that
\begin{equation}
\label{eq:N_varepsilon}
N_\varepsilon
=
\min \biggl\{
    \mathfrak{N} \in \N \colon
    \sup_{ n \in \{ \mathfrak{N}, \mathfrak{N}+1, \ldots \} }
    C
    \biggl[
        \frac{ \mathfrak{b}( 1 + 4 c ) }{ M_n }
    \biggr]^{ \nicefrac{n}{2} }
    \leq \varepsilon
\biggr\}
.
\end{equation}
Observe that
Corollary~\ref{cor:error_analysis}
and the assumption that
$ \forall \, n \in \N \colon
\max_{ k \in \{ 0, 1, \ldots, n \} }
    \nicefrac{ ( M_n )^{ k } }{ \mathfrak{ c }_k }
\leq
\mathfrak{b}^{ n }
$
establish
for all
$ n \in \N $
that
\begin{equation}
\label{eq:error_Y_nn}
\bigl(
\E\bigl[
    \lVert \psi_0(
        Y_{ n, n }^0 - y
    ) \rVert_{ \mathcal{H} }^2
\bigr]
\bigr)^{ \nicefrac{1}{2} }
\leq
C
\biggl[
    \frac{ 1 + 4 c }{ M_n }
\biggr]^{ \nicefrac{n}{2} }
\max_{ k \in \{ 0, 1, \ldots, n \} }
\sqrt{
    \tfrac{ ( M_n )^k }{ \mathfrak{c}_k }
}
\leq
C
\biggl[
    \frac{ \mathfrak{ b }( 1 + 4 c ) }{ M_n }
\biggr]^{ \nicefrac{n}{2} }
< \infty
.
\end{equation}
This proves~(\ref{item:complexity_analysis1}).
In addition,
\eqref{eq:cost_complexity}
and
Corollary~\ref{cor:cost_analysis}
demonstrate
for all
$ n \in \N $
that
\begin{equation}
\label{eq:cost_Y_nn}
\Cost_{ n, n }
\leq
( 3 + \gamma )^n
( M_n )^{ n }
\mathfrak{z}
.
\end{equation}
This finishes the proof of~(\ref{item:complexity_analysis2}).
It thus remains to show~(\ref{item:complexity_analysis3}).
Observe that
\eqref{eq:error_Y_nn}
and
\eqref{eq:N_varepsilon}
ensure
for all
$ \varepsilon \in ( 0, 1 ] $
that
\begin{equation}
\sup_{ n \in \{ N_\varepsilon, N_\varepsilon + 1, \ldots \} }
\bigl(
\E\bigl[
    \lVert \psi_0(
        Y_{ n, n }^0 - y
    ) \rVert_{ \mathcal{H} }^2
\bigr]
\bigr)^{ \nicefrac{1}{2} }
\leq
\sup_{ n \in \{ N_\varepsilon, N_\varepsilon + 1, \ldots \} }
C
\biggl[
    \frac{ \mathfrak{ b }( 1 + 4 c ) }{ M_n }
\biggr]^{ \nicefrac{n}{2} }
\leq
\varepsilon
.
\end{equation}
Furthermore,
note that
\eqref{eq:N_varepsilon}
implies
for all
$ \varepsilon \in ( 0, 1 ] $
with $ N_\varepsilon \geq 2 $
that
\begin{equation}
C
\biggl[
    \frac{ \mathfrak{ b }( 1 + 4 c ) }{ M_{ N_\varepsilon - 1 } }
\biggr]^{ \nicefrac{ ( N_\varepsilon - 1 ) }{2} }
>
\varepsilon
.
\end{equation}
This,
\eqref{eq:cost_Y_nn},
the assumption that
$ \sup_{ j \in \N }
\nicefrac{ M_{ j + 1 } }{ M_j } \leq \mathfrak{B} $,
and
the fact that
$ \forall \, n \in \N \colon
    \nicefrac{ M_n }{ \mathfrak{ c }_1 }
\leq
\mathfrak{b}^{ n }
$
show
for all
$ \varepsilon \in ( 0, 1 ] $,
$ \delta \in ( 0, \infty ) $
with $ N_\varepsilon \geq 2 $
that
\begin{equation}
\label{eq:cost_estimate}
\begin{split}
\Cost_{ N_\varepsilon, N_\varepsilon }
& \leq
( 3 + \gamma )^{ N_\varepsilon }
( M_{ N_\varepsilon } )^{ N_\varepsilon }
\mathfrak{z}
\\ &
\leq
( 3 + \gamma )^{ N_\varepsilon }
( M_{ N_\varepsilon } )^{ N_\varepsilon }
\mathfrak{z}
\Biggl[
C
\biggl[
    \frac{ \mathfrak{ b }( 1 + 4 c ) }{ M_{ N_\varepsilon - 1 } }
\biggr]^{ \nicefrac{ ( N_\varepsilon - 1 ) }{2} }
\varepsilon^{ -1 }
\Biggr]^{ 2 ( 1 + \delta ) }
\\ &
=
\mathfrak{z}
C^{ 2 ( 1 + \delta ) }
\varepsilon^{ -2 ( 1 + \delta ) }
\biggl[
    \frac{
    ( 3 + \gamma )^{ N_\varepsilon }
    ( M_{ N_\varepsilon } )^{ N_\varepsilon }
    [ \mathfrak{ b }( 1 + 4 c ) ]^{ ( N_\varepsilon - 1 )( 1 + \delta ) }
    }{
    ( M_{ N_\varepsilon - 1 } )^{ ( N_\varepsilon - 1 )( 1 + \delta ) }
    }
\biggr]
\\ &
\leq
\mathfrak{z}
C^{ 2 ( 1 + \delta ) }
\varepsilon^{ -2 ( 1 + \delta ) }
\sup_{ n \in \N }
\biggl[
    \frac{
    ( 3 + \gamma )^{ n + 1 }
    ( M_{ n + 1 } )^{ n + 1 }
    [ \mathfrak{ b }( 1 + 4 c ) ]^{ n ( 1 + \delta ) }
    }{
    ( M_{ n } )^{ n( 1 + \delta ) }
    }
\biggr]
\\ &
\leq
\mathfrak{z}
( 3 + \gamma )
C^{ 2 ( 1 + \delta ) }
\varepsilon^{ -2 ( 1 + \delta ) }
\sup_{ n \in \N }
\biggl[
    \frac{
    M_{ n + 1 }
    ( M_{ n + 1 } )^{ n }
    [ \mathfrak{ b }( 3 + \gamma )( 1 + 4 c ) ]^{ n ( 1 + \delta ) }
    }{
    ( M_{ n } )^{ n } ( M_{ n } )^{ n \delta }
    }
\biggr]
\\ &
\leq
\mathfrak{z}
\mathfrak{c}_1
( 3 + \gamma )
C^{ 2 ( 1 + \delta ) }
\varepsilon^{ -2 ( 1 + \delta ) }
\sup_{ n \in \N }
\biggl[
    \frac{
    \mathfrak{b}^{ n + 1 }
    \mathfrak{B}^n
    [ \mathfrak{ b }( 3 + \gamma )( 1 + 4 c ) ]^{ n ( 1 + \delta ) }
    }{
    ( M_{ n } )^{ n \delta }
    }
\biggr]
\\ &
\leq
\mathfrak{z}
\mathfrak{b}
\mathfrak{c}_1
( 3 + \gamma )
C^{ 2 ( 1 + \delta ) }
\varepsilon^{ -2 ( 1 + \delta ) }
\sup_{ n \in \N }
\biggl[
    \frac{
    [ \mathfrak{B} \mathfrak{ b }^2 ( 3 + \gamma )( 1 + 4 c ) ]^{ ( 1 + \delta ) }
    }{
    ( M_{ n } )^{ \delta }
    }
\biggr]^n
\\ &
\leq
\mathfrak{z}
\mathfrak{b}
\mathfrak{c}_1
( 3 + \gamma )
C^{ 2 ( 1 + \delta ) }
\biggl(
    1
    +
    \sup_{ n \in \N }
    \biggl[
        \frac{
        [ \mathfrak{B} \mathfrak{ b }^2 ( 3 + \gamma )( 1 + 4 c ) ]^{ ( 1 + \delta ) }
        }{
        ( M_{ n } )^{ \delta }
        }
    \biggr]^n
\biggr)
\varepsilon^{ -2 ( 1 + \delta ) }
.
\end{split}
\end{equation}
Moreover,
\eqref{eq:cost_complexity},
the fact that
$ \nicefrac{ M_1 }{ \mathfrak{ c }_1 }
\leq
\mathfrak{b} $,
and the fact that
$ C \geq 1 $
ensure
for all
$ \varepsilon \in ( 0, 1 ] $,
$ \delta \in ( 0, \infty ) $
that
\begin{equation}
\begin{split}
\Cost_{ 1, 1 }
& \leq
\mathfrak{z}
( 1 + \gamma )
M_1
\leq
\mathfrak{z}
\mathfrak{b}
\mathfrak{c}_1
( 3 + \gamma )
\\ &
\leq
\mathfrak{z}
\mathfrak{b}
\mathfrak{c}_1
( 3 + \gamma )
C^{ 2 ( 1 + \delta ) }
\biggl(
    1
    +
    \sup_{ n \in \N }
    \biggl[
        \frac{
        [ \mathfrak{B} \mathfrak{ b }^2 ( 3 + \gamma )( 1 + 4 c ) ]^{ ( 1 + \delta ) }
        }{
        ( M_{ n } )^{ \delta }
        }
    \biggr]^n
\biggr)
\varepsilon^{ -2 ( 1 + \delta ) }
.
\end{split}
\end{equation}
Combining this
with \eqref{eq:cost_estimate}
establishes
for all
$ \varepsilon \in ( 0, 1 ] $,
$ \delta \in ( 0, \infty ) $
that
\begin{equation}
\Cost_{ N_\varepsilon, N_\varepsilon }
\leq
\mathfrak{z}
\mathfrak{b}
\mathfrak{c}_1
( 3 + \gamma )
C^{ 2 ( 1 + \delta ) }
\biggl(
    1
    +
    \sup_{ n \in \N }
    \biggl[
        \frac{
        [ \mathfrak{B} \mathfrak{ b }^2 ( 3 + \gamma )( 1 + 4 c ) ]^{ ( 1 + \delta ) }
        }{
        ( M_{ n } )^{ \delta }
        }
    \biggr]^n
\biggr)
\varepsilon^{ -2 ( 1 + \delta ) }
<
\infty
.
\end{equation}
The proof of Theorem~\ref{thm:complexity_analysis} is thus complete.
\end{proof}

\begin{corollary}
\label{cor:complexity_analysis}
Let
$ ( \Omega, \mathscr{F}, \P ) $
be a probability space,
let $ ( \mathcal{Y}, \lVert \cdot \rVert_{ \mathcal{Y} } ) $
be a separable $ \R $-Banach space,
let
$ \mathfrak{z}, \gamma \in [ 0, \infty ) $,
$ \mathfrak{B}, \kappa, C \in [ 1, \infty ) $,
$ c \in ( 0, \infty ) $,
$ \Theta = \cup_{ n = 1 }^\infty \Z^n $,
$ ( M_j )_{ j \in \N } \subseteq \N $,
$ y, \mathfrak{y}_{ -1 }, \mathfrak{y}_{ 0 } \in \mathcal{Y} $
satisfy
$ \liminf_{ j \to \infty } M_j = \infty $,
$ \sup_{ j \in \N }
\nicefrac{ M_{ j + 1 } }{ M_j } \leq \mathfrak{B} $,
and
$ \sup_{ j \in \N }
\nicefrac{ M_j }{ j } \leq \kappa $,
let
$ ( \mathcal{Z}, \mathscr{Z} ) $ be a measurable space,
let
$ Z^{ \theta } \colon \Omega \to \mathcal{Z} $,
$ \theta \in \Theta $,
be i.i.d.\
$ \mathscr{F} $/$ \mathscr{Z} $-measurable functions,
let $ ( \mathcal{H}, \langle \cdot, \cdot \rangle_{ \mathcal{H} }, \lVert \cdot \rVert_{ \mathcal{H} } ) $
be a separable $ \R $-Hilbert space,
let
%$ \mathscr{S} \subseteq \{ \mathcal{S} \colon \mathcal{S} \subseteq L( \mathcal{Y}, \mathcal{H} ) \} $
%be the set given by
$ \mathscr{S}
=
\sigma_{ L( \mathcal{Y}, \mathcal{H} ) } \bigl( \bigl\{
    \{ \varphi \in L( \mathcal{Y}, \mathcal{H} ) \colon \varphi( x ) \in \mathcal{B} \} \subseteq L( \mathcal{Y}, \mathcal{H} ) \colon
    x \in \mathcal{Y},\, \mathcal{ B } \in \mathscr{B}( \mathcal{H} )
\bigr\} \bigr) $,
let
$ \psi_k \colon \Omega \to L( \mathcal{Y}, \mathcal{H} ) $, $ k \in \N_0 $,
be
$ \mathscr{F} $/$ \mathscr{S} $-measurable functions,
assume that
$ ( Z^{ \theta } )_{ \theta \in \Theta } $
and
$ ( \psi_k )_{ k \in \N_0 } $
are independent,
let
$ \Phi_l \colon \mathcal{Y} \times \mathcal{Y} \times \mathcal{Z} \to \mathcal{Y} $, $ l \in \N_0 $,
be
$ ( \mathscr{B}( \mathcal{Y} ) \otimes \mathscr{B}( \mathcal{Y} ) \otimes \mathscr{Z} ) $/$ \mathscr{B}( \mathcal{Y} ) $-measurable
functions,
let
$ Y_{ n, j }^\theta \colon \Omega \to \mathcal{Y} $,
$ \theta \in \Theta $,
$ j \in \N $,
$ n \in ( \N_0 \cup \{ -1 \} ) $,
satisfy
for all
$ n, j \in \N $,
$ \theta \in \Theta $
that
$ Y_{ -1, j }^\theta = \mathfrak{y}_{ -1 } $,
$ Y_{ 0, j }^\theta = \mathfrak{y}_{ 0 } $,
and
\begin{equation}
\label{eq:def_Y_abstract}
Y_{ n, j }^\theta
=
\sum_{l=0}^{n-1}
\tfrac{1}{ ( M_j )^{ n - l } }
\Biggl[
    \sum_{i=1}^{ ( M_j )^{ n - l } }
        \Phi_{ l } \bigl(
        Y^{ ( \theta, l, i ) }_{ l, j },
        Y^{ ( \theta, -l, i ) }_{ l - 1, j },
        Z^{ ( \theta, l, i ) }
        \bigr)
\Biggr]
,
\end{equation}
let
$ ( \Cost_{ n, j } )_{ ( n, j ) \in ( \N_0 \cup \{ -1 \} ) \times \N }
\subseteq [ 0, \infty ) $
satisfy
for all $ n, j \in \N $
that
$ \Cost_{ -1, j } = \Cost_{ 0, j } = 0 $
and
\begin{equation}
\Cost_{ n, j }
\leq
( M_j )^{ n }
\mathfrak{z}
+
\sum_{ l = 0 }^{ n - 1 }
    \bigl[
    ( M_j )^{ n - l }
    ( \Cost_{ l, j } + \Cost_{ l - 1, j } + \gamma \mathfrak{z} )
    \bigr]
,
\end{equation}
and
assume
for all
$ k \in \N_0 $,
$ n, j \in \N $,
$ u, v \in \mathcal{Y} $
that
$ \E\bigl[
    \lVert 
        \Phi_{ k }(
        Y^{ 0 }_{ k, j },
        Y^{ 1 }_{ k - 1, j },
        Z^{ 0 }
        )
    \rVert_{ \mathcal{Y} }
\bigr]
< \infty $
and
\begin{gather}
\label{eq:CA_hypothesis_I}
\max\bigl\{
\E\bigl[
\lVert \psi_k(
    \Phi_{ 0 }(
    \mathfrak{y}_{ 0 },
    \mathfrak{y}_{ -1 },
    Z^{ 0 }
    )
) \rVert_{ \mathcal{H} }^2
\bigr]
,
\mathbbm{1}_{ \N }( k )
\, \E\bigl[
\lVert \psi_k(
    \mathfrak{y}_{ 0 } - y
) \rVert_{ \mathcal{H} }^2
\bigr]
\bigr\}
\leq
\tfrac{ C^2 }{ k! }
,
\\[0.5\baselineskip]
\label{eq:CA_hypothesis_II}
\E\bigl[
\lVert \psi_k(
    \Phi_{ n }( u, v, Z^{ 0 } )
) \rVert_{ \mathcal{H} }^2
\bigr]
\leq
c
\, \E\bigl[
\lVert \psi_{ k + 1 }(
    u - v
) \rVert_{ \mathcal{H} }^2
\bigr]
,
\\[0.5\baselineskip]
\label{eq:CA_hypothesis_III}
\E\Biggl[
\biggl\lVert \psi_{ k }
\biggl(
    y -
    \smallsum_{l=0}^{n-1}
        \E\bigl[
            \Phi_{ l }\bigl(
            Y^0_{ l, j },
            Y^1_{ l - 1, j },
            Z^0
            \bigr)
        \bigr]
\biggr)
\biggr\rVert_{ \mathcal{H} }^2
\Biggr]
\leq
2 c
\, \E\Bigl[
\bigl\lVert \psi_{ k + 1 }\bigl(
    Y_{ n - 1, j }^0 - y
\bigr) \bigr\rVert_{ \mathcal{H} }^2
\Bigr]
.
\end{gather}
Then
\begin{enumerate}[(i)]
\item
\label{item:cor:complexity_analysis1}
it holds
for all
$ n \in \N $
that
\begin{equation}
\bigl(
\E\bigl[
    \lVert \psi_0(
        Y_{ n, n }^0 - y
    ) \rVert_{ \mathcal{H} }^2
\bigr]
\bigr)^{ \nicefrac{1}{2} }
\leq
C
\biggl[
    \frac{ e^\kappa ( 1 + 4 c ) }{ M_n }
\biggr]^{ \nicefrac{n}{2} }
< \infty
,
\end{equation}
\item
\label{item:cor:complexity_analysis2}
it holds
for all
$ n \in \N $
that
$ \Cost_{ n, n }
\leq
( 3 + \gamma )^n
( M_n )^{ n }
\mathfrak{z}
$,
and
\item
\label{item:cor:complexity_analysis3}
there exists
$ ( N_\varepsilon )_{ \varepsilon \in ( 0, 1 ] } \subseteq \N $
such that
it holds
for all
$ \varepsilon \in ( 0, 1 ] $,
$ \delta \in ( 0, \infty ) $
that
$ \sup_{ n \in \{ N_\varepsilon, N_\varepsilon + 1, \ldots \} }
\bigl(
\E\bigl[
    \lVert \psi_0(
        Y_{ n, n }^0 - y
    ) \rVert_{ \mathcal{H} }^2
\bigr]
\bigr)^{ \nicefrac{1}{2} }
\leq
\varepsilon $
and
\begin{equation}
\!%\linebreak
\Cost_{ N_\varepsilon, N_\varepsilon }
\leq
\mathfrak{z}
( 3 + \gamma )
e^\kappa
C^{ 2 ( 1 + \delta ) }
\biggl(
    1
    +
    \sup_{ n \in \N }
    \biggl[
        \frac{
        [ \mathfrak{B} e^{ 2 \kappa } ( 3 + \gamma )( 1 + 4 c ) ]^{ ( 1 + \delta ) }
        }{
        ( M_{ n } )^{ \delta }
        }
    \biggr]^n
\biggr)
\varepsilon^{ -2 ( 1 + \delta ) }
<
\infty
.
\end{equation}
\end{enumerate}
\end{corollary}
\begin{proof}[Proof of Corollary~\ref{cor:complexity_analysis}]
Throughout this proof
let
$ ( \mathfrak{c}_k )_{ k \in \N_0 } \subseteq ( 0, \infty ) $
be the family of real numbers which satisfies
for all
$ k \in \N_0 $
that
$ \mathfrak{c}_k = k! $.
Note that
the assumption that
$ \sup_{ j \in \N }
\nicefrac{ M_j }{ j } \leq \kappa $
ensures
for all
$ n \in \N $
that
\begin{equation}
\label{eq:Mn_condition}
\max_{ k \in \{ 0, 1, \ldots, n \} }
    \tfrac{ ( M_n )^{ k } }{ \mathfrak{ c }_k }
=
\max_{ k \in \{ 0, 1, \ldots, n \} }
    \tfrac{ ( M_n )^{ k } }{ k! }
\leq
\smallsum_{ k = 0 }^{\infty}
    \tfrac{ ( M_n )^{ k } }{ k! }
=
e^{ M_n }
\leq
e^{ \kappa n }
=
( e^{ \kappa } )^n
.
\end{equation}
Next observe that
(\ref{item:independence2})
in Proposition~\ref{prop:independence}
implies
for all
$ n, j \in \N $,
$ \theta \in \Theta $
that
$ \sigma_{ \Omega }( Y_{ n, j }^\theta )
\subseteq
\sigma_{ \Omega }\bigl(
    ( Z^{ ( \theta, \vartheta ) } )_{ \vartheta \in \Theta }
\bigr)
$.
This demonstrates
for all
$ n, j \in \N $
that
$
\sigma_{ \Omega }\bigl(
    Y^{ 0 }_{ n, j }, Y^{ 1 }_{ n - 1, j }
\bigr)
\subseteq
\sigma_{ \Omega }\bigl(
    ( Z^{ ( 0, \theta ) } )_{ \theta \in \Theta },
\allowbreak
    ( Z^{ ( 1, \theta ) } )_{ \theta \in \Theta }
\bigr)
$.
The fact that it holds
for every
$ k \in \N_0 $
that
$
\sigma_{ \Omega }\bigl(
    ( Z^{ ( 0, \theta ) } )_{ \theta \in \Theta },
    ( Z^{ ( 1, \theta ) } )_{ \theta \in \Theta }
\bigr)
$,
$ Z^0 $,
and
$ \psi_k $
are independent
hence
shows
for every
$ k \in \N_0 $,
$ n, j \in \N $
that
\begin{equation}
\label{eq:Y0_Y1_psi_Z0_independent}
\sigma_{ \Omega }\bigl(
    Y^{ 0 }_{ n, j }, Y^{ 1 }_{ n - 1, j }
\bigr)
\quad
\text{and}
\quad
\sigma_\Omega( \psi_k, Z^0 )
\end{equation}
are independent.
Furthermore,
the fact that it holds
for all
$ k \in \N_0 $,
$ x \in \mathcal{Y} $
that
$ \Omega \ni \omega
\mapsto [ \psi_k( \omega ) ]( x ) \in \mathcal{H} $
is an
$ \mathscr{F} $/$ \mathscr{B}( \mathcal{H} ) $-measurable function
(cf.~Lemma~\ref{lem:strong_sigma_algebra})
and
the fact that it holds
for all
$ k \in \N_0 $,
$ \omega \in \Omega $
that
$ \mathcal{Y} \ni x
\mapsto [ \psi_k( \omega ) ]( x ) \in \mathcal{H} $
is a continuous function
yield that
\begin{equation}
\label{eq:product_random_field}
\mathcal{Y} \times \mathcal{Y} \times \Omega \ni ( u, v, \omega )
\mapsto [ \psi_k( \omega ) ]( u - v ) \in \mathcal{H}
\end{equation}
is a continuous random field.
Moreover,
note that
the fact that
$ ( \mathcal{Y}, \lVert \cdot \rVert_{ \mathcal{Y} } ) $ is separable
ensures that
$ \mathscr{B}( \mathcal{Y} ) \otimes \mathscr{B}( \mathcal{Y} )
= \mathscr{B}( \mathcal{Y} \times \mathcal{Y} ) $.
This,
(\ref{item:prod_measurability1}) in Corollary~\ref{cor:prod_measurability},
\eqref{eq:Y0_Y1_psi_Z0_independent},
\cite[Lemma~2.2]{HutzenthalerJentzenKruseNguyenVonWurstemberger2018arXivv2}
(with
%$ ( \Omega, \mathcal{F}, \P ) = ( \Omega, \mathscr{F}, \P ) $,
$ \mathcal{G} = \sigma_\Omega( \psi_k, Z^0 ) $,
$ ( S, \mathcal{S} ) = ( \mathcal{Y} \times \mathcal{Y}, \mathscr{B}( \mathcal{Y} ) \otimes \mathscr{B}( \mathcal{Y} ) ) $,
$ U = ( \mathcal{Y} \times \mathcal{Y} \times \Omega \ni ( u, v, \omega ) \mapsto
\lVert [ \psi_k( \omega ) ]( \Phi_{ n }( u, v, Z^{ 0 }( \omega ) ) ) \rVert_{ \mathcal{H} }^2
\in [ 0, \infty ) ) $,
$ Y = ( Y^{ 0 }_{ n, j }, Y^{ 1 }_{ n - 1, j } ) $
for
$ j, n \in \N $,
$ k \in \N_0 $
in the notation of~\cite[Lemma~2.2]{HutzenthalerJentzenKruseNguyenVonWurstemberger2018arXivv2}),
\eqref{eq:CA_hypothesis_II},
\eqref{eq:product_random_field},
and
\cite[Lemma~2.3]{HutzenthalerJentzenKruseNguyenVonWurstemberger2018arXivv2}
(with
%$ ( \Omega, \mathcal{F}, \P ) = ( \Omega, \mathscr{F}, \P ) $,
$ S = \mathcal{Y} \times \mathcal{Y} $,
$ U = ( \mathcal{Y} \times \mathcal{Y} \times \Omega \ni ( u, v, \omega ) \mapsto
\lVert [ \psi_k( \omega ) ]( u - v ) \rVert_{ \mathcal{H} }^2
\in [ 0, \infty ) ) $,
$ Y = ( Y^{ 0 }_{ n, j }, Y^{ 1 }_{ n - 1, j } ) $
for
$ j, n \in \N $,
$ k \in \N_0 $
in the notation of~\cite[Lemma~2.3]{HutzenthalerJentzenKruseNguyenVonWurstemberger2018arXivv2})
establish
for all
$ k \in \N_0 $,
$ n, j \in \N $
that
\begin{equation}
\label{eq:hypothesis_II_reformulation}
\begin{split}
& \E\Bigl[
\bigl\lVert \psi_k\bigl(
    \Phi_{ n } \bigl(
    Y^{ 0 }_{ n, j },
    Y^{ 1 }_{ n - 1, j },
    Z^{ 0 }
    \bigr)
\bigr) \bigr\rVert_{ \mathcal{H} }^2
\Bigr]
\\ &
=
\int_{ \mathcal{Y} \times \mathcal{Y} }
    \E\bigl[
    \lVert \psi_k(
        \Phi_{ n }( u, v, Z^{ 0 } )
    ) \rVert_{ \mathcal{H} }^2
    \bigr]
\, \bigl( \bigl( Y^{ 0 }_{ n, j }, Y^{ 1 }_{ n - 1, j } \bigr)( \P )_{ \mathscr{B}( \mathcal{Y} ) \otimes \mathscr{B}( \mathcal{Y} ) } \bigr)
( \mathrm{d} u, \mathrm{d} v )
\\ &
=
\int_{ \mathcal{Y} \times \mathcal{Y} }
    \E\bigl[
    \lVert \psi_k(
        \Phi_{ n }( u, v, Z^{ 0 } )
    ) \rVert_{ \mathcal{H} }^2
    \bigr]
\, \bigl( \bigl( Y^{ 0 }_{ n, j }, Y^{ 1 }_{ n - 1, j } \bigr)( \P )_{ \mathscr{B}( \mathcal{Y} \times \mathcal{Y} ) } \bigr)
( \mathrm{d} u, \mathrm{d} v )
\\ &
\leq
c
\int_{ \mathcal{Y} \times \mathcal{Y} }
    \E\bigl[
    \lVert \psi_{ k + 1 }(
        u - v
    ) \rVert_{ \mathcal{H} }^2
    \bigr]
\, \bigl( \bigl( Y^{ 0 }_{ n, j }, Y^{ 1 }_{ n - 1, j } \bigr)( \P )_{ \mathscr{B}( \mathcal{Y} \times \mathcal{Y} ) } \bigr)
( \mathrm{d} u, \mathrm{d} v )
\\ &
=
c
\, \E\Bigl[
\bigl\lVert \psi_{ k + 1 }\bigl(
    Y^{ 0 }_{ n, j }
    -
    Y^{ 1 }_{ n - 1, j }
\bigr) \bigr\rVert_{ \mathcal{H} }^2
\Bigr]
.
\end{split}
\end{equation}
Combining
\eqref{eq:Mn_condition}
and
\eqref{eq:hypothesis_II_reformulation}
with
Theorem~\ref{thm:complexity_analysis}
shows
(\ref{item:cor:complexity_analysis1})--(\ref{item:complexity_analysis3}).
The proof of Corollary~\ref{cor:complexity_analysis} is thus complete.
\end{proof}

\begin{lemma}
\label{lem:M_j_conditions}
Let
$ \kappa \in [ 1, \infty ) $,
$ ( M_j )_{ j \in \N } \subseteq \N $
satisfy
for all
$ j \in \N $
that
$ M_j < M_{ j + 1 } $
and
$ M_j \leq \kappa j $.
Then
\begin{enumerate}[(i)]
\item
\label{item:lem:M_j_conditions1}
it holds
for all
$ j \in \N $
that
$ j \leq M_j \leq \kappa j $,
\item
\label{item:lem:M_j_conditions2}
it holds that
$ \liminf_{ j \to \infty } M_j = \infty $,
and
\item
\label{item:lem:M_j_conditions3}
it holds that
$ \sup_{ j \in \N }
\nicefrac{ M_{ j + 1 } }{ M_j } \leq 2 \kappa $.
\end{enumerate}
\end{lemma}
\begin{proof}[Proof of Lemma~\ref{lem:M_j_conditions}]
Note that
the assumption that
$ \forall \, j \in \N \colon M_j < M_{ j + 1 } $
and induction
show~(\ref{item:lem:M_j_conditions1}).
Next observe that~(\ref{item:lem:M_j_conditions1}) implies~(\ref{item:lem:M_j_conditions2}).
Furthermore,
the assumption that
$ \forall \, j \in \N \colon M_j \leq \kappa j $
and~(\ref{item:lem:M_j_conditions1})
ensure
for all
$ j \in \N $
that
\begin{equation}
\frac{ M_{ j + 1 } }{ M_j }
\leq
\frac{ \kappa( j + 1 ) }{ j }
=
\kappa + \frac{\kappa}{j}
\leq
2\kappa
.
\end{equation}
The proof of Lemma~\ref{lem:M_j_conditions} is thus complete.
\end{proof}

\section{MLP for semi-linear heat equations}
\label{sec:MLP_heat}

In this section we employ the abstract framework for generalised MLP approximations
developed in Section~\ref{sec:MLP_generalised}
to prove that
appropriate MLP approximations,
which essentially are generalised versions
of the MLP approximations proposed in
Hutzenthaler et al.~\cite{HutzenthalerJentzenKruseNguyenVonWurstemberger2018arXivv2},
are able to overcome the curse of dimensionality
in the numerical approximation of semi-linear heat equations
(see Theorem~\ref{thm:MLP_heat} and Corollary~\ref{cor:MLP_heat}
in Subsection~\ref{sec:MLP_heat_variable} below).

In the context of applying
the abstract complexity result about generalised MLP approximations
in Corollary~\ref{cor:complexity_analysis} above
to numerical approximations for semi-linear heat equations,
the separable $ \R $-Banach space
$ ( \mathcal{Y}, \lVert \cdot \rVert_{ \mathcal{Y} } ) $
in Corollary~\ref{cor:complexity_analysis}
is chosen to be a subspace of the vector space of real-valued at most polynomially growing continuous functions
defined on
$ [ 0, T ] \times \R^d $
equipped with a suitable polynomial growth norm,
where $ T \in ( 0, \infty ) $, $ d \in \N $
(see~\eqref{eq:def_Y}--\eqref{eq:def_Y_norm} below).
In Subsection~\ref{sec:polynomial_spaces}
we derive several elementary and well-known properties of these and related function spaces and their elements.
In particular,
Subsection~\ref{sec:completeness_separability}
deals with completeness and separability of such function spaces.
Lemma~\ref{lem:complete_polynomially_growing}
recalls that the vector space of real-valued at most polynomially growing continuous functions
defined on a non-empty subset of $ \R^d $
equipped with an appropriate polynomial growth norm is complete.
Thereafter,
we state in Proposition~\ref{prop:separable_compact}
the well-known fact
that the vector space of real-valued continuous functions
defined on a non-empty compact subset of $ \R^d $
equipped with the uniform norm
is a separable $ \R $-Banach space,
which follows directly from Lemma~\ref{lem:complete_polynomially_growing}
and, e.g.,
Conway~\cite[Theorem~6.6 in Chapter~V]{Conway1990}.
Using Proposition~\ref{prop:separable_compact}
we deduce the elementary fact
that also the vector space
of real-valued continuous functions with compact support
defined on a non-empty closed subset of $ \R^d $
equipped with a suitable polynomial growth norm is separable
(see Lemma~\ref{lem:separable_compact_supp}).
Subsection~\ref{sec:completeness_separability}
is concluded by the well-known result in Proposition~\ref{prop:Y_separable},
which establishes
a characterisation of the above mentioned
choice for the vector space
$ ( \mathcal{Y}, \lVert \cdot \rVert_{ \mathcal{Y} } ) $
(see~\eqref{eq:def_Y}--\eqref{eq:def_Y_norm} below)
and shows that it is indeed a separable $ \R $-Banach space.
Subsequently,
we provide in
Lemmas~\ref{lem:slower_growth}--\ref{lem:well_defined_Phi}
and Corollary~\ref{cor:well_defined_Phi}
in Subsection~\ref{sec:slower_growth}
three elementary results about sufficient conditions under which
suitable functions and suitable compositions of functions
grow strictly slower than a given polynomial order.
These results are used to ensure well-definedness
of certain functions introduced in Subsection~\ref{sec:setting_MLP_heat}
(see~\eqref{eq:definition_Phi_heat} below).
Furthermore,
Lemma~\ref{lem:boundedness_Phi} in Subsection~\ref{sec:compositions_growth}
offers an elementary polynomial growth estimate
for suitable compositions of functions.

In Subsection~\ref{sec:verification}
we specify a number of the objects appearing in Corollary~\ref{cor:complexity_analysis} above
for the example of MLP approximations for semi-linear heat equations
and
verify that the main assumptions of Corollary~\ref{cor:complexity_analysis} are fulfilled in this context.
In particular,
we first present
in Setting~\ref{setting:MLP_heat} in Subsection~\ref{sec:setting_MLP_heat}
the framework which we refer to throughout Subsection~\ref{sec:verification}.
In the subsection that follows,
Subsection~\ref{sec:measurability_heat},
we establish measurability properties of several of the involved functions
(see Lemmas~\ref{lem:measurability_psi_heat}--\ref{lem:measurability_Phi_heat}).
Subsequently,
Lemma~\ref{lem:scheme_description_heat} in Subsection~\ref{sec:formulation}
shows that
the MLP approximations introduced in
\eqref{eq:definition_Y_heat} in Setting~\ref{setting:MLP_heat}
fit into the abstract framework for generalised MLP approximations
developed in Section~\ref{sec:MLP_generalised}
(see~\eqref{eq:def_Y_abstract} above).
Moreover,
Subsection~\ref{sec:integrability}
is devoted to proving certain integrability properties
of the MLP approximations introduced in
\eqref{eq:definition_Y_heat} in Setting~\ref{setting:MLP_heat}
(see Lemma~\ref{lem:expectation_finite_heat}),
while
in Subsection~\ref{sec:estimates}
we verify that the estimates assumed
in~\eqref{eq:CA_hypothesis_I}--\eqref{eq:CA_hypothesis_III}
in Corollary~\ref{cor:complexity_analysis} hold true
for the functions introduced in Setting~\ref{setting:MLP_heat}
(see Lemmas~\ref{lem:hypothesis_I_heat}--\ref{lem:hypothesis_III_heat}).

Finally,
in Subsection~\ref{sec:complexity_analysis_heat}
we combine the results from
Subsection~\ref{sec:verification}
with Corollary~\ref{cor:complexity_analysis}
to obtain a complexity analysis for MLP approximations for semi-linear heat equations.
In Proposition~\ref{prop:MLP_heat_fixed} in Subsection~\ref{sec:MLP_heat_fixed}
this is done for semi-linear heat equations
of fixed space dimension $ d \in \N $
(cf.\ \cite[Theorem~3.8]{HutzenthalerJentzenKruseNguyenVonWurstemberger2018arXivv2}).
Thereafter,
Proposition~\ref{prop:MLP_heat_fixed} is used to establish
Theorem~\ref{thm:MLP_heat} in Subsection~\ref{sec:MLP_heat_variable},
which reveals that the MLP approximations in~\eqref{eq:def_Y_heat}
overcome the curse of dimensionality
in the numerical approximation of semi-linear heat equations
and
which essentially is a slight generalisation of
\cite[Theorem~1.1]{HutzenthalerJentzenKruseNguyenVonWurstemberger2018arXivv2}.
The last result in this section, Corollary~\ref{cor:MLP_heat},
is a direct consequence of Theorem~\ref{thm:MLP_heat}
and describes the special case of Theorem~\ref{thm:MLP_heat}
in which the non-linearity in the semi-linear heat equations
is the same for every dimension
(see~(\ref{item:cor:MLP_heat1}) in Corollary~\ref{cor:MLP_heat})
and in which the constants in the complexity estimate are not given explicitly
(see~(\ref{item:cor:MLP_heat2}) in Corollary~\ref{cor:MLP_heat}).

\subsection{Properties of spaces of at most polynomially growing continuous functions}
\label{sec:polynomial_spaces}

\subsubsection{Completeness and separability}
\label{sec:completeness_separability}

\begin{lemma}
\label{lem:complete_polynomially_growing}
Let $ d \in \N $,
$ p \in [ 0, \infty ) $,
let
$ \mathcal{A} \subseteq \R^d $
be a non-empty set,
let
$ \mathcal{V}
=
\bigl\{
    v \in C( \mathcal{A}, \R ) \colon
    \sup_{ x \in \mathcal{A} }
    \nicefrac{ \lvert v( x ) \rvert }{ \max\{ 1, \lVert x \rVert_{ \R^d }^p \} }
    < \infty
\bigr\} $,
and
let
$ \lVert \cdot \rVert_{ \mathcal{V} } \colon \mathcal{V} \to [ 0, \infty ) $
satisfy
for all
$ v \in \mathcal{V} $
that
$ \lVert v \rVert_{ \mathcal{V} }
=
\sup_{ x \in \mathcal{A} }
\nicefrac{ \lvert v( x ) \rvert }{ \max\{ 1, \lVert x \rVert_{ \R^d }^p \} }
$.
Then it holds that
$ ( \mathcal{V}, \lVert \cdot \rVert_{ \mathcal{V} } ) $
is an $ \R $-Banach space.
\end{lemma}
\begin{proof}[Proof of Lemma~\ref{lem:complete_polynomially_growing}]
Observe that it holds that
$ ( \mathcal{V}, \lVert \cdot \rVert_{ \mathcal{V} } ) $
is a normed $ \R $-vector space.
It thus remains to prove that
$ ( \mathcal{V}, \lVert \cdot \rVert_{ \mathcal{V} } ) $
is complete.
For this let
$ \mathcal{W} \subseteq C( \mathcal{A}, \R ) $
be the set given by
\begin{equation}
\textstyle
\mathcal{W}
=
\bigl\{
    w \in C( \mathcal{A}, \R ) \colon
    \sup_{ x \in \mathcal{A} }
    \, \lvert w( x ) \rvert
    < \infty
\bigr\}
,
\end{equation}
let
$ \lVert \cdot \rVert_{ \mathcal{W} } \colon \mathcal{W} \to [ 0, \infty ) $
satisfy
for all
$ w \in \mathcal{W} $
that
$ \lVert w \rVert_{ \mathcal{W} }
=
\sup_{ x \in \mathcal{A} }
\, \lvert w( x ) \rvert
$,
and
let
$ I \colon \mathcal{V} \to \mathcal{W} $
and
$ J \colon \mathcal{W} \to \mathcal{V} $
satisfy
for all
$ v \in \mathcal{V} $,
$ w \in \mathcal{W} $,
$ x \in \mathcal{A} $
that
$ [ I( v ) ]( x )
=
\nicefrac{ v( x ) }{ \max\{ 1, \lVert x \rVert_{ \R^d }^p \} } $
and
$ [ J( w ) ]( x )
=
w( x ) \max\{ 1, \lVert x \rVert_{ \R^d }^p \} $.
Note that
$ ( \mathcal{W}, \lVert \cdot \rVert_{ \mathcal{W} } ) $
is a normed $ \R $-vector space.
Furthermore,
Jentzen, Mazzonetto, \& Salimova~\cite[Corollary~2.3]{JentzenMazzonettoSalimova2018arXiv}
shows that
\begin{equation}
\label{eq:W_complete}
( \mathcal{W}, \lVert \cdot \rVert_{ \mathcal{W} } ) 
\end{equation}
is complete.
Next observe that it holds
for all
$ v \in \mathcal{V} $
that
\begin{equation}
\label{eq:isomtery}
\lVert I( v ) \rVert_{ \mathcal{W} }
=
\sup_{ x \in \mathcal{A} }
\, \lvert [ I( v ) ]( x ) \rvert
=
\sup_{ x \in \mathcal{A} }
\biggl[
\frac{ \lvert v( x ) \rvert }{ \max\{ 1, \lVert x \rVert_{ \R^d }^p \} }
\biggr]
=
\lVert v \rVert_{ \mathcal{V} }
.
\end{equation}
In addition,
note that it holds
for all
$ w \in \mathcal{W} $,
$ x \in \mathcal{A} $
that
\begin{equation}
[ I( J( w ) ) ]( x )
=
\frac{ [ J( w ) ]( x ) }{ \max\{ 1, \lVert x \rVert_{ \R^d }^p \} }
=
\frac{ w( x ) \max\{ 1, \lVert x \rVert_{ \R^d }^p \} }{ \max\{ 1, \lVert x \rVert_{ \R^d }^p \} }
=
w( x )
.
\end{equation}
Combining this with~\eqref{eq:isomtery}
ensures that
$ I \colon \mathcal{V} \to \mathcal{W} $
is a bijective linear isometry
and $ I^{ -1 } = J $.
This and~\eqref{eq:W_complete}
establish that
$ ( \mathcal{V}, \lVert \cdot \rVert_{ \mathcal{V} } )
=
( I^{ -1 }( \mathcal{W} ), \lVert \cdot \rVert_{ \mathcal{V} } ) $
is complete
and thus finish
the proof of Lemma~\ref{lem:complete_polynomially_growing}.
\end{proof}

\begin{proposition}
\label{prop:separable_compact}
Let
$ d \in \N $,
let
$ \mathcal{A} \subseteq \R^d $ be a non-empty compact set,
and
let
$ \lVert \cdot \rVert_{ C( \mathcal{A}, \R ) } \colon
\allowbreak
C( \mathcal{A}, \R )
\to [ 0, \infty ) $
satisfy
for all $ f \in C( \mathcal{A}, \R ) $
that
$ \lVert f \rVert_{ C( \mathcal{A}, \R ) }
=
\sup_{ x \in \mathcal{A} }
\vert f( x ) \rvert $.
Then it holds that
$ ( C( \mathcal{A}, \R ), \lVert \cdot \rVert_{ C( \mathcal{A}, \R ) } ) $
is a separable $ \R $-Banach space.
\end{proposition}

\begin{lemma}
\label{lem:separable_compact_supp}
Let $ d \in \N $,
$ p \in [ 0, \infty ) $,
let
$ \mathcal{A} \subseteq \R^d $
be a non-empty closed set,
and let
$ \ltriplevert \cdot \rtriplevert \colon
C_{ \mathrm{c} }( \mathcal{A}, \R )
\to [ 0, \infty ) $
satisfy
for all
$ f \in C_{ \mathrm{c} }( \mathcal{A}, \R ) $
that
$ \ltriplevert f \rtriplevert
=
\sup_{ x \in \mathcal{A} }
\nicefrac{ \lvert f( x ) \rvert }{ \max\{ 1, \lVert x \rVert_{ \R^d }^p \} }
$.
Then it holds that
$ ( C_{ \mathrm{c} }( \mathcal{A}, \R ), \ltriplevert \cdot \rtriplevert ) $
is a separable normed $ \R $-vector space.
\end{lemma}
\begin{proof}[Proof of Lemma~\ref{lem:separable_compact_supp}]
Throughout this proof
let $ \mathbf{y} \in \mathcal{A} $,
let
$ N = \min( [ \lVert \mathbf{y} \rVert_{ \R^d }, \infty ) \cap \N ) $,
let
$ \mathcal{S}_n \subseteq C_{ \mathrm{c} }( \mathcal{A}, \R ) $,
$ n \in \{ N, N + 1, \ldots \} $,
be the sets which satisfy
for all
$ n \in \{ N, N + 1, \ldots \} $
that
\begin{equation}
\label{eq:Sn_definition}
\mathcal{S}_n
=
\bigl\{
f \in C( \mathcal{A}, \R ) \colon
\{ x \in \mathcal{A} \colon f( x ) \neq 0 \}
\subseteq [ -n, n ]^d
\bigr\}
,
\end{equation}
let
$ \llbracket \cdot \rrbracket_{ n } \colon
C( \mathcal{A} \cap [ -n, n ]^d, \R ) \to [ 0, \infty ) $,
$ n \in \{ N, N + 1, \ldots \} $,
satisfy
for all
$ n \in \{ N, N + 1, \ldots \} $,
$ f \in C( \mathcal{A} \cap [ -n, n ]^d, \R ) $
that
\begin{equation}
\llbracket f \rrbracket_{ n }
=
\sup_{ x \in \mathcal{A} \cap [ -n, n ]^d }
\biggl[
\frac{ \lvert f( x ) \rvert }{ \max\{ 1, \lVert x \rVert_{ \R^d }^p \} }
\biggr]
,
\end{equation}
and
let
$ I_n \colon \mathcal{S}_n \to C( \mathcal{A} \cap [ -n, n ]^d, \R ) $,
$ n \in \{ N, N + 1, \ldots \} $,
satisfy
for all
$ n \in \{ N, N + 1, \ldots \} $,
$ f \in \mathcal{S}_n $
that
$ I_n( f ) = f \vert_{ \mathcal{A} \cap [ -n, n ]^d } $.
Note that
\eqref{eq:Sn_definition}
proves
for all
$ n \in \{ N, N + 1, \ldots \} $,
$ f \in \mathcal{S}_n $
that
\begin{equation}
\llbracket I_n( f ) \rrbracket_{ n }
=
\sup_{ x \in \mathcal{A} \cap [ -n, n ]^d }
\biggl[
\frac{ \lvert f( x ) \rvert }{ \max\{ 1, \lVert x \rVert_{ \R^d }^p \} }
\biggr]
=
\sup_{ x \in \mathcal{A} }
\biggl[
\frac{ \lvert f( x ) \rvert }{ \max\{ 1, \lVert x \rVert_{ \R^d }^p \} }
\biggr]
=
\ltriplevert f \rtriplevert
.
\end{equation}
This and the fact that it holds
for all
$ n \in \{ N, N + 1, \ldots \} $
that
$ ( \mathcal{S}_n, \ltriplevert \cdot \rtriplevert \vert_{ \mathcal{S}_n } ) $
and
$ ( C( \mathcal{A} \cap [ -n, n ]^d, \R ), \llbracket \cdot \rrbracket_{ n } ) $
are normed $ \R $-vector spaces
ensure
for all
$ n \in \{ N, N + 1, \ldots \} $
that
\begin{equation}
\label{eq:In_isometry}
I_n \colon \mathcal{S}_n \to C( \mathcal{A} \cap [ -n, n ]^d, \R )
\end{equation}
is a linear isometry.
Next observe that it holds
for all
$ n \in \{ N, N + 1, \ldots \} $,
$ f \in C( \mathcal{A} \cap [ -n, n ]^d, \R ) $
that
\begin{equation}
\label{eq:norm_equiv}
\llbracket f \rrbracket_{ n }
\leq
\sup_{ x \in \mathcal{A} \cap [ -n, n ]^d }
\lvert f( x ) \rvert
\leq
\sup_{ x \in \mathcal{A} \cap [ -n, n ]^d }
\Biggl[
\frac{ \lvert f( x ) \rvert \bigl( n \sqrt{ d } \bigr)^p }{ \max\{ 1, \lVert x \rVert_{ \R^d }^p \} }
\Biggr]
=
\llbracket f \rrbracket_{ n }
\bigl( n \sqrt{ d } \bigr)^p
.
\end{equation}
In addition,
the assumption that
$ \mathcal{A} \subseteq \R^d $
is a closed set
and the fact that $ \mathbf{y} \in \mathcal{A} $
ensure
for all
$ n \in \{ N, N + 1, \ldots \} $
that
$ \mathcal{A} \cap [ -n, n ]^d $
is a non-empty compact set.
Proposition~\ref{prop:separable_compact}
and~\eqref{eq:norm_equiv}
hence
show
for all
$ n \in \{ N, N + 1, \ldots \} $
that
$ ( C( \mathcal{A} \cap [ -n, n ]^d, \R ), \llbracket \cdot \rrbracket_{ n } ) $
is a separable $ \R $-Banach space.
This implies
for all
$ n \in \{ N, N + 1, \ldots \} $
that
$ ( I_n( \mathcal{S}_n ), \llbracket \cdot \rrbracket_{ n } \vert_{ I_n( \mathcal{S}_n ) }) $
is a separable normed $ \R $-vector space.
Combining this with~\eqref{eq:In_isometry}
hence
establishes
for all
$ n \in \{ N, N + 1, \ldots \} $
that
\begin{equation}
\label{eq:Sn_separable}
( \mathcal{S}_n, \ltriplevert \cdot \rtriplevert \vert_{ \mathcal{S}_n } )
\end{equation}
is a separable normed $ \R $-vector space.
Furthermore,
the assumption that
$ \mathcal{A} \subseteq \R^d $
is a closed set
and
\eqref{eq:Sn_definition}
demonstrate that
\begin{equation}
\begin{split}
& C_{ \mathrm{c} }( \mathcal{A}, \R )
=
\Bigl\{
f \in C( \mathcal{A}, \R ) \colon
\Bigl( \exists \, n \in \N \colon
\overline{ \{ x \in \mathcal{A} \colon f( x ) \neq 0 \} }^{ \R^d }
\subseteq \mathcal{A} \cap [ -n, n ]^d \Bigr)
\Bigr\}
\\ &
=
\bigl\{
f \in C( \mathcal{A}, \R ) \colon
\bigl( \exists \, n \in \N \colon
\{ x \in \mathcal{A} \colon f( x ) \neq 0 \}
\subseteq [ -n, n ]^d \bigr)
\bigr\}
\\ &
=
\bigl\{
f \in C( \mathcal{A}, \R ) \colon
\bigl( \exists \, n \in \{ N, N + 1, \ldots \} \colon
\{ x \in \mathcal{A} \colon f( x ) \neq 0 \}
\subseteq [ -n, n ]^d \bigr)
\bigr\}
\\ &
=
\bigcup_{ n = N }^{ \infty }
\bigl\{
f \in C( \mathcal{A}, \R ) \colon
\{ x \in \mathcal{A} \colon f( x ) \neq 0 \}
\subseteq [ -n, n ]^d
\bigr\}
=
\bigcup_{ n = N }^{ \infty }
\mathcal{S}_n
.
\end{split}
\end{equation}
This and~\eqref{eq:Sn_separable}
establish
that
$ ( C_{ \mathrm{c} }( \mathcal{A}, \R ), \ltriplevert \cdot \rtriplevert ) $
is a separable normed $ \R $-vector space.
The proof of Lemma~\ref{lem:separable_compact_supp} is thus complete.
\end{proof}

\begin{proposition}
\label{prop:Y_separable}
Let $ d \in \N $,
$ T \in ( 0, \infty ) $,
$ p \in [ 0, \infty ) $,
let
$ \mathcal{Y}
=
\bigl\{
    y \in C( [ 0, T ] \times \R^d, \R ) \colon
\allowbreak
    \limsup_{ \N \ni n \to \infty }
    \sup_{ ( t, x ) \in [ 0, T ] \times \R^d, \, \lVert x \rVert_{ \R^d } \geq n }
    \nicefrac{ \lvert y( t, x ) \rvert }{ \lVert x \rVert_{ \R^d }^p }
    =
    0
\bigr\}
$,
and
let
$ \lVert \cdot \rVert_{ \mathcal{Y} } \colon \mathcal{Y} \to [ 0, \infty ) $
satisfy
for all
$ y \in \mathcal{Y} $
that
$ \lVert y \rVert_{ \mathcal{Y} }
=
\sup_{ ( t, x ) \in [ 0, T ] \times \R^d }
\nicefrac{ \lvert y( t, x ) \rvert }{ \max\{ 1, \lVert x \rVert_{ \R^d }^p \} }
$.
Then
\begin{enumerate}[(i)]
\item
\label{item:Y_separable1}
it holds that
$ \mathcal{Y}
= \overline{ C_{ \mathrm{c} }( [ 0, T ] \times \R^d, \R ) }^{ \mathcal{Y} } $
and
\item
\label{item:Y_separable2}
it holds that
$ ( \mathcal{Y}, \lVert \cdot \rVert_{ \mathcal{Y} } ) $
is a separable $ \R $-Banach space.
\end{enumerate}
\end{proposition}
\begin{proof}[Proof of Proposition~\ref{prop:Y_separable}]
Throughout this proof
let
$ \tau_n \in C( \R^d, [ 0, 1 ] ) $,
$ n \in \N $,
satisfy
for all
$ n \in \N $,
$ x \in \R^d $
that
\begin{equation}
\tau_n( x )
=
\max\{
    \min\{ n + 1 - \lVert x \rVert_{ \R^d }, 1 \}, 0
\}
=
\begin{cases}
1 & \colon \lVert x \rVert_{ \R^d } \leq n \\
n + 1 - \lVert x \rVert_{ \R^d } & \colon n \leq \lVert x \rVert_{ \R^d } \leq n + 1 \\
0 & \colon n + 1 \leq \lVert x \rVert_{ \R^d }
\end{cases}
,
\end{equation}
let $ \mathbf{y} \in \mathcal{Y} $,
and
let
$ y_n \in C( [ 0, T ] \times \R^d, \R ) $,
$ n \in \N $,
satisfy
for all
$ n \in \N $,
$ t \in [ 0, T] $,
$ x \in \R^d $
that
$ y_n( t, x ) = \tau_n( x ) \, \mathbf{y}( t, x ) $.
Note that it holds that
$ ( y_n )_{ n \in \N } \subseteq C_{ \mathrm{c} }( [ 0, T ] \times \R^d, \R ) \subseteq \mathcal{Y} $
and
\begin{equation}
\begin{split}
\limsup_{ \N \ni n \to \infty }
\, \lVert \mathbf{y} - y_n \rVert_{ \mathcal{Y} }
& =
\limsup_{ \N \ni n \to \infty }
\sup_{ ( t, x ) \in [ 0, T ] \times \R^d }
\biggl[
\frac{ \lvert \mathbf{y}( t, x ) - \tau_n( x ) \, \mathbf{y}( t, x ) \rvert }{ \max\{ 1, \lVert x \rVert_{ \R^d }^p \} }
\biggr]
\\ &
=
\limsup_{ \N \ni n \to \infty }
\sup_{ ( t, x ) \in [ 0, T ] \times \R^d, \, \lVert x \rVert_{ \R^d } \geq n }
\biggl[
\frac{ ( 1 - \tau_n( x ) ) \lvert \mathbf{y}( t, x ) \rvert }{ \lVert x \rVert_{ \R^d }^p }
\biggr]
\\ &
\leq
\limsup_{ \N \ni n \to \infty }
\sup_{ ( t, x ) \in [ 0, T ] \times \R^d, \, \lVert x \rVert_{ \R^d } \geq n }
\frac{ \lvert \mathbf{y}( t, x ) \rvert }{ \lVert x \rVert_{ \R^d }^p }
= 0
.
\end{split}
\end{equation}
This proves that
$ \mathcal{Y}
\subseteq \overline{ C_{ \mathrm{c} }( [ 0, T ] \times \R^d, \R ) }^{ \mathcal{Y} } $.
In addition,
observe that
the fact that
$ \mathcal{Y}
\supseteq C_{ \mathrm{c} }( [ 0, T ] \times \R^d, \R ) $
ensures
that
\smash{$ \mathcal{Y}
%= \overline{ \mathcal{Y} }^{ \mathcal{Y} }
\supseteq \overline{ C_{ \mathrm{c} }( [ 0, T ] \times \R^d, \R ) }^{ \mathcal{Y} } $}.
This finishes the proof of~(\ref{item:Y_separable1}).
It thus remains to show~(\ref{item:Y_separable2}).
For this
let
$ \mathcal{V} \subseteq C( [ 0, T ] \times \R^d, \R ) $
be the set given by
\begin{equation}
\textstyle
\mathcal{V}
=
\Bigl\{
    v \in C( [ 0, T ] \times \R^d, \R ) \colon
    \sup_{ ( t, x ) \in [ 0, T ] \times \R^d }
    \Bigl[
    \tfrac{ \lvert v( t, x ) \rvert }{ \max\{ 1, \lVert x \rVert_{ \R^d }^p \} }
    \Bigr]
    < \infty
\Bigr\}
,
\end{equation}
let
$ \lVert \cdot \rVert_{ \mathcal{V} } \colon \mathcal{V} \to [ 0, \infty ) $
satisfy
for all
$ v \in \mathcal{V} $
that
\begin{equation}
\lVert v \rVert_{ \mathcal{V} }
=
\sup_{ ( t, x ) \in [ 0, T ] \times \R^d }
\biggl[
\frac{ \lvert v( t, x ) \rvert }{ \max\{ 1, \lVert x \rVert_{ \R^d }^p \} }
\biggr]
,
\end{equation}
let $ \mathbf{v} \in \mathcal{V} $,
and
let
$ ( v_n )_{ n \in \N } \subseteq \mathcal{Y} \subseteq \mathcal{V} $
be a sequence which satisfies
$ \limsup_{ \N \ni n \to \infty }
\, \lVert \mathbf{v} - v_n \rVert_{ \mathcal{V} }
\allowbreak
= 0 $.
Note that this implies that
\begin{equation}
\begin{split}
&
\limsup_{ \N \ni n \to \infty }
\sup_{ ( t, x ) \in [ 0, T ] \times \R^d, \, \lVert x \rVert_{ \R^d } \geq n }
\frac{ \lvert \mathbf{v}( t, x ) \rvert }{ \lVert x \rVert_{ \R^d }^p }
\\ &
\leq
\limsup_{ \N \ni m \to \infty }
\limsup_{ \N \ni n \to \infty }
\sup_{ ( t, x ) \in [ 0, T ] \times \R^d, \, \lVert x \rVert_{ \R^d } \geq n }
\frac{ \lvert \mathbf{v}( t, x ) - v_m( t, x ) \rvert }{ \lVert x \rVert_{ \R^d }^p }
\\ & \quad
+
\limsup_{ \N \ni m \to \infty }
\limsup_{ \N \ni n \to \infty }
\sup_{ ( t, x ) \in [ 0, T ] \times \R^d, \, \lVert x \rVert_{ \R^d } \geq n }
\frac{ \lvert v_m( t, x ) \rvert }{ \lVert x \rVert_{ \R^d }^p }
\\ &
\leq
\limsup_{ \N \ni m \to \infty }
\sup_{ ( t, x ) \in [ 0, T ] \times \R^d }
\biggl[
\frac{ \lvert \mathbf{v}( t, x ) - v_m( t, x ) \rvert }{ \max\{ 1, \lVert x \rVert_{ \R^d }^p \} }
\biggr]
=
\limsup_{ \N \ni m \to \infty }
\, \lVert \mathbf{v} - v_m \rVert_{ \mathcal{V} }
= 0.
\end{split}
\end{equation}
This establishes that
$ \mathbf{v} \in \mathcal{Y} $.
Therefore, it holds that
$ \mathcal{Y} \subseteq \mathcal{V} $
is a closed set.
The fact that
$ ( \mathcal{V}, \lVert \cdot \rVert_{ \mathcal{V} } ) $
is an $ \R $-Banach space
(cf.~Lemma~\ref{lem:complete_polynomially_growing})
hence
demonstrates that
$ ( \mathcal{Y}, \lVert \cdot \rVert_{ \mathcal{Y} } )
=
( \mathcal{Y}, \lVert \cdot \rVert_{ \mathcal{V} } \vert_{ \mathcal{Y} } ) $
is an $ \R $-Banach space.
Moreover,
note that
the fact that
$ ( C_{ \mathrm{c} }( [ 0, T ] \times \R^d, \R ),
\allowbreak
\lVert \cdot \rVert_{ \mathcal{Y} }
\vert_{ C_{ \mathrm{c} }( [ 0, T ] \times \R^d, \R ) } ) $
is a separable normed $ \R $-vector space
(cf.~Lemma~\ref{lem:separable_compact_supp})
and~(\ref{item:Y_separable1})
assure that
\begin{equation}
( \mathcal{Y}, \lVert \cdot \rVert_{ \mathcal{Y} } )
=
\Bigl( \overline{ C_{ \mathrm{c} }( [ 0, T ] \times \R^d, \R ) }^{ \mathcal{Y} },
\lVert \cdot \rVert_{ \mathcal{Y} }
\big\vert_{ \overline{ C_{ \mathrm{c} }( [ 0, T ] \times \R^d, \R ) }^{ \mathcal{Y} } } \Bigr)
\end{equation}
is separable.
This establishes~(\ref{item:Y_separable2}).
The proof of Proposition~\ref{prop:Y_separable}
is thus complete.
\end{proof}

\subsubsection{Sufficient conditions for strictly slower growth}
\label{sec:slower_growth}

\begin{lemma}
\label{lem:slower_growth}
Let $ d \in \N $,
$ T \in ( 0, \infty ) $,
$ p \in [ 0, \infty ) $,
$ q \in ( p, \infty ) $
and
let
$ y \in C( [ 0, T ] \times \R^d, \R ) $
satisfy
$
\sup_{ ( t, x ) \in [ 0, T ] \times \R^d }
\nicefrac{ \lvert y( t, x ) \rvert }{ \max\{ 1, \lVert x \rVert_{ \R^d }^p \} }
< \infty
$.
Then
\begin{equation}
\limsup_{ \N \ni n \to \infty }
\sup_{ ( t, x ) \in [ 0, T ] \times \R^d, \, \lVert x \rVert_{ \R^d } \geq n }
\frac{ \lvert y( t, x ) \rvert }{ \lVert x \rVert_{ \R^d }^q }
=
0
.
\end{equation}
\end{lemma}
\begin{proof}[Proof of Lemma~\ref{lem:slower_growth}]
Throughout this proof
let $ C \in [ 0, \infty ) $
be the real number which satisfies
$ C =
\sup_{ ( t, x ) \in [ 0, T ] \times \R^d }
\nicefrac{ \lvert y( t, x ) \rvert }{ \max\{ 1, \lVert x \rVert_{ \R^d }^p \} }
$.
Observe that it holds that
\begin{equation}
\begin{split}
& \limsup_{ \N \ni n \to \infty }
\sup_{ ( t, x ) \in [ 0, T ] \times \R^d, \, \lVert x \rVert_{ \R^d } \geq n }
\frac{ \lvert y( t, x ) \rvert }{ \lVert x \rVert_{ \R^d }^q }
\\ & =
\limsup_{ \N \ni n \to \infty }
\sup_{ ( t, x ) \in [ 0, T ] \times \R^d, \, \lVert x \rVert_{ \R^d } \geq n }
\biggl[
    \frac{ \lvert y( t, x ) \rvert }{ \max\{ 1, \lVert x \rVert_{ \R^d }^p \} }
    \frac{ \max\{ 1, \lVert x \rVert_{ \R^d }^p \} }{ \lVert x \rVert_{ \R^d }^q }
\biggr]
\\ & \leq
C \,
\limsup_{ \N \ni n \to \infty }
\sup_{ x \in \R^d, \, \lVert x \rVert_{ \R^d } \geq n }
\biggr[
    \frac{ \max\{ 1, \lVert x \rVert_{ \R^d }^p \} }{ \lVert x \rVert_{ \R^d }^q }
\biggr]
\\ & =
C \,
\limsup_{ \N \ni n \to \infty }
\sup_{ x \in \R^d, \, \lVert x \rVert_{ \R^d } \geq n }
    \frac{ 1 }{ \lVert x \rVert_{ \R^d }^{ q - p } }
=
C \,
\limsup_{ \N \ni n \to \infty }
    \frac{ 1 }{ n^{ q - p } }
=
0
.
\end{split}
\end{equation}
The proof of Lemma~\ref{lem:slower_growth} is thus complete.
\end{proof}

\begin{lemma}
\label{lem:well_defined_Phi}
Let $ d \in \N $,
$ T, q \in ( 0, \infty ) $,
let
$ \mathcal{Y}
=
\bigl\{
    y \in C( [ 0, T ] \times \R^d, \R ) \colon
    \limsup_{ \N \ni n \to \infty }
\allowbreak
    \sup_{ ( t, x ) \in [ 0, T ] \times \R^d, \, \lVert x \rVert_{ \R^d } \geq n }
    \nicefrac{ \lvert y( t, x ) \rvert }{ \lVert x \rVert_{ \R^d }^q }
    =
    0
\bigr\}
$,
let
$ \varrho = ( \varrho_1, \varrho_2 ) \in C( [ 0, T ] \times \R^d, [ 0, T ] \times \R^d ) $
satisfy
$ \sup_{ ( t, x ) \in [ 0, T ] \times \R^d }
\nicefrac{ \lVert \varrho_2( t, x ) \rVert_{ \R^d } }{ \max\{ 1, \lVert x \rVert_{ \R^d } \} }
< \infty $,
and let
$ y \in \mathcal{Y} $.
Then
it holds that
$ y \circ \varrho \in \mathcal{Y} $.
\end{lemma}
\begin{proof}[Proof of Lemma~\ref{lem:well_defined_Phi}]
Throughout this proof
let $ \varepsilon \in ( 0, \infty ) $
and
let
$ L, \mathfrak{N}, N \in \N $
satisfy
$ \sup_{ ( t, x ) \in [ 0, T ] \times \R^d }
\nicefrac{ \lVert \varrho_2( t, x ) \rVert_{ \R^d } }{ \max\{ 1, \lVert x \rVert_{ \R^d } \} }
\leq L $,
$ \sup_{ ( t, x ) \in [ 0, T ] \times \R^d, \, \lVert x \rVert_{ \R^d } \geq \mathfrak{N} }
\nicefrac{ \lvert y( t, x ) \rvert }{ \lVert x \rVert_{ \R^d }^q }
\leq
\frac{ \varepsilon }{ L^q }
$,
and
$ \varepsilon^{ -\nicefrac{1}{q} }
\sup_{ ( t, x ) \in [ 0, T ] \times \R^d, \, \lVert x \rVert_{ \R^d } \leq \mathfrak{N} }
\, \lvert y( t, x ) \rvert^{ \nicefrac{1}{q} }
\leq
N
$.
Observe that it holds for all
$ t \in [ 0, T ] $,
$ x \in \R^d $
with
$ \lVert x \rVert_{ \R^d } \geq N $
and
$ \lVert \varrho_2( t, x ) \rVert_{ \R^d } \leq \mathfrak{N} $
that
\begin{equation}
\label{eq:small_rho2}
\lvert y( \varrho( t, x ) ) \rvert
=
\lvert y( \varrho_1( t, x ), \varrho_2( t, x ) ) \rvert
\leq
\sup_{ ( s, \mathbf{x} ) \in [ 0, T ] \times \R^d, \, \lVert \mathbf{x} \rVert_{ \R^d } \leq \mathfrak{N} }
\lvert y( s, \mathbf{x} ) \rvert
\leq
\varepsilon
N^q
\leq
\varepsilon
 \lVert x \rVert_{ \R^d }^q
.
\end{equation}
In addition, note that it holds for all
$ t \in [ 0, T ] $,
$ x \in \R^d $
with
$ \lVert x \rVert_{ \R^d } \geq 1 $
and
$ \lVert \varrho_2( t, x ) \rVert_{ \R^d } \geq \mathfrak{N} $
that
\begin{equation}
\begin{split}
\lvert y( \varrho( t, x ) ) \rvert
& \leq
\Biggl[
\sup_{ ( s, \mathbf{x} ) \in [ 0, T ] \times \R^d, \, \lVert \mathbf{x} \rVert_{ \R^d } \geq \mathfrak{N} }
\frac{ \lvert y( s, \mathbf{x} ) \rvert }{ \lVert \mathbf{x} \rVert_{ \R^d }^q }
\Biggr]
\lVert \varrho_2( t, x ) \rVert_{ \R^d }^q
\\ & \leq
\frac{ \varepsilon }{ L^q }
\cdot
L^q \max\{ 1, \lVert x \rVert_{ \R^d }^q \}
=
\varepsilon
\lVert x \rVert_{ \R^d }^q
.
\end{split}
\end{equation}
This and
\eqref{eq:small_rho2}
establish
for all
$ n \in \{ N, N + 1, \ldots \} $
that
\begin{equation}
\sup_{ ( t, x ) \in [ 0, T ] \times \R^d, \, \lVert x \rVert_{ \R^d } \geq n }
\frac{ \lvert y( \varrho( t, x ) ) \rvert }{ \lVert x \rVert_{ \R^d }^q }
\leq
\sup_{ ( t, x ) \in [ 0, T ] \times \R^d, \, \lVert x \rVert_{ \R^d } \geq N }
\frac{ \lvert y( \varrho( t, x ) ) \rvert }{ \lVert x \rVert_{ \R^d }^q }
\leq
\varepsilon
.
\end{equation}
The fact that
$ y \circ \varrho \in C( [ 0, T ] \times \R^d, \R ) $
thus completes
the proof of Lemma~\ref{lem:well_defined_Phi}.
\end{proof}

\begin{corollary}
\label{cor:well_defined_Phi}
Let $ d \in \N $,
$ T \in ( 0, \infty ) $,
$ L, p \in [ 0, \infty ) $,
$ q \in ( p, \infty ) $,
let
$ \mathcal{Y}
=
\bigl\{
    y \in C( [ 0, T ] \times \R^d, \R ) \colon
    \limsup_{ \N \ni n \to \infty }
    \sup_{ ( t, x ) \in [ 0, T ] \times \R^d, \, \lVert x \rVert_{ \R^d } \geq n }
    \nicefrac{ \lvert y( t, x ) \rvert }{ \lVert x \rVert_{ \R^d }^q }
    =
    0
\bigr\}
$,
and let
$ \varrho = ( \varrho_1, \varrho_2 ) \in C( [ 0, T ], [ 0, T ] \times \R^d ) $,
$ f \in C( [ 0, T ] \times \R^d \times \R, \R ) $,
$ y \in \mathcal{Y} $
satisfy for all
$ t \in [ 0, T ] $,
$ x \in \R^d $,
$ v, w \in \R $
that
$ \lvert f( t, x, 0 ) \rvert
\leq
L \max\{ 1, \lVert x \rVert^p_{ \R^d } \} $
and
$ \lvert f( t, x, v ) - f( t, x, w ) \rvert
\leq
L \lvert v - w \rvert $.
Then
\begin{enumerate}[(i)]
\item
\label{item:well_defined_Phi1}
it holds that
$ ( [ 0, T ] \times \R^d \ni ( t, x ) \mapsto y( \varrho_1( t ), x + \varrho_2( t ) ) \in \R )
\in \mathcal{Y} $
and
\item
\label{item:well_defined_Phi2}
it holds that
$ ( [ 0, T ] \times \R^d \ni ( t, x ) \mapsto f( \varrho_1( t ), x + \varrho_2( t ), y( t, x ) ) \in \R )
\in \mathcal{Y} $.
\end{enumerate}
\end{corollary}
\begin{proof}[Proof of Corollary~\ref{cor:well_defined_Phi}]
Note that it holds that
\begin{equation}
\begin{split}
\sup_{ ( t, x ) \in [ 0, T ] \times \R^d }
\biggl[
\frac{ \lVert x + \varrho_2( t ) \rVert_{ \R^d } }{ \max\{ 1, \lVert x \rVert_{ \R^d } \} }
\biggr]
& \leq
\sup_{ ( t, x ) \in [ 0, T ] \times \R^d }
\biggl[
\frac{ \lVert x \rVert_{ \R^d } }{ \max\{ 1, \lVert x \rVert_{ \R^d } \} }
+
\frac{ \lVert \varrho_2( t ) \rVert_{ \R^d } }{ \max\{ 1, \lVert x \rVert_{ \R^d } \} }
\biggr]
\\ & \leq
1 +
\sup_{ t \in [ 0, T ] }
\lVert \varrho_2( t ) \rVert_{ \R^d }
< \infty.
\end{split}
\end{equation}
Lemma~\ref{lem:well_defined_Phi}
(with
$ d = d $,
$ T = T $,
$ q = q $,
$ \mathcal{Y} = \mathcal{Y} $,
$ \varrho =
( [ 0, T ] \times \R^d \ni ( t, x ) \mapsto ( \varrho_1( t ), x + \varrho_2( t ) ) \in [ 0, T ] \times \R^d ) $,
$ y = y $
in the notation of Lemma~\ref{lem:well_defined_Phi})
hence shows that
$ ( [ 0, T ] \times \R^d \ni ( t, x ) \mapsto y( \varrho_1( t ), x + \varrho_2( t ) ) \in \R )
\in \mathcal{Y} $.
This proves~(\ref{item:well_defined_Phi1}).
Next observe that
Lemma~\ref{lem:slower_growth}
(with
$ d = d $,
$ T = T $,
$ p = p $,
$ q = q $,
$ y =
( [ 0, T ] \times \R^d \ni ( t, x ) \mapsto f( t, x, 0 ) \in \R ) $
in the notation of Lemma~\ref{lem:slower_growth})
ensures that
$ ( [ 0, T ] \times \R^d \ni ( t, x ) \mapsto f( t, x, 0 ) \in \R )
\in \mathcal{Y} $.
Combining
this with~(\ref{item:well_defined_Phi1})
implies that
$ ( [ 0, T ] \times \R^d \ni ( t, x ) \mapsto f( \varrho_1( t ), x + \varrho_2( t ), 0 ) \in \R )
\in \mathcal{Y} $.
Therefore,
we obtain that
\begin{align*}
& \limsup_{ \N \ni n \to \infty }
\sup_{ ( t, x ) \in [ 0, T ] \times \R^d, \, \lVert x \rVert_{ \R^d } \geq n }
\biggl[
\frac{ \lvert f( \varrho_1( t ), x + \varrho_2( t ), y( t, x ) ) \rvert }{ \lVert x \rVert_{ \R^d }^q }
\biggr]
\\ \yesnumber & \leq
\limsup_{ \N \ni n \to \infty }
\sup_{ ( t, x ) \in [ 0, T ] \times \R^d, \, \lVert x \rVert_{ \R^d } \geq n }
\biggl[
\frac{ \lvert f( \varrho_1( t ), x + \varrho_2( t ), 0 ) \rvert }{ \lVert x \rVert_{ \R^d }^q }
\biggr]
\\ & \quad
\mathop{+}
\limsup_{ \N \ni n \to \infty }
\sup_{ ( t, x ) \in [ 0, T ] \times \R^d, \, \lVert x \rVert_{ \R^d } \geq n }
\biggl[
\frac{ \lvert f( \varrho_1( t ), x + \varrho_2( t ), y( t, x ) ) - f( \varrho_1( t ), x + \varrho_2( t ), 0 ) \rvert }{ \lVert x \rVert_{ \R^d }^q }
\biggr]
\\ & \leq
L \,
\limsup_{ \N \ni n \to \infty }
\sup_{ ( t, x ) \in [ 0, T ] \times \R^d, \, \lVert x \rVert_{ \R^d } \geq n }
\frac{ \lvert y( t, x ) \rvert }{ \lVert x \rVert_{ \R^d }^q }
=
0
.
\end{align*}
This and the fact that
$ ( [ 0, T ] \times \R^d \ni ( t, x ) \mapsto f( \varrho_1( t ), x + \varrho_2( t ), y( t, x ) ) \in \R )
\in C( [ 0, T ] \times \R^d, \R ) $
establish~(\ref{item:well_defined_Phi2}).
The proof of Corollary~\ref{cor:well_defined_Phi}
is thus complete.
\end{proof}

\subsubsection{Growth estimate for compositions}
\label{sec:compositions_growth}

\begin{lemma}
\label{lem:boundedness_Phi}
Let $ d \in \N $,
$ T \in ( 0, \infty ) $,
$ p \in [ 0, \infty ) $,
$ L \in [ 1, \infty ) $,
let
$ \llbracket \cdot \rrbracket \colon
C( [ 0, T ] \times \R^d, \R ) \to [ 0, \infty ] $
satisfy
for all
$ v \in C( [ 0, T ] \times \R^d, \R ) $
that
$ \llbracket v \rrbracket
=
\sup_{ ( t, x ) \in [ 0, T ] \times \R^d }
\nicefrac{ \lvert v( t, x ) \rvert }{ \max\{ 1, \lVert x \rVert_{ \R^d }^p \} }
$,
let
$ \varrho = ( \varrho_1, \varrho_2 ) \in C( [ 0, T ] \times \R^d, [ 0, T ] \times \R^d ) $
satisfy for all
$ t \in [ 0, T ] $,
$ x \in \R^d $
that
$  \lVert \varrho_2( t, x ) \rVert_{ \R^d } \leq L \max\{ 1, \lVert x \rVert_{ \R^d } \} $,
and let
$ v \in C( [ 0, T ] \times \R^d, \R ) $.
Then
it holds that
$ v \circ \varrho \in C( [ 0, T ] \times \R^d, \R ) $
and
$ \llbracket v \circ \varrho \rrbracket \leq L^p \llbracket v \rrbracket $.
\end{lemma}
\begin{proof}[Proof of Lemma~\ref{lem:boundedness_Phi}]
Observe that it holds that
$ v \circ \varrho \in C( [ 0, T ] \times \R^d, \R ) $.
In addition, note that it holds that
\begin{equation}
\begin{split}
\llbracket v \circ \varrho \rrbracket
& =
\sup_{ ( t, x ) \in [ 0, T ] \times \R^d }
\biggl[
\frac{ \lvert v( \varrho( t, x ) ) \rvert }{ \max\{ 1, \lVert x \rVert_{ \R^d }^p \} }
\biggr]
\leq
\llbracket v \rrbracket
\sup_{ ( t, x ) \in [ 0, T ] \times \R^d }
\biggl[
\frac{ \max\{ 1, \lVert \varrho_2( t, x ) \rVert_{ \R^d }^p \} }{ \max\{ 1, \lVert x \rVert_{ \R^d }^p \} }
\biggr]
\\ & \leq
\llbracket v \rrbracket
\sup_{ x \in \R^d }
\biggl[
\frac{ \max\bigl\{ 1, L^p \max\{ 1, \lVert x \rVert_{ \R^d }^p \} \bigr\} }{ \max\{ 1, \lVert x \rVert_{ \R^d }^p \} }
\biggr]
=
L^p
\llbracket v \rrbracket
.
\end{split}
\end{equation}
The proof of Lemma~\ref{lem:boundedness_Phi} is thus complete.
\end{proof}

\subsection{Verification of the assumed properties}
\label{sec:verification}

\subsubsection{Setting}
\label{sec:setting_MLP_heat}

\begin{setting}
\label{setting:MLP_heat}
Let
$ d \in \N $,
$ \xi \in \R^d $,
$ T \in ( 0, \infty ) $,
$ L, p \in [ 0, \infty ) $,
$ q \in ( p, \infty ) $,
$ \Theta = \cup_{ n = 1 }^\infty \Z^n $,
$ ( M_j )_{ j \in \N } \subseteq \N $,
let
$ ( \Omega, \mathscr{F}, \P ) $
be a probability space,
let
$ \mathbf{U} \colon \Omega \to [ 0, 1 ] $
and
$ U^\theta \colon \Omega \to [ 0, 1 ] $, $ \theta \in \Theta $,
be on $ [ 0, 1 ] $ uniformly distributed random variables,
let
$ \mathbf{W} \colon [ 0, T ] \times \Omega \to \R^d $
and
$ W^\theta \colon [ 0, T ] \times \Omega \to \R^d $, $ \theta \in \Theta $,
be standard Brownian motions with continuous sample paths,
%assume that
%%the family
%$ U^\theta $, $ \theta \in \Theta $,
%are independent,
%assume that
%%the family
%$ W^\theta $, $ \theta \in \Theta $,
%are independent,
assume that
$ ( U^\theta, W^\theta ) $, $ \theta \in \Theta $,
are independent,
assume that
$ \mathbf{U} $,
$ \mathbf{W} $,
$ ( U^\theta )_ { \theta \in \Theta } $,
and
$ ( W^\theta )_ { \theta \in \Theta } $
are independent,
let
$ f \in C( [ 0, T ] \times \R^d \times \R, \R ) $,
$ g \in C( \R^d, \R ) $,
$ y \in C( [ 0, T ] \times \R^d, \R ) $
satisfy for all
$ t \in [ 0, T ] $,
$ x \in \R^d $,
$ v, w \in \R $
that
$ \max\{ \lvert f( t, x, 0 ) \rvert, \lvert g( x ) \rvert \}
\leq
L \max\{ 1, \lVert x \rVert^p_{ \R^d } \} $,
$ \lvert f( t, x, v ) - f( t, x, w ) \rvert
\leq
L \lvert v - w \rvert $,
$ \sup_{ ( s, \mathbf{x} ) \in [ 0, T ] \times \R^d }
\nicefrac{ \lvert y( s, \mathbf{x} ) \rvert }{ \max\{ 1, \lVert \mathbf{x} \rVert_{ \R^d }^p \} }
< \infty $,
and
\begin{equation}
\label{eq:definition_y_heat}
y( t, x )
=
\E\biggl[
    g( x + \mathbf{W}_{ T - t } )
    +
    \int_{t}^{T}
        f\bigl(
            s, x + \mathbf{W}_{ s - t }, y( s, x + \mathbf{W}_{ s - t } )
        \bigr)
    \ud s
\biggr]
,
\end{equation}
let
$ Y_{ n, j }^\theta \colon [ 0, T ] \times \R^d \times \Omega \to \R $,
$ \theta \in \Theta $,
$  j \in \N $,
$ n \in ( \N_0 \cup \{ -1 \} ) $,
satisfy
for all
$ n, j \in \N $,
$ \theta \in \Theta $,
$ t \in [ 0, T ] $,
$ x \in \R^d $
that
$ Y_{ -1, j }^\theta( t, x ) = Y_{ 0, j }^\theta( t, x ) = 0 $
and
%\begin{align}
%& Y_{ n, j }^\theta( t, x )
%=
%\tfrac{1}{ ( M_j )^{ n } }
%\Biggl[
%\sum_{ i = 1 }^{ (M_j )^n }
%    g\bigl( x + W^{ ( \theta, 0, i ) }_{ T - t } \bigr)
%\Biggr]
%+
%\sum_{l=0}^{n-1}
%\tfrac{T - t}{ ( M_j )^{ n - l } }
%\Biggl[
%    \sum_{i=1}^{ ( M_j )^{ n - l } }
%\\ \nonumber &
%    \Bigl[
%        f\Bigl(
%            t + U^{ ( \theta, l, i ) } ( T - t ),
%            x + W^{ ( \theta, l, i ) }_{ U^{ ( \theta, l, i ) } ( T - t ) },
%            Y^{ ( \theta, l, i ) }_{ l, j }\bigl(
%                t + U^{ ( \theta, l, i ) } ( T - t ),
%                x + W^{ ( \theta, l, i ) }_{ U^{ ( \theta, l, i ) } ( T - t ) }
%            \bigr)
%        \Bigr)
%\\ \nonumber &
%        -
%        \mathbbm{1}_{ \N }( l )
%        f\Bigl(
%            t + U^{ ( \theta, l, i ) } ( T - t ),
%            x + W^{ ( \theta, l, i ) }_{ U^{ ( \theta, l, i ) } ( T - t ) },
%            Y^{ ( \theta, -l, i ) }_{ l - 1, j }\bigl(
%                t + U^{ ( \theta, l, i ) } ( T - t ),
%                x + W^{ ( \theta, l, i ) }_{ U^{ ( \theta, l, i ) } ( T - t ) }
%            \bigr)
%        \Bigr)
%    \Bigr]
%\Biggr]
%,
%\end{align}
\begin{align*}
\label{eq:definition_Y_heat}
& Y_{ n, j }^\theta( T - t, x )
=
\tfrac{1}{ ( M_j )^{ n } }
\Biggl[
\sum_{ i = 1 }^{ (M_j )^n }
    g\bigl( x + W^{ ( \theta, 0, i ) }_{ t } \bigr)
\Biggr]
+
\sum_{l=0}^{n-1}
\tfrac{t}{ ( M_j )^{ n - l } }
\Biggl[
    \sum_{i=1}^{ ( M_j )^{ n - l } }
\\ & \yesnumber
\quad
    \Bigl[
        f\Bigl(
            T - t + U^{ ( \theta, l, i ) } t,
            x + W^{ ( \theta, l, i ) }_{ U^{ ( \theta, l, i ) } t },
            Y^{ ( \theta, l, i ) }_{ l, j }\bigl(
                T - t + U^{ ( \theta, l, i ) } t,
                x + W^{ ( \theta, l, i ) }_{ U^{ ( \theta, l, i ) } t }
            \bigr)
        \Bigr)
\\ &
\quad
        -
        \mathbbm{1}_{ \N }( l )
        f\Bigl(
            T - t + U^{ ( \theta, l, i ) } t,
            x + W^{ ( \theta, l, i ) }_{ U^{ ( \theta, l, i ) } t },
            Y^{ ( \theta, -l, i ) }_{ l - 1, j }\bigl(
                T - t + U^{ ( \theta, l, i ) } t,
                x + W^{ ( \theta, l, i ) }_{ U^{ ( \theta, l, i ) } t }
            \bigr)
        \Bigr)
    \Bigr]
\Biggr]
,
\end{align*}
let
$ \mathcal{Y} \subseteq C( [ 0, T ] \times \R^d, \R ) $
be the set given by
\begin{equation}
\label{eq:def_Y}
\mathcal{Y}
=
\Biggl\{
    v \in C( [ 0, T ] \times \R^d, \R ) \colon
    \limsup_{ \N \ni n \to \infty }
    \sup_{ ( t, x ) \in [ 0, T ] \times \R^d, \, \lVert x \rVert_{ \R^d } \geq n }
    \frac{ \lvert v( t, x ) \rvert }{ \lVert x \rVert_{ \R^d }^q }
    =
    0
\Biggr\}
,
\end{equation}
let
$ \lVert \cdot \rVert_{ \mathcal{Y} } \colon \mathcal{Y} \to [ 0, \infty ) $
satisfy
for all
$ v \in \mathcal{Y} $
that
\begin{equation}
\label{eq:def_Y_norm}
\lVert v \rVert_{ \mathcal{Y} }
=
\sup_{ ( t, x ) \in [ 0, T ] \times \R^d }
\biggl[
\frac{ \lvert v( t, x ) \rvert }{ \max\{ 1, \lVert x \rVert_{ \R^d }^q \} }
\biggr]
,
\end{equation}
let
$ \mathcal{Z}
=
[ 0, 1 ] \times C( [ 0, T ], \R^d )
$,
let
$ \mathbf{d}_{ \mathcal{Z} } \colon \mathcal{Z} \times \mathcal{Z} \to [ 0, \infty ) $
satisfy
for all
$ \mathfrak{z} = ( \mathfrak{u}, \mathfrak{w} ),
\mathfrak{Z} = ( \mathfrak{U}, \mathfrak{W} )
\in \mathcal{Z} $
that
\begin{equation}
\mathbf{d}_{ \mathcal{Z} }( \mathfrak{z}, \mathfrak{Z} )
=
\lvert \mathfrak{u} - \mathfrak{U} \rvert
+
\lVert \mathfrak{w} - \mathfrak{W} \rVert_{ C( [ 0, T ], \R^d ) }
=
\lvert \mathfrak{u} - \mathfrak{U} \rvert
+
\sup_{ t \in [ 0, T ] }
\lVert \mathfrak{w}( t ) - \mathfrak{W}( t ) \rVert_{ \R^d }
,
\end{equation}
let
$ Z^{ \theta } \colon \Omega \to \mathcal{Z} $,
$ \theta \in \Theta $,
satisfy
for all
$ \theta \in \Theta $
that
$ Z^{ \theta } = ( U^\theta, W^\theta ) $,
let
$ \psi_k \colon \Omega \to \mathcal{Y}^* $, $ k \in \N_0 $,
satisfy
for all
$ k \in \N_0 $,
$ \omega \in \Omega $,
$ v \in \mathcal{Y} $
that
\begin{equation}
\label{eq:definition_psi_k}
[ \psi_k( \omega ) ] ( v )
=
\begin{cases}
v( 0, \xi )
& \colon k = 0
\\
\sqrt{ \frac{ ( \mathbf{U}( \omega ) )^{ k -1 } }{ ( k - 1 )! } }
v\bigl( \mathbf{U}( \omega ) T, \xi + \mathbf{W}_{ \mathbf{U}( \omega ) T }( \omega ) \bigr)
& \colon k \in \N
\end{cases}
,
\end{equation}
and
let
$ \Phi_l \colon \mathcal{Y} \times \mathcal{Y} \times \mathcal{Z} \to \mathcal{Y} $, $ l \in \N_0 $,
satisfy
for all
$ l \in \N_0 $,
$ v, w \in \mathcal{Y} $,
$ \mathfrak{z} = ( \mathfrak{u}, \mathfrak{w} ) \in \mathcal{Z} $,
$ t \in [ 0, T ] $,
$ x \in \R^d $
that
\begin{equation}
\label{eq:definition_Phi_heat}
\begin{split}
& [ \Phi_l( v, w, \mathfrak{z} ) ]( T - t, x )
\\ & =
\begin{cases}
g( x + \mathfrak{w}_{ t } )
+
t
f\bigl(
    T - t + \mathfrak{u}t, x + \mathfrak{w}_{ \mathfrak{u}t },
    v( T - t + \mathfrak{u}t, x + \mathfrak{w}_{ \mathfrak{u} t } )
\bigr)
& \colon l = 0
\\
\!\!
\begin{array}{l}
t \bigl[
f\bigl(
    T - t + \mathfrak{u}t, x + \mathfrak{w}_{ \mathfrak{u}t },
    v( T - t + \mathfrak{u}t, x + \mathfrak{w}_{ \mathfrak{u} t } )
\bigr)
\\
\quad
\mathop{-}
f\bigl(
    T - t + \mathfrak{u}t, x + \mathfrak{w}_{ \mathfrak{u}t },
    w( T - t + \mathfrak{u}t, x + \mathfrak{w}_{ \mathfrak{u} t } )
\bigr)
\bigr]
\end{array}
& \colon l \in \N
\end{cases}
\end{split}
\end{equation}
(cf.\ Lemma~\ref{lem:slower_growth} and Corollary~\ref{cor:well_defined_Phi}).
\end{setting}

\subsubsection{Measurability}
\label{sec:measurability_heat}

\begin{lemma}
\label{lem:measurability_psi_heat}
Assume Setting~\ref{setting:MLP_heat}
and
let
%$ \mathscr{S} \subseteq \{ \mathcal{S} \colon \mathcal{S} \subseteq \mathcal{Y}^* \} $
%be the set given by
$ \mathscr{S}
=
\sigma_{ \mathcal{Y}^* } \bigl( \bigl\{
    \{ \varphi \in \mathcal{Y}^* \colon \varphi( v ) \in \mathcal{B} \} \subseteq \mathcal{Y}^* \colon
    v \in \mathcal{Y},\, \mathcal{ B } \in \mathscr{B}( \R )
\bigr\} \bigr) $.
Then
it holds
for all
$ k \in \N_0 $
that
$ \psi_k \colon \Omega \to \mathcal{Y}^* $
is an $ \mathscr{F} $/$ \mathscr{S} $-measurable function.
\end{lemma}
\begin{proof}[Proof of Lemma~\ref{lem:measurability_psi_heat}]
Note that it holds for all
$ k \in \N_0 $,
$ v \in \mathcal{Y} $
that
$ \Omega \ni \omega \mapsto [ \psi_k( \omega ) ] ( v ) \in \R $
is an $ \mathscr{F} $/$ \mathscr{B}( \R ) $-measurable function.
Lemma~\ref{lem:strong_sigma_algebra}
(with
$ \mathcal{E} = \mathcal{Y} $,
$ ( \mathcal{F}, \mathscr{F} ) = ( \R, \mathscr{B}( \R ) ) $,
$ ( \mathcal{G}, \mathscr{G} ) = ( \Omega, \mathscr{F} ) $,
$ \mathcal{S} = \mathcal{Y}^* $,
$ \mathscr{S} = \mathscr{S} $,
$ \psi  = \psi_k $
for
$ k \in \N_0 $
in the notation of Lemma~\ref{lem:strong_sigma_algebra})
hence proves for all
$ k \in \N_0 $
that
$ \psi_k \colon \Omega \to \mathcal{Y}^* $
is an $ \mathscr{F} $/$ \mathscr{S} $-measurable function.
The proof of Lemma~\ref{lem:measurability_psi_heat} is thus complete.
\end{proof}

\begin{lemma}
\label{lem:measurability_Phi_heat}
Assume Setting~\ref{setting:MLP_heat}.
Then
\begin{enumerate}[(i)]
\item
\label{item:measurability_Phi_heat1}
it holds for all
$ l \in \N_0 $
that
$ \Phi_l \colon \mathcal{Y} \times \mathcal{Y} \times \mathcal{Z} \to \mathcal{Y} $
is a
continuous
function
and
\item
\label{item:measurability_Phi_heat2}
it holds for all
$ l \in \N_0 $
that
$ \Phi_l \colon \mathcal{Y} \times \mathcal{Y} \times \mathcal{Z} \to \mathcal{Y} $
is a
$ ( \mathscr{B}( \mathcal{Y} ) \otimes \mathscr{B}( \mathcal{Y} ) \otimes \mathscr{Z} ) $/$ \mathscr{B}( \mathcal{Y} ) $-measurable function.
\end{enumerate}
\end{lemma}
\begin{proof}[Proof of Lemma~\ref{lem:measurability_Phi_heat}]
Throughout this proof
let
$ \varphi_1 \colon \mathcal{Z} \to C( [ 0, T ], \R^d ) $,
$ \varphi_2, \varphi_3, F \colon \mathcal{Y} \times \mathcal{Z} \to \mathcal{Y} $,
$ \varphi_4 \colon \mathcal{Y} \to \mathcal{Y} $,
$ \Psi_1, \Psi_2 \colon \mathcal{Y} \times \mathcal{Z} \to \mathcal{Y} \times \mathcal{Z} $,
$ \mathfrak{g} \in \mathcal{Y} $,
$ G \colon \mathcal{Z} \to \mathcal{Y} $
satisfy
for all
$ v, w \in \mathcal{Y} $,
$ \mathfrak{z} = ( \mathfrak{u}, \mathfrak{w} ) \in \mathcal{Z} $,
$ t \in [ 0, T ] $,
$ x \in \R^d $
that
\begin{align}
[ \varphi_1( \mathfrak{z} ) ]( t ) & = \mathfrak{w}_{ \mathfrak{u}t }, &
\Psi_1( v, \mathfrak{z} )
& =
( v, \mathfrak{u}, \varphi_1( \mathfrak{z} ) ),
\\
[ \varphi_2( v, \mathfrak{z} ) ]( t, x )
& =
v( T - t + \mathfrak{u}t, x + \mathfrak{w}_{ t } ), &
\Psi_2( v, \mathfrak{z} )
& =
( \varphi_2( v, \mathfrak{z} ), \mathfrak{z} ),
\\
[ \varphi_3( v, \mathfrak{z} ) ]( t, x )
& =
t
f\bigl(
    T - t + \mathfrak{u}t, x + \mathfrak{w}_{ t },
    v( t, x )
\bigr), &
\mathfrak{g}( t, x ) & = g( x ),
\\
[ \varphi_4( v ) ]( t, x ) & = v( T - t, x ), &
G( \mathfrak{z} )
& =
\varphi_4( \varphi_2( \mathfrak{g}, \mathfrak{z} ) ),
\end{align}
and
$ F = \varphi_4 \circ \varphi_3 \circ \Psi_2 \circ \Psi_1 $
(cf.~Corollary~\ref{cor:well_defined_Phi}).
Note that it holds
for all
$ v \in \mathcal{Y} $,
$ \mathfrak{z} = ( \mathfrak{u}, \mathfrak{w} ) \in \mathcal{Z} $,
$ t \in [ 0, T ] $,
$ x \in \R^d $
that
\begin{equation}
\label{eq:composition_G}
\begin{split}
[ G( \mathfrak{z} ) ]( T - t, x )
& =
[ \varphi_4( \varphi_2( \mathfrak{g}, \mathfrak{z} ) ) ]( T - t, x )
=
[ \varphi_2( \mathfrak{g}, \mathfrak{z} ) ]( t, x )
\\ &
=
\mathfrak{g}( T - t + \mathfrak{u}t, x + \mathfrak{w}_{ t } )
=
g( x + \mathfrak{w}_{ t } )
\end{split}
\end{equation}
and
\begin{equation}
\label{eq:composition_F}
\begin{split}
[ F( v, \mathfrak{z} ) ]( T - t, x )
& =
\bigl[ \varphi_4\bigl(
( \varphi_3 \circ \Psi_2 \circ \Psi_1 )( v, \mathfrak{z} ) \bigr) \bigr]( T - t, x )
\\ &
=
[ ( \varphi_3 \circ \Psi_2 \circ \Psi_1 )( v, \mathfrak{z} ) ]( t, x )
=
\bigl[ ( \varphi_3 \circ \Psi_2 )\bigl( v, \mathfrak{u}, \varphi_1( \mathfrak{z} ) \bigr) \bigr]( t, x )
\\ &
=
\bigl[ \varphi_3\bigl(
    \varphi_2\bigl( v, \mathfrak{u}, \varphi_1( \mathfrak{z} ) \bigr),
    \mathfrak{u}, \varphi_1( \mathfrak{z} )
\bigr) \bigr]( t, x )
\\ &
=
t
f\bigl(
    T - t + \mathfrak{u}t, x + [ \varphi_1( \mathfrak{z} ) ]( t ),
    \bigl[ \varphi_2\bigl( v, \mathfrak{u}, \varphi_1( \mathfrak{z} ) \bigr) \bigr]( t, x )
\bigr)
\\ &
=
t
f\bigl(
    T - t + \mathfrak{u}t, x + [ \varphi_1( \mathfrak{z} ) ]( t ),
    v\bigl( T - t + \mathfrak{u}t, x + [ \varphi_1( \mathfrak{z} ) ]( t ) \bigr)
\bigr)
\\ &
=
t
f\bigl(
    T - t + \mathfrak{u}t, x + \mathfrak{w}_{ \mathfrak{u}t },
    v( T - t + \mathfrak{u}t, x + \mathfrak{w}_{ \mathfrak{u} t } )
\bigr)
.
\end{split}
\end{equation}
Combining~\eqref{eq:composition_G}--\eqref{eq:composition_F}
with~\eqref{eq:definition_Phi_heat}
ensures for all
$ l \in \N $,
$ v, w \in \mathcal{Y} $,
$ \mathfrak{z} \in \mathcal{Z} $
that
\begin{equation}
\label{eq:composition_Phi}
\Phi_0( v, w, \mathfrak{z} )
=
G( \mathfrak{z} )
+
F( v, \mathfrak{z} )
\qquad\text{and}\qquad
\Phi_l( v, w, \mathfrak{z} )
=
F( v, \mathfrak{z} )
-
F( w, \mathfrak{z} )
.
\end{equation}
In the following
we establish that
$ G \colon \mathcal{Z} \to \mathcal{Y} $
and
$ F \colon \mathcal{Y} \times \mathcal{Z} \to \mathcal{Y} $
are continuous functions.

First, we show that
$ \varphi_1 \colon \mathcal{Z} \to C( [ 0, T ], \R^d ) $
is a continuous function.
Throughout this paragraph
let
$ \varepsilon \in ( 0, \infty ) $,
$ \mathfrak{Z} = ( \mathfrak{U}, \mathfrak{W} ) \in \mathcal{Z} $
and
let
$ \Delta, \delta \in ( 0, \infty ) $
be real numbers which satisfy
$
\sup_{
\substack{ s, t \in [ 0, T ], \,
\lvert s - t \rvert \leq \Delta
}
}
\lVert
    \mathfrak{W}_{ s } - \mathfrak{W}_{ t }
\rVert_{ \R^d }
\leq
\tfrac{\varepsilon}{2}
$
and
$ \delta = \min\bigl\{ \tfrac{\Delta}{T}, \tfrac{\varepsilon}{2} \bigr\} $.
Observe that it holds for all
$ \mathfrak{z} = ( \mathfrak{u}, \mathfrak{w} ) \in \mathcal{Z} $
with
$ \mathbf{d}_{ \mathcal{Z} }(
    \mathfrak{z},
    \mathfrak{Z}
)
=
\lvert \mathfrak{u} - \mathfrak{U} \rvert
+
\lVert \mathfrak{w} - \mathfrak{W} \rVert_{ C( [ 0, T ], \R^d ) }
\leq \delta $
that
\begin{equation}
\begin{split}
& \lVert
    \varphi_1( \mathfrak{z} ) - \varphi_1( \mathfrak{Z} )
\rVert_{ C( [ 0, T ], \R^d ) }
=
\sup_{ t \in [ 0, T ] }
\lVert
    \mathfrak{w}_{ \mathfrak{u}t } - \mathfrak{W}_{ \mathfrak{U}t }
\rVert_{ \R^d }
\\ &
\leq
\biggl[
\sup_{ t \in [ 0, T ] }
\lVert
    \mathfrak{w}_{ \mathfrak{u}t } - \mathfrak{W}_{ \mathfrak{u}t }
\rVert_{ \R^d }
\biggr]
+
\biggl[
\sup_{ t \in [ 0, T ] }
\lVert
    \mathfrak{W}_{ \mathfrak{u}t } - \mathfrak{W}_{ \mathfrak{U}t }
\rVert_{ \R^d }
\biggr]
\\ &
\leq
\biggl[
\sup_{ t \in [ 0, T ] }
\lVert
    \mathfrak{w}_{ t } - \mathfrak{W}_{ t }
\rVert_{ \R^d }
\biggr]
+
\biggl[
\sup_{
\substack{ s, t \in [ 0, T ], \,
\lvert s - t \rvert \leq  T \delta
}
}
\lVert
    \mathfrak{W}_{ s } - \mathfrak{W}_{ t }
\rVert_{ \R^d }
\biggr]
\\ &
\leq
\lVert \mathfrak{w} - \mathfrak{W} \rVert_{ C( [ 0, T ], \R^d ) }
+
\biggl[
\sup_{
\substack{ s, t \in [ 0, T ], \,
\lvert s - t \rvert \leq  \Delta
}
}
\lVert
    \mathfrak{W}_{ s } - \mathfrak{W}_{ t }
\rVert_{ \R^d }
\biggr]
\\ &
\leq 
\delta
+
\tfrac{\varepsilon}{2}
\leq
\tfrac{\varepsilon}{2}
+
\tfrac{\varepsilon}{2}
=
\varepsilon
.
\end{split}
\end{equation}
It thus holds that
$ \varphi_1 \colon \mathcal{Z} \to C( [ 0, T ], \R^d ) $
is a continuous function.
Note that this ensures that
$ \Psi_1 \colon \mathcal{Y} \times \mathcal{Z} \to \mathcal{Y} \times \mathcal{Z} $
is a continuous function.

Second, we claim that
$ \varphi_2 \colon \mathcal{Y} \times \mathcal{Z} \to \mathcal{Y} $
is a continuous function.
Throughout this paragraph
let
$ \varepsilon \in ( 0, \infty ) $,
$ \mathbf{v} \in \mathcal{Y} $,
$ \mathfrak{Z} = ( \mathfrak{U}, \mathfrak{W} ) \in \mathcal{Z} $
and
let
$ N \in \N $,
$ R, \Delta, \delta \in ( 0, \infty ) $
be real numbers which satisfy
\begin{equation}
\label{eq:definition_N}
\sup_{ ( t, x ) \in [ 0, T ] \times \R^d, \, \lVert x \rVert_{ \R^d } \geq N }
\frac{ \lvert
    \mathbf{v}( t, x )
\rvert }{ \lVert x \rVert_{ \R^d }^q }
\leq
\frac{\varepsilon}{
12
+
6 \,
\lVert \mathfrak{W} \rVert_{ C( [ 0, T ], \R^d ) }
},
\end{equation}
\begin{equation}
\label{eq:definition_R,Delta}
R = 1 + \lVert \mathfrak{W} \rVert_{ C( [ 0, T ], \R^d ) },
\qquad
\sup_{
\substack{
( s, \mathbf{x} ), ( t, x ) \in [ 0, T ] \times \{ w \in \R^d \colon \lVert w \rVert_{ \R^d } \leq N + 2R \}, \\
\lvert s - t \rvert + \lVert \mathbf{x} - x \rVert_{ \R^d } \leq  \Delta
} }
\lvert
    \mathbf{v}( s, \mathbf{x} )
    -
    \mathbf{v}( t, x )
\rvert
\leq
\tfrac{\varepsilon}{3},
\end{equation}
and
\begin{equation}
\label{eq:definition_delta}
\delta
=
\min\biggl\{
    1,
    \frac{\Delta}{ \max\{ 1, T \} },
    \frac{\varepsilon}{ 3 \, ( 2 + \lVert \mathfrak{W} \rVert_{ C( [ 0, T ], \R^d ) } )^q }
\biggr\}
.
\end{equation}
Note that it holds for all
$ \mathfrak{w} \in C( [ 0, T ], \R^d ) $,
$ t \in [ 0, T ] $,
$ x \in \R^d $
with
$ \lVert \mathfrak{w} - \mathfrak{W} \rVert_{ C( [ 0, T ], \R^d ) }
\leq 1 $
that
\begin{equation}
\label{eq:linear_growth1}
\begin{split}
\lVert x + \mathfrak{w}_{ t } \rVert_{ \R^d }
& \leq
\lVert x \rVert_{ \R^d }
+
\lVert \mathfrak{w}_{ t } \rVert_{ \R^d }
\leq
\bigl(
1 +
\lVert \mathfrak{w} \rVert_{ C( [ 0, T ], \R^d ) }
\bigr)
\max\{ 1, \lVert x \rVert_{ \R^d } \}
\\ &
\leq
\bigl(
1 +
\lVert \mathfrak{w} - \mathfrak{W} \rVert_{ C( [ 0, T ], \R^d ) }
+
\lVert \mathfrak{W} \rVert_{ C( [ 0, T ], \R^d ) }
\bigr)
\max\{ 1, \lVert x \rVert_{ \R^d } \}
\\ & \leq
\bigl(
2
+
\lVert \mathfrak{W} \rVert_{ C( [ 0, T ], \R^d ) }
\bigr)
\max\{ 1, \lVert x \rVert_{ \R^d } \}
.
\end{split}
\end{equation}
This
and Lemma~\ref{lem:boundedness_Phi}
(with
$ d = d $,
$ T = T $,
$ p = q $,
$ L = 2 +
\lVert \mathfrak{W} \rVert_{ C( [ 0, T ], \R^d ) } $,
$ \varrho =
( [ 0, T ] \times \R^d \ni ( t, x ) \mapsto ( T - t + \mathfrak{u}t, x + \mathfrak{w}_{ t } ) \in [ 0, T ] \times \R^d ) $,
$ v = v - \mathbf{v} $
for
$ ( \mathfrak{u}, \mathfrak{w} ) \in \mathcal{Z} $,
$ v \in \mathcal{Y} $
with
$ \lVert \mathfrak{w} - \mathfrak{W} \rVert_{ C( [ 0, T ], \R^d ) }
\leq 1 $
in the notation of Lemma~\ref{lem:boundedness_Phi})
imply for all
$ v \in \mathcal{Y} $,
$ ( \mathfrak{u}, \mathfrak{w} ) \in \mathcal{Z} $
with
$ \lVert \mathfrak{w} - \mathfrak{W} \rVert_{ C( [ 0, T ], \R^d ) }
\leq 1 $
that
\begin{equation}
\label{eq:psi_2_estimateI}
\begin{split}
& \sup_{ ( t, x ) \in [ 0, T ] \times \R^d }
\biggl[
\frac{ \lvert
    v( T - t + \mathfrak{u}t, x + \mathfrak{w}_{ t } )
    -
    \mathbf{v}( T - t + \mathfrak{u}t, x + \mathfrak{w}_{ t } )
\rvert }{ \max\{ 1, \lVert x \rVert_{ \R^d }^q \} }
\biggr]
\\ &
\leq
\bigl(
2 +
\lVert \mathfrak{W} \rVert_{ C( [ 0, T ], \R^d ) }
\bigr)^q
\lVert v - \mathbf{v} \rVert_{ \mathcal{Y} }
.
\end{split}
\end{equation}
In addition,
observe that it holds
for all
$ \mathfrak{w} \in C( [ 0, T ], \R^d ) $,
$ t \in [ 0, T ] $,
$ x \in \R^d $
with
$ \lVert \mathfrak{w} - \mathfrak{W} \rVert_{ C( [ 0, T ], \R^d ) }
\leq 1 $
and
$ \lVert x \rVert_{ \R^d } \geq N + R $
that
\begin{equation}
\begin{split}
\lVert x + \mathfrak{w}_{ t } \rVert_{ \R^d }
& \geq
\lVert x \rVert_{ \R^d }
-
\lVert \mathfrak{w}_{ t } \rVert_{ \R^d }
\geq
N + 1 + \lVert \mathfrak{W} \rVert_{ C( [ 0, T ], \R^d ) }
-
\lVert \mathfrak{w} \rVert_{ C( [ 0, T ], \R^d ) }
\geq
N
.
\end{split}
\end{equation}
This, \eqref{eq:linear_growth1}, and~\eqref{eq:definition_N}
establish for all
$ v \in \mathcal{Y} $,
$ ( \mathfrak{u}, \mathfrak{w} ) \in \mathcal{Z} $
with
$ \lVert \mathfrak{w} - \mathfrak{W} \rVert_{ C( [ 0, T ], \R^d ) }
\leq 1 $
that
\begin{equation}
\label{eq:psi_2_estimateII}
\begin{split}
& \sup_{ ( t, x ) \in [ 0, T ] \times \R^d, \, \lVert x \rVert_{ \R^d } \geq N + R }
\biggl[
\frac{ \lvert
    \mathbf{v}( T - t + \mathfrak{u}t, x + \mathfrak{w}_{ t } )
\rvert }{ \lVert x \rVert_{ \R^d }^q }
\biggr]
\\ &
=
\sup_{ ( t, x ) \in [ 0, T ] \times \R^d, \, \lVert x \rVert_{ \R^d } \geq N + R }
\biggl[
\frac{ \lVert x + \mathfrak{w}_{ t } \rVert_{ \R^d }^q }{ \max\{ 1, \lVert x \rVert_{ \R^d }^q \} }
\frac{ \lvert
    \mathbf{v}( T - t + \mathfrak{u}t, x + \mathfrak{w}_{ t } )
\rvert }{ \lVert x + \mathfrak{w}_{ t } \rVert_{ \R^d }^q }
\biggr]
\\ &
\leq
\biggl[
\sup_{ ( t, x ) \in [ 0, T ] \times \R^d }
\frac{ \lVert x + \mathfrak{w}_{ t } \rVert_{ \R^d }^q }{ \max\{ 1, \lVert x \rVert_{ \R^d }^q \} }
\biggr]
\biggl[
\sup_{ ( t, x ) \in [ 0, T ] \times \R^d, \, \lVert x \rVert_{ \R^d } \geq N }
\frac{ \lvert
    \mathbf{v}( t, x )
\rvert }{ \lVert x \rVert_{ \R^d }^q }
\biggr]
\\ &
\leq 
\bigl(
2
+
\lVert \mathfrak{W} \rVert_{ C( [ 0, T ], \R^d ) }
\bigr)
\frac{\varepsilon}{
12
+
6 \,
\lVert \mathfrak{W} \rVert_{ C( [ 0, T ], \R^d ) }
}
=
\tfrac{\varepsilon}{6}
.
\end{split}
\end{equation}
Furthermore,
note that it holds for all
$ \mathfrak{z} = ( \mathfrak{u}, \mathfrak{w} ) \in \mathcal{Z} $,
$ t \in [ 0, T ] $,
$ x \in \R^d $
with
$ \mathbf{d}_{ \mathcal{Z} }(
    \mathfrak{z},
    \mathfrak{Z}
)
=
\lvert \mathfrak{u} - \mathfrak{U} \rvert
+
\lVert \mathfrak{w} - \mathfrak{W} \rVert_{ C( [ 0, T ], \R^d ) }
\leq \delta $
and
$ \lVert x \rVert_{ \R^d } \leq N + R $
that
\begin{equation}
\label{eq:auxiliary_estimate1}
\begin{split}
\lVert x + \mathfrak{w}_{ t } \rVert_{ \R^d }
& \leq
\lVert x \rVert_{ \R^d }
+
\lVert \mathfrak{w}_{ t } \rVert_{ \R^d } 
\leq
N + R
+
\lVert \mathfrak{w} \rVert_{ C( [ 0, T ], \R^d ) }
\\ & \leq
N + R
+
\lVert \mathfrak{w} - \mathfrak{W} \rVert_{ C( [ 0, T ], \R^d ) }
+
\lVert \mathfrak{W} \rVert_{ C( [ 0, T ], \R^d ) }
\\ & \leq
N + R
+
1
+
\lVert \mathfrak{W} \rVert_{ C( [ 0, T ], \R^d ) }
=
N + 2 R
\end{split}
\end{equation}
and
\begin{equation}
\label{eq:auxiliary_estimate2}
\begin{split}
& \lvert
    T - t + \mathfrak{u}t - ( T - t + \mathfrak{U}t )
\rvert
+
\lVert
    x + \mathfrak{w}_{ t } - ( x + \mathfrak{W}_{ t } )
\rVert_{ \R^d }
=
\lvert
    \mathfrak{u} - \mathfrak{U}
\rvert t
+
\lVert
    \mathfrak{w}_{ t } - \mathfrak{W}_{ t }
\rVert_{ \R^d }
\\ &
\leq
\lvert
    \mathfrak{u} - \mathfrak{U}
\rvert T
+
\lVert \mathfrak{w} - \mathfrak{W} \rVert_{ C( [ 0, T ], \R^d ) }
\leq
\max\{ 1, T \}
\delta
.
\end{split}
\end{equation}
Combining~\eqref{eq:auxiliary_estimate1}--\eqref{eq:auxiliary_estimate2}
with
\eqref{eq:psi_2_estimateI},
\eqref{eq:psi_2_estimateII},
\eqref{eq:definition_delta},
and~\eqref{eq:definition_R,Delta}
ensures for all
$ v \in \mathcal{Y} $,
$ \mathfrak{z} = ( \mathfrak{u}, \mathfrak{w} ) \in \mathcal{Z} $
with
$ \lVert v - \mathbf{v} \rVert_{ \mathcal{Y} }
+
\mathbf{d}_{ \mathcal{Z} }(
    \mathfrak{z},
    \mathfrak{Z}
)
\leq 
\delta
$
that
\begin{align*}
& \lVert
    \varphi_2( v, \mathfrak{z} )
    -
    \varphi_2( \mathbf{v}, \mathfrak{Z} )
\rVert_{ \mathcal{Y} }
\\ &
=
\sup_{ ( t, x ) \in [ 0, T ] \times \R^d }
\biggl[
\frac{ \lvert
    v( T - t + \mathfrak{u}t, x + \mathfrak{w}_{ t } )
    -
    \mathbf{v}( T - t + \mathfrak{U}t, x + \mathfrak{W}_{ t } )
\rvert }{ \max\{ 1, \lVert x \rVert_{ \R^d }^q \} }
\biggr]
\\ &
\leq
\biggl[
\sup_{ ( t, x ) \in [ 0, T ] \times \R^d }
\frac{ \lvert
    v( T - t + \mathfrak{u}t, x + \mathfrak{w}_{ t } )
    -
    \mathbf{v}( T - t + \mathfrak{u}t, x + \mathfrak{w}_{ t } )
\rvert }{ \max\{ 1, \lVert x \rVert_{ \R^d }^q \} }
\biggr]
\\ & \quad
+
\biggl[
\sup_{ ( t, x ) \in [ 0, T ] \times \R^d, \, \lVert x \rVert_{ \R^d } \leq N + R }
\frac{ \lvert
    \mathbf{v}( T - t + \mathfrak{u}t, x + \mathfrak{w}_{ t } )
    -
    \mathbf{v}( T - t + \mathfrak{U}t, x + \mathfrak{W}_{ t } )
\rvert }{ \max\{ 1, \lVert x \rVert_{ \R^d }^q \} }
\biggr]
\\ & \quad \yesnumber
+
\biggl[
\sup_{ ( t, x ) \in [ 0, T ] \times \R^d, \, \lVert x \rVert_{ \R^d } \geq N + R }
\frac{ \lvert
    \mathbf{v}( T - t + \mathfrak{u}t, x + \mathfrak{w}_{ t } )
    -
    \mathbf{v}( T - t + \mathfrak{U}t, x + \mathfrak{W}_{ t } )
\rvert }{ \max\{ 1, \lVert x \rVert_{ \R^d }^q \} }
\biggr]
\\ &
\leq
\bigl(
2 +
\lVert \mathfrak{W} \rVert_{ C( [ 0, T ], \R^d ) }
\bigr)^q
\lVert v - \mathbf{v} \rVert_{ \mathcal{Y} }
+
\sup_{
\substack{
( s, \mathbf{x} ), ( t, x ) \in [ 0, T ] \times \{ w \in \R^d \colon \lVert w \rVert_{ \R^d } \leq N + 2R \}, \\
\lvert s - t \rvert + \lVert \mathbf{x} - x \rVert_{ \R^d } \leq  \max\{ 1, T \} \delta
} }
\lvert
    \mathbf{v}( s, \mathbf{x} )
    -
    \mathbf{v}( t, x )
\rvert
\\ & \quad
+
\sup_{ ( t, x ) \in [ 0, T ] \times \R^d, \, \lVert x \rVert_{ \R^d } \geq N + R }
\biggl[
\frac{ \lvert
    \mathbf{v}( T - t + \mathfrak{u}t, x + \mathfrak{w}_{ t } )
\rvert }{ \lVert x \rVert_{ \R^d }^q }
    +
\frac{ \lvert
    \mathbf{v}( T - t + \mathfrak{U}t, x + \mathfrak{W}_{ t } )
\rvert }{ \lVert x \rVert_{ \R^d }^q }
\biggr]
\\ &
\leq
\bigl(
2
+
\lVert \mathfrak{W} \rVert_{ C( [ 0, T ], \R^d ) }
\bigr)^q
\delta
+
\tfrac{2\varepsilon}{6}
+
\sup_{
\substack{
( s, \mathbf{x} ), ( t, x ) \in [ 0, T ] \times \{ w \in \R^d \colon \lVert w \rVert_{ \R^d } \leq N + 2R \}, \\
\lvert s - t \rvert + \lVert \mathbf{x} - x \rVert_{ \R^d } \leq  \Delta
} }
\lvert
    \mathbf{v}( s, \mathbf{x} )
    -
    \mathbf{v}( t, x )
\rvert
\\ &
\leq 
\tfrac{\varepsilon}{3} + \tfrac{\varepsilon}{3} + \tfrac{\varepsilon}{3}
=
\varepsilon
.
\end{align*}
This proves that
$ \varphi_2 \colon \mathcal{Y} \times \mathcal{Z} \to \mathcal{Y} $
is a continuous function.
Observe that this implies that
$ \Psi_2 \colon \mathcal{Y} \times \mathcal{Z} \to \mathcal{Y} \times \mathcal{Z} $
is a continuous function.

Third, we establish that
$ \varphi_3 \colon \mathcal{Y} \times \mathcal{Z} \to \mathcal{Y} $
is a continuous function.
Throughout this paragraph
let
$ \varepsilon \in ( 0, \infty ) $,
$ \mathbf{v} \in \mathcal{Y} $,
$ \mathfrak{Z} = ( \mathfrak{U}, \mathfrak{W} ) \in \mathcal{Z} $
and
let
$ N \in \N $,
$ R, \Delta, \delta \in ( 0, \infty ) $
be real numbers which satisfy
$ N \geq
( 6 L T (
2
+
\lVert \mathfrak{W} \rVert_{ C( [ 0, T ], \R^d ) } )^p
\varepsilon^{ -1 }
)^{ \nicefrac{1}{ ( q - p ) } }
$
and
\begin{equation}
\label{eq:definition_N2}
\sup_{
( t, x ) \in [ 0, T ] \times \R^d, \,
\lVert x \rVert_{ \R^d } \geq N
}
\biggl[
\frac{ L \lvert \mathbf{v}( t, x ) \rvert
+ \lvert
    f(
        T - t + \mathfrak{U}t, x + \mathfrak{W}_{ t },
        \mathbf{v}( t, x )
    )
\rvert }{ \lVert x \rVert_{ \R^d }^q }
\biggr]
\leq
\tfrac{\varepsilon}{6T},
\end{equation}
\begin{equation}
\label{eq:definition_R,Delta2}
R = 1 + \lVert \mathfrak{W} \rVert_{ C( [ 0, T ], \R^d ) },
\quad\!\!%\linebreak
\sup_{
\substack{
( s, \mathbf{x} ), ( t, x )
\in [ 0, T ] \times \{ w \in \R^d \colon \lVert w \rVert_{ \R^d } \leq N + R \}, \\
\mathfrak{v} \in \R, \,
\lvert \mathfrak{v} \rvert \leq \lVert \mathbf{v} \rVert_{ \mathcal{Y} } N^q, \,
\lvert s - t \rvert + \lVert \mathbf{x} - x \rVert_{ \R^d } \leq \Delta
} }
\lvert
    f( s, \mathbf{x}, \mathfrak{v} )
    -
    f( t, x, \mathfrak{v} )
\rvert
\leq
\tfrac{\varepsilon}{3T},
\end{equation}
and
\begin{equation}
\label{eq:definition_delta2}
\delta
=
\min\biggl\{
    1,
    \frac{\Delta}{ \max\{ 1, T \} },
    \frac{\varepsilon}{ 3 \max\{ 1, LT \} }
\biggr\}
\end{equation}
(cf.~(\ref{item:well_defined_Phi2}) in~Corollary~\ref{cor:well_defined_Phi}).
Note that it holds for all
$ v \in \mathcal{Y} $,
$ \mathfrak{z} = ( \mathfrak{u}, \mathfrak{w} ) \in \mathcal{Z} $
that
\begin{align*}
\label{eq:psi_3_estimateI}
& \lVert
    \varphi_3( v, \mathfrak{z} )
    -
    \varphi_3( \mathbf{v}, \mathfrak{Z} )
\rVert_{ \mathcal{Y} }
\\ &
=
\sup_{ ( t, x ) \in [ 0, T ] \times \R^d }
\biggl[
\frac{ t \lvert
    f(
        T - t + \mathfrak{u}t, x + \mathfrak{w}_{ t },
        v( t, x )
    )
    -
    f(
        T - t + \mathfrak{U}t, x + \mathfrak{W}_{ t },
        \mathbf{v}( t, x )
    )
\rvert }{ \max\{ 1, \lVert x \rVert_{ \R^d }^q \} }
\biggr]
\\ & \yesnumber
\leq
T
\sup_{ ( t, x ) \in [ 0, T ] \times \R^d }
\biggl[
\frac{ \lvert
    f(
        T - t + \mathfrak{u}t, x + \mathfrak{w}_{ t },
        v( t, x )
    )
    -
    f(
        T - t + \mathfrak{u}t, x + \mathfrak{w}_{ t },
        \mathbf{v}( t, x )
    )
\rvert }{ \max\{ 1, \lVert x \rVert_{ \R^d }^q \} }
\biggr]
\\ & \quad
+
T
\sup_{
\substack{
( t, x ) \in [ 0, T ] \times \R^d, \\
\lVert x \rVert_{ \R^d } \leq N
} }
\biggl[
\frac{ \lvert
    f(
        T - t + \mathfrak{u}t, x + \mathfrak{w}_{ t },
        \mathbf{v}( t, x )
    )
    -
    f(
        T - t + \mathfrak{U}t, x + \mathfrak{W}_{ t },
        \mathbf{v}( t, x )
    )
\rvert }{ \max\{ 1, \lVert x \rVert_{ \R^d }^q \} }
\biggr]
\\ & \quad
+
T
\sup_{
\substack{
( t, x ) \in [ 0, T ] \times \R^d, \\
\lVert x \rVert_{ \R^d } \geq N
} }
\biggl[
\frac{ \lvert
    f(
        T - t + \mathfrak{u}t, x + \mathfrak{w}_{ t },
        \mathbf{v}( t, x )
    )
    -
    f(
        T - t + \mathfrak{U}t, x + \mathfrak{W}_{ t },
        \mathbf{v}( t, x )
    )
\rvert }{ \lVert x \rVert_{ \R^d }^q }
\biggr]
.
\end{align*}
Next observe that it holds for all
$ \mathfrak{z} = ( \mathfrak{u}, \mathfrak{w} ) \in \mathcal{Z} $,
$ t \in [ 0, T ] $,
$ x \in \R^d $
with
$ \mathbf{d}_{ \mathcal{Z} }(
    \mathfrak{z},
    \mathfrak{Z}
)
=
\lvert \mathfrak{u} - \mathfrak{U} \rvert
+
\lVert \mathfrak{w} - \mathfrak{W} \rVert_{ C( [ 0, T ], \R^d ) }
\leq \delta $
and
$ \lVert x \rVert_{ \R^d } \leq N $
that
$ \lVert x + \mathfrak{w}_{ t } \rVert_{ \R^d } \leq N + R $,
$ \lvert \mathbf{v}( t, x ) \rvert
\leq
\lVert \mathbf{v} \rVert_{ \mathcal{Y} }
\max\{ 1, \lVert x \rVert_{ \R^d }^q \}
\leq
\lVert \mathbf{v} \rVert_{ \mathcal{Y} }
N^q
$,
and
\begin{equation}
\lvert
    T - t + \mathfrak{u}t - ( T - t + \mathfrak{U}t )
\rvert
+
\lVert
    x + \mathfrak{w}_{ t } - ( x + \mathfrak{W}_{ t } )
\rVert_{ \R^d }
\leq
\max\{ 1, T \}
\delta
.
\end{equation}
This and~\eqref{eq:definition_R,Delta2} show for all
$ \mathfrak{z} = ( \mathfrak{u}, \mathfrak{w} ) \in \mathcal{Z} $
with
$ \mathbf{d}_{ \mathcal{Z} }(
    \mathfrak{z},
    \mathfrak{Z}
)
\leq 
\delta
$
that
\begin{align*}
\label{eq:psi_3_estimateII}
&
\sup_{
\substack{
( t, x ) \in [ 0, T ] \times \R^d, \\
\lVert x \rVert_{ \R^d } \leq N
} }
\biggl[
\frac{ \lvert
    f(
        T - t + \mathfrak{u}t, x + \mathfrak{w}_{ t },
        \mathbf{v}( t, x )
    )
    -
    f(
        T - t + \mathfrak{U}t, x + \mathfrak{W}_{ t },
        \mathbf{v}( t, x )
    )
\rvert }{ \max\{ 1, \lVert x \rVert_{ \R^d }^q \} }
\biggr]
\\ &
\leq
\sup_{
\substack{
( t, x ) \in [ 0, T ] \times \R^d, \,
\lVert x \rVert_{ \R^d } \leq N
} }
\lvert
    f(
        T - t + \mathfrak{u}t, x + \mathfrak{w}_{ t },
        \mathbf{v}( t, x )
    )
    -
    f(
        T - t + \mathfrak{U}t, x + \mathfrak{W}_{ t },
        \mathbf{v}( t, x )
    )
\rvert
\\ & \yesnumber
\leq
\sup_{
\substack{
( s, \mathbf{x} ), ( t, x )
\in [ 0, T ] \times \{ w \in \R^d \colon \lVert w \rVert_{ \R^d } \leq N + R \}, \\
\mathfrak{v} \in \R, \,
\lvert \mathfrak{v} \rvert \leq \lVert \mathbf{v} \rVert_{ \mathcal{Y} } N^q, \,
\lvert s - t \rvert + \lVert \mathbf{x} - x \rVert_{ \R^d } \leq \max\{ 1, T \} \delta
} }
\lvert
    f( s, \mathbf{x}, \mathfrak{v} )
    -
    f( t, x, \mathfrak{v} )
\rvert
\\ &
\leq
\sup_{
\substack{
( s, \mathbf{x} ), ( t, x )
\in [ 0, T ] \times \{ w \in \R^d \colon \lVert w \rVert_{ \R^d } \leq N + R \}, \\
\mathfrak{v} \in \R, \,
\lvert \mathfrak{v} \rvert \leq \lVert \mathbf{v} \rVert_{ \mathcal{Y} } N^q, \,
\lvert s - t \rvert + \lVert \mathbf{x} - x \rVert_{ \R^d } \leq \Delta
} }
\lvert
    f( s, \mathbf{x}, \mathfrak{v} )
    -
    f( t, x, \mathfrak{v} )
\rvert
\leq
\tfrac{\varepsilon}{3T}
.
\end{align*}
Furthermore,
\eqref{eq:definition_N2}
and
\eqref{eq:linear_growth1}
ensure for all
$ ( \mathfrak{u}, \mathfrak{w} ) \in \mathcal{Z} $
with
$ \lVert \mathfrak{w} - \mathfrak{W} \rVert_{ C( [ 0, T ], \R^d ) }
\leq 1 $
that
\begin{align*}
\label{eq:psi_3_estimateIII}
&
\sup_{
\substack{
( t, x ) \in [ 0, T ] \times \R^d, \\
\lVert x \rVert_{ \R^d } \geq N
} }
\biggl[
\frac{ \lvert
    f(
        T - t + \mathfrak{u}t, x + \mathfrak{w}_{ t },
        \mathbf{v}( t, x )
    )
    -
    f(
        T - t + \mathfrak{U}t, x + \mathfrak{W}_{ t },
        \mathbf{v}( t, x )
    )
\rvert }{ \lVert x \rVert_{ \R^d }^q }
\biggr]
\\ &
\leq
\sup_{
( t, x ) \in [ 0, T ] \times \R^d, \,
\lVert x \rVert_{ \R^d } \geq N
}
\biggl[
\frac{ \lvert
    f(
        T - t + \mathfrak{u}t, x + \mathfrak{w}_{ t },
        \mathbf{v}( t, x )
    )
    -
    f(
        T - t + \mathfrak{u}t, x + \mathfrak{w}_{ t },
        0
    )
\rvert }{ \lVert x \rVert_{ \R^d }^q }
\\ & \quad
\phantom{
\sup_{
( t, x ) \in [ 0, T ] \times \R^d, \,
\lVert x \rVert_{ \R^d } \geq N
} \biggl[
}
\mathop{+}
\frac{ \lvert
    f(
        T - t + \mathfrak{u}t, x + \mathfrak{w}_{ t },
        0
    )
    -
    f(
        T - t + \mathfrak{U}t, x + \mathfrak{W}_{ t },
        \mathbf{v}( t, x )
    )
\rvert }{ \lVert x \rVert_{ \R^d }^q }
\biggr]
\\ & \yesnumber
\leq
\biggl[
\sup_{
( t, x ) \in [ 0, T ] \times \R^d, \,
\lVert x \rVert_{ \R^d } \geq N
}
\frac{ L \lvert \mathbf{v}( t, x ) \rvert
+
\lvert
    f(
        T - t + \mathfrak{U}t, x + \mathfrak{W}_{ t },
        \mathbf{v}( t, x )
    )
\rvert }{ \lVert x \rVert_{ \R^d }^q }
\biggr]
\\ & \quad
+
\biggl[
\sup_{
( t, x ) \in [ 0, T ] \times \R^d, \,
\lVert x \rVert_{ \R^d } \geq N
}
\frac{ \lvert
    f( T - t + \mathfrak{u}t, x + \mathfrak{w}_{ t }, 0 )
\rvert }{ \lVert x \rVert_{ \R^d }^q }
\biggr]
\\ &
\leq
\tfrac{\varepsilon}{6T}
+
L
\biggl[
\sup_{
( t, x ) \in [ 0, T ] \times \R^d, \,
\lVert x \rVert_{ \R^d } \geq N
}
\frac{ \max\{ 1, \lVert x + \mathfrak{w}_{ t } \rVert_{ \R^d }^p \} }{ \lVert x \rVert_{ \R^d }^q }
\biggr]
\\ &
\leq
\tfrac{\varepsilon}{6T}
+
L
\bigl(
2
+
\lVert \mathfrak{W} \rVert_{ C( [ 0, T ], \R^d ) }
\bigr)^p
\biggl[
\sup_{
x \in \R^d, \,
\lVert x \rVert_{ \R^d } \geq N
}
\frac{ 1 }{ \lVert x \rVert_{ \R^d }^{ q - p } }
\biggr]
\\ &
=
\tfrac{\varepsilon}{6T}
+
L
\bigl(
2
+
\lVert \mathfrak{W} \rVert_{ C( [ 0, T ], \R^d ) }
\bigr)^p
\frac{ 1 }{ N^{ q - p } }
\leq
\tfrac{\varepsilon}{6T}
+
\tfrac{\varepsilon}{6T}
=
\tfrac{\varepsilon}{3T}
.
\end{align*}
Combining~\eqref{eq:psi_3_estimateI}
with~\eqref{eq:psi_3_estimateII},
\eqref{eq:psi_3_estimateIII},
and
\eqref{eq:definition_delta2}
establishes for all
$ v \in \mathcal{Y} $,
$ \mathfrak{z} = ( \mathfrak{u}, \mathfrak{w} ) \in \mathcal{Z} $
with
$ \lVert v - \mathbf{v} \rVert_{ \mathcal{Y} }
+
\mathbf{d}_{ \mathcal{Z} }(
    \mathfrak{z},
    \mathfrak{Z}
)
\leq 
\delta
$
that
\begin{equation}
\begin{split}
& \lVert
    \varphi_3( v, \mathfrak{z} )
    -
    \varphi_3( \mathbf{v}, \mathfrak{Z} )
\rVert_{ \mathcal{Y} }
\leq
L T
\lVert v - \mathbf{v} \rVert_{ \mathcal{Y} }
+
\tfrac{2\varepsilon}{3}
\leq
L T \delta
+
\tfrac{2\varepsilon}{3}
\leq
\tfrac{\varepsilon}{3}
+
\tfrac{2\varepsilon}{3}
=
\varepsilon
.
\end{split}
\end{equation}
From this it follows that
$ \varphi_3 \colon \mathcal{Y} \times \mathcal{Z} \to \mathcal{Y} $
is a continuous function.

As a next step
observe that
the fact that
$ \varphi_2 $, $ \Psi_1 $, $ \Psi_2 $, and $ \varphi_3 $
are continuous functions,
the fact that
$ \mathcal{Z} \ni \mathfrak{z} \mapsto
( \mathfrak{g}, \mathfrak{z} ) \in \mathcal{Y} \times \mathcal{Z} $
is a continuous function,
and the fact that
$ \varphi_4 \colon \mathcal{Y} \to \mathcal{Y} $
is a linear isometry
demonstrate that
$ G \colon \mathcal{Z} \to \mathcal{Y} $
and
$ F \colon \mathcal{Y} \times \mathcal{Z} \to \mathcal{Y} $
are continuous functions.
Combining
this with~\eqref{eq:composition_Phi}
proves~(\ref{item:measurability_Phi_heat1}).
Finally,
the fact that
$ ( \mathcal{Y}, \lVert \cdot \rVert_{ \mathcal{Y} } ) $
is a separable $ \R $-Banach space
(cf.~(\ref{item:Y_separable2}) in Proposition~\ref{prop:Y_separable}),
the fact that
$ ( \mathcal{Z}, \mathbf{d}_{ \mathcal{Z} } ) $
is a separable metric space,
and~(\ref{item:measurability_Phi_heat1})
establish~(\ref{item:measurability_Phi_heat2}).
The proof of Lemma~\ref{lem:measurability_Phi_heat} is thus complete.
\end{proof}

\subsubsection{Recursive formulation}
\label{sec:formulation}

\begin{lemma}
\label{lem:scheme_description_heat}
Assume Setting~\ref{setting:MLP_heat}.
Then
\begin{enumerate}[(i)]
\item
\label{item:scheme_description_heat1}
it holds
for all
$ n \in ( \N_0 \cup \{ -1 \} ) $,
$ j \in \N $,
$ \theta \in \Theta $
that
$ Y_{ n, j }^\theta( \Omega ) \subseteq \mathcal{Y} $,
\item
\label{item:scheme_description_heat2}
it holds
for all
$ n, j \in \N $,
$ \theta \in \Theta $
that
$ Y_{ -1, j }^\theta = Y_{ 0, j }^\theta = 0 $
and
\begin{equation}
Y_{ n, j }^\theta
=
\sum_{l=0}^{n-1}
\tfrac{1}{ ( M_j )^{ n - l } }
\Biggl[
    \sum_{i=1}^{ ( M_j )^{ n - l } }
        \Phi_{ l } \bigl(
        Y^{ ( \theta, l, i ) }_{ l, j },
        Y^{ ( \theta, -l, i ) }_{ l - 1, j },
        Z^{ ( \theta, l, i ) }
        \bigr)
\Biggr],
\end{equation}
and
\item
\label{item:scheme_description_heat3}
it holds for all
$ n \in ( \N_0 \cup \{ - 1 \} ) $,
$ j \in \N $,
$ \theta \in \Theta $
that
$ \Omega \ni \omega \mapsto Y_{ n, j }^\theta( \omega ) \in \mathcal{Y} $
is an
$ \mathscr{F} $/$ \mathscr{B}( \mathcal{Y} ) $-measurable function.
\end{enumerate}
\end{lemma}
\begin{proof}[Proof of Lemma~\ref{lem:scheme_description_heat}]
We show~(\ref{item:scheme_description_heat1})--(\ref{item:scheme_description_heat2})
by induction on $ n \in \N $.
For the base case $ n = 1 $
note that the fact that
$ \forall \, j \in \N, \, \theta \in \Theta \colon
Y_{ -1, j }^\theta = Y_{ 0, j }^\theta = 0 $
implies
for all
$ j \in \N $,
$ \theta \in \Theta $
that
\begin{equation}
\label{eq:Y_1,Y_0_in_Y}
Y_{ -1, j }^\theta, Y_{ 0, j }^\theta \in \mathcal{Y}
.
\end{equation}
Next observe that
\eqref{eq:definition_Y_heat}
and
\eqref{eq:definition_Phi_heat}
ensure
for all
$ j \in \N $,
$ \theta \in \Theta $,
$ t \in [ 0, T ] $,
$ x \in \R^d $
that
\begin{align*}
&
Y_{ 1, j }^\theta( T - t, x )
=
\tfrac{1}{ M_j }
\Biggl[
\sum_{ i = 1 }^{ M_j }
    g\bigl( x + W^{ ( \theta, 0, i ) }_{ t } \bigr)
\Biggr]
\\ & \quad
+
\tfrac{t}{ M_j }
\Biggl[
    \sum_{i=1}^{ M_j }
        f\Bigl(
            T - t + U^{ ( \theta, 0, i ) } t,
            x + W^{ ( \theta, 0, i ) }_{ U^{ ( \theta, 0, i ) } t },
            Y^{ ( \theta, 0, i ) }_{ 0, j }\bigl(
                T - t + U^{ ( \theta, 0, i ) } t,
                x + W^{ ( \theta, 0, i ) }_{ U^{ ( \theta, 0, i ) } t }
            \bigr)
        \Bigr)
\Biggr]
\\ & \yesnumber
=
\tfrac{1}{ M_j }
\Biggl[
    \sum_{i=1}^{ M_j }
    \bigl[
        \Phi_{ 0 } \bigl(
        Y^{ ( \theta, 0, i ) }_{ 0, j },
        Y^{ ( \theta, 0, i ) }_{ -1, j },
        Z^{ ( \theta, 0, i ) }
        \bigr)
    \bigr]( T - t, x )
\Biggr]
.
\end{align*}
This
and~\eqref{eq:Y_1,Y_0_in_Y}
prove~(\ref{item:scheme_description_heat1})--(\ref{item:scheme_description_heat2})
in the base case $ n = 1 $.
For the induction step $ \N \ni n - 1 \to n \in \{ 2, 3, \ldots \} $
let $ n \in \{ 2, 3, \ldots \} $
and assume for all
$ l \in \{ -1, 0, 1, \ldots, n - 1 \} $,
$ j \in \N $,
$ \theta \in \Theta $
that
$ Y_{ l, j }^\theta( \Omega ) \subseteq \mathcal{Y} $.
Equations~\eqref{eq:definition_Y_heat}
and
\eqref{eq:definition_Phi_heat}
hence
demonstrate
for all
$ j \in \N $,
$ \theta \in \Theta $,
$ t \in [ 0, T ] $,
$ x \in \R^d $
that
\begin{align*}
& Y_{ n, j }^\theta( T - t, x )
=
\tfrac{1}{ ( M_j )^{ n } }
\Biggl[
\sum_{ i = 1 }^{ (M_j )^n }
    g\bigl( x + W^{ ( \theta, 0, i ) }_{ t } \bigr)
\Biggr]
+
\sum_{l=0}^{n-1}
\tfrac{t}{ ( M_j )^{ n - l } }
\Biggl[
    \sum_{i=1}^{ ( M_j )^{ n - l } }
\\ &
\quad
    \Bigl[
        f\Bigl(
            T - t + U^{ ( \theta, l, i ) } t,
            x + W^{ ( \theta, l, i ) }_{ U^{ ( \theta, l, i ) } t },
            Y^{ ( \theta, l, i ) }_{ l, j }\bigl(
                T - t + U^{ ( \theta, l, i ) } t,
                x + W^{ ( \theta, l, i ) }_{ U^{ ( \theta, l, i ) } t }
            \bigr)
        \Bigr)
\\ &
\quad
        -
        \mathbbm{1}_{ \N }( l )
        f\Bigl(
            T - t + U^{ ( \theta, l, i ) } t,
            x + W^{ ( \theta, l, i ) }_{ U^{ ( \theta, l, i ) } t },
            Y^{ ( \theta, -l, i ) }_{ l - 1, j }\bigl(
                T - t + U^{ ( \theta, l, i ) } t,
                x + W^{ ( \theta, l, i ) }_{ U^{ ( \theta, l, i ) } t }
            \bigr)
        \Bigr)
    \Bigr]
\Biggr]
\\ &
=
\tfrac{1}{ ( M_j )^{ n } }
\Biggl[
\sum_{ i = 1 }^{ (M_j )^n }
    g\bigl( x + W^{ ( \theta, 0, i ) }_{ t } \bigr)
\Biggr]
\\ & \quad
+
\tfrac{t}{ ( M_j )^{ n } }
\Biggl[
    \sum_{i=1}^{ ( M_j )^{ n } }
        f\Bigl(
            T - t + U^{ ( \theta, 0, i ) } t,
            x + W^{ ( \theta, 0, i ) }_{ U^{ ( \theta, 0, i ) } t },
            Y^{ ( \theta, 0, i ) }_{ 0, j }\bigl(
                T - t + U^{ ( \theta, 0, i ) } t,
                x + W^{ ( \theta, 0, i ) }_{ U^{ ( \theta, 0, i ) } t }
            \bigr)
        \Bigr)
\Biggr]
\\ & \quad \yesnumber
+
\sum_{l=1}^{n-1}
\tfrac{t}{ ( M_j )^{ n - l } }
\Biggl[
    \sum_{i=1}^{ ( M_j )^{ n - l } }
\\ &
\quad\quad
    \Bigl[
        f\Bigl(
            T - t + U^{ ( \theta, l, i ) } t,
            x + W^{ ( \theta, l, i ) }_{ U^{ ( \theta, l, i ) } t },
            Y^{ ( \theta, l, i ) }_{ l, j }\bigl(
                T - t + U^{ ( \theta, l, i ) } t,
                x + W^{ ( \theta, l, i ) }_{ U^{ ( \theta, l, i ) } t }
            \bigr)
        \Bigr)
\\ &
\quad\quad
        -
        f\Bigl(
            T - t + U^{ ( \theta, l, i ) } t,
            x + W^{ ( \theta, l, i ) }_{ U^{ ( \theta, l, i ) } t },
            Y^{ ( \theta, -l, i ) }_{ l - 1, j }\bigl(
                T - t + U^{ ( \theta, l, i ) } t,
                x + W^{ ( \theta, l, i ) }_{ U^{ ( \theta, l, i ) } t }
            \bigr)
        \Bigr)
    \Bigr]
\Biggr]
\\ &
=
\tfrac{1}{ ( M_j )^{ n } }
\Biggl[
    \sum_{i=1}^{ ( M_j )^{ n } }
    \bigl[
        \Phi_{ 0 } \bigl(
        Y^{ ( \theta, 0, i ) }_{ 0, j },
        Y^{ ( \theta, 0, i ) }_{ -1, j },
        Z^{ ( \theta, 0, i ) }
        \bigr)
    \bigr]( T - t, x )
\Biggr]
\\ & \quad
+
\sum_{l=1}^{n-1}
\tfrac{1}{ ( M_j )^{ n - l } }
\Biggl[
    \sum_{i=1}^{ ( M_j )^{ n - l } }
    \bigl[
        \Phi_{ l } \bigl(
        Y^{ ( \theta, l, i ) }_{ l, j },
        Y^{ ( \theta, -l, i ) }_{ l - 1, j },
        Z^{ ( \theta, l, i ) }
        \bigr)
    \bigr]( T - t, x )
\Biggr]
\\ &
=
\sum_{l=0}^{n-1}
\tfrac{1}{ ( M_j )^{ n - l } }
\Biggl[
    \sum_{i=1}^{ ( M_j )^{ n - l } }
    \bigl[
        \Phi_{ l } \bigl(
        Y^{ ( \theta, l, i ) }_{ l, j },
        Y^{ ( \theta, -l, i ) }_{ l - 1, j },
        Z^{ ( \theta, l, i ) }
        \bigr)
    \bigr]( T - t, x )
\Biggr]
.
\end{align*}
Induction hence establishes~(\ref{item:scheme_description_heat1})--(\ref{item:scheme_description_heat2}).

Furthermore,
combining~(\ref{item:scheme_description_heat1})--(\ref{item:scheme_description_heat2})
with~(\ref{item:measurability_Phi_heat2}) in Lemma~\ref{lem:measurability_Phi_heat}
and~(\ref{item:independence1}) in Proposition~\ref{prop:independence}
shows~(\ref{item:scheme_description_heat3}).
The proof of Lemma~\ref{lem:scheme_description_heat} is thus complete.
\end{proof}

\subsubsection{Integrability}
\label{sec:integrability}

\begin{lemma}
\label{lem:expectation_finite_heat}
Assume Setting~\ref{setting:MLP_heat}.
Then
it holds
for all
$ l \in \N_0 $,
$ j \in \N $,
$ r \in [ 0, \infty ) $
that
\begin{equation}
\E\Bigl[
    \bigl\lVert
        \Phi_{ l }\bigl(
        Y^{ 0 }_{ l, j },
        Y^{ 1 }_{ l - 1, j },
        Z^{ 0 }
        \bigr)
    \bigr\rVert_{ \mathcal{Y} }^r
    +
    \bigl\lVert
        Y^{ 0 }_{ l - 1, j }
    \bigr\rVert_{ \mathcal{Y} }^r
\Bigr]
< \infty
\end{equation}
(cf.~(\ref{item:scheme_description_heat3}) in Lemma~\ref{lem:scheme_description_heat}).
\end{lemma}
\begin{proof}[Proof of Lemma~\ref{lem:expectation_finite_heat}]
First of all,
note that it holds for all
$ ( \mathfrak{u}, \mathfrak{w} ) \in \mathcal{Z} $,
$ t \in [ 0, T ] $,
$ x \in \R^d $
that
\begin{equation}
\label{eq:linear_growth}
\begin{split}
\lVert
    x + \mathfrak{w}_{ \mathfrak{u}t }
\rVert_{ \R^d }
\leq
\lVert
    x
\rVert_{ \R^d }
+
\lVert
    \mathfrak{w}_{ \mathfrak{u}t }
\rVert_{ \R^d }
\leq
\biggl(
1 +
\sup_{ s \in [ 0, T ] }
\lVert
    \mathfrak{w}_{ s }
\rVert_{ \R^d }
\biggr)
\max\{ 1, \lVert x \rVert_{ \R^d } \}
.
\end{split}
\end{equation}
This and
Lemma~\ref{lem:boundedness_Phi}
(with
$ d = d $,
$ T = T $,
$ p = p $,
$ L = 1 +
\sup_{ t \in [ 0, T ] }
\lVert
    \mathfrak{w}_{ t }
\rVert_{ \R^d } $,
$ \varrho =
( [ 0, T ] \times \R^d \ni ( t, x ) \mapsto ( t, x + \mathfrak{w}_{ t } ) \in [ 0, T ] \times \R^d ) $,
$ v =
( [ 0, T ] \times \R^d \ni ( t, x ) \mapsto g( x ) \in \R ) $
for
$ \mathfrak{w} \in C( [ 0, T ], \R^d ) $
in the notation of Lemma~\ref{lem:boundedness_Phi})
show for all
$ \mathfrak{w} \in C( [ 0, T ], \R^d ) $
that
\begin{equation}
\label{eq:boundedness_g}
\begin{split}
&
\sup_{ ( t, x ) \in [ 0, T ] \times \R^d }
\biggl[
\frac{ \lvert
    g( x + \mathfrak{w}_{ t } )
\rvert }{ \max\{ 1, \lVert x \rVert_{ \R^d }^q \} }
\biggr]
\leq
\sup_{ ( t, x ) \in [ 0, T ] \times \R^d }
\biggl[
\frac{ \lvert
    g( x + \mathfrak{w}_{ t } )
\rvert }{ \max\{ 1, \lVert x \rVert_{ \R^d }^p \} }
\biggr]
\\ &
\leq
\biggl(
1 +
\sup_{ t \in [ 0, T ] }
\lVert
    \mathfrak{w}_{ t }
\rVert_{ \R^d }
\biggr)^p
\sup_{ ( t, x ) \in [ 0, T ] \times \R^d }
\biggl[
    \frac{ \lvert g( x ) \rvert }{ \max\{ 1, \lVert x \rVert_{ \R^d }^p \} }
\biggr]
\leq
L
\biggl(
1 +
\sup_{ t \in [ 0, T ] }
\lVert
    \mathfrak{w}_{ t }
\rVert_{ \R^d }
\biggr)^p
.
\end{split}
\end{equation}
Similarly,
\eqref{eq:linear_growth}
and
Lemma~\ref{lem:boundedness_Phi}
(with
$ d = d $,
$ T = T $,
$ p = p $,
$ L = 1 +
\sup_{ t \in [ 0, T ] }
\lVert
    \mathfrak{w}_{ t }
\rVert_{ \R^d } $,
$ \varrho =
( [ 0, T ] \times \R^d \ni ( t, x ) \mapsto ( T - t + \mathfrak{u}t, x + \mathfrak{w}_{ \mathfrak{u}t } ) \in [ 0, T ] \times \R^d ) $,
$ v =
( [ 0, T ] \times \R^d \ni ( t, x ) \mapsto t f( t, x, 0 ) \in \R ) $
for
$ ( \mathfrak{u}, \mathfrak{w} ) \in \mathcal{Z} $
in the notation of Lemma~\ref{lem:boundedness_Phi})
ensure for all
$ ( \mathfrak{u}, \mathfrak{w} ) \in \mathcal{Z} $
that
\begin{align*}
\label{eq:boundedness_f0}
&
\sup_{ ( t, x ) \in [ 0, T ] \times \R^d }
\biggl[
\frac{ \lvert
    t
    f(
        T - t + \mathfrak{u}t, x + \mathfrak{w}_{ \mathfrak{u}t },
        0
    )
\rvert }{ \max\{ 1, \lVert x \rVert_{ \R^d }^q \} }
\biggr]
\leq
\sup_{ ( t, x ) \in [ 0, T ] \times \R^d }
\biggl[
\frac{ \lvert
    t
    f(
        T - t + \mathfrak{u}t, x + \mathfrak{w}_{ \mathfrak{u}t },
        0
    )
\rvert }{ \max\{ 1, \lVert x \rVert_{ \R^d }^p \} }
\biggr]
\\ & \yesnumber
\leq
\biggl(
1 +
\sup_{ t \in [ 0, T ] }
\lVert
    \mathfrak{w}_{ t }
\rVert_{ \R^d }
\biggr)^p
\sup_{ ( t, x ) \in [ 0, T ] \times \R^d }
\biggl[
    \frac{ \lvert t f( t, x, 0 ) \rvert }{ \max\{ 1, \lVert x \rVert_{ \R^d }^p \} }
\biggr]
\\ &
\leq
T
\biggl(
1 +
\sup_{ t \in [ 0, T ] }
\lVert
    \mathfrak{w}_{ t }
\rVert_{ \R^d }
\biggr)^p
\sup_{ ( t, x ) \in [ 0, T ] \times \R^d }
\biggl[
    \frac{ \lvert f( t, x, 0 ) \rvert }{ \max\{ 1, \lVert x \rVert_{ \R^d }^p \} }
\biggr]
\leq
L
T
\biggl(
1 +
\sup_{ t \in [ 0, T ] }
\lVert
    \mathfrak{w}_{ t }
\rVert_{ \R^d }
\biggr)^p
.
\end{align*}
Combining~\eqref{eq:definition_Phi_heat},
\eqref{eq:boundedness_g},
and~\eqref{eq:boundedness_f0}
implies
for all
$ w \in \mathcal{Y} $,
$ \mathfrak{z} = ( \mathfrak{u}, \mathfrak{w} ) \in \mathcal{Z} $
that
\begin{equation}
\label{eq:boundedness_Phi_0}
\begin{split}
& \lVert
    \Phi_0( 0, w, \mathfrak{z} )
\rVert_{ \mathcal{Y} }
=
\sup_{ ( t, x ) \in [ 0, T ] \times \R^d }
\biggl[
\frac{ \lvert
    [ \Phi_0( 0, w, \mathfrak{z} ) ]( T - t, x )
\rvert }{ \max\{ 1, \lVert x \rVert_{ \R^d }^q \} }
\biggr]
\\ &
\leq
\biggl[
\sup_{ ( t, x ) \in [ 0, T ] \times \R^d }
\frac{ \lvert
    g( x + \mathfrak{w}_{ t } )
\rvert }{ \max\{ 1, \lVert x \rVert_{ \R^d }^q \} }
\biggr]
+
\biggl[
\sup_{ ( t, x ) \in [ 0, T ] \times \R^d }
\frac{ \lvert
    t
    f(
        T - t + \mathfrak{u}t, x + \mathfrak{w}_{ \mathfrak{u}t },
        0
    )
\rvert }{ \max\{ 1, \lVert x \rVert_{ \R^d }^q \} }
\biggr]
\\ &
\leq
L
( T + 1 )
\biggl(
1 +
\sup_{ t \in [ 0, T ] }
\lVert
    \mathfrak{w}_{ t }
\rVert_{ \R^d }
\biggr)^p
.
\end{split}
\end{equation}
In addition,
\eqref{eq:definition_Phi_heat},
\eqref{eq:linear_growth}
and Lemma~\ref{lem:boundedness_Phi}
(with
$ d = d $,
$ T = T $,
$ p = q $,
$ L = 1 +
\sup_{ t \in [ 0, T ] }
\lVert
    \mathfrak{w}_{ t }
\rVert_{ \R^d } $,
$ \varrho =
( [ 0, T ] \times \R^d \ni ( t, x ) \mapsto ( T - t + \mathfrak{u}t, x + \mathfrak{w}_{ \mathfrak{u}t } ) \in [ 0, T ] \times \R^d ) $,
$ v = v - w $
for
$ ( \mathfrak{u}, \mathfrak{w} ) \in \mathcal{Z} $,
$ v, w \in \mathcal{Y} $
in the notation of Lemma~\ref{lem:boundedness_Phi})
prove for all
$ l \in \N $,
$ v, w \in \mathcal{Y} $,
$ \mathfrak{z} = ( \mathfrak{u}, \mathfrak{w} ) \in \mathcal{Z} $
that
\begin{equation}
\label{eq:boundedness_Phi_l}
\begin{split}
& \lVert
    \Phi_l( v, w, \mathfrak{z} )
\rVert_{ \mathcal{Y} }
=
\sup_{ ( t, x ) \in [ 0, T ] \times \R^d }
\biggl[
\frac{ \lvert
    [ \Phi_l( v, w, \mathfrak{z} ) ]( T - t, x )
\rvert }{ \max\{ 1, \lVert x \rVert_{ \R^d }^q \} }
\biggr]
\\ &
=
\sup_{ ( t, x ) \in [ 0, T ] \times \R^d }
\biggl[
\frac{ t }{ \max\{ 1, \lVert x \rVert_{ \R^d }^q \} }
\bigl\lvert
    f\bigl(
        T - t + \mathfrak{u}t, x + \mathfrak{w}_{ \mathfrak{u}t },
        v( T - t + \mathfrak{u}t, x + \mathfrak{w}_{ \mathfrak{u} t } )
    \bigr)
\\ &
\hphantom{
=
\sup_{ ( t, x ) \in [ 0, T ] \times \R^d }
\biggl[
\frac{ t }{ \max\{ 1, \lVert x \rVert_{ \R^d }^q \} }
\bigl\lvert
}
    \mathop{-}
    f\bigl(
        T - t + \mathfrak{u}t, x + \mathfrak{w}_{ \mathfrak{u}t },
        w( T - t + \mathfrak{u}t, x + \mathfrak{w}_{ \mathfrak{u} t } )
    \bigr)
\bigr\rvert
\biggr]
\\ & \leq
L T
\sup_{ ( t, x ) \in [ 0, T ] \times \R^d }
\biggl[
\frac{ \lvert
    v( T - t + \mathfrak{u}t, x + \mathfrak{w}_{ \mathfrak{u} t } )
    -
    w( T - t + \mathfrak{u}t, x + \mathfrak{w}_{ \mathfrak{u} t } )
\rvert }{ \max\{ 1, \lVert x \rVert_{ \R^d }^q \} }
\biggr]
\\ &
\leq
L T
\biggl(
1 +
\sup_{ t \in [ 0, T ] }
\lVert
    \mathfrak{w}_{ t }
\rVert_{ \R^d }
\biggr)^q
\sup_{ ( t, x ) \in [ 0, T ] \times \R^d }
\biggl[
    \frac{ \lvert v( t, x ) - w( t, x ) \rvert }{ \max\{ 1, \lVert x \rVert_{ \R^d }^q \} }
\biggr]
\\ &
=
L
T
\biggl(
1 +
\sup_{ t \in [ 0, T ] }
\lVert
    \mathfrak{w}_{ t }
\rVert_{ \R^d }
\biggr)^q
\lVert v - w \rVert_{ \mathcal{Y} }
.
\end{split}
\end{equation}

Next we claim that it holds for all
$ l \in \N_0 $,
$ j \in \N $,
$ r \in [ 0, \infty ) $
that
\begin{equation}
\label{eq:expectation_finite_heat_induction}
\E\Bigl[
    \bigl\lVert
        \Phi_{ l }\bigl(
        Y^{ 0 }_{ l, j },
        Y^{ 1 }_{ l - 1, j },
        Z^{ 0 }
        \bigr)
    \bigr\rVert_{ \mathcal{Y} }^r
    +
    \bigl\lVert
        Y^{ 0 }_{ l, j }
    \bigr\rVert_{ \mathcal{Y} }^r
    +
    \bigl\lVert
        Y^{ 0 }_{ l - 1, j }
    \bigr\rVert_{ \mathcal{Y} }^r
\Bigr]
< \infty
.
\end{equation}
We establish~\eqref{eq:expectation_finite_heat_induction}
by induction on $ l \in \N_0 $.
For the base case $ l = 0 $
observe that~\eqref{eq:boundedness_Phi_0}
and
the fact that
$ \forall \, a, b, r \in [ 0, \infty ) \colon
( a + b )^r \leq 2^{ \max\{ r - 1, 0 \} }( a^r + b^r ) $
show for all
$ j \in \N $,
$ r \in [ 0, \infty ) $
that
\begin{equation}
\begin{split}
& \E\Bigl[
    \bigl\lVert
        \Phi_{ 0 }\bigl(
        Y^{ 0 }_{ 0, j },
        Y^{ 1 }_{ -1, j },
        Z^{ 0 }
        \bigr)
    \bigr\rVert_{ \mathcal{Y} }^r
\Bigr]
=
\E\Bigl[
    \bigl\lVert
        \Phi_{ 0 }\bigl(
        0,
        0,
        U^0, W^0
        \bigr)
    \bigr\rVert_{ \mathcal{Y} }^r
\Bigr]
\\ &
\leq
L^r
( T + 1 )^r
\, \E\biggl[
\biggl(
1 +
\sup_{ t \in [ 0, T ] }
\lVert
    W^0_{ t }
\rVert_{ \R^d }
\biggr)^{pr}
\biggr]
\\ & \leq
2^{ \max\{ pr - 1, 0 \} }
L^r
( T + 1 )^r
\biggl(
1 +
\E\biggl[
\sup_{ t \in [ 0, T ] }
\lVert
    W^0_{ t }
\rVert_{ \R^d }^{pr}
\biggr]
\biggr)
<
\infty.
\end{split}
\end{equation}
This and the fact that
$ \forall \, j \in \N, \, r \in [ 0, \infty ) \colon
\E\bigl[
    \lVert
        Y^{ 0 }_{ 0, j }
    \rVert_{ \mathcal{Y} }^r
    +
    \lVert
        Y^{ 0 }_{ -1, j }
    \rVert_{ \mathcal{Y} }^r
\bigr]
= 0 < \infty
$
prove~\eqref{eq:expectation_finite_heat_induction}
in the base case $ l = 0 $.
For the induction step $ \N_0 \ni l - 1 \to l \in \N $
let $ l \in \N $
and assume that it holds for all
$ k \in \{ 0, 1, \ldots, l - 1 \} $,
$ j \in \N $,
$ r \in [ 0, \infty ) $
that
\begin{equation}
\label{eq:expectation_finite_heat_induction_step}
\E\Bigl[
    \bigl\lVert
        \Phi_{ k }\bigl(
        Y^{ 0 }_{ k, j },
        Y^{ 1 }_{ k - 1, j },
        Z^{ 0 }
        \bigr)
    \bigr\rVert_{ \mathcal{Y} }^r
    +
    \bigl\lVert
        Y^{ 0 }_{ k, j }
    \bigr\rVert_{ \mathcal{Y} }^r
    +
    \bigl\lVert
        Y^{ 0 }_{ k - 1, j }
    \bigr\rVert_{ \mathcal{Y} }^r
\Bigr]
< \infty
.
\end{equation}
Note that
this,
(\ref{item:scheme_description_heat2}) in Lemma~\ref{lem:scheme_description_heat},
and~(\ref{item:independence6}) in Proposition~\ref{prop:independence}
ensure for all
$ j \in \N $,
$ r \in [ 1, \infty ) $
that
\begin{equation}
\label{eq:boundedness_Y_l_j}
\begin{split}
\Bigl( \E\Bigl[
    \bigl\lVert
        Y^{ 0 }_{ l, j }
    \bigr\rVert_{ \mathcal{Y} }^r
\Bigr] \Bigr)^{ \nicefrac{1}{r} }
& =
\Biggl( \E\Biggl[
    \Biggl\lVert
        \sum_{k=0}^{l-1}
        \tfrac{1}{ ( M_j )^{ l - k } }
        \Biggl[
            \sum_{i=1}^{ ( M_j )^{ l - k } }
                \Phi_{ k } \bigl(
                Y^{ ( 0, k, i ) }_{ k, j },
                Y^{ ( 0, -k, i ) }_{ k - 1, j },
                Z^{ ( 0, k, i ) }
                \bigr)
        \Biggr]
    \Biggr\rVert_{ \mathcal{Y} }^r
\Biggr] \Biggr)^{ \nicefrac{1}{r} }
\\ &
\leq
\sum_{k=0}^{l-1}
\tfrac{1}{ ( M_j )^{ l - k } }
\Biggl[
    \sum_{i=1}^{ ( M_j )^{ l - k } }
    \Bigl( \E\Bigl[
        \bigl\lVert
        \Phi_{ k } \bigl(
        Y^{ ( 0, k, i ) }_{ k, j },
        Y^{ ( 0, -k, i ) }_{ k - 1, j },
        Z^{ ( 0, k, i ) }
        \bigr)
    \bigr\rVert_{ \mathcal{Y} }^r
\Bigr] \Bigr)^{ \nicefrac{1}{r} }
\Biggr]
\\ &
=
\sum_{k=0}^{l-1}
    \Bigl( \E\Bigl[
        \bigl\lVert
        \Phi_{ k } \bigl(
        Y^{ 0 }_{ k, j },
        Y^{ 1 }_{ k - 1, j },
        Z^{ 0 }
        \bigr)
    \bigr\rVert_{ \mathcal{Y} }^r
\Bigr] \Bigr)^{ \nicefrac{1}{r} }
< \infty
.
\end{split}
\end{equation}
H\"older's inequality,
\eqref{eq:boundedness_Phi_l},
the fact that
$ \forall \, a, b, r \in [ 0, \infty ) \colon
( a + b )^r \leq 2^{ \max\{ r - 1, 0 \} }( a^r + b^r ) $,
(\ref{item:independence5}) in Proposition~\ref{prop:independence},
and
\eqref{eq:expectation_finite_heat_induction_step}
hence
demonstrate for all
$ j \in \N $,
$ r \in [ 1, \infty ) $
that
\begin{equation}
\begin{split}
& \Bigl( \E\Bigl[
    \bigl\lVert
        \Phi_{ l }\bigl(
        Y^{ 0 }_{ l, j },
        Y^{ 1 }_{ l - 1, j },
        Z^{ 0 }
        \bigr)
    \bigr\rVert_{ \mathcal{Y} }^r
\Bigr] \Bigr)^{ \nicefrac{1}{r} }
\\ & \leq
L
T
\biggl( \E\biggl[
    \biggl(
    1 +
    \sup_{ t \in [ 0, T ] }
    \lVert
        W_{ t }^0
    \rVert_{ \R^d }
    \biggr)^{qr}
    \bigl\lVert
        Y^{ 0 }_{ l, j } - Y^{ 1 }_{ l - 1, j }
    \bigr\rVert_{ \mathcal{Y} }^r
\biggr] \biggr)^{ \nicefrac{1}{r} }
\\ & \leq
L
T
\biggl( \E\biggl[
    \biggl(
    1 +
    \sup_{ t \in [ 0, T ] }
    \lVert
        W_{ t }^0
    \rVert_{ \R^d }
    \biggr)^{2qr}
\biggr] \biggr)^{ \nicefrac{1}{(2r)} }
\Bigl( \E\Bigl[
    \bigl\lVert
        Y^{ 0 }_{ l, j } - Y^{ 1 }_{ l - 1, j }
    \bigr\rVert_{ \mathcal{Y} }^{2r}
\Bigr] \Bigr)^{ \nicefrac{1}{(2r)} }
\\ & \leq
2^{ \max\{ q - \nicefrac{1}{(2r)}, 0 \} }
L
T
\biggl(
    1 +
    \E\biggl[
    \sup_{ t \in [ 0, T ] }
    \lVert
        W_{ t }^0
    \rVert_{ \R^d }^{2qr}
\biggr] \biggr)^{ \nicefrac{1}{(2r)} }
\\ & \quad
\cdot
\biggl[
\Bigl( \E\Bigl[
    \bigl\lVert
        Y^{ 0 }_{ l, j }
    \bigr\rVert_{ \mathcal{Y} }^{2r}
\Bigr] \Bigr)^{ \nicefrac{1}{(2r)} }
+
\Bigl( \E\Bigl[
    \bigl\lVert
        Y^{ 0 }_{ l - 1, j }
    \bigr\rVert_{ \mathcal{Y} }^{2r}
\Bigr] \Bigr)^{ \nicefrac{1}{(2r)} }
\biggr]
< \infty
.
\end{split}
\end{equation}
Combining this with~\eqref{eq:boundedness_Y_l_j} and~\eqref{eq:expectation_finite_heat_induction_step}
establishes for all
$ j \in \N $,
$ r \in [ 0, \infty ) $
that
\begin{equation}
\E\Bigl[
    \bigl\lVert
        \Phi_{ l }\bigl(
        Y^{ 0 }_{ l, j },
        Y^{ 1 }_{ l - 1, j },
        Z^{ 0 }
        \bigr)
    \bigr\rVert_{ \mathcal{Y} }^r
    +
    \bigl\lVert
        Y^{ 0 }_{ l, j }
    \bigr\rVert_{ \mathcal{Y} }^r
    +
    \bigl\lVert
        Y^{ 0 }_{ l - 1, j }
    \bigr\rVert_{ \mathcal{Y} }^r
\Bigr]
< \infty
.
\end{equation}
Induction hence proves~\eqref{eq:expectation_finite_heat_induction}.
The proof of Lemma~\ref{lem:expectation_finite_heat} is thus complete.
\end{proof}

\subsubsection{Estimates}
\label{sec:estimates}

\begin{lemma}
\label{lem:hypothesis_I_heat}
Assume Setting~\ref{setting:MLP_heat}
and
let $ C \in [ 0, \infty ) $
be given by
\begin{equation}
C
=
e^{ LT }
\Bigl[
    \bigl( \E \bigl[
    \lvert g( \xi + W_{ T }^0 ) \rvert^2
    \bigr] \bigr)^{ \nicefrac{1}{2} }
    +
    \sqrt{T}
    \bigl(
    \textstyle\int_{0}^{T}
        \E \bigl[
        \lvert f( t, \xi + W_{ t }^0, 0 ) \rvert^2
        \bigr]
    \ud t
    \bigr)^{ \nicefrac{1}{2} }
\Bigr]
.
\end{equation}
Then
it holds
for all
$ k \in \N_0 $
%$ \mathfrak{y}_{ -1 } \in \mathcal{Y} $
that
\begin{equation}
\max\bigl\{
\E\bigl[ \lvert \psi_k (
    \Phi_{ 0 }( 0, 0, Z^{ 0 } )
) \rvert^2 \bigr]
,
\E\bigl[ \lvert
    \psi_k( y )
\rvert^2 \bigr]
\bigr\}
\leq
\tfrac{ C^2 }{ k! }
.
\end{equation}
\end{lemma}
\begin{proof}[Proof of Lemma~\ref{lem:hypothesis_I_heat}]
Throughout this proof
let
$ F \colon C( [ 0, T ] \times \R^d, \R ) \to C( [ 0, T ] \times \R^d, \R ) $
satisfy
for all
$ v \in C( [ 0, T ] \times \R^d, \R ) $,
$ t \in [ 0, T] $,
$ x \in \R^d $
that
\begin{equation}
[ F( v ) ]( t, x )
=
f( t, x, v( t, x ) )
.
\end{equation}
Observe that
\eqref{eq:definition_psi_k},
the fact that
$ \mathbf{U} $ and $ \mathbf{W} $ are independent,
and
Hutzenthaler et al.~\cite[Lemma~2.3]{HutzenthalerJentzenKruseNguyenVonWurstemberger2018arXivv2}
(with
%$ ( \Omega, \mathcal{F}, \P ) = ( \Omega, \mathscr{F}, \P ) $,
$ S = [ 0, 1 ] $,
$ U = ( [ 0, 1 ] \times \Omega \ni ( s, \omega ) \mapsto
\nicefrac{ s^{ k - 1 } }{ ( k - 1 )! }
\lvert y( s T, \xi + \mathbf{W}_{ s T }( \omega ) ) \rvert^2
\in [ 0, \infty ) ) $,
$ Y = \mathbf{U} $
for $ k \in \N $
in the notation of~\cite[Lemma~2.3]{HutzenthalerJentzenKruseNguyenVonWurstemberger2018arXivv2})
imply for all
$ k \in \N $
that
\begin{equation}
\begin{split}
& \E\bigl[ \lvert
    \psi_k( y )
\rvert^2 \bigr]
=
\E \Bigl[
\tfrac{ \mathbf{U}^{ k - 1 } }{ ( k - 1 )! }
\lvert y( \mathbf{U} T, \xi + \mathbf{W}_{ \mathbf{U} T } ) \rvert^2
\Bigr]
\\ &
=
\tfrac{ 1 }{ ( k - 1 )! } \,
\int_{0}^{1}
    s^{ k - 1 } \,
    \E \bigl[
    \lvert y( sT, \xi + \mathbf{W}_{ sT } ) \rvert^2
    \bigr]
\ud s
\leq
\tfrac{1}{ k! }
\biggl[
\sup_{ t \in [ 0, T ] }
    \E \bigl[
    \lvert y( t, \xi + \mathbf{W}_{ t } ) \rvert^2
    \bigr]
\biggr]
.
\end{split}
\end{equation}
This,
the fact that
$ \E[ \lvert
    \psi_0( y )
\rvert^2 ]
=
\lvert y( 0, \xi ) \rvert^2
=
\E[ \lvert
    y( 0, \xi + \mathbf{W}_{ 0 } )
\rvert^2 ]
$,
and
\cite[Lemma~3.4]{HutzenthalerJentzenKruseNguyenVonWurstemberger2018arXivv2}
establish for all
$ k \in \N_0 $
that
\begin{equation}
\begin{split}
& \E\bigl[ \lvert
    \psi_k( y )
\rvert^2 \bigr]
\leq
\tfrac{1}{ k! }
\biggl[
\sup_{ t \in [ 0, T ] }
    \E \bigl[
    \lvert y( t, \xi + \mathbf{W}_{ t } ) \rvert^2
    \bigr]
\biggr]
\\ &
\leq
\tfrac{ e^{ 2LT } }{ k! }
\Bigl[
    \bigl( \E \bigl[
    \lvert g( \xi + \mathbf{W}_{ T } ) \rvert^2
    \bigr] \bigr)^{ \nicefrac{1}{2} }
    +
    \sqrt{T}
    \bigl(
    \textstyle\int_{0}^{T}
        \E \bigl[
        \lvert [ F( 0 ) ]( t, \xi + \mathbf{W}_{ t } ) \rvert^2
        \bigr]
    \ud t
    \bigr)^{ \nicefrac{1}{2} }
\Bigr]^2
\\ &
=
\tfrac{ e^{ 2LT } }{ k! }
\Bigl[
    \bigl( \E \bigl[
    \lvert g( \xi + W_{ T }^0 ) \rvert^2
    \bigr] \bigr)^{ \nicefrac{1}{2} }
    +
    \sqrt{T}
    \bigl(
    \textstyle\int_{0}^{T}
        \E \bigl[
        \lvert f( t, \xi + W_{ t }^0, 0 ) \rvert^2
        \bigr]
    \ud t
    \bigr)^{ \nicefrac{1}{2} }
\Bigr]^2
=
\frac{ C^2 }{ k! }
.
\end{split}
\end{equation}
Next note that
\eqref{eq:definition_Phi_heat}
shows
for all
$ t \in [ 0, T ] $,
$ x \in \R^d $
that
\begin{equation}
\label{eq:Phi_0_explicit}
[ \Phi_{ 0 }( 0, 0, Z^{ 0 } ) ]( t, x )
=
g( x + W^0_{ T - t } )
+
( T - t )
f\bigl(
    t + ( T - t )U^0, x + W^0_{ ( T - t )U^0 }, 0
\bigr)
.
\end{equation}
This,
\eqref{eq:definition_psi_k},
and H\"older's inequality demonstrate for all
$ k \in \N $
that
\begin{equation}
\label{eq:psi_k_Phi_0_estimate}
\begin{split}
&
\bigl( \E\bigl[ \lvert \psi_k (
    \Phi_{ 0 }( 0, 0, Z^{ 0 } )
) \rvert^2 \bigr] \bigr)^{ \nicefrac{1}{2} }
=
\E \Bigl[
\tfrac{ \mathbf{U}^{ k - 1 } }{ ( k - 1 )! }
\lvert [ \Phi_{ 0 }( 0, 0, Z^{ 0 } ) ]( \mathbf{U} T, \xi + \mathbf{W}_{ \mathbf{U} T } ) \rvert^2
\Bigr]
\\ & =
\Bigl( \E\Bigl[
\tfrac{ \mathbf{U}^{ k - 1 } }{ ( k - 1 )! }
\bigl\lvert 
    g( \xi + \mathbf{W}_{ \mathbf{U} T } + W^0_{ ( 1 - \mathbf{U} )T } )
\\ & \quad
\phantom{\Bigl( \E\Bigl[}
    +
    ( 1 - \mathbf{U} )T
    f\bigl(
        \mathbf{U} T + ( 1 - \mathbf{U} ) U^0 T,
        \xi + \mathbf{W}_{ \mathbf{U} T } + W^0_{ ( 1 - \mathbf{U} ) U^0 T },
        0
    \bigr)
\bigr\rvert^2 \Bigr] \Bigr)^{ \nicefrac{1}{2} }
\\ &
\leq
\Bigl( \E\Bigl[
\tfrac{ \mathbf{U}^{ k - 1 } }{ ( k - 1 )! }
\lvert 
    g( \xi + \mathbf{W}_{ \mathbf{U} T } + W^0_{ ( 1 - \mathbf{U} )T } )
\rvert^2 \Bigr] \Bigr)^{ \nicefrac{1}{2} }
\\ & \quad
+
\Bigl( \E\Bigl[
\tfrac{ \mathbf{U}^{ k - 1 } }{ ( k - 1 )! }
\bigl\lvert 
    ( 1 - \mathbf{U} )T
    f\bigl(
        \mathbf{U} T + ( 1 - \mathbf{U} ) U^0 T,
        \xi + \mathbf{W}_{ \mathbf{U} T } + W^0_{ ( 1 - \mathbf{U} ) U^0 T },
        0
    \bigr)
\bigr\rvert^2 \Bigr] \Bigr)^{ \nicefrac{1}{2} }
.
\end{split}
\end{equation}
The fact that
$ \mathbf{U} $, $ \mathbf{W} $, and $ W^0 $ are independent
and
\cite[Lemma~2.3]{HutzenthalerJentzenKruseNguyenVonWurstemberger2018arXivv2}
(with
%$ ( \Omega, \mathcal{F}, \P ) = ( \Omega, \mathscr{F}, \P ) $,
$ S = [ 0, 1 ] $,
$ U = ( [ 0, 1 ] \times \Omega \ni ( s, \omega ) \mapsto
\nicefrac{ s^{ k - 1 } }{ ( k - 1 )! }
\lvert g( \xi + \mathbf{W}_{ sT }( \omega ) + W^0_{ ( 1 - s )T }( \omega ) ) \rvert^2 
\in [ 0, \infty ) ) $,
$ Y = \mathbf{U} $
for $ k \in \N $
in the notation of~\cite[Lemma~2.3]{HutzenthalerJentzenKruseNguyenVonWurstemberger2018arXivv2})
ensure for all
$ k \in \N $
that
\begin{equation}
\begin{split}
\label{eq:psi_k_g_estimate}
& \E\Bigl[
\tfrac{ \mathbf{U}^{ k - 1 } }{ ( k - 1 )! }
\lvert 
    g( \xi + \mathbf{W}_{ \mathbf{U} T } + W^0_{ ( 1 - \mathbf{U} )T } )
\rvert^2 \Bigr]
=
\int_{0}^{1}
    \tfrac{ s^{ k - 1 } }{ ( k - 1 )! } \,
    \E\bigl[
    \lvert 
        g( \xi + \mathbf{W}_{ sT } + W^0_{ ( 1 - s )T } )
    \rvert^2 \bigr]
\ud s
\\ &
=
\tfrac{ 1 }{ ( k - 1 )! }
\biggl[
\int_{0}^{1}
    s^{ k - 1 }
\ud s
\biggr]
\E\bigl[
\lvert 
    g( \xi + W^0_{ T } )
\rvert^2 \bigr]
=
\tfrac{ 1 }{ k! }
\, \E\bigl[
\lvert 
    g( \xi + W^0_{ T } )
\rvert^2 \bigr]
.
\end{split}
\end{equation}
In addition,
the fact that
$ \mathbf{U} $, $ U^0 $, $ \mathbf{W} $, and $ W^0 $ are independent,
\cite[Lemma~2.3]{HutzenthalerJentzenKruseNguyenVonWurstemberger2018arXivv2}
(with
%$ ( \Omega, \mathcal{F}, \P ) = ( \Omega, \mathscr{F}, \P ) $,
$ S = [ 0, 1 ] $,
$ U = ( [ 0, 1 ] \times \Omega \ni ( s, \omega ) \mapsto
\nicefrac{ s^{ k - 1 } }{ ( k - 1 )! }
\lvert ( 1 - s )T f( sT + ( 1 - s ) U^0( \omega ) T, \xi + \mathbf{W}_{ sT }( \omega ) + W^0_{ ( 1 - s ) U^0( \omega ) T }( \omega ), 0 ) \rvert^2 
\in [ 0, \infty ) ) $,
$ Y = \mathbf{U} $
for $ k \in \N $
in the notation of~\cite[Lemma~2.3]{HutzenthalerJentzenKruseNguyenVonWurstemberger2018arXivv2}),
and
\cite[Lemma~2.10]{HutzenthalerJentzenKruseNguyenVonWurstemberger2018arXivv2}
(with
$ k = k $,
$ U = ( [ 0, T ] \times \R^d \times \Omega \ni ( t, x, \omega ) \mapsto
f( t, x, 0 )
\in \R ) $,
$ \mathfrak{r} = U^0 $,
$ \mathbb{W} = W^0 $
for $ k \in \N $
in the notation of~\cite[Lemma~2.10]{HutzenthalerJentzenKruseNguyenVonWurstemberger2018arXivv2})
establish for all
$ k \in \N $
that
\begin{equation}
\label{eq:psi_k_f_estimate}
\begin{split}
& \E\Bigl[
\tfrac{ \mathbf{U}^{ k - 1 } }{ ( k - 1 )! }
\bigl\lvert 
    ( 1 - \mathbf{U} )T
    f\bigl(
        \mathbf{U} T + ( 1 - \mathbf{U} ) U^0 T,
        \xi + \mathbf{W}_{ \mathbf{U} T } + W^0_{ ( 1 - \mathbf{U} ) U^0 T },
        0
    \bigr)
\bigr\rvert^2 \Bigr]
\\ &
=
\int_{0}^{1}
    \tfrac{ s^{ k - 1 } }{ ( k - 1 )! } \,
    \E\Bigl[
    \bigl\lvert 
        ( 1 - s )T
        f\bigl(
            sT + ( 1 - s ) U^0 T,
            \xi + \mathbf{W}_{ sT } + W^0_{ ( 1 - s ) U^0 T },
            0
        \bigr)
    \bigr\rvert^2 \Bigr]
\ud s
\\ &
=
\tfrac{1}{T^k}
\int_{0}^{T}
    \tfrac{ t^{ k - 1 } }{ ( k - 1 )! } \,
    \E\Bigl[
    \bigl\lvert 
        ( T - t )
        f\bigl(
            t + ( T - t ) U^0,
            \xi + \mathbf{W}_{ t } + W^0_{ ( T - t ) U^0 },
            0
        \bigr)
    \bigr\rvert^2 \Bigr]
\ud t
\\ &
=
\tfrac{1}{T^k}
\int_{0}^{T}
    \tfrac{ t^{ k - 1 } }{ ( k - 1 )! } \,
    \E\Bigl[
    \bigl\lvert 
        ( T - t )
        f\bigl(
            t + ( T - t ) U^0,
            \xi + \mathbf{W}_{ t } + W^0_{ t + ( T - t ) U^0 } - W^0_{ t },
            0
        \bigr)
    \bigr\rvert^2 \Bigr]
\ud t
\\ &
\leq
\tfrac{T^2}{T^{ k + 1 } }
\int_{0}^{T}
    \tfrac{ t^{ k } }{ k! } \,
    \E\bigl[
    \lvert
        f(
            t,
            \xi + \mathbf{W}_{ t },
            0
        )
    \rvert^2 \bigr]
\ud t
\leq
\tfrac{ T }{ k! }
\int_{0}^{T}
    \E\bigl[
    \lvert
        f(
            t,
            \xi + \mathbf{W}_{ t },
            0
        )
    \rvert^2 \bigr]
\ud t
\\ &
=
\tfrac{ T }{ k! }
\int_{0}^{T}
    \E\bigl[
    \lvert
        f(
            t,
            \xi + W^0_{ t },
            0
        )
    \rvert^2 \bigr]
\ud t
.
\end{split}
\end{equation}
Combining~\eqref{eq:psi_k_Phi_0_estimate}
with~\eqref{eq:psi_k_g_estimate}--\eqref{eq:psi_k_f_estimate}
yields for all
$ k \in \N $
that
\begin{equation}
\begin{split}
&
\E\bigl[ \lvert \psi_k (
    \Phi_{ 0 }( 0, 0, Z^{ 0 } )
) \rvert^2 \bigr]
\\ &
\leq
\tfrac{ 1 }{ k! }
\Bigl[
    \bigl( \E \bigl[
    \lvert g( \xi + W_{ T }^0 ) \rvert^2
    \bigr] \bigr)^{ \nicefrac{1}{2} }
    +
    \sqrt{T}
    \bigl(
    \textstyle\int_{0}^{T}
        \E \bigl[
        \lvert f( t, \xi + W_{ t }^0, 0 ) \rvert^2
        \bigr]
    \ud t
    \bigr)^{ \nicefrac{1}{2} }
\Bigr]^2
\\ &
\leq
\tfrac{ e^{ 2LT } }{ k! }
\Bigl[
    \bigl( \E \bigl[
    \lvert g( \xi + W_{ T }^0 ) \rvert^2
    \bigr] \bigr)^{ \nicefrac{1}{2} }
    +
    \sqrt{T}
    \bigl(
    \textstyle\int_{0}^{T}
        \E \bigl[
        \lvert f( t, \xi + W_{ t }^0, 0 ) \rvert^2
        \bigr]
    \ud t
    \bigr)^{ \nicefrac{1}{2} }
\Bigr]^2
=
\frac{ C^2 }{ k! }
.
\end{split}
\end{equation}
Moreover,
\eqref{eq:Phi_0_explicit},
\eqref{eq:definition_psi_k},
H\"older's inequality,
the fact that
$ U^0 $ and $ W^0 $ are independent,
and \cite[Lemma~2.3]{HutzenthalerJentzenKruseNguyenVonWurstemberger2018arXivv2}
imply that
\begin{equation}
\begin{split}
&
\E\bigl[ \lvert \psi_0 (
    \Phi_{ 0 }( 0, 0, Z^{ 0 } )
) \rvert^2 \bigr]
=
\E\bigl[ \lvert
g( \xi + W^0_{ T } )
+
T
f( U^0 T, \xi + W^0_{ U^0 T }, 0 )
\rvert^2 \bigr]
\\ &
\leq
\Bigl[
    \bigl( \E \bigl[
    \lvert g( \xi + W_{ T }^0 ) \rvert^2
    \bigr] \bigr)^{ \nicefrac{1}{2} }
    +
    \bigl(
    T^2 \,
    \E \bigl[
    \lvert f( U^0 T, \xi + W^0_{ U^0 T }, 0 ) \rvert^2
    \bigr]
    \bigr)^{ \nicefrac{1}{2} }
\Bigr]^2
\\ &
=
\Bigl[
    \bigl( \E \bigl[
    \lvert g( \xi + W_{ T }^0 ) \rvert^2
    \bigr] \bigr)^{ \nicefrac{1}{2} }
    +
    \bigl(
    T^2
    \textstyle\int_{0}^{1}
        \E \bigl[
        \lvert f( s T, \xi + W_{ s T }^0, 0 ) \rvert^2
        \bigr]
    \ud s
    \bigr)^{ \nicefrac{1}{2} }
\Bigr]^2
\\ &
=
\Bigl[
    \bigl( \E \bigl[
    \lvert g( \xi + W_{ T }^0 ) \rvert^2
    \bigr] \bigr)^{ \nicefrac{1}{2} }
    +
    \bigl(
    T
    \textstyle\int_{0}^{T}
        \E \bigl[
        \lvert f( t, \xi + W_{ t }^0, 0 ) \rvert^2
        \bigr]
    \ud t
    \bigr)^{ \nicefrac{1}{2} }
\Bigr]^2
\\ &
\leq
e^{ 2LT }
\Bigl[
    \bigl( \E \bigl[
    \lvert g( \xi + W_{ T }^0 ) \rvert^2
    \bigr] \bigr)^{ \nicefrac{1}{2} }
    +
    \sqrt{T}
    \bigl(
    \textstyle\int_{0}^{T}
        \E \bigl[
        \lvert f( t, \xi + W_{ t }^0, 0 ) \rvert^2
        \bigr]
    \ud t
    \bigr)^{ \nicefrac{1}{2} }
\Bigr]^2
=
\frac{ C^2 }{ 0! }
.
\end{split}
\end{equation}
The proof of Lemma~\ref{lem:hypothesis_I_heat} is thus complete.
\end{proof}

\begin{lemma}
\label{lem:hypothesis_II_heat}
Assume Setting~\ref{setting:MLP_heat}.
%and
%let $ c \in [ 0, \infty ) $
%be given by
%$ c = LT $.
Then
it holds
for all
$ k \in \N_0 $,
$ n \in \N $,
$ u, v \in \mathcal{Y} $
that
\begin{equation}
\E\bigl[ \lvert \psi_k(
        \Phi_{ n }(
        u,
        v,
        Z^{ 0 }
        )
) \rvert^2 \bigr]
\leq
( LT )^2 \,
\E\bigl[ \lvert \psi_{ k + 1 } (
    u
    -
    v
) \rvert^2 \bigr]
.
\end{equation}
\end{lemma}
\begin{proof}[Proof of Lemma~\ref{lem:hypothesis_II_heat}]
Throughout this proof
let $ u, v \in \mathcal{Y} $.
Observe that
\eqref{eq:definition_Phi_heat}
shows
for all
$ t \in [ 0, T ] $,
$ x \in \R^d $
that
\begin{equation}
\label{eq:Phi_1_explicit}
\begin{split}
& \lvert [ \Phi_{ 1 }( u, v, Z^{ 0 } ) ]( t, x ) \rvert
\\ & =
( T - t )
\bigl\lvert
    f\bigl(
        t + ( T - t )U^0, x + W^0_{ ( T - t )U^0 },
        u\bigl( t + ( T - t )U^0, x + W^0_{ ( T - t )U^0 } \bigr)
    \bigr)
\\ & \quad
\phantom{
( T - t )
\bigl\lvert
}
    \mathop{-}
    f\bigl(
        t + ( T - t )U^0, x + W^0_{ ( T - t )U^0 },
        v\bigl( t + ( T - t )U^0, x + W^0_{ ( T - t )U^0 } \bigr)
    \bigr)
\bigr\rvert
\\ &
\leq
L( T - t )
\bigl\lvert
    u\bigl( t + ( T - t )U^0, x + W^0_{ ( T - t )U^0 } \bigr)
    -
    v\bigl( t + ( T - t )U^0, x + W^0_{ ( T - t )U^0 } \bigr)
\bigr\rvert
\\ &
=
L
\bigl\lvert
    ( T - t ) \cdot
    [ u - v ]\bigl( t + ( T - t )U^0, x + W^0_{ ( T - t )U^0 } \bigr)
\bigr\rvert
.
\end{split}
\end{equation}
Equation~\eqref{eq:definition_psi_k},
the fact that
$ \mathbf{U} $, $ U^0 $, $ \mathbf{W} $, and $ W^0 $ are independent,
\cite[Lemma~2.3]{HutzenthalerJentzenKruseNguyenVonWurstemberger2018arXivv2},
and
\cite[Lemma~2.10]{HutzenthalerJentzenKruseNguyenVonWurstemberger2018arXivv2}
(with
$ k = k $,
$ U = ( [ 0, T ] \times \R^d \times \Omega \ni ( t, x, \omega ) \mapsto
u( t, x ) - v( t, x )
\in \R ) $,
$ \mathfrak{r} = U^0 $,
$ \mathbb{W} = W^0 $
for $ k \in \N $
in the notation of~\cite[Lemma~2.10]{HutzenthalerJentzenKruseNguyenVonWurstemberger2018arXivv2})
hence
prove for all
$ k \in \N $
that
\begin{align*}
\label{eq:psi_k_Phi_1}
&
\E\bigl[ \lvert \psi_k(
        \Phi_{ 1 }(
        u,
        v,
        Z^{ 0 }
        )
) \rvert^2 \bigr]
=
\E \Bigl[
\tfrac{ \mathbf{U}^{ k - 1 } }{ ( k - 1 )! }
\lvert [ \Phi_{ 1 }( u, v, Z^{ 0 } ) ]( \mathbf{U} T, \xi + \mathbf{W}_{ \mathbf{U} T } ) \rvert^2
\Bigr]
\\ &
\leq
L^2 \,
\E\Bigl[
\tfrac{ \mathbf{U}^{ k - 1 } }{ ( k - 1 )! }
\bigl\lvert
    ( 1 - \mathbf{U} )T \cdot
    [ u - v ]\bigl(
        \mathbf{U} T + ( 1 - \mathbf{U} ) U^0 T,
        \xi + \mathbf{W}_{ \mathbf{U} T } + W^0_{ ( 1 - \mathbf{U} ) U^0 T }
    \bigr)
\bigr\rvert^2 \Bigr]
\\ &
=
L^2
\int_{0}^{1}
    \tfrac{ s^{ k - 1 } }{ ( k - 1 )! } \,
    \E\Bigl[
    \bigl\lvert
        ( 1 - s )T \cdot
        [ u - v ]\bigl(
            sT + ( 1 - s ) U^0 T,
            \xi + \mathbf{W}_{ sT } + W^0_{ ( 1 - s ) U^0 T }
        \bigr)
    \bigr\rvert^2 \Bigr]
\ud s
\\ &
=
\tfrac{L^2}{T^k}
\int_{0}^{T}
    \tfrac{ t^{ k - 1 } }{ ( k - 1 )! } \,
    \E\Bigl[
    \bigl\lvert
        ( T - t ) \cdot
        [ u - v ]\bigl(
            t + ( T - t ) U^0,
            \xi + \mathbf{W}_{ t } + W^0_{ ( T - t ) U^0 }
        \bigr)
    \bigr\rvert^2 \Bigr]
\ud t
\\ &
=
\tfrac{L^2}{T^k}
\int_{0}^{T}
    \tfrac{ t^{ k - 1 } }{ ( k - 1 )! } \,
    \E\Bigl[
    \bigl\lvert
        ( T - t ) \cdot
        [ u - v ]\bigl(
            t + ( T - t ) U^0,
            \xi + \mathbf{W}_{ t } + W^0_{ t + ( T - t ) U^0 } - W^0_{ t }
        \bigr)
    \bigr\rvert^2 \Bigr]
\ud t
\\ & \yesnumber
\leq
\tfrac{( LT )^2}{T^{ k + 1 } }
\int_{0}^{T}
    \tfrac{ t^{ k } }{ k! } \,
    \E\bigl[
    \lvert
        [ u - v ](
            t,
            \xi + \mathbf{W}_{ t }
        )
    \rvert^2 \bigr]
\ud t
\\ &
=
( LT )^2
\int_{0}^{1}
    \tfrac{ s^{ k } }{ k! } \,
    \E\bigl[
    \lvert
        [ u - v ](
            sT,
            \xi + \mathbf{W}_{ sT }
        )
    \rvert^2 \bigr]
\ud s
\\ &
=
( LT )^2 \,
\E \Bigl[
\tfrac{ \mathbf{U}^{ k } }{ k! }
\lvert [ u - v ]( \mathbf{U} T, \xi + \mathbf{W}_{ \mathbf{U} T } ) \rvert^2
\Bigr]
=
( LT )^2 \,
\E\bigl[ \lvert \psi_{ k + 1 } (
    u
    -
    v
) \rvert^2 \bigr]
.
\end{align*}
In addition,
\eqref{eq:definition_psi_k},
\eqref{eq:Phi_1_explicit}
and
the fact that
$ ( U^0, W^0 ) $ and $ ( \mathbf{U}, \mathbf{W} ) $ are identically distributed
ensure that
\begin{equation}
\begin{split}
&
\E\bigl[ \lvert \psi_0(
        \Phi_{ 1 }( u, v, Z^{ 0 } )
) \rvert^2 \bigr]
=
\E \bigl[
\lvert [ \Phi_{ 1 }( u, v, Z^{ 0 } ) ]( 0, \xi ) \rvert^2
\bigr]
\\ &
\leq
( LT )^2 \,
\E\bigl[ \lvert
    [ u - v ]( U^0 T, \xi + W^0_{ U^0 T } )
\rvert^2 \bigr]
\\ &
=
( LT )^2 \,
\E \bigl[
\lvert [ u - v ]( \mathbf{U} T, \xi + \mathbf{W}_{ \mathbf{U} T } ) \rvert^2
\bigr]
=
( LT )^2 \,
\E\bigl[ \lvert \psi_{ 1 } (
    u
    -
    v
) \rvert^2 \bigr]
.
\end{split}
\end{equation}
This, \eqref{eq:psi_k_Phi_1},
and the fact that
$ \forall \, n \in \N \colon \Phi_n = \Phi_1 $
complete the proof of Lemma~\ref{lem:hypothesis_II_heat}.
\end{proof}

\begin{lemma}
\label{lem:hypothesis_III_heat}
Assume Setting~\ref{setting:MLP_heat}.
Then
it holds
for all
$ k \in \N_0 $,
$ n, j \in \N $
that
\begin{equation}
\E\Biggl[ \biggl\lvert \psi_{ k }
\biggl(
    y -
    \smallsum_{l=0}^{n-1}
        \E\bigl[
            \Phi_{ l } \bigl(
            Y^0_{ l, j },
            Y^1_{ l - 1, j },
            Z^0
            \bigr)
        \bigr]
\biggr)
\biggr\rvert^2 \Biggr]
\leq
( LT )^2 \,
\E\Bigl[ \bigl\lvert \psi_{ k + 1 } \bigl(
    Y_{ n - 1, j }^0 - y
\bigr) \bigr\rvert^2 \Bigr]
.
\end{equation}
\end{lemma}
\begin{proof}[Proof of Lemma~\ref{lem:hypothesis_III_heat}]
Throughout this proof
let
$ \Psi_{ n, j } \colon
[ 0, T ] \times \R^d \to [ 0, \infty ) $,
$ j \in \N $,
$ n \in \N_0 $,
satisfy
for all
$ n, j \in \N $,
$ t \in [ 0, T ] $,
$ x \in \R^d $
that
\begin{equation}
\label{eq:definition_Psi}
\Psi_{ n -1, j }( t, x )
=
\E\Bigl[
\bigl\lvert
    ( T - t ) \cdot
    \bigl[ Y^0_{ n -1, j } - y \bigr]\bigl( t + ( T - t )U^0, x + W^0_{ ( T - t )U^0 } \bigr)
\bigr\rvert^2
\Bigr]
\end{equation}
(cf.\ Lemma~\ref{lem:expectation_finite_heat}).
To start with,
observe that
\eqref{eq:definition_Phi_heat},
(\ref{item:scheme_description_heat1})--(\ref{item:scheme_description_heat2}) in
Lemma~\ref{lem:scheme_description_heat},
(\ref{item:measurability_Phi_heat2}) in Lemma~\ref{lem:measurability_Phi_heat},
and
(\ref{item:independence3}) and (\ref{item:independence5})~in Proposition~\ref{prop:independence}
show
for all
$ l, j \in \N $,
$ t \in [ 0, T ] $,
$ x \in \R^d $
that
\begin{equation}
\begin{split}
&
\E\bigl[
    \bigl[ \Phi_{ l } \bigl(
    Y^0_{ l, j },
    Y^1_{ l - 1, j },
    Z^0
    \bigr) \bigr]
    ( T - t, x )
\bigr]
\\ &
=
t \,
\E \bigl[
f\bigl(
    T - t + U^0t, x + W^0_{ U^0t },
    Y^0_{ l, j }\bigl( T - t + U^0t, x + W^0_{ U^0 t } \bigr)
\bigr)
\\ & \quad
\phantom{
t \,
\E \bigl[
}
\mathop{-}
f\bigl(
    T - t + U^0t, x + W^0_{ U^0t },
    Y^1_{ l - 1, j }\bigl( T - t + U^0t, x + W^0_{ U^0 t } \bigr)
\bigr)
\bigr]
\\ &
=
t \,
\E \bigl[
f\bigl(
    T - t + U^0t, x + W^0_{ U^0t },
    Y^0_{ l, j }\bigl( T - t + U^0t, x + W^0_{ U^0 t } \bigr)
\bigr)
\\ & \quad
\phantom{
t \,
\E \bigl[
}
\mathop{-}
f\bigl(
    T - t + U^0t, x + W^0_{ U^0t },
    Y^0_{ l - 1, j }\bigl( T - t + U^0t, x + W^0_{ U^0 t } \bigr)
\bigr)
\bigr]
.
\end{split}
\end{equation}
Again~\eqref{eq:definition_Phi_heat}
hence
ensures
for all
$ n, j \in \N $,
$ t \in [ 0, T ] $,
$ x \in \R^d $
that
\begin{equation}
\label{eq:sum_Phi}
\begin{split}
&
\biggl[
\smallsum_{l=0}^{n-1}
\E\bigl[
    \Phi_{ l } \bigl(
    Y^0_{ l, j },
    Y^1_{ l - 1, j },
    Z^0
    \bigr)
\bigr]
\biggr]
( T - t, x )
\\ & =
\smallsum_{l=0}^{n-1}
\E\bigl[
    \bigl[ \Phi_{ l } \bigl(
    Y^0_{ l, j },
    Y^1_{ l - 1, j },
    Z^0
    \bigr) \bigr]
    ( T - t, x )
\bigr]
\\ &
=
\E[
g( x + W^0_{ t } )
]
+
t \,
\E \bigl[
f\bigl(
    T - t + U^0t, x + W^0_{ U^0t },
    Y^0_{ 0, j }\bigl( T - t + U^0t, x + W^0_{ U^0 t } \bigr)
\bigr)
\bigr]
\\ & \quad
+
t
\smallsum_{l=1}^{n-1}
\E \bigl[
f\bigl(
    T - t + U^0t, x + W^0_{ U^0t },
    Y^0_{ l, j }\bigl( T - t + U^0t, x + W^0_{ U^0 t } \bigr)
\bigr)
\\ & \quad
\hphantom{
+
t
\smallsum_{l=1}^{n-1}
\E \bigl[
}
\mathop{-}
f\bigl(
    T - t + U^0t, x + W^0_{ U^0t },
    Y^0_{ l - 1, j }\bigl( T - t + U^0t, x + W^0_{ U^0 t } \bigr)
\bigr)
\bigr]
\\ &
=
\E[
g( x + W^0_{ t } )
]
+
t \,
\E \bigl[
f\bigl(
    T - t + U^0t, x + W^0_{ U^0t },
    Y^0_{ n - 1, j }\bigl( T - t + U^0t, x + W^0_{ U^0 t } \bigr)
\bigr)
\bigr]
.
\end{split}
\end{equation}
In addition,
\eqref{eq:definition_y_heat},
the fact that
$ \mathbf{W} $ and $ W^0 $ are identically distributed,
the fact that
$ W^0 $ and $ U^0 $ are independent,
and
\cite[Lemma~2.4]{HutzenthalerJentzenKruseNguyenVonWurstemberger2018arXivv2}
(with
$ S = [ 0, 1 ] $,
$ U = ( [ 0, 1 ] \times \Omega \ni ( u, \omega ) \mapsto
f( t + ( T - t )u, x + W^0_{ ( T - t )u }( \omega ), y( t + ( T - t )u, x + W^0_{ ( T - t )u }( \omega ) ) )
\in \R ) $,
$ Y = U^0 $
for
$ x \in \R^d $,
$ t \in [ 0, T ] $
in the notation of~\cite[Lemma~2.4]{HutzenthalerJentzenKruseNguyenVonWurstemberger2018arXivv2})
imply
for all
$ t \in [ 0, T ] $,
$ x \in \R^d $
that
\begin{equation}
\begin{split}
& y( t, x )
=
\E[ g( x + \mathbf{W}_{ T - t } ) ]
+
\int_{t}^{T}
    \E\bigl[
    f\bigl(
        s, x + \mathbf{W}_{ s - t }, y( s, x + \mathbf{W}_{ s - t } )
    \bigr)
    \bigr]
\ud s
\\ &
=
\E[ g( x + W^0_{ T - t } ) ]
+
\int_{t}^{T}
    \E\bigl[
    f\bigl(
        s, x + W^0_{ s - t }, y( s, x + W^0_{ s - t } )
    \bigr)
    \bigr]
\ud s
\\ &
=
\E[ g( x + W^0_{ T - t } ) ]
\\ & \quad
+
( T - t )
\int_{0}^{1}
    \E\bigl[
    f\bigl(
        t + ( T - t )u, x + W^0_{ ( T - t )u }, y\bigl( t + ( T - t )u, x + W^0_{ ( T - t )u } \bigr)
    \bigr)
    \bigr]
\ud u
\\ &
=
\E[ g( x + W^0_{ T - t } ) ]
\\ & \quad
+
( T - t ) \,
\E\bigl[
    f\bigl(
        t + ( T - t )U^0, x + W^0_{ ( T - t )U^0 }, y\bigl( t + ( T - t )U^0, x + W^0_{ ( T - t )U^0 } \bigr)
    \bigr)
\bigr]
.
\end{split}
\end{equation}
This
and
\eqref{eq:sum_Phi}
demonstrate
for all
$ n, j \in \N $,
$ t \in [ 0, T ] $,
$ x \in \R^d $
that
\begin{align*}
&
\biggl\lvert
\biggl[
y
-
\textstyle\sum\limits_{l=0}^{n-1}
\E\bigl[
    \Phi_{ l } \bigl(
    Y^0_{ l, j },
    Y^1_{ l - 1, j },
    Z^0
    \bigr)
\bigr]
\biggr]
( t, x )
\biggr\rvert
\\ &
\leq
( T - t ) \,
\E\bigl[
\bigl\lvert
    f\bigl(
        t + ( T - t )U^0, x + W^0_{ ( T - t )U^0 },
        y\bigl( t + ( T - t )U^0, x + W^0_{ ( T - t )U^0 } \bigr)
    \bigr)
\\ & \quad
\hphantom{
( T - t ) \,
\E\bigl[
\bigl\lvert
}
    \mathop{-}
    f\bigl(
        t + ( T - t )U^0, x + W^0_{ ( T - t )U^0 },
        Y^0_{ n - 1, j }\bigl( t + ( T - t )U^0, x + W^0_{ ( T - t )U^0 } \bigr)
    \bigr)
\bigr\rvert
\bigr]
\\ &
\leq
L( T - t ) \,
\E\bigl[
\bigl\lvert
    y\bigl( t + ( T - t )U^0, x + W^0_{ ( T - t )U^0 } \bigr)
    -
    Y^0_{ n - 1, j }\bigl( t + ( T - t )U^0, x + W^0_{ ( T - t )U^0 } \bigr)
\bigr\rvert
\bigr]
\\ & \yesnumber
=
L ( T - t ) \,
\E\bigl[
\bigl\lvert
    \bigl[ Y^0_{ n - 1, j } - y \bigr]\bigl( t + ( T - t )U^0, x + W^0_{ ( T - t )U^0 } \bigr)
\bigr\rvert
\bigr]
.
\end{align*}
Jensen's inequality
and
\eqref{eq:definition_Psi}
hence
ensure
for all
$ n, j \in \N $,
$ t \in [ 0, T ] $,
$ x \in \R^d $
that
\begin{equation}
\label{eq:expectation_difference}
\begin{split}
&
\biggl\lvert
\biggl[
y
-
\smallsum_{l=0}^{n-1}
\E\bigl[
    \Phi_{ l } \bigl(
    Y^0_{ l, j },
    Y^1_{ l - 1, j },
    Z^0
    \bigr)
\bigr]
\biggr]
( t, x )
\biggr\rvert^2
\\ &
\leq
L^2 ( T - t )^2
\bigl(
\E\bigl[
\bigl\lvert
    \bigl[ Y^0_{ n - 1, j } - y \bigr]\bigl( t + ( T - t )U^0, x + W^0_{ ( T - t )U^0 } \bigr)
\bigr\rvert
\bigr]
\bigr)^2
\\ &
\leq
L^2 ( T - t )^2 \,
\E\Bigl[
\bigl\lvert
    \bigl[ Y^0_{ n - 1, j } - y \bigr]\bigl( t + ( T - t )U^0, x + W^0_{ ( T - t )U^0 } \bigr)
\bigr\rvert^2
\Bigr]
=
L^2
\Psi_{ n - 1, j }( t, x )
.
\end{split}
\end{equation}
Furthermore,
\eqref{eq:definition_Psi},
the fact that it holds
for every $ n, j \in \N $
that
$ \mathbf{W} $, $ Y^0_{ n - 1, j } $, $ U^0 $, and $ W^0 $ are independent
(cf.~Lemma~\ref{lem:scheme_description_heat}
and~(\ref{item:independence2})--(\ref{item:independence3})~in Proposition~\ref{prop:independence}),
and
\cite[Lemma~2.3]{HutzenthalerJentzenKruseNguyenVonWurstemberger2018arXivv2}
(with
$ S = \R^d $,
$ U = ( \R^d \times \Omega \ni ( w, \omega ) \mapsto
\lvert
    ( T - t ) \cdot
    [ Y^0_{ n - 1, j }( \omega ) - y ]( t + ( T - t )U^0( \omega ), \xi + w + W^0_{ ( T - t )U^0( \omega ) }( \omega ) )
\rvert^2
\in [ 0, \infty ) ) $,
$ Y = \mathbf{W}_t $
for
$ t \in [ 0, T ] $,
$ j, n \in \N $
in the notation of~\cite[Lemma~2.3]{HutzenthalerJentzenKruseNguyenVonWurstemberger2018arXivv2})
prove
for all
$ n, j \in \N $,
$ t \in [ 0, T ] $
that
\begin{align*}
\label{eq:double_expectation}
& \E[
\Psi_{ n - 1, j }
( t, \xi + \mathbf{W}_{ t } )
]
=
\int_{ \R^d }
\Psi_{ n - 1, j }
( t, \xi + w )
\, \bigl( \mathbf{W}_{ t } ( \P )_{ \mathscr{B}( \R^d ) } \bigr)( \mathrm{d} w )
\\ &
=
\int_{ \R^d }
\E\Bigl[
\bigl\lvert
    ( T - t ) \cdot
    \bigl[ Y^0_{ n - 1, j } - y \bigr]\bigl( t + ( T - t )U^0, \xi + w + W^0_{ ( T - t )U^0 } \bigr)
\bigr\rvert^2
\Bigr]
\, \bigl( \mathbf{W}_{ t } ( \P )_{ \mathscr{B}( \R^d ) } \bigr)( \mathrm{d} w )
\\ & \yesnumber
=
\E\Bigl[
\bigl\lvert
    ( T - t ) \cdot
    \bigl[ Y^0_{ n - 1, j } - y \bigr]\bigl( t + ( T - t )U^0, \xi + \mathbf{W}_{ t } + W^0_{ ( T - t )U^0 } \bigr)
\bigr\rvert^2
\Bigr]
\\ &
=
\E\Bigl[
\bigl\lvert
    ( T - t ) \cdot
    \bigl[ Y^0_{ n - 1, j } - y \bigr]\bigl( t + ( T - t )U^0, \xi + \mathbf{W}_{ t } + W^0_{ t + ( T - t )U^0 } - W^0_{ t } \bigr)
\bigr\rvert^2
\Bigr]
.
\end{align*}
Combining
\eqref{eq:definition_psi_k}
with
\eqref{eq:expectation_difference},
the fact that
$ \mathbf{U} $ and $ \mathbf{W} $ are independent,
\cite[Lemma~2.3]{HutzenthalerJentzenKruseNguyenVonWurstemberger2018arXivv2}
(with
$ S = [ 0, 1 ] $,
$ U = ( [ 0, 1 ] \times \Omega \ni ( s, \omega ) \mapsto
\nicefrac{ s^{ k - 1 } }{ ( k - 1 )! } \,
\Psi_{ n - 1, j }
( s T, \xi + \mathbf{W}_{ s T }( \omega ) )
\in [ 0, \infty ) ) $,
$ Y = \mathbf{U} $
for $ j, n, k \in \N $
in the notation of~\cite[Lemma~2.3]{HutzenthalerJentzenKruseNguyenVonWurstemberger2018arXivv2}),
\eqref{eq:double_expectation},
again
the fact that it holds
for every $ n, j \in \N $
that
$ \mathbf{W} $, $ Y^0_{ n - 1, j } $, $ U^0 $, and $ W^0 $ are independent,
and
\cite[Lemma~2.10]{HutzenthalerJentzenKruseNguyenVonWurstemberger2018arXivv2}
(with
$ k = k $,
$ U = ( [ 0, T ] \times \R^d \times \Omega \ni ( t, x, \omega ) \mapsto
[ Y^0_{ n - 1, j }( \omega ) ]( t, x ) - y( t, x )
\in \R ) $,
$ \mathfrak{r} = U^0 $,
$ \mathbb{W} = W^0 $
for $ j, n, k \in \N $
in the notation of~\cite[Lemma~2.10]{HutzenthalerJentzenKruseNguyenVonWurstemberger2018arXivv2})
establishes
for all
$ k, n, j \in \N $
that
\begin{align*}
\label{eq:psi_k_sum_Phi}
&
\E\Biggl[ \biggl\lvert \psi_{ k }
\biggl(
    y -
    \smallsum_{l=0}^{n-1}
        \E\bigl[
            \Phi_{ l } \bigl(
            Y^0_{ l, j },
            Y^1_{ l - 1, j },
            Z^0
            \bigr)
        \bigr]
\biggr)
\biggr\rvert^2 \Biggr]
\\ &
=
\E \Biggl[
\tfrac{ \mathbf{U}^{ k - 1 } }{ ( k - 1 )! }
\biggl\lvert
\biggl[
    y -
    \smallsum_{l=0}^{n-1}
        \E\bigl[
            \Phi_{ l } \bigl(
            Y^0_{ l, j },
            Y^1_{ l - 1, j },
            Z^0
            \bigr)
        \bigr]
\biggr]
( \mathbf{U} T, \xi + \mathbf{W}_{ \mathbf{U} T } )
\biggr\rvert^2
\Biggr]
\\ & \yesnumber
\leq
L^2 \,
\E\Bigl[
\tfrac{ \mathbf{U}^{ k - 1 } }{ ( k - 1 )! }
\Psi_{ n - 1, j }
( \mathbf{U} T, \xi + \mathbf{W}_{ \mathbf{U} T } )
\Bigr]
=
L^2
\int_{0}^{1}
\tfrac{ s^{ k - 1 } }{ ( k - 1 )! } \,
\E[
\Psi_{ n - 1, j }
( s T, \xi + \mathbf{W}_{ s T } )
]
\ud s
\\ &
=
\tfrac{L^2}{T^k}
\int_{0}^{T}
\tfrac{ t^{ k - 1 } }{ ( k - 1 )! } \,
\E[
\Psi_{ n - 1, j }
( t, \xi + \mathbf{W}_{ t } )
]
\ud t
\\ &
=
\tfrac{L^2}{T^k}
\int_{0}^{T}
\tfrac{ t^{ k - 1 } }{ ( k - 1 )! } \,
    \E\Bigl[
    \bigl\lvert
        ( T - t ) \cdot
        \bigl[ Y^0_{ n - 1, j } - y \bigr]\bigl( t + ( T - t )U^0, \xi + \mathbf{W}_{ t } + W^0_{ t + ( T - t )U^0 } - W^0_{ t } \bigr)
    \bigr\rvert^2
    \Bigr]
\ud t
\\ &
\leq
\tfrac{( LT )^2}{T^{ k + 1 } }
\int_{0}^{T}
    \tfrac{ t^{ k } }{ k! } \,
    \E\Bigl[
    \bigl\lvert
        \bigl[ Y^0_{ n - 1, j } - y \bigr](
            t,
            \xi + \mathbf{W}_{ t }
        )
    \bigr\rvert^2 \Bigr]
\ud t
.
\end{align*}
This,
the fact that it holds
for every $ n, j \in \N $
that
$ Y^0_{ n - 1, j } $, $ \mathbf{W} $, and $ \mathbf{U} $ are independent,
\cite[Lemma~2.3]{HutzenthalerJentzenKruseNguyenVonWurstemberger2018arXivv2}
(with
$ S = [ 0, 1 ] $,
$ U = ( [ 0, 1 ] \times \Omega \ni ( s, \omega ) \mapsto
\nicefrac{ s^k }{ k! }
\lvert
    [ Y^0_{ n - 1, j }( \omega ) - y ]( sT, \xi + \mathbf{W}_{ sT }( \omega ) )
\rvert^2
\in [ 0, \infty ) ) $,
$ Y = \mathbf{U} $
for $ j, n, k \in \N $
in the notation of~\cite[Lemma~2.3]{HutzenthalerJentzenKruseNguyenVonWurstemberger2018arXivv2}),
and
\eqref{eq:definition_psi_k}
show
for all
$ k, n, j \in \N $
that
\begin{equation}
\begin{split}
&
\E\Biggl[ \biggl\lvert \psi_{ k }
\biggl(
    y -
    \smallsum_{l=0}^{n-1}
        \E\bigl[
            \Phi_{ l } \bigl(
            Y^0_{ l, j },
            Y^1_{ l - 1, j },
            Z^0
            \bigr)
        \bigr]
\biggr)
\biggr\rvert^2 \Biggr]
\\ &
\leq
( LT )^2
\int_{0}^{1}
    \tfrac{ s^{ k } }{ k! } \,
    \E\Bigl[
    \bigl\lvert
        \bigl[ Y^0_{ n - 1, j } - y \bigr](
            sT,
            \xi + \mathbf{W}_{ sT }
        )
    \bigr\rvert^2 \Bigr]
\ud s
\\ &
=
( LT )^2 \,
\E \Bigl[
\tfrac{ \mathbf{U}^{ k } }{ k! }
\bigl\lvert \bigl[ Y^0_{ n - 1, j } - y \bigr]( \mathbf{U} T, \xi + \mathbf{W}_{ \mathbf{U} T } ) \bigr\rvert^2
\Bigr]
=
( LT )^2 \,
\E\Bigl[ \bigl\lvert \psi_{ k + 1 } \bigl(
    Y_{ n - 1, j }^0 - y
\bigr) \bigr\rvert^2 \Bigr]
.
\end{split}
\end{equation}
Moreover,
\eqref{eq:definition_psi_k},
\eqref{eq:expectation_difference},
and
the fact that it holds
for all $ n, j \in \N $
that
$ ( Y^0_{ n - 1, j }, U^0, W^0 ) $ and $ ( Y^0_{ n - 1, j }, \mathbf{U}, \mathbf{W} ) $ are identically distributed
demonstrate
for all
$ n, j \in \N $
that
\begin{align*}
&
\E\Biggl[ \biggl\lvert \psi_{ 0 }
\biggl(
    y -
    \smallsum_{l=0}^{n-1}
        \E\bigl[
            \Phi_{ l } \bigl(
            Y^0_{ l, j },
            Y^1_{ l - 1, j },
            Z^0
            \bigr)
        \bigr]
\biggr)
\biggr\rvert^2 \Biggr]
=
\biggl\lvert
\biggl[
    y -
    \smallsum_{l=0}^{n-1}
        \E\bigl[
            \Phi_{ l } \bigl(
            Y^0_{ l, j },
            Y^1_{ l - 1, j },
            Z^0
            \bigr)
        \bigr]
\biggr]
( 0, \xi )
\biggr\rvert^2
\\ & \yesnumber
\leq
( LT )^2 \,
\E\Bigl[
\bigl\lvert
    \bigl[ Y^0_{ n - 1, j } - y \bigr]\bigl( U^0 T, \xi + W^0_{ U^0 T } \bigr)
\bigr\rvert^2
\Bigr]
\\ &
=
( LT )^2 \,
\E\Bigl[
\bigl\lvert
    \bigl[ Y^0_{ n - 1, j } - y \bigr]\bigl( \mathbf{U} T, \xi + \mathbf{W}_{ \mathbf{U} T } \bigr)
\bigr\rvert^2
\Bigr]
=
( LT )^2 \,
\E\Bigl[ \bigl\lvert \psi_{ 1 } \bigl(
    Y_{ n - 1, j }^0 - y
\bigr) \bigr\rvert^2 \Bigr]
.
\end{align*}
The proof of Lemma~\ref{lem:hypothesis_III_heat} is thus complete.
\end{proof}

\subsection{Complexity analysis}
\label{sec:complexity_analysis_heat}

\subsubsection{MLP approximations in fixed space dimensions}
\label{sec:MLP_heat_fixed}

\begin{proposition}
\label{prop:MLP_heat_fixed}
Let
$ d \in \N $,
$ \xi \in \R^d $,
$ T \in ( 0, \infty ) $,
$ L, p, \mathfrak{B}, \kappa, C \in [ 0, \infty ) $,
$ \Theta = \cup_{ n = 1 }^\infty \Z^n $,
$ ( M_j )_{ j \in \N } \subseteq \N $
satisfy
$ \liminf_{ j \to \infty } M_j = \infty $,
$ \sup_{ j \in \N }
\nicefrac{ M_{ j + 1 } }{ M_j } \leq \mathfrak{B} $,
and
$ \sup_{ j \in \N }
\nicefrac{ M_j }{ j } \leq \kappa $,
let
$ f \in C( [ 0, T ] \times \R^d \times \R, \R ) $,
$ g \in C( \R^d, \R ) $
satisfy for all
$ t \in [ 0, T ] $,
$ x \in \R^d $,
$ v, w \in \R $
that
$ \max\{ \lvert f( t, x, 0 ) \rvert, \lvert g( x ) \rvert \}
\leq
L \max\{ 1, \lVert x \rVert^p_{ \R^d } \} $
and
$ \lvert f( t, x, v ) - f( t, x, w ) \rvert
\leq
L \lvert v - w \rvert $,
let
$ ( \Omega, \mathscr{F}, \P ) $
be a probability space,
let
$ U^\theta \colon \Omega \to [ 0, 1 ] $, $ \theta \in \Theta $,
be independent on $ [ 0, 1 ] $ uniformly distributed random variables,
let
$ W^\theta \colon [ 0, T ] \times \Omega \to \R^d $, $ \theta \in \Theta $,
be independent standard Brownian motions with continuous sample paths,
assume that
$ ( U^\theta )_ { \theta \in \Theta } $
and
$ ( W^\theta )_ { \theta \in \Theta } $
are independent,
assume that
\begin{equation}
\label{eq:definition_C_MLP_heat_fixed}
C
=
\max\Bigl\{
1,
e^{ LT }
\Bigl[
    \bigl( \E \bigl[
    \lvert g( \xi + W_{ T }^0 ) \rvert^2
    \bigr] \bigr)^{ \nicefrac{1}{2} }
    +
    \sqrt{T}
    \bigl(
    \textstyle\int_{0}^{T}
        \E \bigl[
        \lvert f( t, \xi + W_{ t }^0, 0 ) \rvert^2
        \bigr]
    \ud t
    \bigr)^{ \nicefrac{1}{2} }
\Bigr]
\Bigr\}
,
\end{equation}
let
$ Y_{ n, j }^\theta \colon [ 0, T ] \times \R^d \times \Omega \to \R $,
$ \theta \in \Theta $,
$  j \in \N $,
$ n \in ( \N_0 \cup \{ -1 \} ) $,
satisfy
for all
$ n, j \in \N $,
$ \theta \in \Theta $,
$ t \in [ 0, T ] $,
$ x \in \R^d $
that
$ Y_{ -1, j }^\theta( t, x ) = Y_{ 0, j }^\theta( t, x ) = 0 $
and
\begin{align*}
\label{eq:definition_Y_MLP_heat_fixed}
& Y_{ n, j }^\theta( T - t, x )
=
\tfrac{1}{ ( M_j )^{ n } }
\Biggl[
\sum_{ i = 1 }^{ (M_j )^n }
    g\bigl( x + W^{ ( \theta, 0, i ) }_{ t } \bigr)
\Biggr]
+
\sum_{l=0}^{n-1}
\tfrac{t}{ ( M_j )^{ n - l } }
\Biggl[
    \sum_{i=1}^{ ( M_j )^{ n - l } }
\\ & \yesnumber
\quad
    \Bigl[
        f\Bigl(
            T - t + U^{ ( \theta, l, i ) } t,
            x + W^{ ( \theta, l, i ) }_{ U^{ ( \theta, l, i ) } t },
            Y^{ ( \theta, l, i ) }_{ l, j }\bigl(
                T - t + U^{ ( \theta, l, i ) } t,
                x + W^{ ( \theta, l, i ) }_{ U^{ ( \theta, l, i ) } t }
            \bigr)
        \Bigr)
\\ &
\quad
        -
        \mathbbm{1}_{ \N }( l )
        f\Bigl(
            T - t + U^{ ( \theta, l, i ) } t,
            x + W^{ ( \theta, l, i ) }_{ U^{ ( \theta, l, i ) } t },
            Y^{ ( \theta, -l, i ) }_{ l - 1, j }\bigl(
                T - t + U^{ ( \theta, l, i ) } t,
                x + W^{ ( \theta, l, i ) }_{ U^{ ( \theta, l, i ) } t }
            \bigr)
        \Bigr)
    \Bigr]
\Biggr]
,
\end{align*}
and
let
$ ( \Cost_{ n, j } )_{ ( n, j ) \in ( \N_0 \cup \{ -1 \} ) \times \N }
\subseteq \N_0 $
satisfy
for all $ n, j \in \N $
that
$ \Cost_{ -1, j } = \Cost_{ 0, j } = 0 $
and
\begin{equation}
\label{eq:cost_MLP_heat_fixed}
\Cost_{ n, j }
\leq
( M_j )^{ n }
d
+
\sum_{ l = 0 }^{ n - 1 }
    \bigl[
    ( M_j )^{ n - l }
    ( \Cost_{ l, j } + \Cost_{ l - 1, j } + d + 1 )
    \bigr]
.
\end{equation}
Then
\begin{enumerate}[(i)]
\item
\label{item:prop:MLP_heat_fixed1}
there exists a unique at most polynomially growing viscosity solution
$ y \in C( [ 0, T ] \times \R^d, \R ) $
of
\begin{equation}
\bigl( \tfrac{ \partial y }{ \partial t } \bigr)( t, x )
+
\tfrac{1}{2}
( \Delta_x y )( t, x )
+
f( t, x, y( t, x ) )
= 0
\end{equation}
with
$ y( T, x ) = g( x ) $
for $ ( t, x ) \in ( 0, T ) \times \R^d $,
\item
\label{item:prop:MLP_heat_fixed2}
it holds
for all
$ n \in \N $
that
\begin{equation}
\bigl(
\E\bigl[
    \lvert Y_{ n, n }^0( 0, \xi ) - y( 0, \xi ) \rvert^2
\bigr]
\bigr)^{ \nicefrac{1}{2} }
\leq
C
\biggl[
    \frac{ e^\kappa ( 1 + ( 2 L T)^2 ) }{ M_n }
\biggr]^{ \nicefrac{n}{2} }
< \infty
,
\end{equation}
\item
\label{item:prop:MLP_heat_fixed3}
it holds
for all
$ n \in \N $
that
$ \Cost_{ n, n }
\leq
( 5 M_n )^{ n } d
$,
and
\item
\label{item:prop:MLP_heat_fixed4}
there exists
$ ( N_\varepsilon )_{ \varepsilon \in ( 0, 1 ] } \subseteq \N $
such that
it holds
for all
$ \varepsilon \in ( 0, 1 ] $,
$ \delta \in ( 0, \infty ) $
that
$ \sup_{ n \in \{ N_\varepsilon, N_\varepsilon + 1, \ldots \} }
\bigl(
\E\bigl[
    \lvert Y_{ n, n }^0( 0, \xi ) - y( 0, \xi ) \rvert^2
\bigr]
\bigr)^{ \nicefrac{1}{2} }
\leq
\varepsilon $
and
\begin{equation}
\Cost_{ N_\varepsilon, N_\varepsilon }
\leq
5 d
e^\kappa
C^{ 2 ( 1 + \delta ) }
\Bigl(
    1
    +
    \sup\nolimits_{ n \in \N }
    \Bigl[
        \tfrac{
        [ 5 \mathfrak{B} e^{ 2 \kappa } ( 1 + ( 2 L T )^2 ) ]^{ ( 1 + \delta ) }
        }{
        ( M_{ n } )^{ \delta }
        }
    \Bigr]^n
\Bigr)
\varepsilon^{ -2 ( 1 + \delta ) }
<
\infty
.
\end{equation}
\end{enumerate}
\end{proposition}
\begin{proof}[Proof of Proposition~\ref{prop:MLP_heat_fixed}]
Throughout this proof
assume w.l.o.g.\
that $ L > 0 $,
assume w.l.o.g.\
that there exist
an on $ [ 0, 1 ] $ uniformly distributed random variable
$ \mathbf{U} \colon \Omega \to [ 0, 1 ] $
and a standard Brownian motion
$ \mathbf{W} \colon [ 0, T ] \times \Omega \to \R^d $
with continuous sample paths
such that
$ \mathbf{U} $,
$ \mathbf{W} $,
$ ( U^\theta )_ { \theta \in \Theta } $,
and
$ ( W^\theta )_ { \theta \in \Theta } $
are independent,
let
$ \mathfrak{z}, \gamma \in [ 0, \infty ) $,
$ c \in ( 0, \infty ) $,
$ \mathfrak{y}_{ -1 }, \mathfrak{y}_{ 0 } \in C( [ 0, T ] \times \R^d, \R ) $
be given by
$ \mathfrak{z} = d $,
$ \gamma = 2 $,
$ c = ( L T )^2 $,
and
$ \mathfrak{y}_{ -1 } = \mathfrak{y}_{ 0 } = 0 $,
let
$ q \in ( p, \infty ) $,
let
$ \mathcal{Y} \subseteq C( [ 0, T ] \times \R^d, \R ) $
be the set given by
\begin{equation}
\mathcal{Y}
=
\Biggl\{
    v \in C( [ 0, T ] \times \R^d, \R ) \colon
    \limsup_{ \N \ni n \to \infty }
    \sup_{ ( t, x ) \in [ 0, T ] \times \R^d, \, \lVert x \rVert_{ \R^d } \geq n }
    \frac{ \lvert v( t, x ) \rvert }{ \lVert x \rVert_{ \R^d }^q }
    =
    0
\Biggr\}
,
\end{equation}
let
$ \lVert \cdot \rVert_{ \mathcal{Y} } \colon \mathcal{Y} \to [ 0, \infty ) $
satisfy
for all
$ v \in \mathcal{Y} $
that
\begin{equation}
\lVert v \rVert_{ \mathcal{Y} }
=
\sup_{ ( t, x ) \in [ 0, T ] \times \R^d }
\biggl[
\frac{ \lvert v( t, x ) \rvert }{ \max\{ 1, \lVert x \rVert_{ \R^d }^q \} }
\biggr]
,
\end{equation}
let
$ ( \mathcal{Z}, \mathscr{Z} )
=
\bigl( [ 0, 1 ] \times C( [ 0, T ], \R^d ),
\mathscr{ B }( [ 0, 1 ] ) \otimes \mathscr{ B }( C( [ 0, T ], \R^d ) )
\bigr) $,
let
$ Z^{ \theta } \colon \Omega \to \mathcal{Z} $,
$ \theta \in \Theta $,
satisfy
for all
$ \theta \in \Theta $
that
$ Z^{ \theta } = ( U^\theta, W^\theta ) $,
let
$ ( \mathcal{H}, \langle \cdot, \cdot \rangle_{ \mathcal{H} }, \lVert \cdot \rVert_{ \mathcal{H} } )
=
( \R, \langle \cdot, \cdot \rangle_{ \R }, \lvert \cdot \rvert ) $,
let
%$ \mathscr{S} \subseteq \{ \mathcal{S} \colon \mathcal{S} \subseteq L( \mathcal{Y}, \mathcal{H} ) \} $
%be the set given by
$ \mathscr{S}
=
\sigma_{ L( \mathcal{Y}, \mathcal{H} ) } \bigl( \bigl\{
    \{ \varphi \in L( \mathcal{Y}, \mathcal{H} ) \colon \varphi( v ) \in \mathcal{B} \} \subseteq L( \mathcal{Y}, \mathcal{H} ) \colon
    v \in \mathcal{Y},\, \mathcal{ B } \in \mathscr{B}( \mathcal{H} )
\bigr\} \bigr) $,
let
$ \psi_k \colon \Omega \to L( \mathcal{Y}, \mathcal{H} ) $, $ k \in \N_0 $,
satisfy
for all
$ k \in \N_0 $,
$ \omega \in \Omega $,
$ v \in \mathcal{Y} $
that
\begin{equation}
[ \psi_k( \omega ) ] ( v )
=
\begin{cases}
v( 0, \xi )
& \colon k = 0
\\
\sqrt{ \frac{ ( \mathbf{U}( \omega ) )^{ k -1 } }{ ( k - 1 )! } }
v\bigl( \mathbf{U}( \omega ) T, \xi + \mathbf{W}_{ \mathbf{U}( \omega ) T }( \omega ) \bigr)
& \colon k \in \N
\end{cases}
,
\end{equation}
and
let
$ \Phi_l \colon \mathcal{Y} \times \mathcal{Y} \times \mathcal{Z} \to \mathcal{Y} $, $ l \in \N_0 $,
satisfy
for all
$ l \in \N_0 $,
$ v, w \in \mathcal{Y} $,
$ \mathfrak{z} = ( \mathfrak{u}, \mathfrak{w} ) \in \mathcal{Z} $,
$ t \in [ 0, T ] $,
$ x \in \R^d $
that
\begin{equation}
\begin{split}
& [ \Phi_l( v, w, \mathfrak{z} ) ]( T - t, x )
\\ & =
\begin{cases}
g( x + \mathfrak{w}_{ t } )
+
t
f\bigl(
    T - t + \mathfrak{u}t, x + \mathfrak{w}_{ \mathfrak{u}t },
    v( T - t + \mathfrak{u}t, x + \mathfrak{w}_{ \mathfrak{u} t } )
\bigr)
& \colon l = 0
\\
\!\!
\begin{array}{l}
t \bigl[
f\bigl(
    T - t + \mathfrak{u}t, x + \mathfrak{w}_{ \mathfrak{u}t },
    v( T - t + \mathfrak{u}t, x + \mathfrak{w}_{ \mathfrak{u} t } )
\bigr)
\\
\quad
\mathop{-}
f\bigl(
    T - t + \mathfrak{u}t, x + \mathfrak{w}_{ \mathfrak{u}t },
    w( T - t + \mathfrak{u}t, x + \mathfrak{w}_{ \mathfrak{u} t } )
\bigr)
\bigr]
\end{array}
& \colon l \in \N
\end{cases}
\end{split}
\end{equation}
(cf.\ Lemma~\ref{lem:slower_growth} and Corollary~\ref{cor:well_defined_Phi}).
Note that the assumption that
$ \forall \, t \in [ 0, T ],
\, x \in \R^d,
\, v, w \in \R \colon
\bigl(
\max\{ \lvert f( t, x, 0 ) \rvert, \lvert g( x ) \rvert \}
\leq
L \max\{ 1, \lVert x \rVert^p_{ \R^d } \} $
and
$ \lvert f( t, x, v ) - f( t, x, w ) \rvert
\leq
L \lvert v - w \rvert
\bigr) $
ensures that
there exists a unique at most polynomially growing
viscosity solution
$ y \in C( [ 0, T ] \times \R^d, \R ) $
of
\begin{equation}
\bigl( \tfrac{ \partial y }{ \partial t } \bigr)( t, x )
+
\tfrac{1}{2}
( \Delta_x y )( t, x )
+
f( t, x, y( t, x ) )
= 0
\end{equation}
with
$ y( T, x ) = g( x ) $
for $ ( t, x ) \in ( 0, T ) \times \R^d $
(cf., e.g., Hairer, Hutzenthaler, \& Jentzen~\cite[Section~4]{HairerHutzenthalerJentzen2015},
Hutzenthaler et al.~\cite[Corollary~3.11]{HutzenthalerJentzenKruseNguyenVonWurstemberger2018arXivv2},
and
Beck et al.~\cite[Theorem~1.1]{BeckGononHutzenthalerJentzen2019arXiv}).
This shows~(\ref{item:prop:MLP_heat_fixed1}).
Moreover,
the Feynman--Kac formula proves
for all
$ t \in [ 0, T ] $,
$ x \in \R^d $
that
\begin{equation}
\label{eq:Feynman-Kac_solution}
y( t, x )
=
\E\biggl[
    g( x + \mathbf{W}_{ T - t } )
    +
    \int_{t}^{T}
        f\bigl(
            s, x + \mathbf{W}_{ s - t }, y( s, x + \mathbf{W}_{ s - t } )
        \bigr)
    \ud s
\biggr]
\end{equation}
(cf., e.g.,~\cite[Section~4]{HairerHutzenthalerJentzen2015},
\cite[Corollary~3.11]{HutzenthalerJentzenKruseNguyenVonWurstemberger2018arXivv2},
and
\cite[Theorem~1.1]{BeckGononHutzenthalerJentzen2019arXiv}).
Combining this
with \cite[Corollary~3.11]{HutzenthalerJentzenKruseNguyenVonWurstemberger2018arXivv2}
demonstrates that
\begin{equation}
\label{eq:polynomial_growth_solution}
\sup_{ ( t, x ) \in [ 0, T ] \times \R^d }
\biggl[
\frac{ \lvert y( t, x ) \rvert }{ \max\{ 1, \lVert x \rVert_{ \R^d }^p \} }
\biggr]
< \infty
.
\end{equation}
Next observe that
\begin{itemize}
\item
it holds that
$ ( \mathcal{Y}, \lVert \cdot \rVert_{ \mathcal{Y} } ) $
is a separable $ \R $-Banach space
(cf.~(\ref{item:Y_separable2}) in Proposition~\ref{prop:Y_separable}),
\item
it holds that
$ \min\{ \mathfrak{B}, \kappa, C \} \geq 1 $,
$ y \in \mathcal{Y} $
(cf.~\eqref{eq:polynomial_growth_solution} and Lemma~\ref{lem:slower_growth}),
and
$ \mathfrak{y}_{ -1 }, \mathfrak{y}_{ 0 } \in \mathcal{Y} $,
\item
it holds that
$ ( \mathcal{Z}, \mathscr{Z} ) $ is a measurable space,
\item
it holds that
$ Z^{ \theta } \colon \Omega \to \mathcal{Z} $,
$ \theta \in \Theta $,
are i.i.d.\
$ \mathscr{F} $/$ \mathscr{Z} $-measurable functions,
\item
it holds that
$ ( \mathcal{H}, \langle \cdot, \cdot \rangle_{ \mathcal{H} }, \lVert \cdot \rVert_{ \mathcal{H} } ) $
is a separable $ \R $-Hilbert space,
\item
it holds that
$ \psi_k \colon \Omega \to L( \mathcal{Y}, \mathcal{H} ) $, $ k \in \N_0 $,
are
$ \mathscr{F} $/$ \mathscr{S} $-measurable functions
(cf.~Lem\-ma~\ref{lem:measurability_psi_heat}),
\item
it holds that
$ ( Z^{ \theta } )_{ \theta \in \Theta } $
and
$ ( \psi_k )_{ k \in \N_0 } $
are independent,
\item
it holds that
$ \Phi_l \colon \mathcal{Y} \times \mathcal{Y} \times \mathcal{Z} \to \mathcal{Y} $, $ l \in \N_0 $,
are
$ ( \mathscr{B}( \mathcal{Y} ) \otimes \mathscr{B}( \mathcal{Y} ) \otimes \mathscr{Z} ) $/$ \mathscr{B}( \mathcal{Y} ) $-measurable
functions
(cf.~(\ref{item:measurability_Phi_heat2}) in Lemma~\ref{lem:measurability_Phi_heat}),
\item
it holds
for all
$ n \in ( \N_0 \cup \{ -1 \} ) $,
$ j \in \N $,
$ \theta \in \Theta $
that
$ Y_{ n, j }^\theta( \Omega ) \subseteq \mathcal{Y} $
(cf.~assumption~\eqref{eq:definition_Y_MLP_heat_fixed}
and (\ref{item:scheme_description_heat1}) in Lemma~\ref{lem:scheme_description_heat}),
\item
it holds
for all
$ n, j \in \N $,
$ \theta \in \Theta $
that
$ Y_{ -1, j }^\theta = \mathfrak{y}_{ -1 } $,
$ Y_{ 0, j }^\theta = \mathfrak{y}_{ 0 } $,
and
\begin{equation}
Y_{ n, j }^\theta
=
\sum_{l=0}^{n-1}
\tfrac{1}{ ( M_j )^{ n - l } }
\Biggl[
    \sum_{i=1}^{ ( M_j )^{ n - l } }
        \Phi_{ l } \bigl(
        Y^{ ( \theta, l, i ) }_{ l, j },
        Y^{ ( \theta, -l, i ) }_{ l - 1, j },
        Z^{ ( \theta, l, i ) }
        \bigr)
\Biggr]
\end{equation}
(cf.~assumption~\eqref{eq:definition_Y_MLP_heat_fixed}
and (\ref{item:scheme_description_heat2}) in Lemma~\ref{lem:scheme_description_heat}),
\item
it holds
for all $ n, j \in \N $
that
$ \Cost_{ -1, j } = \Cost_{ 0, j } = 0 $
and
\begin{equation}
\begin{split}
\Cost_{ n, j }
&
\leq
( M_j )^{ n }
d
+
\sum_{ l = 0 }^{ n - 1 }
    \bigl[
    ( M_j )^{ n - l }
    ( \Cost_{ l, j } + \Cost_{ l - 1, j } + d + 1 )
    \bigr]
\\
& \leq
( M_j )^{ n }
\mathfrak{z}
+
\sum_{ l = 0 }^{ n - 1 }
    \bigl[
    ( M_j )^{ n - l }
    ( \Cost_{ l, j } + \Cost_{ l - 1, j } + \gamma \mathfrak{z} )
    \bigr]
\end{split}
\end{equation}
(cf.~assumption~\eqref{eq:cost_MLP_heat_fixed}),
\item
it holds
for all
$ k \in \N_0 $,
$ j \in \N $
that
$ \E\bigl[
    \lVert 
        \Phi_{ k }(
        Y^{ 0 }_{ k, j },
        Y^{ 1 }_{ k - 1, j },
        Z^{ 0 }
        )
    \rVert_{ \mathcal{Y} }
\bigr]
< \infty $
(cf.~Lemma~\ref{lem:expectation_finite_heat}),
and
\item
it holds
for all
$ k \in \N_0 $,
$ n, j \in \N $,
$ u, v \in \mathcal{Y} $
that
\begin{gather}
\max\bigl\{
\E\bigl[
\lVert \psi_k(
    \Phi_{ 0 }(
    \mathfrak{y}_{ 0 },
    \mathfrak{y}_{ -1 },
    Z^{ 0 }
    )
) \rVert_{ \mathcal{H} }^2
\bigr]
,
\mathbbm{1}_{ \N }( k )
\, \E\bigl[
\lVert \psi_k(
    \mathfrak{y}_{ 0 } - y
) \rVert_{ \mathcal{H} }^2
\bigr]
\bigr\}
\leq
\tfrac{ C^2 }{ k! }
,
\\[0.5\baselineskip]
\E\bigl[
\lVert \psi_k(
    \Phi_{ n }( u, v, Z^{ 0 } )
) \rVert_{ \mathcal{H} }^2
\bigr]
\leq
c
\, \E\bigl[
\lVert \psi_{ k + 1 }(
    u - v
) \rVert_{ \mathcal{H} }^2
\bigr]
,
\\[0.5\baselineskip]
\E\Biggl[
\biggl\lVert \psi_{ k }
\biggl(
    y -
    \smallsum_{l=0}^{n-1}
        \E\bigl[
            \Phi_{ l }\bigl(
            Y^0_{ l, j },
            Y^1_{ l - 1, j },
            Z^0
            \bigr)
        \bigr]
\biggr)
\biggr\rVert_{ \mathcal{H} }^2
\Biggr]
\leq
2 c
\, \E\Bigl[
\bigl\lVert \psi_{ k + 1 }\bigl(
    Y_{ n - 1, j }^0 - y
\bigr) \bigr\rVert_{ \mathcal{H} }^2
\Bigr]
\end{gather}
(cf.~\eqref{eq:Feynman-Kac_solution},
assumption~\eqref{eq:definition_C_MLP_heat_fixed},
Lemma~\ref{lem:hypothesis_I_heat},
Lemma~\ref{lem:hypothesis_II_heat},
and
Lemma~\ref{lem:hypothesis_III_heat}).
\end{itemize}
Corollary~\ref{cor:complexity_analysis}
hence
establishes~(\ref{item:prop:MLP_heat_fixed2})--(\ref{item:prop:MLP_heat_fixed4}).
The proof of Proposition~\ref{prop:MLP_heat_fixed} is thus complete.
\end{proof}

\subsubsection{MLP approximations in variable space dimensions}
\label{sec:MLP_heat_variable}

\begin{theorem}
\label{thm:MLP_heat}
Let
$ T \in ( 0, \infty ) $,
$ K, L, p, \mathfrak{B}, \kappa \in [ 0, \infty ) $,
$ \Theta = \cup_{ n = 1 }^\infty \Z^n $,
$ ( M_j )_{ j \in \N } \subseteq \N $
satisfy
$ \liminf_{ j \to \infty } M_j = \infty $,
$ \sup_{ j \in \N }
\nicefrac{ M_{ j + 1 } }{ M_j } \leq \mathfrak{B} $,
and
$ \sup_{ j \in \N }
\nicefrac{ M_j }{ j } \leq \kappa $,
let
$ \xi_d \in \R^d $, $ d \in \N $,
satisfy
$ \sup_{ d \in \N } \lVert \xi_d \rVert_{ \R^d } \leq K $,
for every $ d \in \N $
let
$ f_d \in C( [ 0, T ] \times \R^d \times \R, \R ) $,
$ g_d \in C( \R^d, \R ) $
satisfy for all
$ t \in [ 0, T ] $,
$ x \in \R^d $,
$ v, w \in \R $
that
$ \max\{ \lvert f_d( t, x, 0 ) \rvert, \lvert g_d( x ) \rvert \}
\leq
L \max\{ 1, \lVert x \rVert^p_{ \R^d } \} $
and
$ \lvert f_d( t, x, v ) - f_d( t, x, w ) \rvert
\leq
L \lvert v - w \rvert $,
let
$ ( \Omega, \mathscr{F}, \P ) $
be a probability space,
let
$ U^\theta \colon \Omega \to [ 0, 1 ] $, $ \theta \in \Theta $,
be independent on $ [ 0, 1 ] $ uniformly distributed random variables,
for every $ d \in \N $
let
$ W^{ d, \theta } \colon [ 0, T ] \times \Omega \to \R^d $, $ \theta \in \Theta $,
be independent standard Brownian motions,
assume
for every $ d \in \N $
that
$ ( U^\theta )_ { \theta \in \Theta } $
and
$ ( W^{ d, \theta } )_ { \theta \in \Theta } $
are independent,
let
$ Y_{ n, j }^{ d, \theta } \colon [ 0, T ] \times \R^d \times \Omega \to \R $,
$ \theta \in \Theta $,
$ d, j \in \N $,
$ n \in ( \N_0 \cup \{ -1 \} ) $,
satisfy
for all
$ n, j, d \in \N $,
$ \theta \in \Theta $,
$ t \in [ 0, T ] $,
$ x \in \R^d $
that
$ Y_{ -1, j }^{ d, \theta }( t, x ) = Y_{ 0, j }^{ d, \theta }( t, x ) = 0 $
and
\begin{align*}
\label{eq:def_Y_heat}
& Y_{ n, j }^{ d, \theta }( T - t, x )
=
\tfrac{1}{ ( M_j )^{ n } }
\Biggl[
\sum_{ i = 1 }^{ (M_j )^n }
    g_d\bigl( x + W^{ d, ( \theta, 0, i ) }_{ t } \bigr)
\Biggr]
+
\sum_{l=0}^{n-1}
\tfrac{t}{ ( M_j )^{ n - l } }
\Biggl[
    \sum_{i=1}^{ ( M_j )^{ n - l } }
\\ & \yesnumber
\quad
    \Bigl[
        f_d\Bigl(
            T - t + U^{ ( \theta, l, i ) } t,
            x + W^{ d, ( \theta, l, i ) }_{ U^{ ( \theta, l, i ) } t },
            Y^{ d, ( \theta, l, i ) }_{ l, j }\bigl(
                T - t + U^{ ( \theta, l, i ) } t,
                x + W^{ d, ( \theta, l, i ) }_{ U^{ ( \theta, l, i ) } t }
            \bigr)
        \Bigr)
\\ &
\quad
        -
        \mathbbm{1}_{ \N }( l )
        f_d\Bigl(
            T - t + U^{ d, ( \theta, l, i ) } t,
            x + W^{ d, ( \theta, l, i ) }_{ U^{ ( \theta, l, i ) } t },
            Y^{ d, ( \theta, -l, i ) }_{ l - 1, j }\bigl(
                T - t + U^{ ( \theta, l, i ) } t,
                x + W^{ d, ( \theta, l, i ) }_{ U^{ ( \theta, l, i ) } t }
            \bigr)
        \Bigr)
    \Bigr]
\Biggr]
,
\end{align*}
and
let
$ ( \Cost_{ d, n, j } )_{ ( d, n, j ) \in \N \times ( \N_0 \cup \{ -1 \} ) \times \N }
\subseteq \N_0 $
satisfy
for all $ d, n, j \in \N $
that
$ \Cost_{ d, -1, j } = \Cost_{ d, 0, j } = 0 $
and
\begin{equation}
\Cost_{ d, n, j }
\leq
( M_j )^{ n }
d
+
\sum_{ l = 0 }^{ n - 1 }
    \bigl[
    ( M_j )^{ n - l }
    ( \Cost_{ d, l, j } + \Cost_{ d, l - 1, j } + d + 1 )
    \bigr]
.
\end{equation}
Then
\begin{enumerate}[(i)]
\item
\label{item:thm:MLP_heat1}
for every $ d \in \N $
there exists a unique at most polynomially growing
viscosity solution
$ y_d \in C( [ 0, T ] \times \R^d, \R ) $
of
\begin{equation}
\bigl( \tfrac{ \partial y_d }{ \partial t } \bigr)( t, x )
+
\tfrac{1}{2}
( \Delta_x y_d )( t, x )
+
f_d( t, x, y_d( t, x ) )
= 0
\end{equation}
with
$ y_d( T, x ) = g_d( x ) $
for $ ( t, x ) \in ( 0, T ) \times \R^d $
and
\item
\label{item:thm:MLP_heat2}
there exists
$ ( N_{ d, \varepsilon } )_{ ( d, \varepsilon ) \in \N \times ( 0, 1 ] } \subseteq \N $
such that
it holds
for all
$ d \in \N $,
$ \varepsilon \in ( 0, 1 ] $,
$ \delta \in ( 0, \infty ) $
that
$ \sup_{ n \in \{ N_{ d, \varepsilon }, N_{ d, \varepsilon } + 1, \ldots \} }
\bigl(
\E\bigl[
    \lvert Y_{ n, n }^{ d, 0 }( 0, \xi_d ) - y_d( 0, \xi_d ) \rvert^2
\bigr]
\bigr)^{ \nicefrac{1}{2} }
\leq
\varepsilon $
and
\begin{equation}
\begin{split}
\Cost_{ d, N_{ d, \varepsilon }, N_{ d, \varepsilon } }
& \leq
\Bigl[
[ 4^{ p + 2 }
\max\{ L, 1 \}
( 1 + T )^{ \nicefrac{p}{2} + 1 }
e^{ LT }
( \max\{ K, p, 1 \} )^p
]^{ 2 ( 1 + \delta ) }
e^\kappa
\\ & \quad \cdot
\Bigl(
    1
    +
    \sup\nolimits_{ n \in \N }
    \Bigl[
        \tfrac{
        [ 5 \mathfrak{B} e^{ 2 \kappa } ( 1 + ( 2 L T )^2 ) ]^{ ( 1 + \delta ) }
        }{
        ( M_{ n } )^{ \delta }
        }
    \Bigr]^n
\Bigr)
\Bigr]
d^{ 1 + p( 1 + \delta ) }
\varepsilon^{ -2 ( 1 + \delta ) }
<
\infty
.
\end{split}
\end{equation}
\end{enumerate}
\end{theorem}
\begin{proof}[Proof of Theorem~\ref{thm:MLP_heat}]
Throughout this proof
assume w.l.o.g.\
for every $ d \in \N $
that
$ W^{ d, \theta } \colon [ 0, T ] \times \Omega \to \R^d $, $ \theta \in \Theta $,
are independent standard Brownian motions with continuous sample paths
(cf., e.g., Klenke~\cite[Definition 21.8]{Klenke2014})
and
throughout this proof
let
$ C_d \in [ 1, \infty ) $, $ d \in \N $,
be the real numbers which satisfy
for all
$ d \in \N $
that
\begin{equation}
\label{eq:definition_C_MLP_heat}
C_d
=
\max\Bigl\{
1,
e^{ LT }
\Bigl[
    \bigl( \E \bigl[
    \lvert g_d( \xi_d + W_{ T }^{ d, 0 } ) \rvert^2
    \bigr] \bigr)^{ \nicefrac{1}{2} }
    +
    \sqrt{T}
    \bigl(
    \textstyle\int_{0}^{T}
        \E \bigl[
        \lvert f_d( t, \xi_d + W_{ t }^{ d, 0 }, 0 ) \rvert^2
        \bigr]
    \ud t
    \bigr)^{ \nicefrac{1}{2} }
\Bigr]
\Bigr\}
.
\end{equation}
First of all,
observe that the
Burkholder--Davis--Gundy-type inequality in
Da~Prato \& Zabczyk~\cite[Lemma~7.7]{DaPratoZabczyk1992}
establishes
for all
$ r \in [ 2, \infty ) $,
$ d \in \N $,
$ t \in [ 0, T ] $
that
\begin{equation}
\begin{split}
\bigl( \E\bigl[
    \lVert W_t^{ d, 0 } \rVert_{ \R^d }^r
\bigr] \bigr)^{ \nicefrac{1}{r} }
\leq
\sqrt{ \tfrac{1}{2} r ( r - 1 ) t d }
\leq
r \sqrt{ \tfrac{ T d }{2} }
.
\end{split}
\end{equation}
Jensen's inequality
and
the fact that
$ \forall \, a, b, r \in [ 0, \infty ) \colon
( a + b )^r \leq 2^{ \max\{ r - 1, 0 \} }( a^r + b^r ) $
hence
prove
for all
$ d \in \N $
that
\begin{equation}
\begin{split}
&
\bigl( \E \bigl[
\lvert g_d( \xi_d + W_{ T }^{ d, 0 } ) \rvert^2
\bigr] \bigr)^{ \nicefrac{1}{2} }
+
\sqrt{T}
\bigl(
\textstyle\int_{0}^{T}
    \E \bigl[
    \lvert f_d( t, \xi_d + W_{ t }^{ d, 0 }, 0 ) \rvert^2
    \bigr]
\ud t
\bigr)^{ \nicefrac{1}{2} }
\\ &
\leq
L
\bigl( \E \bigl[
    \max\bigl\{ 1, \lVert \xi_d + W_{ T }^{ d, 0 } \rVert_{ \R^d }^{ 2p } \bigr\}
\bigr] \bigr)^{ \nicefrac{1}{2} }
+
L
\sqrt{T}
\bigl(
\textstyle\int_{0}^{T}
    \E \bigl[
    \max\bigl\{ 1, \lVert \xi_d + W_{ t }^{ d, 0 } \rVert_{ \R^d }^{ 2p } \bigr\}
    \bigr]
\ud t
\bigr)^{ \nicefrac{1}{2} }
\\ &
\leq
L
\Bigl(
1 +
\bigl( \E \bigl[
    \lVert \xi_d + W_{ T }^{ d, 0 } \rVert_{ \R^d }^{ 2p }
\bigr] \bigr)^{ \nicefrac{1}{2} }
\Bigr)
+
L T
\bigl(
1 +
\tfrac{1}{T}
\textstyle\int_{0}^{T}
    \E \bigl[
        \lVert \xi_d + W_{ t }^{ d, 0 } \rVert_{ \R^d }^{ 2p }
    \bigr]
\ud t
\bigr)^{ \nicefrac{1}{2} }
\\ &
\leq
L( 1 + T )
+
L
\Bigl(
\bigl( \E \bigl[
    \lVert \xi_d + W_{ T }^{ d, 0 } \rVert_{ \R^d }^{ 2\max\{p,1\} }
\bigr] \bigr)^{ \frac{1}{ 2\max\{p,1\} } }
\Bigr)^p
\\ & \quad
+
L T
\Bigl(
\bigl(
\tfrac{1}{T}
\textstyle\int_{0}^{T}
    \E \bigl[
        \lVert \xi_d + W_{ t }^{ d, 0 } \rVert_{ \R^d }^{ 2\max\{p,1\} }
    \bigr]
\ud t
\bigr)^{ \frac{1}{ 2\max\{p,1\} } }
\Bigr)^p
\\ &
\leq
L( 1 + T )
+
L
\Bigl(
\lVert \xi_d \rVert_{ \R^d }
+
\bigl( \E \bigl[
    \lVert W_{ T }^{ d, 0 } \rVert_{ \R^d }^{ 2\max\{p,1\} }
\bigr] \bigr)^{ \frac{1}{ 2\max\{p,1\} } }
\Bigr)^p
\\ & \quad
+
L T
\Bigl(
\lVert \xi_d \rVert_{ \R^d }
+
\bigl(
\tfrac{1}{T}
\textstyle\int_{0}^{T}
    \E \bigl[
        \lVert W_{ t }^{ d, 0 } \rVert_{ \R^d }^{ 2\max\{p,1\} }
    \bigr]
\ud t
\bigr)^{ \frac{1}{ 2\max\{p,1\} } }
\Bigr)^p
\\ &
\leq
L( 1 + T )
+
L
\bigl(
\lVert \xi_d \rVert_{ \R^d }
+
\max\{p,1\}
\sqrt{ 2 T d }
\bigr)^p
+
L T
\bigl(
\lVert \xi_d \rVert_{ \R^d }
+
\max\{p,1\}
\sqrt{ 2 T d }
\bigr)^p
\\ &
\leq
L( 1 + T )
+
L( 1 + T )
2^{ \max\{ p - 1, 0 \} }
\bigl(
\lVert \xi_d \rVert_{ \R^d }^p
+
\max\{p^p,1\}
( 2 T d )^{ \nicefrac{p}{2} }
\bigr)
\\ &
\leq
d^{ \nicefrac{p}{2} }
2^{ \max\{ p, 1 \} + \nicefrac{p}{2} }
L( 1 + T )^{ \nicefrac{p}{2} + 1 }
\bigl(
K^p
+
\max\{p^p,1\}
\bigr)
\\ &
\leq
d^{ \nicefrac{p}{2} }
4^{ p + 1 }
L( 1 + T )^{ \nicefrac{p}{2} + 1 }
( \max\{ K, p, 1 \} )^p
.
\end{split}
\end{equation}
This and~\eqref{eq:definition_C_MLP_heat}
show
for all
$ d \in \N $,
$ \delta \in ( 0, \infty ) $
that
\begin{equation}
\label{eq:C_d_estimate}
\begin{split}
5
(C_d)^{ 2 ( 1 + \delta ) }
& \leq
5
[ d^{ \nicefrac{p}{2} }
4^{ p + 1 }
\max\{ L, 1 \}
( 1 + T )^{ \nicefrac{p}{2} + 1 }
e^{ LT }
( \max\{ K, p, 1 \} )^p
]^{ 2 ( 1 + \delta ) }
\\ &
\leq
[ 4^{ p + 2 }
\max\{ L, 1 \}
( 1 + T )^{ \nicefrac{p}{2} + 1 }
e^{ LT }
( \max\{ K, p, 1 \} )^p
]^{ 2 ( 1 + \delta ) }
d^{ p( 1 + \delta ) }
.
\end{split}
\end{equation}
Combining this with
Proposition~\ref{prop:MLP_heat_fixed}
completes the proof of Theorem~\ref{thm:MLP_heat}.
\end{proof}

\begin{corollary}
\label{cor:MLP_heat}
Let
$ T \in ( 0, \infty ) $,
$ p \in [ 0, \infty ) $,
$ \Theta = \cup_{ n = 1 }^\infty \Z^n $,
$ ( M_j )_{ j \in \N } \subseteq \N $,
$ ( \xi_d )_{ d \in \N } \subseteq \R^d $
satisfy
$ \sup_{ j \in \N }
( \nicefrac{ M_{ j + 1 } }{ M_j }
+
\nicefrac{ M_j }{ j }
+
\lVert \xi_j \rVert_{ \R^j } )
< \infty
=
\liminf_{ j \to \infty } M_j $,
let
$ f \colon \R \to \R $
be a Lipschitz continuous function,
let
$ g_d \in C( \R^d, \R ) $, $ d \in \N $,
satisfy
$ \sup_{ d \in \N, \, x \in \R^d }
\nicefrac{ \lvert g_d( x ) \rvert }{ \max\{ 1, \lVert x \rVert^p_{ \R^d } \} }
< \infty $,
let
$ ( \Omega, \mathscr{F}, \P ) $
be a probability space,
let
$ U^\theta \colon \Omega \to [ 0, 1 ] $, $ \theta \in \Theta $,
be independent on $ [ 0, 1 ] $ uniformly distributed random variables,
let
$ W^{ d, \theta } \colon [ 0, T ] \times \Omega \to \R^d $, $ \theta \in \Theta $, $ d \in \N $,
be independent standard Brownian motions,
assume
that
$ ( U^\theta )_ { \theta \in \Theta } $
and
$ ( W^{ d, \theta } )_ { ( d, \theta ) \in \N \times \Theta } $
are independent,
let
$ Y_{ n, j }^{ d, \theta } \colon [ 0, T ] \times \R^d \times \Omega \to \R $,
$ \theta \in \Theta $,
$ d, j \in \N $,
$ n \in ( \N_0 \cup \{ -1 \} ) $,
satisfy
for all
$ n, j, d \in \N $,
$ \theta \in \Theta $,
$ t \in [ 0, T ] $,
$ x \in \R^d $
that
$ Y_{ -1, j }^{ d, \theta }( t, x ) = Y_{ 0, j }^{ d, \theta }( t, x ) = 0 $
and
\begin{align}
& Y_{ n, j }^{ d, \theta }( T - t, x )
=
\sum_{l=0}^{n-1}
\tfrac{t}{ ( M_j )^{ n - l } }
\Biggl[
    \sum_{i=1}^{ ( M_j )^{ n - l } }
    \Bigl[
        f\Bigl(
            Y^{ d, ( \theta, l, i ) }_{ l, j }\bigl(
                T - t + U^{ ( \theta, l, i ) } t,
                x + W^{ d, ( \theta, l, i ) }_{ U^{ ( \theta, l, i ) } t }
            \bigr)
        \Bigr)
\\ & \quad \nonumber
        -
        \mathbbm{1}_{ \N }( l )
        f\Bigl(
            Y^{ d, ( \theta, -l, i ) }_{ l - 1, j }\bigl(
                T - t + U^{ ( \theta, l, i ) } t,
                x + W^{ d, ( \theta, l, i ) }_{ U^{ ( \theta, l, i ) } t }
            \bigr)
        \Bigr)
    \Bigr]
\Biggr]
+
\tfrac{1}{ ( M_j )^{ n } }
\Biggl[
\sum_{ i = 1 }^{ (M_j )^n }
    g_d\bigl( x + W^{ d, ( \theta, 0, i ) }_{ t } \bigr)
\Biggr]
,
\end{align}
and
let
$ ( \Cost_{ d, n, j } )_{ ( d, n, j ) \in \N \times ( \N_0 \cup \{ -1 \} ) \times \N }
\subseteq \N_0 $
satisfy
for all $ d, n, j \in \N $
that
$ \Cost_{ d, -1, j } = \Cost_{ d, 0, j } = 0 $
and
\begin{equation}
\label{eq:cost_MLP_heat}
\Cost_{ d, n, j }
\leq
( M_j )^{ n }
d
+
\sum_{ l = 0 }^{ n - 1 }
    \bigl[
    ( M_j )^{ n - l }
    ( \Cost_{ d, l, j } + \Cost_{ d, l - 1, j } + d + 1 )
    \bigr]
.
\end{equation}
Then
\begin{enumerate}[(i)]
\item
\label{item:cor:MLP_heat1}
for every $ d \in \N $
there exists a unique at most polynomially growing
viscosity solution
$ y_d \in C( [ 0, T ] \times \R^d, \R ) $
of
\begin{equation}
\bigl( \tfrac{ \partial y_d }{ \partial t } \bigr)( t, x )
+
\tfrac{1}{2}
( \Delta_x y_d )( t, x )
+
f( y_d( t, x ) )
= 0
\end{equation}
with
$ y_d( T, x ) = g_d( x ) $
for $ ( t, x ) \in ( 0, T ) \times \R^d $
and
\item
\label{item:cor:MLP_heat2}
there exist
$ ( N_{ d, \varepsilon } )_{ ( d, \varepsilon ) \in \N \times ( 0, 1 ] } \subseteq \N $
and
$ ( C_\delta )_{ \delta \in ( 0, \infty ) } \subseteq ( 0, \infty ) $
such that
it holds
for all
$ d \in \N $,
$ \varepsilon \in ( 0, 1 ] $,
$ \delta \in ( 0, \infty ) $
that
$ \Cost_{ d, N_{ d, \varepsilon }, N_{ d, \varepsilon } }
\leq
C_\delta
\:\! d^{ 1 + p( 1 + \delta ) }
\varepsilon^{ -2 ( 1 + \delta ) } $
and
\begin{equation}
\sup_{ n \in \{ N_{ d, \varepsilon }, N_{ d, \varepsilon } + 1, \ldots \} }
\bigl(
\E\bigl[
    \lvert Y_{ n, n }^{ d, 0 }( 0, \xi_d ) - y_d( 0, \xi_d ) \rvert^2
\bigr]
\bigr)^{ \nicefrac{1}{2} }
\leq
\varepsilon
.
\end{equation}
\end{enumerate}
\end{corollary}

\section*{Acknowledgements}

This project has been partially supported through the ETH Research Grant \mbox{ETH-47 15-2}
``Mild stochastic calculus and numerical approximations for nonlinear stochastic evolution equations with L\'evy noise''
and
by the project
``Deep artificial neural network approximations for stochastic partial differential equations: Algorithms and convergence proofs''
(project number 184220)
funded by the Swiss National Science Foundation (SNSF).

\printbibliography

\end{document}